\documentclass[a4paper,11pt,fleqn,oneside]{book}%

\usepackage[utf8]{inputenc} 
\usepackage[UKenglish]{babel} 

\usepackage{amsmath, amsfonts}
\usepackage{amsthm, amssymb} 

\usepackage[T1]{fontenc}

\usepackage{lmodern}

\usepackage{mathrsfs} 
\usepackage{faktor} 
\usepackage{booktabs} 

\usepackage{graphicx, subcaption} 

\usepackage[urlcolor=blue, linkcolor=blue, citecolor=blue, colorlinks=true]{hyperref} 

\usepackage[sort=use,style=long3col]{glossaries} 
\usepackage{fancyhdr} 
\usepackage[textwidth=14.5truecm,textheight=21.5truecm,top=33mm,includeheadfoot]{geometry}

\usepackage{tikz}
\usetikzlibrary{matrix, arrows, cd, patterns}

\usepackage{afterpage} 
\newcommand\blankpage{%
    \null
    \thispagestyle{empty}%
    \addtocounter{page}{-1}%
    \newpage}
\usepackage{setspace} 

\theoremstyle{plain}
\newtheorem{theorem}{Theorem}[chapter]
\newtheorem{corollary}[theorem]{Corollary}
\newtheorem{proposition}[theorem]{Proposition}
\newtheorem{lemma}[theorem]{Lemma}
\newtheorem*{remark*}{Remark*}

\theoremstyle{definition}
\newtheorem{definition}[theorem]{Definition}
\newtheorem*{notation}{Notation}

\newtheorem{examplex}[theorem]{Example} 
\newenvironment{example}
{\pushQED{\qed}\examplex}
{\popQED\endexamplex}

\newtheorem{remark}[theorem]{Remark}

\newtheoremstyle{claim}
{0pt}
{-2pt}
{\itshape}
{}
{\bfseries\slshape}
{:}
{.5em }
{ }

\theoremstyle{claim}

\newtheoremstyle{note}
{0pt}
{0pt}
{}
{\parindent}
{\bfseries\slshape}
{:}
{.5em }
{ }

\theoremstyle{note}

\newcommand{\NN}{\mathbb{N}}
\newcommand{\ZZ}{\mathbb{Z}}

\newcommand{\RR}{\mathbb{R}}
\newcommand{\CC}{\mathbb{C}}

\newcommand{\KK}{\mathbb{K}} 

\renewcommand{\AA}{\mathbb{A}} 
\newcommand{\PP}{\mathbb{P}} 

\newcommand{\m}{\mathfrak{m}} 
\newcommand{\p}{\mathfrak{p}} 
\newcommand{\n}{\mathfrak{n}} 

\newcommand{\F}{\mathscr{F}} 
\newcommand{\G}{\mathscr{G}} 
\renewcommand{\O}{\mathcal{O}} 
\newcommand{\C}{\mathcal{C}} 
\renewcommand{\H}{\mathrm{H}} 
\newcommand{\vC}{\check{\mathcal{C}}} 
\newcommand{\vH}{\check{\mathrm{H}}} 

\newcommand{\R}{\mathcal{R}} 
\newcommand{\U}{\mathscr{U}} 

\newcommand{\midd}{\;\middle\vert\;} 

\DeclareMathOperator{\Cl}{Cl}
\DeclareMathOperator{\Div}{Div}
\DeclareMathOperator{\Prin}{Prin}
\DeclareMathOperator{\CaDiv}{CaDiv}
\DeclareMathOperator{\CaPrin}{CaPrin}
\DeclareMathOperator{\CaCl}{CaCl}

\DeclareMathOperator{\Spec}{Spec}

\DeclareMathOperator{\htnew}{ht}

\DeclareMathOperator{\im}{im}
\DeclareMathOperator{\supp}{supp}

\DeclareMathOperator{\rk}{rk}

\DeclareMathOperator{\Pic}{Pic}

\DeclareMathOperator{\Hom}{Hom}
\DeclareMathOperator{\lk}{lk}

\DeclareMathOperator{\red}{red}
\DeclareMathOperator{\tf}{tf}
\DeclareMathOperator{\nil}{nil}
\DeclareMathOperator{\N}{\mathcal{N}}
\DeclareMathOperator{\loc}{loc}
\DeclareMathOperator{\nerve}{nerve}
\DeclareMathOperator{\comb}{comb}
\DeclareMathOperator{\Top}{Top}
\DeclareMathOperator{\ccc}{CCC}

\makeindex

\definecolor{grey1}{rgb}{0.3, 0.3, 0.3}
\definecolor{grey2}{rgb}{0.5, 0.5, 0.5}
\definecolor{grey3}{rgb}{0.7, 0.7, 0.7}

\usepackage{xparse}
\AtBeginDocument{
    \let\oldchapter\chapter
    \RenewDocumentCommand{\chapter}{s o m}{%
        \clearpage
        \IfBooleanTF{#1}
        {\oldchapter*{#3}}
        {\IfValueTF{#2}
            {\oldchapter[#2]{#3}}
            {\oldchapter{#3}}
            \label{chapter-\thechapter}
        }%
    }
}

\makeatletter
\def\extract@chapter#1.#2\@nil{#1}%
\newcommand{\chapref}[1]{{%
        \edef\x{\noexpand\edef\noexpand\chapnum{\noexpand\extract@chapter\getrefnumber{#1}\noexpand\@nil}}\x%
        \ref{chapter-\chapnum}%
    }}
    \makeatother

    \usepackage[bottom, marginal, multiple]{footmisc} 
    \usepackage{multicol}
    \usepackage{mathtools}
    
    \usepackage{stackengine} 

    \newcommand{\bigcupdot}{\mathop{\ensurestackMath{\stackinset{c}{}{c}{+.25ex}{\cdot}{\bigcup}}}}
    \newcommand{\cupdot}{\mathop{\ensurestackMath{\stackinset{c}{}{c}{+.25ex}{\cdot}{\cup}}}}

    \definecolor{uos-red-full}{cmyk}{0.2, 1, 0.65, 0.25}
    \definecolor{uos-yellow-full}{cmyk}{0, 0.3, 1, 0}
    \newcommand{\uoslogo}[2]{\resizebox{#1}{!}{
            \begin{tikzpicture}

            \path[fill=#2] (45.6990,36.6582) .. controls
            (45.6990,36.6582) and (35.2160,42.1395) .. (35.2160,42.1395) --
            (35.2210,52.7102) -- (45.6990,52.7102) -- (45.6990,36.6582) --
            cycle(34.7270,53.7305) -- (29.6540,56.3809) -- (39.3020,56.3809) ..
            controls (39.3020,56.3809) and (44.2270,53.7305) .. (44.2270,53.7305) --
            (34.7270,53.7305) -- cycle(34.1960,42.6688) .. controls (34.1960,42.6688)
            and (32.3660,43.6277) .. (32.3660,43.6277) .. controls (32.3660,43.6277)
            and (32.3740,48.6105) .. (32.3740,48.6105) .. controls (32.3740,50.4004)
            and (30.9040,51.8574) .. (29.0990,51.8574) .. controls (27.2920,51.8574)
            and (25.8230,50.4004) .. (25.8230,48.6105) .. controls (25.8230,48.6105)
            and (25.8230,43.6633) .. (25.8230,43.6633) .. controls (25.8230,43.6633)
            and (24.0000,42.7156) .. (24.0000,42.7156) -- (24.0000,52.8645) .. controls
            (24.0000,52.8645) and (29.0980,55.5188) .. (29.0980,55.5188) .. controls
            (29.0980,55.5188) and (34.1960,52.8555) .. (34.1960,52.8555) --
            (34.1960,42.6688) -- cycle(23.4550,53.7305) -- (13.8977,53.7305) .. controls
            (13.8977,53.7305) and (18.8059,56.3809) .. (18.8059,56.3809) --
            (28.5420,56.3809) -- (23.4550,53.7305) -- cycle(22.9790,42.1730) ..
            controls (22.9790,42.1730) and (12.4973,36.6602) .. (12.4973,36.6602) --
            (12.4973,52.7102) -- (22.9790,52.7102) -- (22.9790,42.1730) --
            cycle(39.5590,57.4023) -- (18.5480,57.4023) -- (11.4758,53.5891) --
            (11.4758,34.9695) -- (26.8440,43.0422) -- (26.8440,48.6105) .. controls
            (26.8440,49.8379) and (27.8560,50.8352) .. (29.0990,50.8352) .. controls
            (30.3400,50.8352) and (31.3520,49.8379) .. (31.3520,48.6105) .. controls
            (31.3520,48.6086) and (31.3430,43.0105) .. (31.3430,43.0105) --
            (46.7200,34.9719) -- (46.7200,53.5480) -- (39.5590,57.4023) --
            cycle(29.1360,54.0652) .. controls (29.4320,54.0652) and (29.6720,53.8312)
            .. (29.6720,53.5398) .. controls (29.6720,53.2512) and (29.4320,53.0145) ..
            (29.1360,53.0145) .. controls (28.8410,53.0145) and (28.6000,53.2512) ..
            (28.6000,53.5398) .. controls (28.6000,53.8312) and (28.8410,54.0652) ..
            (29.1360,54.0652) -- cycle(27.6810,53.5398) .. controls (27.6810,52.7434)
            and (28.3340,52.0945) .. (29.1360,52.0945) .. controls (29.9390,52.0945)
            and (30.5910,52.7434) .. (30.5910,53.5398) .. controls (30.5910,54.3367)
            and (29.9390,54.9855) .. (29.1360,54.9855) .. controls (28.3340,54.9855)
            and (27.6810,54.3367) .. (27.6810,53.5398) -- cycle(35.2160,39.1730) ..
            controls (35.2160,39.1730) and (45.6970,33.6926) .. (45.6970,33.6926) --
            (45.6970,20.4734) .. controls (45.6970,20.4734) and (29.0990,11.7832) ..
            (29.0990,11.7832) .. controls (29.0990,11.7832) and (12.4992,20.4734) ..
            (12.4992,20.4734) -- (12.4992,33.6934) .. controls (12.4992,33.6934) and
            (22.9770,39.1730) .. (22.9770,39.1730) .. controls (22.9770,39.1730) and
            (22.9770,26.0566) .. (22.9770,26.0566) -- (29.0990,22.8137) --
            (35.2160,26.0574) .. controls (35.2160,26.0574) and (35.2160,39.1730) ..
            (35.2160,39.1730) -- cycle(34.1960,26.6719) .. controls (34.1960,26.6719)
            and (29.0980,23.9688) .. (29.0980,23.9688) .. controls (29.0980,23.9688)
            and (23.9980,26.6719) .. (23.9980,26.6719) .. controls (23.9980,26.6719)
            and (23.9980,40.8602) .. (23.9980,40.8602) -- (11.4777,34.3113) --
            (11.4777,19.8547) -- (29.0990,10.6309) -- (46.7170,19.8547) --
            (46.7170,34.3113) -- (34.1960,40.8602);

            \end{tikzpicture}
        }}

\newglossaryentry{M-bullet}{name={\ensuremath{M^\bullet}},
    description={$M\smallsetminus\{\infty\}$}
}

\newglossaryentry{infty-binoid}{name={\ensuremath{\infty}},
    description={degenerate trivial binoid}
}

\newglossaryentry{zero-binoid}{name={\ensuremath{\{0, \infty\}}},
    description={smallest non trivial binoid, also called zero binoid}
}

\newglossaryentry{difference-binoid}{name={\ensuremath{\Gamma}},
    description={difference binoid of a cancellative binoid $M$, $\Gamma=(-M^\bullet+M)$}
}

\newglossaryentry{difference-group}{name={\ensuremath{\Gamma^\bullet}},
    description={difference group of a cancellative binoid, $\Gamma^\bullet=\Gamma\smallsetminus\{\infty\}$}
}

\newglossaryentry{supp}{name={\ensuremath{\supp}},
    description={support of an element of a semifree binoid. If $\{a_i\}_{i\in I}$ is the semibasis of $M$ and $f=\sum_{i\in I} n_ia_i$ then $\supp f=\{a_i\mid n_i\neq 0\}$}
}

\newglossaryentry{red}{name={\ensuremath{\red}},
    description={reduction operation in a semifree binoid. If $\{a_i\}_{i\in I}$ is the semibasis of $M$ and $f=\sum_{i\in I} n_ia_i$ then $\red f=\sum_{n_i\neq 0} a_i$}
}

\newglossaryentry{nilM}{name={\ensuremath{\nil(M)}},
    description={ideal of nilpotent elements in a binoid $M$}
}

\newglossaryentry{Mred}{name={\ensuremath{M_{\red}}},
    description={reduction of $M$, $M_{\red}=\faktor{M}{\nil(M)}$}
}

\newglossaryentry{Mtf}{name={\ensuremath{M_{\mathrm{tf}}}},
    description={torsion-freeification of $M$}
}

\newglossaryentry{genset}{name={\ensuremath{\mathcal{G}}},
    description={generating set of the binoid $M=(\mathcal{G}\mid \mathcal{R})$}
}

\newglossaryentry{relset}{name={\ensuremath{\mathcal{R}}},
    description={set of relations of the binoid $M=(\mathcal{G}\mid \mathcal{R})$}
}

\newglossaryentry{primeideal}{name={\ensuremath{\p}},
    description={prime ideal in $M$, i.e.\ a non-empty $M$-subset such that $M\smallsetminus\p$ is a monoid}
}

\newglossaryentry{M-plus}{name={\ensuremath{M_+}},
    description={unique maximal ideal of the binoid $M$}
}

\newglossaryentry{height-prime}{name={\ensuremath{\htnew(\p)}},
    description={height of the prime ideal $\p$ of the binoid $M$}
}

\newglossaryentry{Mstar}{name={\ensuremath{M^*}},
    description={group of invertible elements of $M$}
}

\newglossaryentry{localization-element}{name={\ensuremath{M_f}},
    description={localization of the binoid $M$ w.r.t.\ an element $f$, i.e.\ $M_f=-S+M$ with $S=\{nf\mid n\in\NN\}$}
}

\newglossaryentry{localization-prime}{name={\ensuremath{M_{\p}}},
    description={localization of a binoid at a prime ideal, i.e.\ $M_\p=-S+M$ with $S=(M\smallsetminus\p)$}
}

\newglossaryentry{localization-Mset-element}{name={\ensuremath{S_f}},
    description={localization of the $M$-set $S$ w.r.t.\ an element $f\in M$, i.e.\ $S_f=S+M_f$}
}

\newglossaryentry{localization-Mset-prime}{name={\ensuremath{S_{\p}}},
    description={localization of the $M$-set $S$ at a prime ideal of $M$, i.e.\ $S_\p=S+M_\p$}
}

\newglossaryentry{simplicial-complex}{name={\ensuremath{\triangle}},
    description={simplicial complex, a finite subset-closed collection of finite sets}
}

\newglossaryentry{simplicial-complex-restriction}{name={\ensuremath{\triangle_W}},
    description={restriction of a simplicial complex on vertex set $V$ to a subset of vertices $W\subseteq V$}
}

\newglossaryentry{link}{name={\ensuremath{\lk_\triangle(F)}},
    description={link of the face $F$ in $\triangle$, $\lk_\triangle(F)=\{G\in\triangle\mid G\cap F=\varnothing, G\cap F\in\triangle\}$}
}

\newglossaryentry{crosscut-complex}{name={\ensuremath{\ccc(\{G_i\}_{i\in I}, \triangle)}},
    description={the crosscut complex of the set of faces $\{G_i\}\subseteq \triangle$ in $\triangle$}
}

\newglossaryentry{triangle-j}{name={\ensuremath{\triangle_j}},
    description={set of faces of dimension $j$ in the simplicial complex $\triangle$}
}

\newglossaryentry{simplicial-binoid}{name={\ensuremath{M_\triangle}},
    description={simplicial binoid associated to the simplicial complex $\triangle$ on vertex set $V$, $M_\triangle=\left(\{x_i\}_{i\in V}\mid x_{i_1}+\dots x_{i_k}=\infty, \forall \{i_1, \dots, i_k\}\notin \triangle\right)$}
}

\newglossaryentry{simplicial-binoid-empty-set}{name={\ensuremath{M_{\{\varnothing\}}}},
    description={simplicial binoid of the degenerate simplicial complex, $M_{\{\varnothing\}}:=\{0, \infty\}$}
}

\newglossaryentry{binoinded-space}{name={\ensuremath{(X, \O_X)}},
    description={binoided space, a topological space equipped with a sheaf of binoids}
}

\newglossaryentry{structure-sheaf}{name={\ensuremath{\O_X}},
    description={structure sheaf of the binoided space $(X, \O_X)$}
}

\newglossaryentry{spectrum}{name={\ensuremath{\Spec M}},
    description={spectrum of the binoid $M$, i.e.\ the set of its prime ideals}
}

\newglossaryentry{punctured-spectrum}{name={\ensuremath{\Spec^\bullet M}},
    description={punctured spectrum of $M$, i.e.\ $\Spec M\smallsetminus\{M_+\}$}
}

\newglossaryentry{punctured-spectrum-KM}{name={\ensuremath{\Spec^\bullet \KK[M]}},
    description={punctured spectrum of $\KK[M]$, i.e.\ $\Spec \KK[M]\smallsetminus\{\KK[M_+]\}$}
}

\newglossaryentry{fundamental-open-subset}{name={\ensuremath{D(f)}},
    description={open subset of $\Spec M$ defined as $D(f)=\{\p\in M\mid f\notin \p\}$}
}

\newglossaryentry{structure-sheaf-binoid}{name={\ensuremath{\O_M}},
    description={another name for $\O_{\Spec M}$}
}

\newglossaryentry{structure-sheaf-spec-binoid}{name={\ensuremath{\O_{\Spec M}}},
    description={structure sheaf of the affine scheme $(\Spec M, \O_{\Spec M})$}
}

\newglossaryentry{D-of-face}{name={\ensuremath{D(F)}},
    description={open subset of $\Spec M_\triangle$ defined by the face $F$, $D(F)=\cap_{i\in F}D(x_i)$}
}

\newglossaryentry{xF}{name={\ensuremath{x_F}},
    description={set of variables of $M_\triangle$ associated to face $F\in\triangle$. $x_F := \{x_i\mid i\in F\}$}
}

\newglossaryentry{TopX}{name={\ensuremath{\Top_{X}}},
    description={category of open subsets of a topological space $X$, where the maps are either the inclusions or nothing}
}

\newglossaryentry{Bin}{name={\ensuremath{\mathrm{Bin}}},
    description={category of binoids, where the maps are binoid morphisms}
}

\newglossaryentry{Ab}{name={\ensuremath{\mathrm{Ab}}},
    description={category of abelian groups, where the maps are group morphisms}
}

\newglossaryentry{KSpecM}{name={\ensuremath{\KK-\Spec M}},
    description={$\KK$-spectrum of the binoid $M$, $\KK-\Spec M=\Hom_{\mathrm{Bin}}(M, \KK)$}
}

\newglossaryentry{affinespace}{name={\ensuremath{\AA^n}},
    description={affine space. The same symbol is used both in combinatorial and algebraic context}
}

\newglossaryentry{sheafificationMset}{name={\ensuremath{\widetilde{S}}},
    description={sheafification of the $M$-set $S$}
}

\newglossaryentry{PicX}{name={\ensuremath{\Pic(X)}},
    description={Picard group, group of line bundles, or invertible $\O_X$-sheaves, on the scheme $X$}
}

\newglossaryentry{PiclocM}{name={\ensuremath{\Pic^{\loc}(M)}},
    description={local Picard group of $M$, $\Pic^{\loc}(M)=\Pic(\Spec^\bullet M)$}
}

\newglossaryentry{PiclocKM}{name={\ensuremath{\Pic^{\loc}(\KK[M])}},
    description={local Picard group of $\KK[M]$, $\Pic^{\loc}(\KK[M])=\Pic(\Spec^\bullet \KK[M])$}
}

\newglossaryentry{sheaf-units-X}{name={\ensuremath{\O_X^*}},
    description={sheaf of units of the binoid scheme $(X, \O_X)$}
}

\newglossaryentry{sheaf-units-M}{name={\ensuremath{\O_M^*}},
    description={sheaf of units of the binoid scheme $(\Spec M, \O_{\Spec M})$}
}

\newglossaryentry{cohomologyXF}{name={\ensuremath{\H^p(X, \F)}},
    description={sheaf cohomology of the sheaf $\F$ on the topological space $X$ in degree $p$}
}

\newglossaryentry{cech-chain-cpx}{name={\ensuremath{\vC^\bullet(\{U_i\}, \F)}},
    description={chain complex to compute \v{C}ech cohomology of the (pre)sheaf $\F$ w.r.t.\ the covering $\{U_i\}$}
}

\newglossaryentry{cechcohomologyUF}{name={\ensuremath{\vH^p(\{U_i\}, \F)}},
    description={\v{C}ech cohomology of degree $p$ of the (pre)sheaf $\F$ w.r.t.\ the covering $\{U_i\}$}
}

\newglossaryentry{cech-picard-U}{name={\ensuremath{\vC(\{U_i\}, \O_X^*)}},
    description={\v{C}ech-Picard complex to compute the cohomology of $\O^*_X$ w.r.t.\ the covering $\{U_i\}$}
}

\newglossaryentry{simplicial-chain-cpx}{name={\ensuremath{\C^\bullet(\triangle, G)}},
    description={chain complex to compute simplicial cohomology of $\triangle$ with coefficients in the abelian group $G$}
}

\newglossaryentry{simplicial-cohomology}{name={\ensuremath{\H^p(\triangle, G)}},
    description={simplicial cohomology of degree $p$ of $\triangle$ with coefficients in the abelian group $G$}
}

\newglossaryentry{reduced-simplicial-cohomology}{name={\ensuremath{\widetilde{\H}^p(\triangle, G)}},
    description={reduced simplicial cohomology of degree $p$ of $\triangle$ with coefficients in the abelian group $G$}
}

\newglossaryentry{X-height-k}{name={\ensuremath{X^{(k)}}},
    description={set of points in $X$ of height $k$}
}

\newglossaryentry{cartierdivisor}{name={\ensuremath{\{U_i, \gamma_i\}_{i\in I}}},
    description={representation of a Cartier divisor of $X$ on the covering $\{U_i\}_{i\in I}$ of $X$}
}

\newglossaryentry{cartierdivisorsU}{name={\ensuremath{\CaDiv(U)}},
    description={group of Cartier divisors of $U$}
}

\newglossaryentry{principalcartierdivisorsU}{name={\ensuremath{\CaPrin(U)}},
    description={subgroup of principal Cartier divisors of $U$}
}

\newglossaryentry{cartierdivisorsclassU}{name={\ensuremath{\CaCl(U)}},
    description={class group of Cartier divisors of $U$}
}

\newglossaryentry{weildivisorsM}{name={\ensuremath{\Div(M)}},
    description={group of Weil divisors of $M$}
}

\newglossaryentry{principalweildivisorsM}{name={\ensuremath{\Prin(M)}},
    description={subgroup of principal Weil divisors of $M$}
}

\newglossaryentry{weildivisorsclassM}{name={\ensuremath{\Cl(M)}},
    description={class group of Weil divisors of $M$}
}

\newglossaryentry{weildivisorsV}{name={\ensuremath{\Div(V)}},
    description={group of Weil divisors of $V$}
}

\newglossaryentry{principalweildivisorsV}{name={\ensuremath{\Prin(V)}},
    description={subgroup of principal Weil divisors of $V$}
}

\newglossaryentry{weildivisorsclassV}{name={\ensuremath{\Cl(V)}},
    description={class group of Weil divisors of $V$}
}

\newglossaryentry{nerve-U}{name={\ensuremath{\nerve(\{U_i\}_{i\in I})}},
    description={nerve of the finite collection of open subsets $\{U_i\}$ of a topological space $X$}
}

\newglossaryentry{extension-zero-F}{name={\ensuremath{j_!\F}},
    description={extension of a sheaf by zero outside of $U$ open subset of $X$ with inclusion map $U\stackrel{j}{\hookrightarrow}X$}
}

\newglossaryentry{binoid-algebra}{name={\ensuremath{\KK[M]}},
    description={$\KK$-algebra of the binoid $M$}
}

\newglossaryentry{SRring}{name={\ensuremath{\KK[\triangle]}},
    description={$\KK$-Stanley-Reisner algebra of the simplicial complex $\triangle$}
}

\newglossaryentry{cadivX}{name={\ensuremath{\mathcal{C}\mathrm{aDiv}(X)}},
    description={sheaf of Cartier divisors of the space $X$}
}

\newglossaryentry{wdivX}{name={\ensuremath{\mathcal{W}\mathrm{Div}(X)}},
    description={sheaf of Weil divisors of the space $X$}
}

\newglossaryentry{closure-face}{name={\ensuremath{\mathcal{P}(F)}},
    description={subset-closure of a face $F\in\triangle$}
}

\newglossaryentry{functor-KK}{name={\ensuremath{\KK[\hspace{1em}]}},
    description={functor $\mathrm{Bin}\longrightarrow \mathrm{Rings}$ that associates to a binoid its binoid $\KK$-algebra. It induces functors between $M$-sets and $\KK[M]$-modules, between topologies and between sheaves}
}

\makeglossaries

\def\lq{\scalebox{-1}[1]{''}}

\begin{document}

\frontmatter
\pagestyle{empty}

\begin{titlepage} 
    
    \vspace{5mm}
    \begin{center}
        {{\LARGE \bf{\textsc{Universit\"at Osnabr\"uck}}}} \\
        \vspace{3mm}
        {\small Dissertation \\
            zur Erlangung des Doktorgrades (Dr. rer. nat.) \\
            des Fachbereichs Mathematik/Informatik \\
            der Universit\"at Osnabr\"uck
        }
    \end{center}
    \vspace{25mm}
    \begin{center}
        {\huge{\bf Local Picard Group of Binoids\\and Their Algebras}}
    \end{center}
    \vspace{40mm}
    \par
    \noindent
    \begin{minipage}[t]{0.49\textwidth}
        {\large{\bf Candidate:\\
                Davide Alberelli}}
    \end{minipage}
    \hfill
    \begin{minipage}[t]{0.49\textwidth}\raggedleft
        {\large{\bf Advisor:\\
                Prof.\ Dr.\ rer.\ nat.\ Holger Brenner\\
            }}
        \end{minipage}
        \vspace{10mm}
        
        \noindent
        \begin{minipage}[t]{0.49\textwidth}
            {\large Osnabrück, 2016}
        \end{minipage}
        \hfill
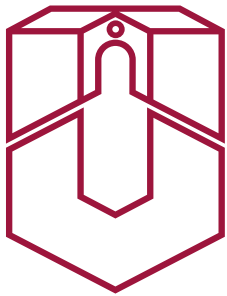
        \begin{minipage}[t]{0.49\textwidth}\raggedleft
        \end{minipage}
        \vspace{15mm}
        
        \begin{center}
            \uoslogo{.2\textwidth}{uos-red-full}
        \end{center}
        
        \afterpage{\blankpage}
        
    \end{titlepage}

    \null\vspace*{\stretch{1}}
    \begin{flushright}
        
        \lq And what do you use this bullsh*t for?''\\
        \textit{Prof.\ Dr.\ Jürgen Herzog}

    \end{flushright}
    \afterpage{\blankpage}
    \vspace*{\stretch{3}}\null

    \pagestyle{fancy}
    \fancyhf{} 
    \renewcommand{\headrulewidth}{0pt} 
    \fancypagestyle{plain}{%
        \fancyhf{}
        \renewcommand{\headrulewidth}{0pt} 
    }
    
    \setcounter{page}{0}
    \tableofcontents
    \clearpage
    
    \newgeometry{textwidth=345.0pt,textheight=598.0pt, top=40mm, includeheadfoot}
    
    \pagestyle{fancy}
    \fancyhf{} 
    \fancyhead[R]{\bfseries\thepage}
    \fancyhead[L]{\bfseries INTRODUCTION}
    \renewcommand{\headrulewidth}{0.5pt}
    \renewcommand{\footrulewidth}{0pt}
    \addtolength{\headheight}{3.5pt} 
    \fancyheadoffset{0pt}
    \fancypagestyle{plain}{%
        \fancyhead{} 
        \renewcommand{\headrulewidth}{0pt} 
    }

    \chapter*{Introduction}
    \addcontentsline{toc}{chapter}{Introduction}
    
    The interplay between commutative algebra, combinatorics and algebraic geometry has been widely exploited in the last decades, providing very deep and interesting results, as well as unlocking completely new fields of mathematics, such as toric geometry and combinatorial optimization, that found applications in areas that were classically considered very far from abstract mathematics, for example in statistics and physics.
    
    \subsection*{Motivations and history}
    
    A key role in combinatorial commutative algebra and combinatorial algebraic geometry is played by monoids and their algebras, that represent the interface between combinatorics and algebra in the playground of toric geometry.
    
    Although toric geometry is very powerful and gives many explicit results, the class of toric varieties is rather small; indeed affine toric geometry tells us mainly about ideals that are generated by binomials of the form $X^\alpha-X^\beta$.
    But many other ideals have interesting combinatorial properties, for example Stanley-Reisner ideals are related to simplicial complexes, and (strongly) stable ideals are related to posets of monomials.
    
    In this direction, many efforts have been spent on trying to adapt and extend the methods of toric geometry to other classes of varieties and ideals; for example the toric face rings and the concepts of log geometry, although from different directions and with completely different methods, go in this direction of expanding the field of applications of toric methods.
    
    The introduction of binoids (also known as pointed monoids) allows us to unify the description of toric geometry and monomial ideals, both reduced and not.
    We refer to the PhD thesis of Simone B\"ottger \cite{Boettger} for basic definitions and properties of binoids.
    
    Binoids are of interest in many active areas of mathematics, and they go often under various names. For example, Jaret Flores in his PhD thesis \cite{flores2015homological} developed homological algebra for pointed monoids, including in his work the paper \cite{flores2014picard} written with Charles Weibel.
    This links binoids with $\mathbb{F}_1$, the field with one element, and algebras over it.
    Geometry and algebra over $\mathbb{F}_1$ has been developed a lot recently; a good review of the land of $\mathbb{F}_1$ geometry is the article by Javier López Peña and Oliver Lorscheid \cite{pena2009mapping}.
    The last author introduced the concept of blueprints in \cite{lorscheid2012geometry}, in order to unify the description of binoids and their algebras and produce a common scheme theory.
    
    We are interested in a geometric invariant of the binoids and their binoid algebras, namely their Picard group, the group of line bundles defined on their spectrum.
    Although the study of the Picard group goes a while back (Steven L. Kleiman in \cite{kleiman2005picard} even says \lq The Picard scheme has roots in the 1600s''), we are interested in a more direct application to algebras and rings, for example the results on factoriality of Dedekind domains, that date back to the 19th century.
    
    More recently, the Picard group has been studied in relation to positively graded and seminormal rings; it is worth quoting the paper from M.\ Pavaman Murthy in 1969~\cite{murthy1969vector} where it is shown that for a positively graded normal ring the Picard group is trivial, and the papers from Carlo Traverso in 1970~\cite{traverso1970seminormality}, Richard G.\ Swan in 1980~\cite{swan1980seminormality} and David F.\ Anderson~\cite{anderson1981seminormal} where the seminormal case is covered, and it is shown that in this case $\Pic(A)=\Pic(A[X_1, \dots, X_n])$.
    
    These results are about rings and their affine spectra, but we will mainly look at what we call the punctured spectrum of a binoid algebra. 
    Every binoid is local, and its spectrum (the collection of its prime ideals), that is finite and can be equipped with a very natural partial order, has a unique maximal point, that we call $M_+$. Although its algebra is not always local, it contains the special ideal $\KK[M_+]$, usually referred to as the irrelevant ideal, that is often a singularity; in order to study it, we remove this point, and we call what is left the punctured spectrum of the binoid algebra, i.e.\ $\Spec^\bullet\KK[M]=\Spec\KK[M]\smallsetminus\KK[M_+]$, where $\KK[M_+]$ is the irrelevant ideal.
    
    Recall that the Picard group can be defined for any scheme $X$ as the group of line bundles on $X$ with operation the tensor product over $\O_X$, and it is denoted by $\Pic(X)$.
    We call the Picard group of the punctured spectrum of a binoid (resp.\ a binoid algebra) its local Picard group, and we denote it by $\Pic^{\loc}(M)$ (resp.\ $\Pic^{\loc}(\KK[M])$).
    One of the main goal of this project was to address the question on whether there are relations between the Picard group of the punctured spectrum of the binoid ($\Pic^{\loc}(M)$) and the Picard group of the punctured spectrum of its algebra ($\Pic^{\loc}(\KK[M])$).
    In some cases, we are able to give closed formulas for the computation of these groups, in terms of the combinatorics of some common objects underlying the binoid and its algebra, through the study of the cohomology of the sheaf of units $\O^*$. Clearly the field does not play any role for a binoid, but we see that it does when we consider $\KK[M]$.
    
    The idea of studying a singularity by removing it and looking at the remaining space, dates a while back in time and it proves to be a powerful one; for example, this is what we do when we study singular points on curves and we approximate them with tangent lines; or we could cite the famous result by Mumford in~\cite{mumford1961topology}, where he proves that a point $P$ in a normal complex projective surface $V$ is smooth if and only if $\pi_1(V\smallsetminus\{P\})=0$.
    Also, it happens often that affine varieties have trivial invariants, while by removing a point we get many of them, that give us informations on the geometric and algebraic structures of the variety we started with.
    Moreover, if $R$ is a $\mathbb{N}$-graded ring, its punctured spectrum $\Spec^\bullet R$ and its projectivization $\mathrm{Proj}\, R$ have much in common, and we can give a map from the first to the second, that we can't from the complex spectrum itself, and most of binoid algebras here considered are of this type.
    
    Since $\O^*$ is not a quasi-coherent sheaf, the well known vanishing theorems of these sheaves do not apply.
    However, in some cases there are some vanishing results for the positive cohomology of this sheaf, that depend on the algebraic property of the underlying ring that we are considering.

    In more details, it is know that the Picard group of a normal affine toric variety and the cohomology of the sheaf of units in higher degrees are trivial, so we can produce an acyclic covering of the punctured spectrum of a normal toric ring, that allows us to compute cohomology of $\O^*$ via \v{C}ech cohomology.
    We will prove that the same can be done for Stanley-Reisner rings and, moreover, we will be able to write the groups of cohomology of $\O^*$ purely in terms of some simplicial cohomology of the underlying simplicial complex.

    \subsection*{Main results}
    In this thesis we introduce some combinatorial tools referring to commutative and finitely generated binoids, in order to use them to produce some results related to their algebras.
    
    We begin by introducing schemes of binoids, invertible $\O_M$-sets and cohomology of sheaves of abelian groups defined on schemes of binoids. An interesting result in this setting is Theorem~\ref{theorem:vanishing-combinatorial-cohomology-affine}, that shows that the cohomology of a sheaf of abelian groups on an affine scheme of binoids vanishes in degree at least one, thus proving that any affine open covering of a scheme of binoids is acyclic for any sheaf of abelian groups.
    We use this result to define the so-called punctured combinatorial \v{C}ech-Picard complex in Definition~\ref{definition:cech-picard-complex}, whose first cohomology computes $\Pic^{\loc}(M)$, the local Picard group of a binoid.
    We then talk about the divisor class group of a (sufficiently nice) binoid and in Proposition~\ref{proposition:CaCl-Cl} we show isomorphisms between Weil divisor class groups and the class group of suitably defined Cartier divisors defined on the set of points of height at most one. We then extend this result in Proposition~\ref{proposition:PicV-ClV}, showing that $\Pic(V)\cong\Cl(V)$ for a subset $V$ of $\Spec M$ that is regular enough.
    
    We look then at a particular type of binoids that arise from simplicial complexes, namely simplicial binoids, whose spectrum presents very nice combinatorial properties. In Proposition~\ref{proposition:intersection-open-subsets} we show that the intersection pattern of the open subsets $D(x_i)$ is given by the faces of the simplicial complex, thus leading us to prove, in Theorem~\ref{thm:simplicial-cohomology} that the cohomology of a constant sheaf can be computed entirely in terms of simplicial cohomology.
    In Theorem~\ref{equation:localization-simpicial-binoid-wedge} we show that the localization of a simplicial binoid at a face is isomorphic to the smash product of the simplicial binoid of the link of that face and a free group on that face. This opens the door to Theorem~\ref{theorem:semifree-decomposition-of-sheaf}, where we show that we can rewrite the sheaf $\O^*_{M_\triangle}$ as a direct sum of smaller sheaves, indexed by the vertices. These sheaves are actually defined as extensions by zeros of the constant sheaf $\ZZ$ on $D(x_i)$, that is proved to be homeomorphic to the spectrum of the link of the corresponding vertex. This brings us to the main theorem of the chapter, Theorem~\ref{theorem:cohomology-simplicial-complex}, where we prove that we can compute sheaf cohomology (and thus the local Picard group) by meaning of reduced simplicial cohomology, via the formula
    \[
    {\H^j\left(\Spec^\bullet M_\triangle, \O^*_{M_\triangle}\right)}\cong\bigoplus_{i\in V}\widetilde{\H}^{j-1}\left(\lk_\triangle(i), \ZZ\right).
    \]
    We then move to the computations of $\Pic(V)$ for different open subsets $V$ of $\Spec^\bullet M_\triangle$, where we are able to rewrite the groups involved in the \v{C}ech complex in terms of crosscut complexes of faces in the simplicial complex.
    
    We then try to use the tools developed so far in order to understand binoid algebras and cohomology of their sheaf of units. In order to do so, we introduce a new topology in Definition~\ref{definition:combinatorial-topology}, namely the combinatorial topology on $\Spec\KK[M]$, that presents some interesting properties. For example, in Proposition~\ref{proposition:split-sheaves-comb-top} we show that if $M$ is a reduced, torsion-free and cancellative binoid in this topology we can decompose the sheaf of units of its algebra in a direct sum
    \[
    (\O_{\KK[M]}^*)^{\comb}\cong (\KK^*)^{\comb}\oplus (i_*\O^*_M)^{\comb},
    \]
    where $\KK^*$ is the constant sheaf, as usual.
    Since the functor between binoids and their algebra induces a continuous map between $\Spec M$ and $\Spec\KK[M]$ (with the Zariski topology) we are able to look at pushforward of sheaves along this map. Indeed, in Theorem~\ref{theorem:push-forward-exact} we are able to prove that this pushforward is exact, thanks to the combinatorial topology. This leads us to prove in Proposition~\ref{proposition:cohomology-pushforward-affine} that the Zariski cohomology of any pushforwarded sheaves vanishes on the affine spectrum $\Spec\KK[M]$.
    We talk about the non-reduced case and in Theorem~\ref{theorem:1+N-acyclic} we show that cohomology in higher degree of the sheaf of unipotent units vanishes for affine noetherian schemes, i.e.\ $\H^i(\Spec R, 1+\N)=0$ for all $i\geq 1$.
    
    We then move to the case of the Stanley-Reisner rings, for which we are able to prove a vanishing result for the affine case, in Theorem~\ref{Thm:vanishingspecstanleyreisner},
    \[
    \H^j(\KK[\triangle], \O^*)=0,
    \]
    for all $j\geq 1$. We then prove in Theorem~\ref{theorem:SR-splits-comb-top} that for a Stanley-Reisner ring the sheaf of units splits in a constant part and in the pushforward of the combinatorial units in the combinatorial topology
    \[
    \O^*_{\KK[\triangle]}=\KK^*\oplus i_*\O^*_{M_\triangle}.
    \]
    This leads us to prove Theorem~\ref{theorem:cohomology-stanley-reisner}, that states that we can compute the Zariski cohomology of the sheaf of units entirely in terms of simplicial cohomology, both usual and reduced as
    \[
    \H^j(\Spec^\bullet(\KK[\triangle]), \O^*_{\KK[\triangle]})=\H^j(\triangle, \KK^*)\oplus \bigoplus_{v\in V}\widetilde{\H}^{j-1}(\lk_\triangle(v), \ZZ),
    \]
    for $j\geq0$.
    This ultimately leads to a generalization to any monomial ideal, and we prove in Theorem~\ref{theorem:split-non-reduced-monomial} that we can compute the cohomology of the sheaf of units on the punctured spectrum of $\KK[M]$, $X$, with the Zariski topology, as simplicial cohomology, reduced simplicial cohomology and \v{C}ech cohomology on a nice covering,
    \[
    \H^j(X, \O^*_X)=\H^j(\triangle, \KK^*)\oplus \bigoplus_{v\in V}\widetilde{\H}^{j-1}(\lk_\triangle(v), \ZZ) \oplus \vH^j(\{D(X_i)\}, 1+\N)
    \]
    for any $j\geq 0$, thus yielding us some purely combinatorial non-vanishing results for $\Pic^{\loc}\KK[M]$ in Corollary~\ref{corollary:non-vanishing}.

    \newpage
    \subsection*{Outline of the work} We describe here and summarize what is in every chapter and, right after its outline, we report some of the main results that it contains, as an easy reference for the reader.
    
    \vspace{2ex}
    
    {\large\textbf{The first chapter}} of the thesis is devoted to introduce the notion of scheme of binoids, in a fashion similar to the classical description for rings.
    We introduce prime ideals, affine schemes, Zariski topology and structure sheaves.
    Then we talk about more general binoid schemes, namely schemes of finite type, and we introduce sheaves of abelian groups on them.
    We exploit the combinatorics of schemes of binoids in order to prove some vanishing results of the cohomology of these sheaves, that lead us to prove that any affine covering of a scheme of binoids can be used to compute sheaf cohomology through \v{C}ech cohomology on it.
    Since a binoid has only one maximal ideal, we introduce the notion of punctured spectrum of the binoid and we discuss its punctured \v{C}ech-Picard complex, i.e.\ the \v{C}ech complex of the sheaf of units $\O^*_M$ with respect to the covering $\{D(x_i)\}$ of $\Spec^\bullet M$.
    We conclude the chapter with a discussion about Cartier divisors, Weil divisors, and some results of isomorphism, leading us to a definition of divisor class group of a general binoid as the Picard group of a specific open subset of its spectrum.
    
    \paragraph{Results.} In Proposition~\ref{proposition:stalk-binoid-scheme} we prove that the stalk of $\O_M$ at a prime ideal $\p$ can be computed directly as the image of the sheaf on the open subset defined by the variables outside $\p$,
    \[
    \O_{M, \p}=M_\p=\O_M(D(x_{k+1}+ \dots+ x_n)),
    \]
    where $x_{k+1}, \dots, x_n\notin \p$.
    
    In Theorem~\ref{theorem:vanishing-combinatorial-cohomology-affine} we then prove for binoids a result already known for monoids, namely that
    \[
    \H^j(\Spec M, \F)=0
    \]
    for any binoid $M$, any sheaf of abelian groups $\F$ and any $j\geq 1$.
    
    In Proposition~\ref{proposition:injection-PicU-PicU'} we prove that for any two open subsets of $\Spec M$ that contain the same prime ideals of height one, there is an injection
    \[
    \begin{tikzcd}[baseline=(current  bounding  box.center)]
    \Pic(U)\rar& \Pic(U').
    \end{tikzcd}
    \]
    In Proposition~\ref{proposition:normal-dimension-2} we prove that if $M$ is normal of dimension at least 2 and regular in codimension 1 then
    \[
    \Pic^{\loc}(M\wedge \NN^\infty)=0.
    \]
    In Proposition~\ref{proposition:PicV-ClV} we extend this result and show that Picard group and divisor class group might agree on larger open subsets, i.e.\ if $M_\p$ is regular for all $\p\in V$ open subset of $\Spec M$ then
    \[
    \Pic(V)\cong \Cl(V)
    \]
    and we get, as Corollary~\ref{corollary:isolated-singularity}, that if $M$ is an isolated singularity of dimension at least 2, then $\Cl(M)\cong\Pic^{\loc}(M)$.

    \vspace{6ex}
    
    {\large\textbf{The second chapter}} of the thesis introduces the notion of simplicial binoid, whose algebra is the Stanley-Reisner algebra of a simplicial complex.
    We talk thoroughly about the deep combinatorics involved in the spectra of simplicial binoids, relating prime ideals and faces of the simplicial complex.
    We can also describe the sheaf of units in term of the combinatorics of the simplicial complex, and we will present a decomposition of the sheaf $\O^*_{M_\triangle}$ into smaller sheaves $\O^*_{v}$ indexed by the vertices in the simplicial complex.
    We then show how this smaller sheaves can be seen as the extension by zeros of the constant sheaf $\ZZ$ along the embedding $D(v)\hookrightarrow \Spec^\bullet M_\triangle$, and how the open set $D(v)$ is homeomorphic to $\Spec M_{\lk_\triangle(v)}$, the spectrum of the link of $v$ in the simplicial complex $\triangle$.
    Then we sum up the results and state an explicit formula for the computations of the cohomology of $\O^*_v$ in terms of reduced simplicial cohomology of $\lk_\triangle(v)$ with coefficients in $\ZZ$.
    Lastly, we carry on explicit computations on many examples, and we conclude with the description of the Picard group for some special subsets of $\Spec M_\triangle$, including the one to compute the divisor class group defined at the end of the first chapter.
    
    \paragraph{Results.} 
    In Corollary~\ref{corollary:cech-covering-nerve} we prove that the nerve of the covering $\{D(x_i)\}$ of $\Spec^\bullet M_\triangle$ is the simplicial complex itself. This leads us to Theorem~\ref{thm:simplicial-cohomology}, where we prove, through a direct study of the involved chain complexes, that the \v{C}ech cohomology of the constant sheaf $\ZZ$ on the covering $\{D(x_i)\}$ of $\Spec^\bullet M_\triangle$ is simplicial cohomology
    \[
    \H^i(\triangle, \ZZ)=\vH^i(\{D(x_i)\}, \ZZ)
    \]
    and then in Corollary~\ref{corollary:sheaf-cohomology} we notice that the latter actually computes sheaf cohomology, so we are successfully computing sheaf cohomology as simplicial cohomology
    \[
    \H^i(\Spec^\bullet M_\triangle, \ZZ)=\H^i(\triangle, \ZZ).
    \]
    In the next Corollary~\ref{corollary:cech-covering-nerve-simplicial-cohomology} we extend this to any open subset, obtaining that, for a specific cover $\mathscr{V}$ of $\mathcal{ U}\subseteq\Spec M$ open subset, we can compute sheaf cohomology as simplicial cohomology
    \[
    \H^i(U, \ZZ)\cong{\vH^i({\mathscr{V}}, \ZZ)}\cong{\H^i(\nerve(\mathscr{V}), \ZZ)}.
    \]
    
    In Theorem~\ref{theorem:localization-simplicial-binoid-multiple} we relate the localization of a simplicial binoid $M_\triangle$ at a face $F$ with the link of that face in $\triangle$, thus proving that
    \[
    (M_\triangle)_{x_F}\cong M_{\lk_\triangle(F)}\wedge(\ZZ^F)^\infty.
    \]
    In Theorem~\ref{theorem:semifree-decomposition-of-sheaf} we prove that, for any semifree binoids (like simplicial binoids), we can split the sheaf of units into smaller sheaves, indexed by the elements of the semibasis
    \[
    \begin{tikzcd}[baseline=(current  bounding  box.center), cramped, row sep = 0ex,
    /tikz/column 1/.append style={anchor=base east},
    /tikz/column 2/.append style={anchor=base west}]
    \displaystyle\bigoplus_{i=1}^n\O^*_{x_i} \rar & \O^*_{M}.
    \end{tikzcd}
    \]
    By putting this things together we are able, several pages later, to prove the main result of this chapter, Theorem~\ref{theorem:cohomology-simplicial-complex}, namely that the cohomology of the sheaf of units of a binoid can be computed entirely in terms of (reduced) simplicial cohomology
    \[
    \H^j\left(\Spec^\bullet M_\triangle, \O^*_{M_\triangle}\right)\cong\bigoplus_{i\in V}\widetilde{\H}^{j-1}\left(\lk_\triangle(i), \ZZ\right),
    \]
    for $j\geq0$.
    
    \vspace{6ex}
    
    {\large\textbf{The third chapter}} of the thesis contains results about interactions and interplays between the combinatorial side of binoids and the algebraic side of their rings.
    We present a functor between $M$-sets and $\KK[M]$-modules, that give rise to an injection between the spectra.
    This allows us to describe a new topology on $\Spec\KK[M]$, that we call combinatorial topology, that is coarser than the Zariski topology.
    Thanks to the injection of $\Spec M$ into $\Spec\KK[M]$ we can talk about pushforward of sheaves, and when the second space is equipped with the combinatorial topology we have that the sheaf $\O^*_{\KK[M]}$ splits into a part which depends on the field and a combinatorial part that is independent from it.
    Since we are interested in $\O^*_{\KK[M]}$ and its cohomology with respect to the Zariski topology, we relate the cohomology of this sheaf in the Zariski and in the combinatorial topology and we show that in some cases (namely, when the combinatorial open subsets define and acyclic covering) the second is fine enough to compute sheaf cohomology in the Zariski topology via \v{C}ech cohomology on a combinatorial covering.
    We cover here the cases of the affine space and the affine normal toric varieties, and we provide a counterexample of a non normal toric variety in which the combinatorial topology is not enough.
    We conclude the chapter with some explicit results about the sheaf $\O^*_{\KK[M]}$ in the non reduced case, and we write down some vanishing condition on the reduction that allow us to use the combinatorial topology for the computations also in the non reduced case.
    
    \paragraph{Results.} In this Chapter the main result is the introduction of the combinatorial topology, followed by the proof in Proposition~\ref{proposition:split-sheaves-comb-top} that for a reduced, torsion-free and cancellative binoid $M$ with binoid algebra $\KK[M]$ there is a splitting of sheaves in the combinatorial topology
    \[
    (\O_{\KK[M]}^*)^{\comb}\cong (\KK^*)^{\comb}\oplus (i_*\O^*_M)^{\comb}.
    \]
    The next important result is Proposition~\ref{proposition:hzar-hcomb}, that states that if $U=D(\mathfrak{A})$ is a combinatorial open subset of $\Spec \KK[M]$ with an acyclic covering of combinatorial open subsets, then
    \[
    \H^j_{\mathrm{Zar}}(U, \F)=\H^j_{\comb}(U, \F),
    \]
    for all $j\geq 0$.
    
    There is a continuous injection $i$ of $\Spec M$ in $\Spec \KK[M]$ and in Theorem~\ref{theorem:push-forward-exact} we prove that the pushforward of a sheaf along this embedding is exact. This allows us to prove in Proposition~\ref{proposition:cohomology-pushforward-affine} that
    \[
    \H^j(\Spec \KK[M], i_*\F)=0
    \] for all $j\geq 1$ and all sheaves of abelian groups $\F$ on $\Spec M$.
    
    The last part of the Chapter is devoted to the non-reduced case, where we are able to prove in Theorem~\ref{theorem:1+N-acyclic} that cohomology of the sheaf of unipotent units vanishes in higher degrees
    \[
    \H^i(X, 1+\N)=0,
    \]for all $i\geq 1$.
    
    \vspace{6ex}
    
    {\large\textbf{The fourth chapter}}, last of this work, deals mainly with explicit results about Stanley-Reisner rings.
    We study the spectra of these rings, explicitly relating them with combinatorial properties of the simplicial complexes.
    We follow the steps of the second chapter, on simplicial binoids, in order to write down the punctured \v{C}ech-Picard complex of a Stanley-Reisner ring.
    We prove that cohomology of positive degree of the sheaf of units vanishes for an affine spectrum of the Stanley-Reisner rings, then we show that it vanishes on affine open subsets $D(X_{i_0}\cdots X_{i_k})$ and this proves that $\{D(X_i)\}$ is an acyclic covering of $\Spec^\bullet \KK[M]$ for this sheaf, like it was for the simplicial binoid.
    We are able to split the sheaf of units on the combinatorial topology as a direct sum of smaller sheaves, and to prove that computing cohomology in the combinatorial topology is the same as computing it in the Zariski topology.
    This allows us to give again explicit formulas for the computation of the cohomology of $\O^*_{\KK[M_\triangle]}$, that relates sheaf cohomology of $\O^*_{\KK[M_\triangle]}$ with (reduced) simplicial cohomology. Since we are dealing with $\KK$-algebras, this time also the field will play a role, but again we are able to tame it and reduce it to simplicial cohomology.
    The last part of the Chapter is devoted to some results about the non reduced monomial case, and we give some examples that show how harder it can become, but still our methods will provide us some tools to determine when this cohomology is non zero, allowing us to prove some non vanishing results.
    
    \paragraph{Results.} In Proposition~\ref{proposition:covering-algebraic-spectrum} and Corollary~\ref{corollary:nerve-covering-stanley-reisner} we prove that the covering $\{D(X_i)\}$ of the punctured spectrum of the Stanley-Reisner ring has the same intersection pattern as the covering $\{D(x_i)\}$ of the punctured spectrum of the simplicial binoid, thus proving that the nerve of the covering $\{D(X_i)\}$ of $\Spec^\bullet\KK[\triangle]$ is again the simplicial complex itself.
    
    In Theorems~\ref{Thm:vanishingspecstanleyreisner} and \ref{Thm:vanishingspecstanleyreisnerA*}, we prove that
    \[
    \H^j(\KK[\triangle], \O^*)=\H^j(\KK[\triangle][y_1^{\pm1}, \dots, y_m^{\pm1}], \O^*)=0,
    \] for all $j\geq 1$, thus yielding the very special case $\Pic(\KK[M])=0$.
    
    In Lemma~\ref{theorem:localization-stanley-reisner-combinatorial-multiple} we prove that, for a face $F$ in $\triangle$,
    \[
    \KK[\triangle]_{X_F}\cong \KK[\triangle'][\ZZ^F]
    \]
    where $\triangle'=\lk_\triangle(F)$.
    
    This leads us to Theorem~\ref{theorem:SR-splits-comb-top}, where we prove that in the combinatorial topology the sheaf of units of a Stanley-Reisner ring splits as
    \[
    \O^*_{\KK[\triangle]}=\KK^*\oplus i_*\O^*_{M_\triangle}.
    \]
    
    In Theorem~\ref{theorem:cohomology-stanley-reisner} we give again a description of the cohomology of $\O^*_{M_\triangle}$ in the Zariski topology entirely in terms of (reduced) simplicial cohomology
    \[
    \H^j(\Spec^\bullet(\KK[\triangle]), \O^*_{\KK[\triangle]})=\H^j(\triangle, \KK^*)\oplus \bigoplus_{v\in V}\widetilde{\H}^{j-1}(\lk_\triangle(v), \ZZ),
    \]
    for any $j\geq 0$.
    
    In the last Section of this Chapter we deal with some results in the monomial non-reduced case, where we are able to split again the cohomology of the sheaf of units in a part that is purely combinatorial. Although we cannot describe combinatorially the unipotent units, we are able to address the problem of computing their via \v{C}ech cohomology on a particular covering. The last Corollary~\ref{corollary:non-vanishing} gives some non-vanishing results that are again entirely combinatoric for the cohomology of the sheaf of units in this situation.
    
    \newpage
    \subsection*{Notations and general assumptions}
    In what follows, we try to stick to the following notations.
    \begin{itemize}
        \item When we write small caps $x$'s, we usually denote variables in the binoid, while big caps $X$'s are reserved for variables in the binoid algebras.
        \item The same goes with prime ideals, small Fraktur letters (like $\p$) are for prime ideals in the binoids, while big cap Fraktur letters (like $\mathfrak{P}$) are for (prime) ideals in the ring; if there is no confusion, we might just use $I$ for ideals.
        \item We tend to oppose \lq combinatorial'' and \lq algebraic''; with the first we refer to properties of the binoids or of the underlying combinatorial objects (simplicial complexes, polytopes, monoids, etc.), while with the second we refer to properties of rings and algebras.
        \item When we refer to a scheme $X$, either combinatorial or algebraic, we are referring to a subscheme of some affine scheme (of finite type), e.g.\ $X=\Spec R\smallsetminus Y$, for some $Y$ subscheme, with the usual restriction of sheaves.
        \item The previous assumption will save us to repeat always the hypothesis that we would assume anyway, because we are always considering separated schemes of finite type.
        \item Unless otherwise specified, moreover, we are always considering finitely generated, commutative binoids and finitely generated algebras.
        \item The binoids that we are considering, unless otherwise specified, are also always torsion-free and cancellative. We do not assume that they are integral, reduced or normal. Indeed, most of the binoids considered do not satisfy any of these hypothesis
    \end{itemize}
    
    \newpage

    \section*{Acknowledgements}
    
    I want first to thank my supervisor Holger Brenner, who helped me in choosing and then dealing with the topic of this thesis. His constant support and understanding have been fundamental for the success of the project. I am really grateful to him; it is partly his fault, that I spent three great years in Osnabrück, dealing with the passion and joy and despair that come with research. He was there in the dark times, when nothing seemed to work, understanding, supporting and suggesting me alternative approaches and new tools to try. He was there in the light times, rejoicing with me and sharpening the results that we found. He could always find some spare time to talk about mathematics and personal problems, and this is something for which I could not thank him enough.
    
    Secondly, I want to thank the other people with whom I crossed my path here at the mathematics department in Osnabrück; talking with them really inspired me a lot during these three years, helping me to grow, both culturally and professionally.
    A particular thought goes to the postdoctoral researchers Lukas Katth\"an, Sara Saeedi Madani and Sean Tilson, whose experience and mutual collaboration, as well as open-mindedness and kindness, helped me find the strength and inspiration necessary to ask dumb questions and reach this point of my career.
    
    A special thank goes to Gilles Bonnet, with whom I shared many experiences and good times; his amazing ability to connect friends from around the world turned to be really helpful during our time here in Osnabrück, creating a positive environment in which life abroad was very easy. I want to thank Richard Sieg, my precious officemate, that was always there when I wanted to talk and with whom I gladly shared the craziness of mixing combinatorics and debugging. I want to also thank Alejo Lòpez Ávila, with whom I shared the office at the very beginning, and together we managed to get a grasp on German bureaucracy.
    
    I want to thank all the people involved in the Graduiertenkolleg; although we did not all spent much time together, it was surely fun and interesting to meet all of them and intersect our different cultures; I am really thankful for having been part of this DFG project.
    
    A special thank goes to {Simone B\"ottger}, who did her PhD here before me and developed the solid basement of the theory of binoids, \mbox{Bayarjargal} Batsukh, who built some walls on this basement, and the professors {Martina Juhnke-Kubitzke}, {Tim R\"omer}, {Markus Spitzweck}, {Oliver Röndigs} and {Matthias Reitzner} for their knowledge and kindness.
    
    My gratitude also goes to Alessio Caminata and Elisa Fascio, with whom I shared many Italian dinners here in Osnabrück; it was a surprise to meet such nice persons with which I could feel connected so easily.
    
    Of the many people that I met during conferences and schools around the world, I want to thank in particular Elisa Palezzato, Michela Ceria and Paolo Lella, whose work and life decisions inspired me a lot; a last thank goes to professors Mario Valenzano and Margherita Roggero, in Torino, who helped me doing the first steps of my career during and after the master degree.
    
    There are also many people who are not related to the mathematics department that I would like to thank.
    I want to begin with the people in Osnabrück, in particular Ann-Christin, Marlene, Santiago, Maren and Arushi; wonderful persons with whom I had the pleasure and honour to share many adventures in and around Osnabrück.
    A special thank goes to Serena, Miriam and  Özge Nilay; although we went along separated ways already some time ago, I always look with pleasure and happiness at our time together here.

    My life-lasting friendships hold a special place in my mind; Miriam, Paola, Matteo, Stefano, Sandro, Valentina, Valentina, Alice, Michele, Daniele, Alessandro, Fiorella. Although we now live in different parts of the world, I always feel connected to you.
    
    The last two special thanks go to my parents and brother, who always supported me and my choices, and my girlfriend Giulia, who bears everyday the weight of being next to me.

    \begin{flushright} Davide Alberelli\end{flushright}
    
    \restoregeometry

    \newgeometry{textwidth=14.5truecm,textheight=21.5truecm,top=40mm,includeheadfoot}
    
    \clearpage
    
    \pagestyle{fancy}
    \renewcommand{\chaptermark}[1]{%
        \markboth{#1}{}}
    \renewcommand{\chaptermark}[1]{%
        \markright{\thesection\ #1}}
    \fancyhf{} 
    \fancyhead[R]{\bfseries\thepage}
    \fancyhead[L]{\bfseries\rightmark}
    \renewcommand{\headrulewidth}{0.5pt}
    \renewcommand{\footrulewidth}{0pt}
    \addtolength{\headheight}{3.5pt} 
    \fancyheadoffset{0pt}
    
    \fancypagestyle{plain}{%
        \fancyhead{} 
        \renewcommand{\headrulewidth}{0pt} 
    }
    
    \mainmatter
    
\chapter{Schemes of binoids}

In our attempt to study some classical objects in Algebraic Geometry in the context of binoids, we will need to define combinatorial counterparts to these objects. As classically done for rings, we start by defining the spectrum of a binoid, then we define the Zariski topology, what a sheaf of binoids is and then we investigate other sheaves. We will also prove some vanishing Theorems for the cohomology of sheaves of abelian groups.

Unless otherwise stated, we always assume that our binoids are finitely generated. This condition is not necessary in many points, but these are the objects that we are interested in, because they yield finitely generated binoid algebras.

\section{Definitions}
Our basic objects will be schemes of binoids. We start from a binoid and we try to develop a kind of algebraic geometry revolving around binoids instead of rings.

\subsection{\texorpdfstring{$\Spec M$}{spectrum of M}}
The first object that we want to define is the spectrum of a binoid. For an extended treatment we refer the reader to \cite[Chapter 4]{Boettger}.

\begin{definition}\label{definition:prime-ideal}
    An ideal $\gls{primeideal}\subseteq M$ is \emph{prime} if $M\smallsetminus \p$ is a monoid. Equivalently, if for any $a, b\in M$ such that $a+b\in\p$ we have that $a\in\p$ or $b\in\p$.
\end{definition}

\begin{definition}
    Let $M$ be a binoid. The set of prime ideals of $M$ is the spectrum of $M$, denoted by $\gls{spectrum}$.
\end{definition}

The spectrum of $M$ is naturally a poset ordered by inclusion, as the next Proposition shows.

\begin{definition}\label{definition:maximal-ideal-binoid}
    $\gls{M-plus}$ is the maximal ideal of $M$, $M_+=M\smallsetminus \gls{Mstar}$, where $M^*$ is the group of units of $M$. This is a prime ideal, because if $a$ and $b$ are both units with respective inverses $c$ and $d$, then $a+b$ has inverse $c+d$.
\end{definition}
\begin{proposition}[{\cite[Proposition 2.2.3]{Boettger}}]\label{proposition:spectrum-is-join-semilattice}
    The union of a set of prime ideals is again prime, so $\Spec M$ is a join-semilattice with largest element $M_+$.
\end{proposition}

\begin{remark}\label{remark:generators-prime-ideals}
    It is elementary to prove that every prime ideal is generated by a subset of the generators of the binoid. This fact allows us to write the next criterion for prime ideals of semifree binoids.
\end{remark}

\begin{definition}\label{definition:semifree}
    Let $M$ be a nonzero commutative binoid. We say that $M$ is \textbf{semifree} with \textbf{semibasis} $(a_i)_{i\in I}$ if $M$ is generated by $\{a_i\mid i\in I\}$ and every element $f\in M$ can be written uniquely as $f=\sum_{i\in I} n_ia_i$ with $n_i=0$ for almost all $i\in I$. The set of $a_i$ such that $n_i\neq 0$ is called the \emph{support} of $f$, $\gls{supp}(f)=\{a_i\mid n_i\neq 0\}$.
\end{definition}

\begin{proposition}\label{proposition:criterion-prime-ideals}
    Let $M=(\gls{genset}\mid \gls{relset})$ be a positive semifree binoid. A subset of the generators $G\subseteq \mathcal{G}$ generates a prime ideal if and only if for every relation $r_L=r_R\in\ \mathcal{R}$, $G\cap\supp r_L\neq \varnothing$ if and only if $G\cap\supp r_R\neq \varnothing$.
\end{proposition}

\begin{remark}\label{algorithm:criterion-prime-ideals}
    This criterion gives rise to a naive algorithm to compute the spectrum of a positive semifree binoid. Namely, we check the criterion for all the subsets of $\mathcal{G}$, starting from the singletons. If we find that $\p$ and $\mathfrak{q}$ are two different prime ideals, we can avoid checking the criterion for $\p\cup \mathfrak{q}$, thanks to Proposition~\ref{proposition:spectrum-is-join-semilattice}.
\end{remark}
An application of this criterion will be presented in Example~\ref{example:spectrum-of-x+y=z+w} below.

\begin{example}[The affine space]\label{example:combinatorial-affine-space}We begin our series of Examples with a very fundamental one, namely the combinatorial equivalent of the $n$-dimensional affine space. The binoid we consider is ${(\NN^{n})}^\infty=(x_1, x_2, \dots, x_n\mid \varnothing)$. Since there are no relations involved, we can easily see that any subset of $\{x_1, \dots, x_n\}$ generates a prime ideal. In particular then the spectrum of this binoid is isomorphic to the power set of $[n]$. When $n=3$ we can easily draw it as a $\subseteq$-poset like
    
    \vspace{1ex}
    \raisebox{\height}{\hspace{\mathindent}$\Spec {(\NN^3)}^\infty=$
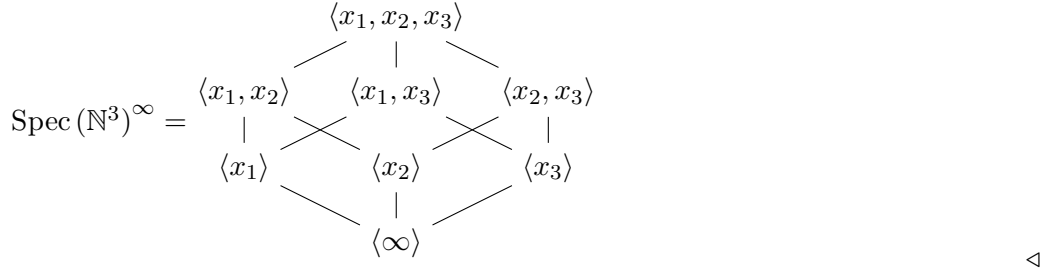
\begin{tikzpicture}[baseline={(current bounding box.center)}]
        \node (xyz) at (0,0){$\langle x_1, x_2, x_3\rangle$};
        \node (xy) at (-2,-1){$\langle x_1, x_2\rangle$};
        \node (xz) at (0,-1){$\langle x_1, x_3\rangle$};
        \node (yz) at (2,-1){$\langle x_2, x_3\rangle$};
        \node (x) at (-2,-2){$\langle x_1\rangle$};
        \node (y) at (0,-2){$\langle x_2\rangle$};
        \node (z) at (2,-2){$\langle x_3\rangle$};
        \node (infty) at (0,-3){$\langle \infty\rangle$};
        
        \draw[-](xyz) to (xy);
        \draw[-](xyz) to (xz);
        \draw[-](xyz) to (yz);
        \draw[-](xy) to (x);
        \draw[-](xy) to (y);
        \draw[-](xz) to (x);
        \draw[-](xz) to (z);
        \draw[-](yz) to (y);
        \draw[-](yz) to (z);
        \draw[-](x) to (infty);
        \draw[-](y) to (infty);
        \draw[-](z) to (infty);
        \end{tikzpicture}}
\end{example}

\begin{example}\label{example:spectrum-of-x+y=z+w}
    We study the spectrum of the binoid $M=(x, y, z, w\mid x+y=z+w)$. By meaning of the Algorithm in Remark~\ref{algorithm:criterion-prime-ideals} we can easily build the spectrum of this binoid and represent it as a $\subseteq$-poset.
    
    Let us start from a subset of the generators, for example the singleton $\{x\}$. There exists a relation such that $x$ is in the left hand side, but not in the right hand side. So $\{x\}$ does not generate a prime ideal.
    Let us consider $\{x, z\}$. The unique relation in this binoid has elements from this set on both sides of the equal sign, so this set defines a prime ideal.
    
    With a single relation it is easy to see the prime ideals, because is suffices to select all the possible combinations of at least one element from each side.
    
    In this case, the spectrum of $M$ is
    
    \vspace{1ex}
    \raisebox{\height}{\hspace{\mathindent}$\Spec M=$
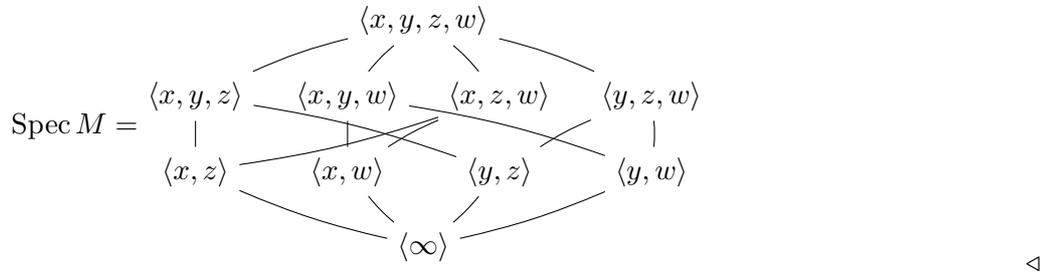
\begin{tikzpicture}[baseline={(current bounding box.center)}]
        \node (xyzw) at (0,0){$\langle x, y, z, w\rangle$};
        \node (xyz) at (-3,-1){$\langle x, y, z\rangle$};
        \node (xyw) at (-1,-1){$\langle x, y, w\rangle$};
        \node (xzw) at (1,-1){$\langle x, z, w\rangle$};
        \node (yzw) at (3,-1){$\langle y, z, w\rangle$};
        \node (xz) at (-3,-2){$\langle x, z\rangle$};
        \node (xw) at (-1,-2){$\langle x, w\rangle$};
        \node (yz) at (1,-2){$\langle y, z\rangle$};
        \node (yw) at (3,-2){$\langle y, w\rangle$};
        \node (infty) at (0,-3){$\langle \infty\rangle$};
        
        \draw[-](xyzw) to [bend right=5](xyz);
        \draw[-](xyzw) to [bend right=5](xyw);
        \draw[-](xyzw) to [bend left=5](xzw);
        \draw[-](xyzw) to [bend left=5](yzw);
        \draw[-](xyz)--(xz);
        \draw[-](xyz) to [bend left = 5](yz);
        \draw[-](xyw)--(xw);
        \draw[-](xyw) to [bend left = 5](yw);
        \draw[-](xzw) to [bend right = 5](xw);
        \draw[-](xzw) to [bend left = 5](xz);
        \draw[-](yzw) to [bend right = 5](yz);
        \draw[-](yzw) to [bend left = 5](yw);
        \draw[-](xz) to [bend right = 5](infty);
        \draw[-](xw) to [bend right = 5](infty);
        \draw[-](yz) to [bend left = 5](infty);
        \draw[-](yw) to [bend left = 5](infty);
        \end{tikzpicture}}
\end{example}

\begin{example}\label{example:spectrum-non-integral-binoid}
    We study now the spectrum of a reduced non-integral binoid, the simplicial binoid $M_\triangle = (x_1, x_2, x_3\mid x_1+x_2+x_3=\infty)$ that we will study in more details in Chapter~\ref{chapter:simplcial-binoids}.
    
    This spectrum is very similar to the one of Example~\ref{example:combinatorial-affine-space} except for the point $\{\infty\}$, because this binoid is not integral.
    
    \vspace{1ex}
    \raisebox{\height}{\hspace{\mathindent}$\Spec M_\triangle=$
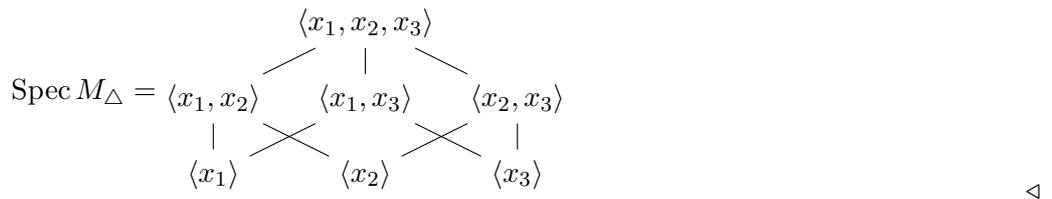
\begin{tikzpicture}[baseline={(current bounding box.center)}]
        \node (xyz) at (0,0){$\langle x_1, x_2, x_3\rangle$};
        \node (xy) at (-2,-1){$\langle x_1, x_2\rangle$};
        \node (xz) at (0,-1){$\langle x_1, x_3\rangle$};
        \node (yz) at (2,-1){$\langle x_2, x_3\rangle$};
        \node (x) at (-2,-2){$\langle x_1\rangle$};
        \node (y) at (0,-2){$\langle x_2\rangle$};
        \node (z) at (2,-2){$\langle x_3\rangle$};
        
        \draw[-](xyz) to (xy);
        \draw[-](xyz) to (xz);
        \draw[-](xyz) to (yz);
        \draw[-](xy) to (x);
        \draw[-](xy) to (y);
        \draw[-](xz) to (x);
        \draw[-](xz) to (z);
        \draw[-](yz) to (y);
        \draw[-](yz) to (z);
        \end{tikzpicture}}
\end{example}

\begin{definition}\label{definition:reduced}
    Let $M$ be a binoid. An element $f\in M$ is called \emph{nilpotent} if $nf=\infty$ for some $n\geq 1$. The set of all nilpotent elements will be denoted by $\gls{nilM}$ and it is easy to show that this is an ideal. We say that $M$ is \emph{reduced} if $\gls{nilM}=\{\infty\}$.
\end{definition}

The following result shows us that nilpotents and torsion elements do no play a role when we consider the spectrum of the binoid
\begin{proposition}[{\cite[Corollary 2.2.11, Lemma 1.7.11]{Boettger}}]\label{proposition:specM-specMred-tf}
    For any binoid $M$,
    \[
    \Spec M \cong \Spec \gls{Mred}\cong \Spec \gls{Mtf}.
    \]
\end{proposition}

It is possible to endow the spectrum of a binoid with the usual Zariski topology.
\begin{definition}\label{definition:zariski-topology-combinatorial-spectrum}
    Let $A\subseteq M$. The \emph{Zariski closed set defined by $A$ in $\Spec M$} is
    \[
    V(A):=\{\p\in\Spec M\mid A\subseteq \p\}.
    \]
    The complement of the closure of $A$ is called \emph{Zariski open set defined by $A$ in $\Spec M$} and it is denoted, as usual, by
    \[
    D(A):=\Spec M\smallsetminus V(A).
    \]
    If $A=\{f\}$, we denote these subsets by $V(f)$ and $D(f)$ respectively.
\end{definition}
\begin{remark}
    It is easy to see that the sets $D(A)$ define a topology on $\Spec M$, because
    \begin{itemize}
        \item $\bigcup_{i\in I}D(A_i)=D\left(\bigcup_{i\in I} A_i\right)$
        \item $D(f)\cap D(g) = D(f+g)$
        \item $D(0) = \Spec M$
        \item $D(\infty) = \varnothing$
    \end{itemize}
\end{remark}
\begin{remark}When $A=\{f\}\subseteq M$ then $\gls{fundamental-open-subset}=\{\p\in\Spec M\mid f\notin \p\}$ and $\{D(f)\mid f\in M\}$ is a basis for this topology.
\end{remark}
\begin{definition}
    This topology is called the \emph{Zarisk topology} on the spectrum of $M$. The open subsets
    \[
    \{D(f)\}_{f\in M}
    \] are called \emph{fundamental open subsets}.
\end{definition}

\begin{remark}
    $D(f)=\varnothing$ if and only if $f\in\nil(M)$, and $D(f)=\Spec M$ if and only if $f\in M^*$.
\end{remark}

\begin{proposition}[{\cite[Remark 4.1.6]{Boettger}}]\label{remark:boettger:closedsubset-closedsuperset}
    $S\subseteq \Spec M$ is closed (respectively open) if and only if it is superset closed (respectively subset closed).
\end{proposition}

\begin{example}\label{example:xyzw}
    Let $M=(x, y, z, w\mid x+y=z+w)$. From Example~\ref{example:spectrum-of-x+y=z+w} we know that its spectrum is
    
    \vspace{1ex}
    \hspace{\mathindent}$\Spec M=$
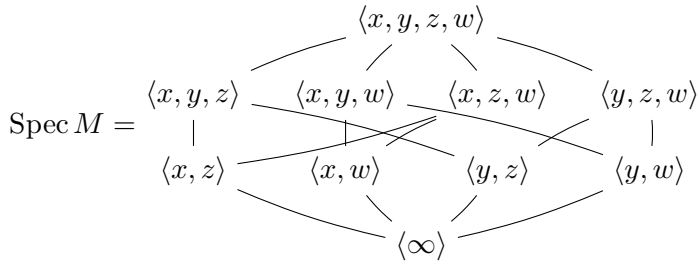
\begin{tikzpicture}[baseline={(current bounding box.center)}]
    \node (xyzw) at (0,0){$\langle x, y, z, w\rangle$};
    \node (xyz) at (-3,-1){$\langle x, y, z\rangle$};
    \node (xyw) at (-1,-1){$\langle x, y, w\rangle$};
    \node (xzw) at (1,-1){$\langle x, z, w\rangle$};
    \node (yzw) at (3,-1){$\langle y, z, w\rangle$};
    \node (xz) at (-3,-2){$\langle x, z\rangle$};
    \node (xw) at (-1,-2){$\langle x, w\rangle$};
    \node (yz) at (1,-2){$\langle y, z\rangle$};
    \node (yw) at (3,-2){$\langle y, w\rangle$};
    \node (infty) at (0,-3){$\langle \infty\rangle$};
    
    \draw[-](xyzw) to [bend right=5](xyz);
    \draw[-](xyzw) to [bend right=5](xyw);
    \draw[-](xyzw) to [bend left=5](xzw);
    \draw[-](xyzw) to [bend left=5](yzw);
    \draw[-](xyz)--(xz);
    \draw[-](xyz) to [bend left = 5](yz);
    \draw[-](xyw)--(xw);
    \draw[-](xyw) to [bend left = 5](yw);
    \draw[-](xzw) to [bend right = 5](xw);
    \draw[-](xzw) to [bend left = 5](xz);
    \draw[-](yzw) to [bend right = 5](yz);
    \draw[-](yzw) to [bend left = 5](yw);
    \draw[-](xz) to [bend right = 5](infty);
    \draw[-](xw) to [bend right = 5](infty);
    \draw[-](yz) to [bend left = 5](infty);
    \draw[-](yw) to [bend left = 5](infty);
    \end{tikzpicture}
    \vspace{1ex}
    
    Let $f=x+w$, then $D(f)=D(x)\cap D(w)$, that is the subset of prime ideals that contain neither $x$ nor $w$, as illustrated in the following picture
    
    
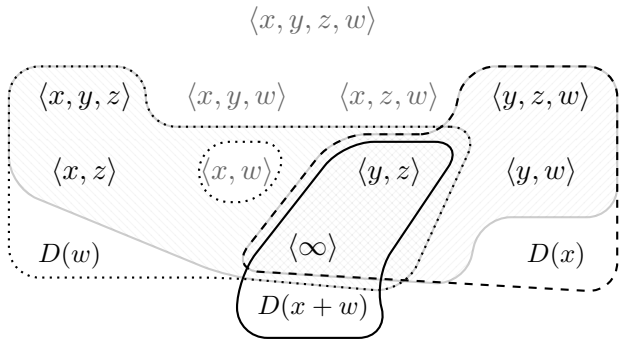
\vspace{1ex}
    \hspace{\mathindent}
    \begin{tikzpicture}[baseline={(current bounding box.center)}]
    \draw[black, rounded corners=4mm, thick](2,-1.6)--(.4,-1.6)--(-1,-3.4)--(-1,-4.2)--(1,-4.2)--(.8, -3.4)--cycle;
    \draw[black, rounded corners=4mm, thick, pattern=north east lines, pattern color=black, opacity=0.2](1.8,-1.5)--(.3,-1.5)--(-1.1,-3.3)--(2, -3.5)--(2.2, -2.6)--(4, -2.6)--(4, -.6)--(2, -.6)--cycle;
    \draw[black, rounded corners=4mm, thick, dashed](1.8,-1.5)--(.3,-1.5)--(-1.1,-3.3)--(4, -3.6)--(4, -.6)--(2, -.6)--cycle;
    \draw[black, rounded corners=4mm, thick, pattern=north west lines, pattern color=black, opacity=0.2](2.2,-1.4)--(-2.2,-1.4)--(-2.2, -.6)--(-4, -.6)--(-4,-2.2)--(-1.1,-3.4)--(1.2, -3.6)--cycle;
    \draw[black, rounded corners=4mm, thick, dotted](2.2,-1.4)--(-2.2,-1.4)--(-2.2, -.6)--(-4, -.6)--(-4,-3.4)--(-1.1,-3.4)--(1.2, -3.6)--cycle;
    \draw[white, rounded corners=4mm, thick, fill](-0.3,-1.6)--(-1.5,-1.6)--(-1.5, -2.4)--(-.6, -2.4)--cycle;
    \draw[black, rounded corners=4mm, thick, dotted](-0.3,-1.6)--(-1.5,-1.6)--(-1.5, -2.4)--(-.6, -2.4)--cycle;
    
    \node[opacity = .6] (xyzw) at (0,0){$\langle x, y, z, w\rangle$};
    \node (xyz) at (-3,-1){$\langle x, y, z\rangle$};
    \node[opacity = .6] (xyw) at (-1,-1){$\langle x, y, w\rangle$};
    \node[opacity = .6] (xzw) at (1,-1){$\langle x, z, w\rangle$};
    \node (yzw) at (3,-1){$\langle y, z, w\rangle$};
    \node (xz) at (-3,-2){$\langle x, z\rangle$};
    \node[opacity = .6] (xw) at (-1,-2){$\langle x, w\rangle$};
    \node (yz) at (1,-2){$\langle y, z\rangle$};
    \node (yw) at (3,-2){$\langle y, w\rangle$};
    \node (infty) at (0,-3){$\langle \infty\rangle$};
    \node (Dxw) at (0, -3.8) {\small $D(x+w)$};
    \node (Dw) at (-3.2, -3.1) {\small $D(w)$};
    \node (Dx) at (3.2, -3.1) {\small $D(x)$};
    
    \end{tikzpicture}
    \vspace{1ex}
    
    Its complement is $V(x+w)=V(x)\cup V(w)$, i.e.\ the set of prime ideals that contain either one or the other, since $x+w\in\p$ if and only if $x\in\p$ or $w\in\p$, because $\p$ is prime.
    
    
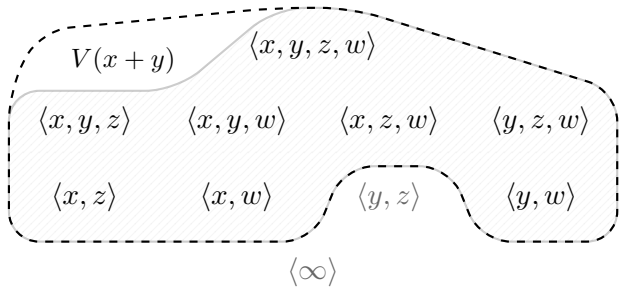
\vspace{1ex}
    \hspace{\mathindent}
    \begin{tikzpicture}[baseline={(current bounding box.center)}]
    \draw[black, rounded corners=4mm, thick, dashed](1.8,-1.6)--(2.2, -2.6)--(4, -2.6)--(4, -.6)--(.5, .5)--(-.5, .5)--(-3.6, 0.2)--(-4, -0.6)--(-4, -2.6)--(0, -2.6)--(.4,-1.6)--cycle;
    
    \draw[black, rounded corners=4mm, thick, pattern=north east lines, pattern color=black, opacity=0.2](1.8,-1.6)--(2.2, -2.6)--(4, -2.6)--(4, -.6)--(.5, .5)--(-.5, .5)--(-1.8, -.6)--(-4, -.6)--(-4, -2.6)--(0, -2.6)--(.4,-1.6)--cycle;
    
    \node (xyzw) at (0,0){$\langle x, y, z, w\rangle$};
    \node (xyz) at (-3,-1){$\langle x, y, z\rangle$};
    \node (xyw) at (-1,-1){$\langle x, y, w\rangle$};
    \node (xzw) at (1,-1){$\langle x, z, w\rangle$};
    \node (yzw) at (3,-1){$\langle y, z, w\rangle$};
    \node (xz) at (-3,-2){$\langle x, z\rangle$};
    \node (xw) at (-1,-2){$\langle x, w\rangle$};
    \node[opacity = .6] (yz) at (1,-2){$\langle y, z\rangle$};
    \node (yw) at (3,-2){$\langle y, w\rangle$};
    \node[opacity = .6] (infty) at (0,-3){$\langle \infty\rangle$};
    \node (Vxy) at (-2.5, -.2){\small $V(x+y)$};
    \end{tikzpicture}
    \vspace{1ex}
    
    From these pictures, we can easily see that open subsets are subset-closed and closed subsets are superset-closed.
\end{example}

\begin{proposition}\label{proposition:minimal-open-set-prime-ideal}
    Let $M$ be a binoid. For any prime ideal $\p\in\Spec M$ there exists a unique minimal open set that contains it.\footnote{Remember that we assumed at the beginning of the Chapter that $M$ is finitely generated.}
\end{proposition}
\begin{proof}
    Let $x_1, \dots, x_k, x_{k+1}, \dots, x_n$ be the generators of $M_+$. Without loss of generality, we can assume that $\p=\langle x_1, \dots, x_k\rangle$ and that $\p$ does not contain any other generator of $M_+$. Let
    \[
    V=D(x_{k+1}+ \dots+ x_n)=\{\mathfrak{q}\in \Spec M\mid \{x_{k+1}, \dots, x_n\}\nsubseteq \mathfrak{q}\}.
    \]
    Clearly $\p\in V$. Moreover, $V$ is subset closed, because it is open, and $\p$ is the unique maximal element of $V$. In order to prove this last statement, assume that there exists another maximal element in $V$, $\mathfrak{q}\neq \p$. Then there exists a $x_j\in\mathfrak{q}\smallsetminus \p$ and $j$ has to be bigger than $k$, so $\mathfrak{q}\notin V$. Since $\p$ is the unique maximal element of $V$, the latter is the unique minimal open subset that contains $\p$. In order to see this, let $U$ be any open subset that contains $\p$. Then $U$ is again subset closed, so $V\subseteq U$.
\end{proof}

\begin{remark}
    The open subset $V$ in the Proposition above is homeomorphic to $\Spec M_\p$.
\end{remark}

\begin{remark}\label{remark:no-open-covers-M+}
    There is no closed proper subset of $\Spec M$ that contains $\{\infty\}$. Conversely, there is no open proper subset that contains $M_+$.
\end{remark}

\subsection{Binoid schemes}
Now that we have understood the basic Zariski topology for the spectrum of a binoid, we are going to introduce the concepts of scheme of binoids, structural sheaf, affine scheme and other objects of algebraic geometry in this combinatorial setting.

\begin{definition}\label{definition:sheaf-of-binoids}
    Let $X$ be a topological space. A \emph{presheaf of binoids} on $X$ is a controvariant functor from the topology of $X$ to the category of binoids
    \begin{equation*}
    \begin{tikzcd}[baseline=(current  bounding  box.center), cramped, row sep = 0ex,
    /tikz/column 1/.append style={anchor=base east},
    /tikz/column 2/.append style={anchor=base west}]
    \F:\gls{TopX}\rar & \gls{Bin}\\
    U\rar[mapsto]& \F(U)
    \end{tikzcd}
    \end{equation*}
    
    A \emph{sheaf of binoids on $X$} is a presheaf of binoids on $X$ that respects locality and gluing axioms.\footnote{See for example \cite[Chapter II.1]{hartshorne1977algebraic} or \cite[Definition III.1.3]{perrin2007algebraic}. Locality and glueing axioms are together also called Serre conditions, for example in \cite[Definition 4.A.2]{patil2010introduction}.}
\end{definition}

\begin{definition}\label{definition:binoided-space}
    A \emph{binoided space} is a pair $\gls{binoinded-space}$ where $X$ is a topological space and $\gls{structure-sheaf}$ is a sheaf of binoids on $X$, called \emph{structure sheaf} of the space.
\end{definition}

Like with rings, we have very special binoided spaces, namely the schemes.

\begin{definition}\label{definition:affine-scheme-binoid}
    Let $M$ be a binoid. The \emph{affine binoid scheme defined by $M$} is the binoided space $(\Spec M, \gls{structure-sheaf-spec-binoid})$, where $\O_{\Spec M}$ is the unique sheaf associated to the presheaf defined on the basis $\{D(f)\}$ as
    \[
    \O_{\Spec M}(D(f))=\Gamma(D(f), \O_{\Spec M})=\gls{localization-element},
    \]
    called the \emph{structure sheaf} of $\Spec M$.
    
    Since there is no confusion, for ease of notation we usually denote it by \gls{structure-sheaf-binoid} and call it simply the \emph{structure sheaf} of the binoid $M$.
\end{definition}

\begin{remark}\label{remark:structure-presheaf}
    Like for rings, we can explicitly describe the presheaf $\O_{\Spec M}$ as
    \[
    \begin{tikzcd}[cramped, cramped, row sep = 0ex,
    /tikz/column 1/.append style={anchor=base east},
    /tikz/column 2/.append style={anchor=base west}]
    D(f_1, \dots, f_r)\rar[mapsto]& \displaystyle\Gamma\left(\bigcup D(f_i), \O_M\right)\\
    &=\left\{(s_1, \dots, s_r)\in M_{f_1}\times\dots\times M_{f_r}\midd s_i=s_j \text{ in } M_{f_i+f_j}\right\},
    \end{tikzcd}
    \]
    so the image is a subbinoid of $M_{f_1}\times\dots\times M_{f_r}$.
\end{remark}

\begin{definition}\label{definition:scheme-binoid}
    The binoided space $(X, \O_X)$ is a \emph{binoid scheme} or \emph{scheme of binoids} if there exist an open cover $\{U_i\}_{i\in I}$ of $X$ and a collection of binoids $\{M_i\}_{i\in I}$ such that
    \[
    U_i \cong \Spec M_i\qquad \text{and} \qquad \O_X\restriction_{U_i} \cong \O_{M_i}.
    \]
\end{definition}

\begin{proposition}\label{proposition:structure-sheaf-on-empty-set}
    For any binoid $M$, evaluating the structure sheaf of $M$ on the empty set yields the trivial binoid, $\O_M(\varnothing)=\gls{infty-binoid}$.
\end{proposition}
\begin{proof}
    We have that $\varnothing = D(\infty)$, so $\O_M(\varnothing)= M_\infty$, that is the trivial binoid \gls{infty-binoid}.
\end{proof}

\begin{proposition}\label{proposition:open-subset-scheme}
    Let $U$ be an open subset of $\Spec M$. Then
    \[
    (U, \O_M\restriction_U)
    \]
    is a scheme of binoids.
\end{proposition}
\begin{proof}
    We have to prove that there exists an affine open cover that respects Definition~\ref{definition:scheme-binoid}. Mimicking \cite[Section I.2.1]{eisenbud2006geometry}, it is enough to prove it for $U=D(f)$. Consider the cover of $D(f)$ given by the $D(g)$'s such that $D(g)\subseteq D(f)$.
    We have to prove that
    \[
    \O_M\restriction_{D(f)}(D(g))=\O_{M_f}(D(g)).
    \]
    Since $D(g)\subseteq D(f)$, we have that
    \[
    \O_M\restriction_{D(f)}(D(g))=\O_M(D(f)\cap D(g))=\O_M(D(g))=M_g.
    \]
    At the same time, we have that
    \[
    \O_M(D(f)\cap D(g)) = \O_M(D(f+g))=M_{f+g}=\left(M_f\right)_g=\O_{M_f}(D(g)),
    \]
    that proves our statement.
\end{proof}

A particular scheme that we are interested in, is the punctured spectrum of a binoid
\begin{definition}\label{definition:punctured-spectrum-scheme}
    Let $M$ be a binoid. Its \emph{punctured spectrum} is the scheme
    \[
    (\gls{punctured-spectrum}, \O_M\restriction_{\Spec^\bullet M}).
    \]
\end{definition}

\begin{remark}
    This is a scheme thanks to Proposition~\ref{proposition:open-subset-scheme}, considering $U=\Spec^\bullet M$.
\end{remark}

\begin{proposition}\label{proposition:covering-punctured-spectrum-binoid}
    Let $M$ be a binoid. Let $\{x_i\}_{i\in I}$ be the generators of the maximal ideal $M_+$. Then $\Spec^\bullet M=\cup_{i\in I} D(x_i)$.\footnote{This result, with the same proof, is true also for binoids that are not finitely generated.}
\end{proposition}
\begin{proof}
    We begin by proving that $M_+$ is not in this union. If it was, then there is at least a $D(x_i)$ such that $M_+$ belongs to it. This implies that $x_i\notin M_+$, a contradiction.
    
    On the other side, let $\p\in\Spec^\bullet M_+$. Since $\p\subsetneq M_+$, there exists $x_k$ such that $x_k\notin \p$, i.e.\ $\p\in D(x_k)$, and  $\p\in\cup_{i\in I} D(x_i)$.
\end{proof}

\begin{example}
    Let $M$ be a binoid and let $M_+=\langle x_1, \dots, x_n\rangle$. It is not true in general that $\langle x_1, \dots, \widehat{x_i}, \dots, x_n\rangle$ belongs to $\Spec M$. Two counterexamples to this are the two binoids  \\
    \textbullet\ $M=(x, y\mid 2x=3y)$, where $M_+=\langle x, y\rangle$ but $\langle x\rangle$ and $\langle y\rangle$ are not in $\Spec M$,\\
    \textbullet\ $M=(x, y, z\mid x+y=2z)$, where $M_+=\langle x, y, z\rangle$ but $\langle x, y\rangle$ is not in $\Spec M$.
    This last example shows us also that the cover in the statement of the Proposition is not minimal, since
    \[
    \Spec^\bullet M = \{\langle \infty\rangle, \langle x, z\rangle, \langle y, z\rangle \}
    \]
    can be covered by just two open subsets instead of three, namely
    \[
    D(x)=\{\langle\infty\rangle, \langle y, z\rangle\}\quad \text{and}\quad D(y)=\{\langle\infty\rangle, \langle x, z\rangle\}.\qedhere
    \]
\end{example}

\begin{proposition}\label{proposition:stalk-binoid-scheme}
    Let $(\Spec M, \O_M)$ be the affine scheme of binoids defined by $M$, let $M_+$ be generated by $x_1, \dots, x_k, x_{k+1}, \dots, x_n$ and let $\p=\langle x_1, \dots, x_k\rangle$ be a prime ideal such that $\p$ does not contain any other generator of $M_+$. The stalk of $\O_M$ at the point $\p$ is
    \[
    \O_{M, \p}=\gls{localization-prime}=\O_M(D(x_{k+1}+ \dots+ x_n)).
    \]
\end{proposition}
\begin{proof}
    This is clear because, thanks to Proposition~\ref{proposition:minimal-open-set-prime-ideal}, we know that $D(x_{k+1}+ \dots+ x_n)$ is the unique minimal open subset of $X$ that contains $\p$.
\end{proof}

Assuming that $p$ is generated by the first $k$ generators does not reduce the generality of the result, but also for this result we need that the binoid is finitely generated.

\begin{remark}
    The previous proposition extends to any scheme of binoids, provided some finiteness conditions, because $(X, \O_X)$ is locally isomorphic to a scheme of binoids, and the minimal open subset that contains a point $\p$ is then the spectrum of the localization at this point.
\end{remark}

\begin{proposition}
    An open subset $W\subseteq \Spec M$ is affine if and only if there exists $f\in M$ such that $W=D(f)$.
\end{proposition}
\begin{proof}
    If $W$ is affine, there exists a binoid $N$ such that $W\cong \Spec N$, in particular there exists a maximal ideal $N_+$ in $W$. Again like in Proposition~\ref{proposition:minimal-open-set-prime-ideal}, without loss of generality we can assume that $N_+=\langle x_1, \dots, x_k\rangle$ and that it does not contain any other generator of $M_+$. Then clearly $D(x_{k+1}+ \dots+ x_n)$ and $W$ are both subset-closed in $\Spec M$ with the same unique maximal point $N_+$, so $W=D(x_{k+1}+ \dots+ x_n)$.
\end{proof}

\begin{remark}
    Obviously all affine subsets of $\Spec M$ will define open affine subsets of $\gls{KSpecM} $, but the contrary is not true. For example let $M=(x, y\mid x+y=\infty)$. Then $\Spec^\bullet M=\{\langle x\rangle, \langle y\rangle\}$ is not affine as a scheme of binoids, because it does not have a unique closed point. But $\KK-\Spec^\bullet M=\KK-\Spec M\smallsetminus \langle x, y\rangle$ is defined by $D(X+Y)\subseteq \Spec^\bullet\KK[M]$. The point is that the element $X+Y$ is not combinatorial, since it involves explicitly the operation $+$.    
\end{remark}

\subsection{Closed affine subschemes}
We introduce now affine schemes and affine subschemes of the combinatorial affine space.
\begin{definition}\label{definition:combinatorial-affine-space}
    The \emph{combinatorial affine space} of dimension $n$ is the binoid scheme
    \[
    \gls{affinespace}=\left(\Spec{\left(\NN^n\right)}^\infty, \O_{{\left(\NN^n\right)}^\infty}\right).
    \]
\end{definition}

\begin{remark}
    As we have seen in Example~\ref{example:combinatorial-affine-space}, its underlying topological space is isomorphic to the power set of $[n]$, with the topology induced by the inclusions.
\end{remark}

\begin{definition}
    A \emph{closed subscheme} of an affine binoid scheme $(\Spec M, \O_M)$ is an affine binoid scheme $(\Spec N, \O_N)$ such that $N=\faktor{M}{I}$, for some $I$ ideal of $M$.
\end{definition}
\begin{remark}
    $\Spec N$ can be identified with $V(I)$ in $\Spec M$, because it contains exactly the prime ideals that contain $I$.
\end{remark}

Unlike what happens with rings, an integral binoid can not always be defined as a quotient of $(\NN^r)^\infty$ by a prime ideal, in which case it does not define a closed subscheme of the affine space. Vice versa, there can be ideals that define integral binoids, and the following example shows both behaviours.

\begin{example}
    The integral binoid $M=(x, y\mid 2x=3y)$ does not define a closed subset of $\Spec\left(\NN^2\right)^\infty$. In order to show this, it is enough to notice that $\Spec M=\{\langle\infty\rangle, \langle x, y\rangle\}$ is not superset-closed in $\Spec\left(\NN^2\right)^\infty=\{\langle\infty\rangle, \langle x\rangle, \langle y\rangle, \langle x, y\rangle\}$.
    
    On the other hand, $M'=\faktor{(x, y\mid\varnothing)}{\langle x\rangle}$ is an integral binoid, because it is isomorphic to $\NN^\infty$, but it is defined by an ideal relation.
\end{example}

There are some positive results relating closeness of $\Spec N\smallsetminus \{\infty\}$ to the combinatorics of the relations involved in $N$, but they are in a rather convoluted form and not so interesting for the scope of this work.

\section{Sheaves}

The arguments covered in this Section can be introduced and studied in more general terms. We will concentrate on a rather specific situation, that allow us to ignore most technicalities and provide results that are useful for the main Chapters of this work.

The interested reader will find more general studies in \cite{PirashviliCohomology} and \cite{flores2014picard}.

Let $M$ be a finitely generated binoid. From now on, we concentrate on schemes of the type $(U, \O_M\restriction _U)$, where $U$ is an open subset of the affine scheme $\Spec M$. When $U$ is a proper subset, we refer to it with the name quasi-affine scheme.

\subsection{Sheaves of \texorpdfstring{$M$}{M}-sets}

\begin{definition}
    An $M$-set is a punctured set $(S, p)$ together with an action of $M$, that satisfies the usual properties. If there is no ambiguity about the special point, we just call it $S$.\footnote{For more properties of $M$-sets, refer to \cite[Section 1.10]{Boettger}.}
\end{definition}

\begin{remark}
    For any $f\in M$, $\gls{localization-Mset-element}$ is an $M_f$-set.
\end{remark}

\begin{definition}\label{definition:sheaf-of-M-sets}
    Let $(X, \O_X)$ be a binoid scheme. A \emph{sheaf of $\O_X$-sets on $X$}, or \emph{$\O_X$-sheaf} is a sheaf $\F$ on $X$ such that $\F(U)$ is a $\O_X(U)$-set for any open subset $U$ of $X$.
\end{definition}

\begin{definition}
    Let $S$ be an $M$-set. The \emph{sheafification of $S$} is the $\O_{\Spec M}$-sheaf $\gls{sheafificationMset}$ associated to the presheaf defined on fundamental open subsets as
    \[
    \gls{sheafificationMset}(D(f))=\Gamma(D(f), \widetilde{S})=S_f.
    \]
\end{definition}

\begin{remark}
    Like for binoids, that we saw in Remark~\ref{remark:structure-presheaf}, we can explicitly describe the presheaf $\widetilde{S}$ as
    \[
    \begin{tikzcd}[cramped, cramped, row sep = 0ex,
    /tikz/column 1/.append style={anchor=base east},
    /tikz/column 2/.append style={anchor=base west}]
    D(f_1, \dots, f_r)\rar[mapsto]& \displaystyle\Gamma\left(\bigcup D(f_i), \widetilde{S}\right)\subseteq S_{f_1}\times\dots\times S_{f_r},
    \end{tikzcd}
    \]    
    where $M_{f_1}\times \dots \times M_{f_r}$ acts on $S_{f_1}\times\dots\times S_{f_r}$ and we have again the compatibility conditions on the intersections.\\
    Similarly, we can look at the stalk $\widetilde{S}_\p$ at a point $\p\in\Spec M$ and it is easy to see that $\widetilde{S}_\p=S_\p=S+M_\p$, so $M_\p$ acts naturally on $\widetilde{S}_\p$.
\end{remark}

\begin{remark}
    If the $M$-set is an ideal $I$, then $\widetilde{I}$ is a sheaf of ideals on $\Spec M$.
\end{remark}

\begin{proposition} \label{proposition:isomorphism-M+-M}
    The sheafification of the maximal ideal of $M$ and $\O_M$ are isomorphic as sheaves on the punctured spectrum, i.e.
    \[
    \widetilde{M_+}\restriction_{\Spec^\bullet M}\cong \O_M\restriction_{\Spec^\bullet M}.
    \]
\end{proposition}
\begin{proof}
    From $M_+\hookrightarrow M$ we get an injective morphism of sheaves
    \begin{equation*}
    \begin{tikzcd}[baseline=(current  bounding  box.center), cramped, row sep = 0ex,
    /tikz/column 1/.append style={anchor=base east},
    /tikz/column 2/.append style={anchor=base west}]
    \widetilde{M_+}\rar[hook] & \O_M
    \end{tikzcd}
    \end{equation*}
    The above morphism is also surjective on $\Spec^\bullet M$, because for any $f\in M_+$ (thus for any proper fundamental open subset) we have
    \begin{equation*}
    \begin{tikzcd}[baseline=(current  bounding  box.center), cramped, row sep = 0ex,
    /tikz/column 1/.append style={anchor=base east},
    /tikz/column 2/.append style={anchor=base west}]
    \widetilde{M_+}(D(f))={(M_+)}_f\rar[hook]&M_f=\O_M(D(f)).
    \end{tikzcd}
    \end{equation*}
    The induced inclusion is locally surjective, because we reach $0$ in $M_f$, since $0=f-f\in {(M_+)}_f$.
    From this, we get the wanted isomorphism.
\end{proof}

\begin{remark}
    It is not true in general that $\widetilde{M_+}\cong \O_M$ on the whole spectrum, since their global sections are different.
\end{remark}

\begin{definition}\label{definition:invertible-sheaf}
    Let $(X, \O_X)$ be a binoid scheme and $\F$ a sheaf of $\O_X$-sets on $X$. We say that $\F$ is \emph{locally free of rank $n$} if there exist an $n\in \NN$ and a cover $\{U_i\}_{i\in I}$ of $X$ such that for every $i$
    \[
    \F\restriction_{U_i}\cong \left(\O_X\restriction_{U_i}\right)^{\cupdot n}.
    \]
    If $n=1$ we say that the sheaf is \emph{invertible}.\footnote{Refer to \cite[Definition 1.9.2]{Boettger} for the definition of the pointed union of $M$-sets}
\end{definition}

\begin{remark}
    Locally free sheaves on $X$ are also referred to as vector bundles on $X$, and if $n=1$ they are called line bundles. We will use both terms indistinctly in this work.
\end{remark}

\begin{definition}
    The category of locally free sheaves of rank $n$ on $X$ is denoted by $\mathbf{Loc}_n(X)$. The set of their isomorphism classes of rank $n$ on $X$ is denoted by $\mathrm{Loc}_n(X)$.\\
    The smash product and pointed union of $M$-sets induce the corresponding operations on the sheaves
    \begin{equation*}
    \begin{tikzcd}[baseline=(current  bounding  box.center), cramped, row sep = 0ex,
    /tikz/column 1/.append style={anchor=base east},
    /tikz/column 2/.append style={anchor=base west}]
    \wedge_{\O_X}:\mathbf{Loc}_m(X)\times \mathbf{Loc}_n(X)\rar & \mathbf{Loc}_{mn}(X)\\
    \bigcupdot:\mathbf{Loc}_m(X)\times \mathbf{Loc}_n(X)\rar & \mathbf{Loc}_{m+n}(X)
    \end{tikzcd}
    \end{equation*}
\end{definition}

\begin{definition}
    $\mathrm{Loc}_1(X)$, equipped with the operation $\wedge_{\O_X}$ is a group, called the \emph{Picard group} of $X$ and it is denoted by $\gls{PicX}$.
\end{definition}

\begin{example}\label{example:vector-bundles-of-x+y=2z}
    Let $M=(x, y, z\mid x+y=2z)$. We are going to show a non-trivial line bundle on $U=\Spec^\bullet M$. The space that we are considering is
    \[
    U=\{\langle\infty\rangle, \langle x, z\rangle, \langle y, z\rangle\}
    \]
    its topology does not have many open subsets, namely $D(x), D(y), D(x)\cup D(y)=U$ and $D(x)\cap D(y)=\{\langle\infty\rangle\}$. So, any line bundle (or locally free sheaf of rank $1$) is trivialized by the covering $\{D(x), D(y)\}$. Let $\mathcal{L}$ be such a line bundle. We know that $\mathcal{L}\restriction_{D(x)}\cong M_x$ as $M_x$-set and $\mathcal{L}\restriction_{D(y)}\cong M_y$ as $M_y$-set.
    \\
    We know that we have isomorphisms $M_x=\Gamma(D(x), \O_M)\cong \Gamma(D(x), \mathcal{L})$ and $M_y=\Gamma(D(y), \O_M)\cong \Gamma(D(y), \mathcal{L})$.
    \\
    On $D(x)$, we invert $x$ and the relation defining $M_x$ becomes $y=2z-x=z+(z-x)$. Similarly on $D(y)$, we have that $x=2z-y=z+(z-y)$. Let $\mathcal{L}$ be $\widetilde{\langle x, z\rangle}$, i.e.\ the sheafification of this ideal, considered on $U$. Then
    \begin{align*}
    \Gamma(D(x), \mathcal{L})&= \langle x, z\rangle_x = \langle 0\rangle_{M_x}\\
    \Gamma(D(y), \mathcal{L})&= \langle x, z\rangle_y = \langle z\rangle_{M_y}
    \end{align*}
    So $\mathcal{L}_x$ is exactly $M_x$, because it contains an invertible element (namely $x$) and $\mathcal{L}_y$ is isomorphic to $M_y$ via the $M_y$-set isomorphism
    \begin{equation*}
    \begin{tikzcd}[baseline=(current  bounding  box.center), cramped, row sep = 0ex,
    /tikz/column 1/.append style={anchor=base east},
    /tikz/column 2/.append style={anchor=base west}]
    \varphi_y:M_y\rar["\sim"] & \Gamma(D(y), \mathcal{L})\\
    0 \rar[mapsto] & z
    \end{tikzcd}
    \end{equation*}
    
    So $\mathcal{L}=\widetilde{\langle x, z\rangle}$ is a line bundle on $U$. In order to prove that it is not trivial, we have to show that $\langle x, z\rangle$ is not a principal ideal un $U$, thus it will not be principal in $M$.\footnote{The converse is not true, for example $M_+$ is principal on $U$ but usually not in $M$.} Assume that there exists $f\in M$ such that $\langle x, z\rangle=\langle f\rangle$. This means that there exist $g, h\in M$ such that $f+g=x$ and $f+h=z$. Since $M$ is positive and cancellative, the only possibility is $f=0$, $g=x$ and $h=z$ and the equality of the ideals does not hold. So $\mathcal{L}$ is non trivial.
    
    What is $\mathcal{L}\wedge_{\O_U}\mathcal{L}$? Since we have a nice global description as the sheafification of an ideal, we have that
    \[
    \mathcal{L}\wedge_{\O_U}\mathcal{L}=\widetilde{\langle x, z\rangle}\wedge_{\O_U}\widetilde{\langle x, z\rangle} \cong \widetilde{\langle x, z\rangle \wedge_M \langle x, z\rangle}.
    \]
    What is this smash product?
    \[
    \langle x, z\rangle \wedge_M \langle x, z\rangle = \langle x\wedge_M x, x\wedge_M z, x\wedge_M z, z\wedge_M z \rangle.
    \]
    From the map  $I\longrightarrow M$ and the properties of the smash product, we get a map
    \[
    \begin{tikzcd}[baseline=(current  bounding  box.center), cramped, row sep = 0ex,
    /tikz/column 1/.append style={anchor=base east},
    /tikz/column 2/.append style={anchor=base west}]
    I\wedge_M I\rar &M\\
    f\wedge_M g\rar[mapsto]&f+g            
    \end{tikzcd}
    \]
    whose image is the ideal $\langle 2x, x+z, 2z\rangle\cong \langle 2x, x+z, x+y\rangle\cong \langle x\rangle+M_+\cong M_+$. This map is not globally injective because, for example, $x\wedge z$ and $z\wedge x$ are mapped to the same element $x+z$. But it is injective on the punctured spectrum, because $I_x\cong M_x$ on $D(x)$ and $I_y\cong M_y$ on $D(y)$, and it is an easy computation to show injectivity in these cases.
    So we proved that these objects are congruent on the punctured spectrum, i.e.
    \[
    \Gamma\left (U, \widetilde{\langle x, z\rangle \wedge_M \langle x, z\rangle}\right)\cong \Gamma\left(U, \widetilde{M_+}\right)
    \]
    Now, since $\widetilde{M_+}\restriction_{U}\cong \O_M\restriction_{U}=\O_U$, thanks to Proposition~\ref{proposition:isomorphism-M+-M}, we get that
    \[
    \mathcal{L}\wedge_{\O_U}\mathcal{L} \cong \widetilde{\langle x, z\rangle \wedge_M \langle x, z\rangle}\cong \O_U,
    \]
    so the order of $\mathcal{L}$ as an element of $\Pic(U)$ is $2$. In general, there is an inclusion between ideals of height one modulo principal divisors into the Picard group of a particular open subset, that we explore in Section~\ref{section:isomorphisms-class-group}.
\end{example}

\subsection{Sheaves of abelian groups}

Another important type of sheaves that we want to consider is sheaves of abelian groups on a binoid scheme. We introduce them briefly here, but we do not define cohomology. Instead, we refer the reader to \cite[Chapter III]{hartshorne1977algebraic} for a compact but extensive treating of the subject.

\begin{definition}\label{definition:sheaf-of-groups}
    Let $X$ be a topological space. A \emph{presheaf of groups} on $X$ is a contravariant functor from the topology of $X$ to the category of abelian groups
    \begin{equation*}
    \begin{tikzcd}[baseline=(current  bounding  box.center), cramped, row sep = 0ex,
    /tikz/column 1/.append style={anchor=base east},
    /tikz/column 2/.append style={anchor=base west}]
    \F:\Top_{X}\rar & \gls{Ab}\\
    U\rar[mapsto]& \F(U)
    \end{tikzcd}
    \end{equation*}
    
    A \emph{sheaf of groups on $X$} is a presheaf of groups on $X$ that respects locality and gluing axioms.
\end{definition}

\begin{definition}\label{definition:constant-sheaf}
    The sheafification of a constant presheaf is called a \emph{constant sheaf}, and we will often denote it by the same symbol as the group.
\end{definition}

\begin{example} Important examples of constant sheaves are $\ZZ$, $(\KK, +)$, $(\KK^*, \cdot)$ for a field $\KK$ and the difference group of a binoid $M$, that we denote by $\gls{difference-group}$.
\end{example}

\begin{definition}\label{definition:gamma}
    Let $M$ be an integral, cancellative binoid. In this case, the subset $\gls{M-bullet}$ is a monoid. The \emph{difference binoid of $M$} is $\gls{difference-binoid}=(-M^\bullet+M)$. The \emph{difference group of $M$} is the group $\gls{difference-group}$.
\end{definition}

\begin{remark}
    When constructing $\Gamma$, what we are doing is inverting all the elements in $M^\bullet$.
\end{remark}

\begin{remark}
    On $\Spec^\bullet M$, if $M$ is integral then the constant presheaf is already a sheaf. If $M$ is not integral, then the sheafification is not trivial.
\end{remark}

\begin{remark}
    If $M$ is an integral and cancellative binoid, then the structure sheaf of Remark~\ref{remark:structure-presheaf} can be defined as
    \[
    D(f_1, \dots, f_r)\longmapsto \bigcap_{i=1}^r M_{f_i}\subseteq \Gamma.
    \]
\end{remark}

\begin{example}
    Let $M=(x, y\mid x+y=\infty)$. $\Spec^\bullet M$ can be covered by the two disjoint open subsets $D(x)$ and $D(y)$. Let $G$ be an abelian group, and consider the constant presheaf\footnote{We denote $G^-$ the constant presheaf, so that $(G^-)^+=G$ is its sheafification.}
    \[
    \begin{tikzcd}[baseline=(current  bounding  box.center), cramped, row sep = 0ex,
    /tikz/column 1/.append style={anchor=base east},
    /tikz/column 2/.append style={anchor=base west}]
    G^-:\Top_{X}\rar & \mathrm{Ab}\\
    U\rar[mapsto]& G
    \end{tikzcd}
    \]
    Clearly $\Gamma(D(x)\cup D(y), G^-)=G$ because it is a constant presheaf. But if we consider its sheafification $G$ we have that $\Gamma(D(x)\cup D(y), G)=\Gamma(D(x), G)\oplus \Gamma(D(y), G)= G\oplus G$ because $D(x)\cap D(y)=\varnothing$.
\end{example}

\begin{remark}
    Let $U$ be a subset of $\Spec M$ and let $G$ be a constant sheaf, so the sheafification of a constant presheaf. Then $G(U)=G^k$ where $k$ is the number of connected components of $U$. 
    In the previous example, $U=\Spec^\bullet M$ and $k=2$ because $V(\langle x\rangle)\cap U$ and $V(\langle y\rangle)\cap U$ are the two components (disjoint closed subsets) from which $U$ is made.
    \\
    Clearly if $M$ is integral then every non-empty open subset is again integral, hence connected, and so the constant presheaf is already a sheaf.
\end{remark}

The sheaf we are most interested to study is the sheaf of units of a binoid scheme, that is non constant.
\begin{definition}
    Let $(X, \O_X)$ be a binoid scheme. Its \emph{sheaf of units} is the sheaf of abelian groups on $X$
    \begin{equation*}
    \begin{tikzcd}[baseline=(current  bounding  box.center), cramped, row sep = 0ex,
    /tikz/column 1/.append style={anchor=base east},
    /tikz/column 2/.append style={anchor=base west}]
    \gls{sheaf-units-X}:\Top_{X}\rar & \mathrm{Ab}\\
    U\rar[mapsto]& \left(\O_X(U)\right)^*
    \end{tikzcd}
    \end{equation*}
    Where, given a binoid $M$, $M^*$ denotes the group of its units.
    If $X=\Spec M$, we denote $\O^*_X$ with $\gls{sheaf-units-M}$.
\end{definition}

There are some important results on the cohomology of sheaves of abelian groups on schemes of binoids. The next Theorem follows almost entirely the proof of \cite[Proposition 2.2.i]{PirashviliCohomology}, where it is proved for schemes of monoids. We will see later in Proposition~\ref{proposition:constant-sheaf-trivial-cohomology} that not all the properties that hold for monoid schemes hold for binoid schemes, because even in an affine situation the binoid scheme might not have a unique minimum element, that is always present in the monoid case.

\begin{theorem}\label{theorem:vanishing-combinatorial-cohomology-affine}
    Let $M$ be a binoid and $\F$ be a sheaf of abelian groups on $\Spec M$. Then
    \[
    \H^i(\Spec M, \F)=0
    \]
    for any $i\geq 1$.
\end{theorem}
\begin{proof}
    Recall from Remark~\ref{remark:no-open-covers-M+} that, since $X=\Spec M$ is affine, there is only one open subset that covers the maximal ideal $M_+$, that is $X$ itself. So for any sheaf $\F$ we have that $\F(X)=\F_{M_+}$, i.e.\ the evaluation of the global sections functor equals the stalk at ${M_+}$. Since $\F\longmapsto\F_{M_+}$ is an exact functor from the the category of sheaves to the category of groups, we see that the global sections functor is exact. Since $i$-th cohomology is the $i$-th derived functor of the global sections, that is an exact functor, we get the claim.
\end{proof}

\begin{theorem}[Leray's Theorem, {\cite[Exercise III.4.11]{hartshorne1977algebraic}}]\label{theorem:leray}
    Let $X$ be a topological space, $\F$ a sheaf of abelian groups on $X$, $\U=\{U_i\}_{i\in I}$ an open cover of $X$. Assume that for any finite intersection $V=U_{i_0}\cap\dots\cap U_{i_p}$ of sets in the cover we have $\H^k(V, \F\restriction_V)=0$ for all $k\geq 1$. Then for all $p\geq 0$ the natural maps
    \[
    \glslink{cechcohomologyUF}{\vH^p(\U, \F)}\longrightarrow \gls{cohomologyXF}
    \]
    from \v{C}ech cohomology to sheaf cohomology are isomorphisms.
\end{theorem}

\begin{definition}
    A cover $U$ that respects the hypothesis in the theorem is called \emph{acyclic for $\F$ on $X$}.
\end{definition}

\begin{corollary}[{\cite[Proposition 2.2.ii]{PirashviliCohomology}}]
    Let $(X, \O_X)$ be a separated binoid scheme, $\U$ an affine cover of $X$ and $\F$ a sheaf of abelian groups on $X$. Then $\H^*(X, \F)=\vH^*(\U, \F)$.
\end{corollary}
\begin{proof}
    The intersection of affine binoid subschemes is again affine, thanks to separateness, and cohomology vanishes on any intersection, thanks to Theorem~\ref{theorem:vanishing-combinatorial-cohomology-affine}. So any affine cover is acyclic, we can apply Leray's Theorem~\ref{theorem:leray} and obtain the wanted result.
\end{proof}

\begin{remark}
    Since, in this work, we are considering open subschemes of an affine spectrum of a finitely generated binoid, we are always considering separated schemes of finite type.
\end{remark}

\begin{proposition}\cite[Lemma 2.4]{PirashviliCohomology}\label{proposition:constant-sheaf-trivial-cohomology}
    Let $M$ be an integral binoid. For any open subscheme $U$ of $\Spec M$ and any constant sheaf of abelian groups $\G$ on $U$ we have that
    \[
    \H^i(U, \G)=0
    \]
    for all $i\geq 1$.
\end{proposition}

\begin{remark}
    The proof of the previous Proposition relies on the fact that a constant sheaf of abelian groups on an open subset of the spectrum of an integral binoid is flasque, hence it has trivial cohomology in degree higher than 0. This is not true if the binoid is not integral, essentially because there is not a unique minimal element in open subsets of its spectrum.
    This implies that the cohomology of a constant sheaf of abelian groups on the spectrum of a non integral binoid is not always zero. Indeed, it is a meaningful object, and this will be exploited in Corollary~\ref{corollary:cech-covering-nerve-simplicial-cohomology}.
\end{remark}

The following Proposition motivates all our subsequent studies on the Picard group of a binoid.
\begin{proposition}[{\cite[Proposition 3.1]{PirashviliCohomology}}]\label{proposition:cohomology-vector-bundles}
    There is a natural bijection
    \[
    \mathrm{Loc}_n(X)\sim  \H^1(X, \mathrm{GL}(n, \O_X))
    \]
    and an isomorphism of groups
    \[
    \Pic(X)\cong \H^1(X, \O^*_X).
    \]
\end{proposition}

\begin{definition}\label{definition:local-picard-group}
    Let $M$ be a binoid. Its \emph{local Picard group} is the Picard group of its punctured spectrum, $\gls{PiclocM}=\Pic(\Spec^\bullet M)$.
\end{definition}

\begin{remark}
    If $M$ is torsion-free, cancellative and reduced, the sheaf $\O^*_M$ can be embedded in a flasque sheaf
    \[
    \begin{tikzcd}[cramped]
    \O^*\rar[hook]&\displaystyle\bigoplus_{\begin{subarray}{c}
        \p \text{ minimal}\\
        \text{prime ideal of $M$}\end{subarray}} M^*_{\p}.
    \end{tikzcd}
    \]
\end{remark}

If $M$ is integral, the sheaf above is just $\Gamma$.

\begin{example}
    Let us consider the binoid $M=(x, y, z, w\mid x+y=z+w)$ and compute its local Picard group. Its punctured spectrum can be covered by the four open subsets $\{D(x), D(y), D(z), D(w)\}$ and, thanks to the intersection pattern, we have the following \v{C}ech complex
    \begin{equation*}
    \begin{tikzcd}[baseline=(current  bounding  box.center), cramped, row sep = 0ex,
    /tikz/column 1/.append style={anchor=base east},
    /tikz/column 2/.append style={anchor=base west}]
    \vC^0\rar["\partial^0"] & \vC^1\rar["\partial^1"]&\vC^2\rar["\partial^2"] & \vC^3\rar["\partial^2"] & 0
    \end{tikzcd}
    \end{equation*}
    
    We begin with the localizations at the variables and their unit groups, that form $\vC^0$
    \begin{align*}
    M_x&\cong (x, -x, y, z, w\mid y=z+w-x) & M_x^*&\cong \ZZ\\
    M_y&\cong (x, y, -y, z, w\mid y=z+w-x) & M_y^*&\cong \ZZ\\
    M_z&\cong (x, y, z, -z, w\mid y=z+w-x) & M_z^*&\cong \ZZ\\
    M_w&\cong (x, y, z, w, -w\mid y=z+w-x) & M_w^*&\cong \ZZ
    \end{align*}
    In the intersections, we have either $\ZZ^3$ or $\ZZ^2$, depending on the variables that we are considering. For example, when we invert both $x$ and $y$, also $z$ and $w$ get inverted, but thanks to the relation we can write one in term of the others and we get a $\ZZ^3$. These are the groups that contribute to $\vC^1$
    \begin{align*}
    M_{x+y}&\cong (x, -x, y, -y, z, w\mid 0=z+w-x-y) & M_{x+y}^*&\cong \ZZ^3\\
    M_{x+z}&\cong (x, -x, y, z, -z, w\mid y-z=w-x) & M_{x+z}^*&\cong \ZZ^2\\
    M_{x+w}&\cong (x, -x, y, z, w, -w\mid y-w=z-x) & M_{x+w}^*&\cong \ZZ^2\\
    M_{y+z}&\cong (x, y, -y, z, -z, w\mid x-z=w-y) & M_{y+z}^*&\cong \ZZ^2\\
    M_{y+w}&\cong (x, y, -y, z, w, -w\mid x-w=z-y) & M_{y+w}^*&\cong \ZZ^2\\
    M_{z+w}&\cong (x, y, z, -z, w, -w\mid x+y-z-w=0) & M_{z+w}^*&\cong \ZZ^3
    \end{align*}
    
    When we intersect further, we get a $\ZZ^3$ everywhere, because we can rewrite one variable in term of the others, and these are the groups in $\vC^2$ and the group in $\vC^3$
    \begin{align*}
    M_{x+y+z}^*&\cong \ZZ^3 & M_{x+y+w}^*&\cong \ZZ^3\\
    M_{x+z+w}^*&\cong \ZZ^3 & M_{y+z+w}^*&\cong \ZZ^3\\
    M_{x+y+z+w}^*&\cong \ZZ^3
    \end{align*}
    
    It is easy to see that
    \begin{equation*}
    \begin{tikzcd}[baseline=(current  bounding  box.center), cramped, row sep = 0ex, column sep=1em,
    /tikz/column 1/.append style={anchor=base east},
    /tikz/column 2/.append style={anchor=base west}]
    \ZZ\oplus \ZZ\oplus \ZZ\oplus \ZZ\rar["\partial^0"] & \ZZ^3\oplus \ZZ^2\oplus \ZZ^2\oplus \ZZ^2\oplus \ZZ^2\oplus \ZZ^3\\
    (\alpha_1, \alpha_2, \alpha_3, \alpha_4)\rar[mapsto]& \left((-\alpha_1, \alpha_2, 0), (-\alpha_1, \alpha_3), (-\alpha_1, \alpha_4), (-\alpha_2, \alpha_3), (-\alpha_2, \alpha_4), (-\alpha_3, \alpha_4, 0)\right)
    \end{tikzcd}
    \end{equation*}
    has trivial kernel, so $\H^0(\O^*_M)=0$. We are interested in $\H^1$, so we start by observing that we have some equalities of the open subsets
    \begin{align*}
    D(x+y)&=D(x+y+z)=D(x+y+w)=D(x+y+z+w)\\
    &=D(x+z+w)=D(y+z+w)=D(z+w).
    \end{align*}    
    
    The \v{C}ech complex now looks like
    \begin{equation*}
    \begin{tikzcd}[baseline=(current  bounding  box.center), cramped, row sep = 0ex,
    /tikz/column 1/.append style={anchor=base east},
    /tikz/column 2/.append style={anchor=base west}]
    \ZZ^4\rar["\partial^0"] & \ZZ^{14}\rar["\partial^1"]&\ZZ^{12}\rar["\partial^2"] & \ZZ^{3}\rar["\partial^2"] & 0
    \end{tikzcd}
    \end{equation*}
    and since $-4+14-12+3=1$ we know that the rank of $\H^1$ will be 1.
    It is not hard to examine the relations between elements in the kernel of $\partial^1$ and in the image of $\partial^0$ to conclude that this group has to be free, so, in particular,
    \[
    \Pic^{\loc}(M)=\ZZ.
    \]
    In a later Section we will compute the divisor class group of this binoid, and we will see that these two groups coincide. A generator of this group is represented by the ideal $\langle x, z\rangle$, that is invertible because on each affine combinatorial open subset it is generated by only one element.
\end{example}

\subsection{\texorpdfstring{\v{C}}{C}ech-Picard complex}

In this Section we are going to study the \v{C}ech complex for the sheaf $\O^*_X$ on the covering of $\Spec^\bullet M$ given by $\{D(x_i)\}$.

\begin{definition}\label{definition:cech-picard-complex}
    Let $(X, \O_X)$ be a binoid scheme. Let $\U=\{U_i\}_{i\in [n]}$ be a finite affine covering of $X$.\footnote{It always exists affine and finite because we are assuming that $X$ is quasiaffine.} The \emph{\v{C}ech-Picard complex of $X$} is the \v{C}ech co-chain complex of $\O_X^*$ with respect to $\U$
    \begin{equation*}
    \begin{tikzcd}[cramped, row sep = 0pt]
    \glslink{cech-picard-U}{\C(\U, \O_X^*)}: \C^0(\U, \O_X^*)\rar["\partial^0"] &\C^1(\U, \O_X^*)\rar["\partial^1"] &\dots \rar["\partial^{p-1}"] & \C^p(\U, \O_X^*)\rar["\partial^p"]&\dots
    \end{tikzcd}
    \end{equation*}
    where the groups are
    \[
    \C^p(\U, \O^*_X)=\bigoplus_{1\leq i_0< i_1<\dots<i_p\leq n} \O^*_X\left(U_{i_0}\cap U_{i_1}\cap\dots\cap U_{i_p}\right)
    \]
    and the coboundary maps are defined, as usually like in \cite[Section III.4]{hartshorne1977algebraic},
    \[
    \left(\partial^{p-1}(\sigma)\right)_{i_0, \dots, i_p}=\sum_{k=0}^{p}(-1)^k\sigma_{i_0, \dots, \widehat{i_k}, \dots, i_p}\restriction_{U_{i_0, \dots, i_p}}.
    \]
\end{definition}

\begin{example}\label{example:pic-x+y=2z}
    Let $M=(x, y, z\mid x+y=2z)$ as in Example~\ref{example:vector-bundles-of-x+y=2z} above. Let again $X=\Spec^\bullet M$. We know that there exists at least an invertible sheaf in $\Pic(X)$, and it has order 2.
    
    Thanks to the Theorems and Propositions so far in this Chapter, we know that we can obtain $\Pic(X)$ as the first cohomology group of the sheaf $\O_X^*$ and, in turn, we can compute this on any affine open covering of $X$. Let $\U$ be the covering $\{D(x), D(y)\}$.
    
    To build the \v{C}ech complex of $\O^*_X$ on $\U$ we need first to evaluate the sheaf on the open subsets $D(x)$, $D(y)$ and $D(x)\cap D(y)=D(x+y)$.
    \begin{align*}
    \O_X^*(D(x))&=\left(M_x\right)^*\\
    \O_X^*(D(y))&=\left(M_y\right)^*\\
    \O_X(D(x)\cap D(y))&=\left(M_{x+y}\right)^*
    \end{align*}
    
    In details, $M_x = (x, -x, y, z \mid x+(-x)=0, y=2z+(-x))$ and so $M_x^*=(x, -x\mid x+(-x)=0)\cong \ZZ$, where this integer represents the coefficient of $x$. Similarly for $y$.
    
    On the intersection, when we invert both $x$ and $y$, also $z$ gets inverted. So we have
    \begin{align*}
    M_{x+y}&=(x, -x, y, -y, z\mid x+(-x)=0, y+(-y)=0, z+(z-y-x)=0)\\
    &\cong (x, -x, z, -z, y\mid x+(-x)=0, z+(-z)=0, y=2z-x)\\
    &\cong (x, -x, z, -z\mid x+(-x)=0, z+(-z)=0)\cong\ZZ^2
    \end{align*}
    
    The maps of the \v{C}ech complex come from the localizations $M_x\stackrel{\iota_y}{\longrightarrow} M_{x+y}$ and $M_y\stackrel{\iota_x}{\longrightarrow} M_{x+y}$ when restricted to the units. In particular, when restricting a unit defined on $D(x)$, namely $f=mx$ to the intersection of $D(x)$ and $D(y)$, this stays the same. Vice versa, a unit defined on $M_y$, say $\beta y$ goes to $2\beta z-\beta x$ in $M_{x+y}$, so the complex looks like
    \begin{equation*}
    \begin{tikzcd}[baseline=(current  bounding  box.center), cramped, row sep = 0,
    /tikz/column 1/.append style={anchor=base east},
    /tikz/column 2/.append style={anchor=base west}]
    M_x^* \oplus M_y^*\rar["\partial^0"]& M_{x+y}^*\rar["\partial^1"] &0\\
    \hspace{.4em}\rotatebox[origin=c]{-90}{$\cong$}\hspace{1.8em}\rotatebox[origin=c]{-90}{$\cong$} \hspace{.4em}& \hspace{1.2em}\rotatebox[origin=c]{-90}{$\cong$}\\
    \hspace{.4em}\ZZ\hspace{.2em}\oplus\hspace{.3em}\ZZ\hspace{.3em} \rar & \ZZ\oplus\ZZ\rar &0\\
    (\hspace{.2em}\alpha\hspace{.6em}, \hspace{.8em}\beta\hspace{.2em})\rar[mapsto]&(\alpha-\beta,2\beta)
    \end{tikzcd}
    \end{equation*}
    To compute the first cohomology, we need to compute the quotient $\faktor{\ker(\partial^1)}{\im(\partial^0)}\cong \faktor{M_{x+y}^*}{\im(\partial^0)}$, so we need to understand the image of $\partial^0$. The image is generated by $(1, 2)$ as a subgroup of $\ZZ^2$, so the quotient is $\faktor{\ZZ}{\ZZ}\oplus\faktor{\ZZ}{2\ZZ}\cong\ZZ_2$.
    
    So we have proved that $\Pic(X)\cong\ZZ_2$, and we already found a representative of the only non-trivial class in this group, in Example~\ref{example:vector-bundles-of-x+y=2z}.
\end{example}

\begin{remark}
    Since $\Spec^\bullet M$ can be covered by $\{D(x_i)\}$ and we know that $\O^*_M$ is acyclic on these affine open subsets, we can compute the local Picard group of $M$ as the first cohomology group of the \v{C}ech-Picard complex on $\{D(x_i)\}$,
    \[
    \Pic^{\loc} M=\vH^1(\{D(x_i)\}, \O^*).
    \]
\end{remark}

\begin{lemma}\label{lemma:M*=M*red}
    Let $M$ be a binoid. Then $M^*\cong M^*_{\red}$.
\end{lemma}
\begin{proof}
    Let $\varphi_{\red}:M\longrightarrow M_{\red}$ be the reduction morphism. We will prove that it is an isomorphism when restricted to the group of units. We have that $\ker(\varphi_{\red})=\nil(M)$, $\varphi_{\red}$ is a bijection outside this ideal and $\nil(M)\cap M^*=\varnothing$. So we only have to prove that $\varphi_{\red}(M^*)\subseteq M^*_{\red}$, but this is true for any binoid homomorphism.
\end{proof}

This Lemma proves that, unlike for rings, the nilpotent elements do not enter at all when computing the units of a binoid. We will use this fact in later discussions in Chapter~\ref{chapter:injections}.

\begin{example}
    The above is not true for $M_{\tf}$. Let $n\in \NN$, $n\geq 2$ and let $M=(x\mid nx=0)$. Then $M^*\cong \ZZ_n$, but $M_{\tf}\cong\{0, \infty\}$, so $M^*_{\tf}\cong 0$.
\end{example}

\begin{proposition}
    Let $(X, \O_X)$ be a binoid scheme. Then $\O^*_X\cong \O^*_{X_{\red}}$.
\end{proposition}
\begin{proof}
    This follows from the previous Lemma~\ref{lemma:M*=M*red}.
\end{proof}

\begin{remark}
    Since $\O^*_X\cong \O^*_{X_{\red}}$, also their cohomologies will be the same. So we can concentrate on studying it in the reduced case, in order to study line bundles and higher cohomology of the sheaf of units.
\end{remark}

\begin{example}
    Let $M=(x, y, z, w\mid x+y=2z, 3w=\infty)$. Its punctured spectrum is
    \[
    X=\{\langle w\rangle, \langle x, z, w\rangle, \langle y, z, w\rangle\}
    \]
    and its reduction is $M_{\red}=(x, y, z, w\mid x+y=2z, w=\infty)\cong (x, y, z\mid x+y=2z)$, whose local Picard group we already studied in Example~\ref{example:pic-x+y=2z}. Since $D(w)=\varnothing$, we can easily see that the groups involved in the computation of \v{C}ech cohomology of $\O^*_X=\O^*_{X_{\red}}$ are exactly the same for $X$ and its reduction.
\end{example}

\section{The divisor class group}
In this Section we will first review some results that can be found in \cite{flores2014picard}, while rewriting them in our language, and then generalise the concept of Divisor Class Group to arbitrary quasiaffine binoid schemes.

\begin{definition}
    Let $\p$ be a prime ideal of a binoid $M$. The \emph{height of $\p$}, denoted $\htnew(\p)$, is the maximum $n\in\NN$ such that there exists a strictly increasing chain of prime ideals $\p_0\subsetneq \p_1\subsetneq \dots\subsetneq \p_n=\p$. The set of points of height $k$ in $\Spec M$ is denoted by $\glslink{X-height-k}{{\Spec M}^{(k)}}$.
\end{definition} 

\begin{remark}
    In the previous definition, $\p_0=\langle\infty\rangle$ if and only if the binoid is integral.
\end{remark}

\begin{definition}
    A binoid $M$ is called \emph{regular} if $M\cong G^\infty\wedge \left(\NN^s\right)^\infty$ for some abelian group $G$ and $s\in\NN$.
\end{definition}
\begin{definition}
    An integral binoid $M$ is called \emph{regular in codimension $s$} if, for any $\p\in\Spec M$ of height at most $s$, the localization $M_{\p}$ is a regular binoid.
\end{definition}

\begin{remark}
    If $M$ is torsion free and regular in codimension 1, then $M_\p\cong\left(\ZZ^r\right)^\infty\wedge\NN^\infty$ for some $r\in\NN$ and all the $\p$'s, prime ideals of height 1.
\end{remark}

\begin{remark}
    For a binoid, normal, torsion free and cancellative implies regular in codimension 1.
\end{remark}

\begin{example}
    Let us again consider $M=(x, y, z\mid x+y=2z)$. Its prime ideals of height one are just $\p_1=\langle x, z\rangle$ and $\p_2=\langle y, z\rangle$. Let us study the regularity of $M_{\p_i}$. 
    By definition $M_{\p_1}=-(M\smallsetminus \p_1)+M$ but, from the evaluation of $\O_M$, we have also that $M_{\p_1}=\O_M(D(y))=M_{y}$, thanks to Proposition~\ref{proposition:stalk-binoid-scheme}.
    $M_y=(x, y, -y, z\mid x=2z-y)\cong\ZZ^\infty\wedge\NN^\infty$. The same holds for $\p_2$, so $M$ is regular in codimension $1$.
\end{example}

\begin{proposition}\label{proposition:injectionM-Gamma}
    Let $M$ be an integral, cancellative and torsion-free binoid. Then its difference binoid is $\Gamma =M_{\langle\infty\rangle} \cong (\ZZ^r)^\infty$ for some $r$ and the map $M \longrightarrow \Gamma$ is injective.
\end{proposition}

\subsection{\texorpdfstring{$\Pic(U)$}{PicU}}
Assume that $M$ is an integral, cancellative, torsion-free, positive and regular in codimension $1$ binoid of dimension at least 1 and $X=\Spec^\bullet M$.
We study the set of line bundles defined on the open subset of the prime ideals of height at most $1$.

\begin{definition}
    Let $U$ be an open subset of $\Spec M$. Then $\Pic(U)$ is the group of invertible $\O_M$-sheaves defined on $U$. A sheaf in $\Pic(U)$ is called \emph{$U$-invertible}.
\end{definition}
\begin{remark}
    Recall from Proposition~\ref{proposition:cohomology-vector-bundles} that $\Pic(U)=\H^1(U, \O_M^*\restriction_U)$.
\end{remark}

\begin{proposition}\label{proposition:W-open}
    Let $W$ be the subset of $\Spec M$ that consists of the prime ideals of height at most $1$. Then $W$ is an open subset.
\end{proposition}
\begin{proof}
    $W$ is an open subset because it is the union of the minimal open subsets that contain these prime ideals  defined in Proposition~\ref{proposition:minimal-open-set-prime-ideal}.
\end{proof}

\begin{proposition}\label{proposition:M-sets-line-bundles-on-W}
    Every $\mathcal{L}\in\Pic(W)$ defines an $M$-set. Vice versa, every $M$-set $S$ such that $S_{\p}\cong M_{\p}$ for every $\p$ of height $1$ defines a line bundle on $W$.
\end{proposition}
\begin{proof}
    Let $\mathcal{L}$ be a line bundle defined on $W$. Then its sections $\Gamma(W, \mathcal{L})$ can be considered as an $M$-set. On the other side, Let $S$ be an $M$-set with this property. Then $M_\p\cong \gls{localization-Mset-prime}=\widetilde{S}(U)$ where $U$ is the unique minimal subset of $\Spec M$ than contains $\p$ and $\widetilde{S}$ is the sheafification of $S$.
\end{proof}

\begin{remark}
    The $M$-set defined by $\mathcal{L}$ is possibly not locally free as a $\O_M$-set on the whole $X$.
\end{remark}

\subsection{Cartier Divisors}
We introduce the notion of Cartier Divisor of a scheme of binoids and we prove an isomorphism with its Picard group. Assume that $M$ is an integral, cancellative binoid and $X=\Spec^\bullet M$.

Recall from Proposition~\ref{proposition:injectionM-Gamma} that there is an injection $M\hookrightarrow \Gamma$. Moreover, $\Gamma^\bullet=\Gamma\smallsetminus\{\infty\}$ defines a constant sheaf on $\Spec M$, that contains the sheaf of units $\O^*_M$ as a subsheaf.
\begin{definition}
    Let $U$ be an open subset of $\Spec M$. A \emph{Cartier Divisor} on $U$ is a global section of the sheaf of groups $\faktor{\Gamma^\bullet}{M^*}$. The group of this global sections is denoted by $\gls{cartierdivisorsU}$.
\end{definition}

\begin{definition}\label{definition:principal-cartier-divisor}
    A \emph{Principal Cartier Divisor} is a divisor that can be globally represented on $U$ by a $g\in\Gamma^\bullet$. They form a subgroup with the usual sum operation, denoted by $\gls{principalcartierdivisorsU}$.
\end{definition}

\begin{definition}
    The \emph{Cartier divisors Class Group} of $U$ is
    \[
    \gls{cartierdivisorsclassU}=\faktor{\CaDiv(U)}{\CaPrin(U)}
    \]
\end{definition}

The next Proposition follows the same proof as in \cite[Proposition 6.1]{flores2014picard}, adapted to our setting.

\begin{proposition}\label{proposition:isomorphism-CaCl-Pic}
    For any open subset $U$ of $X$ there is an isomorphism $\CaCl(U)\longrightarrow\Pic(U)$.
\end{proposition}
\begin{proof}
    Consider the short exact sequence of abelian groups
    \begin{equation*}
    \begin{tikzcd}[baseline=(current  bounding  box.center), cramped, row sep = 0ex,
    /tikz/column 1/.append style={anchor=base east},
    /tikz/column 2/.append style={anchor=base west}]
    0 \rar & M^* \rar &\Gamma^*\rar &\faktor{\Gamma^*}{M^*}\rar & 0
    \end{tikzcd}
    \end{equation*}    
    Since $M$ is integral and $\Gamma^*$ is a constant sheaf on $U$, from Proposition~\ref{proposition:constant-sheaf-trivial-cohomology} we get that $\H^1(U, \Gamma^*)=0$. From Proposition~\ref{proposition:cohomology-vector-bundles} we know that $\H^1(U, \O^*_M)=\Pic(U)$, so when we take cohomology we get the exact sequence
    \begin{equation*}
    \begin{tikzcd}[baseline=(current  bounding  box.center), cramped, row sep = 0ex,
    /tikz/column 1/.append style={anchor=base east},
    /tikz/column 2/.append style={anchor=base west}]
    0 \rar & \O^*_M(U) \rar &\Gamma^*(U)\rar["\mathrm{div}"] & \CaDiv(U)\rar["\delta"] &\Pic(U)\rar & 0.\\
    & 				&    \stackrel{\rotatebox{-60}{=}}{\hspace{2em}\Gamma^*}
    \end{tikzcd}
    \end{equation*}
    Since the image of $\mathrm{div}$ is exactly the subgroup of principal Cartier divisors defined in Definition~\ref{definition:principal-cartier-divisor}, we get the statement.
\end{proof}

\begin{remark}
    On each affine open subset $D(f)$ a Cartier divisor on $D(f)$ is given by a $\gamma_f\in\Gamma^\bullet$ up to a unit in $M_f^*$, since $\H^1(D(f), \O^*)=0$ and the shorth exact sequence above. So on $U=\bigcup D(f)$ we have the usual representation where a Cartier divisor $\glslink{cartierdivisor}{\{D(f), \gamma_f\}}$ is given by a collection of open sets $D(f)$ and sections $\gamma_f\in\Gamma^*$, with  $\gamma_f-\gamma_g\in M_{f+g}^*$.\\
    The sequence that we use in the Proposition is parallel to the sequence that occurs in algebraic geometry
    \begin{equation*}
    \begin{tikzcd}[baseline=(current  bounding  box.center), cramped, row sep = 0ex,
    /tikz/column 1/.append style={anchor=base east},
    /tikz/column 2/.append style={anchor=base west}]
    0 \rar & \O^* \rar & \mathcal{K}^*\rar &\faktor{\mathcal{K}^*}{\O^*}\rar & 0
    \end{tikzcd}
    \end{equation*}
    what we will use later in Lemma~\ref{Lemma:aciclicityforaffinespace}.
\end{remark}

There is an explicit description of $\delta$ in the Proof of the Proposition above. Given a Cartier divisor $D=\{(D(f), \gamma_f)\}$, we define a $U$-invertible subsheaf $\mathcal{L}(D)$ of the constant sheaf $\Gamma^*$ by letting its restriction to $D(f)$ to be generated by $-\gamma_f$. This construction is well defined, and it is an isomorphism from $\CaDiv(U)$ to the $U$-invertible subsheaves of $\Gamma^*$. The map $\delta$ in the Proposition above sends exactly $D$ to the corresponding isomorphism class of $\mathcal{L}(D)$.
In particular, if $W$ is an open subset of $U$ and $U$ is covered by $U_i$, we have that $W\cap U_i\neq \varnothing$ because the binoid is integral, and we can describe the evaluation of $\mathcal{L}(D)$ as
\[
\Gamma(W, \mathcal{L}(D))=\{q\in\Gamma\mid q\in (\gamma_i) \text{ on } U_i\cap W\}=\bigcap_{i}\gamma_i+M_{f_i}
\]
In particular then
\[
\Gamma(U, \mathcal{L}(D))=\bigcap_{i}\gamma_i+M_{f_i}\subseteq \Gamma.
\]
The operation is then the one induced by the embedding in $\Gamma$, i.e.\ if $(U_i, \gamma_i)$ and $(U_i',\gamma_i')$ are two Cartier divisors, that correspond to invertible sheaves $\mathcal{L}$ and $\mathcal{M}$ respectively, then $(U_i, \gamma_i)+(U_i',\gamma_i')$ corresponds to $\mathcal{L}+\mathcal{M}$ where, for any open subset $V\subseteq U$,
\[
(\mathcal{L}+\mathcal{M})(V)=\{s+t\mid s\in\mathcal{L}(V), t\in\mathcal{M}(V)\}.
\]
\begin{remark}\label{remark:CaDiv-ideals}
    Every Cartier divisor can be realized as an ideal in the following way. Let $(U_i, \gamma_i)$ be a Cartier divisor. We can assume that $U_i=D(f_i)$, as above, so $\gamma_i=g_i-h_i$ for some $g_i\in M$ and $h_i\in\Gamma$. Let $f=\sum_i h_i$. Clearly $(U_i, \gamma_i)\sim (U_i, \gamma_i+f)$ because $\gamma_i-(\gamma_i+f)=-f\in\Gamma$ and $\gamma_i+f\in M$. The ideal that realizes it is then generated by the collection of the $(\gamma_i+f)$'s, $\langle \{\gamma_i+f\}\rangle\subseteq M$.
\end{remark}

An example of this behaviour can be seen in the case $M=(x, y, z\mid x+y=nz)$, that will be fully explored in Example~\ref{example:x+y=nz}.

\subsection{Weil divisors}

We introduce now Weil divisors of $M$ and their class group, and we will later see that this group is isomorphic to $\Pic(W)$ and to the class group of Cartier divisors on $W$. Assume again that $M$ is an integral, cancellative, torsion-free, positive and regular in codimension $1$ binoid of dimension at least 1 and $X=\Spec^\bullet M$.
\footnote{A binoid that satisfies these assumptions is locally factorial in the sense of \cite{flores2014picard}.}

\begin{definition}
    Let $M$ be a binoid that satisfies all the hypothesis above. The \emph{group of Weil divisors of $M$} is the free abelian group on $X^{(1)}$
    \[
    \gls{weildivisorsM}=\bigoplus_{\p\in X^{(1)}}\ZZ \p.
    \]
\end{definition}

Since $M$ is regular in codimension $1$, for any prime ideal $\p$ of height $1$ there is a natural injection $M_{\p}^\bullet\cong \ZZ^{r-1}\wedge \NN\longrightarrow \ZZ^{r-1}\wedge \ZZ\cong \Gamma^\bullet$. Compose it now with the projection on the last factor $(\ZZ^{r-1})\wedge \ZZ\longrightarrow\ZZ$ and call this $v_{\p}$.

\begin{definition}
    Let $q\in\Gamma^\bullet$. We call $v_{\p}(q)$ \emph{the valuation of $q$ at $\p$}.
\end{definition}

\begin{remark}
    This is a indeed like a valuation in the ring case, and the following Proposition shows it.
\end{remark}

\begin{proposition} \begin{enumerate}
        \item $v_{\p}$ is a morphism of monoids.
        \item $v_{\p}(q)\in\NN$ if and only if $q\in M_{\p}^\bullet$.
        \item $v_{\p}(q)=0$ if and only if $q$ is a unit in $M_{\p}$.
    \end{enumerate} 
\end{proposition}
\begin{proof}
    1. This is true because $v_{\p}(q_1+q_2)=v_{\p}(q_1)+v_{\p}(q_2)$.\\
    2. This follows directly from the definition.\\
    3. If $q$ is a unit in $M_\p$ then $0=v_{\p}(0)=v_{\p}(q)+v_{\p}(-q)$ and this last values are both in $\NN$, so they are both $0$. On the other side, if $v_{\p}(q)=0$ then we can see it in $M_\p^\bullet\cong\ZZ^r\oplus \NN$ as $(\alpha_1, \dots, \alpha_r, 0)$, that has inverse $(-\alpha_1, \dots, -\alpha_r, 0)$.
\end{proof}

\begin{remark}
    Since $M$ is normal, it is known that
    \[
    M=\cap_{\p\in X^{(1)}}M_\p.
    \]
    Since $M$ is toric (see \cite{flores2014picard} and \cite[Construction 4.2]{cortinas2015toric}), then these prime ideals of height 1 correspond also to the facets of the cone defined by $M$.
\end{remark}

\begin{proposition}\label{proposition:morphism-divisors}
    Let $X^{(1)}=\{\p_1, \dots, \p_l\}$ be the set of all prime ideals of height $1$ in a toric normal positive binoid $M$. There is an injective morphism
    \begin{equation*}
    \begin{tikzcd}[baseline=(current  bounding  box.center), cramped, row sep = 0ex,
    /tikz/column 1/.append style={anchor=base east},
    /tikz/column 2/.append style={anchor=base west}]
    \Gamma^\bullet\rar["\varphi"] & \Div(M)\\
    q\rar[mapsto]&(v_{\p_1}(q), \dots, v_{\p_l}(q))
    \end{tikzcd}
    \end{equation*}
\end{proposition}
\begin{proof}
    Let $q\in \Gamma^\bullet$ such that $\varphi(q)=0$. Then $q$ is a unit in any stalk $M_\p$. In particular, then, $-q$ belongs to all these stalks. Since $M$ is normal, we have the equality above
    \[
    M=\bigcap_{\p\in X^{(1)}}M_\p
    \]
    so $q$ is invertible in $M$. But $M$ is positive, so $q=0$.
\end{proof}

\begin{remark}\label{remark:prin-not-positive}
    If $M$ is not positive, one can prove that this map has in its kernel exactly the elements of $M^*$, so we can then prove the injectivity between $\faktor{\Gamma^\bullet}{M^*}$ and $\Div(M)$. So we assume that $M$ is positive without losing too much information, see also \cite[Lemma 4.2]{flores2014picard}.
\end{remark}

This map in general is not surjective, and we define the divisor class group as follows.
\begin{definition}
    Let $M$ be a normal toric binoid. The group of principal divisors is the image of $\varphi$, a subgroup of $\Div(M)$, $\gls{principalweildivisorsM}=\im(\varphi)$. The \emph{Divisor Class Group} of $M$ is defined as the quotient
    \[
    \gls{weildivisorsclassM}=\faktor{\Div(M)}{\Prin(M)}.
    \]
\end{definition}

\subsection{Isomorphisms}\label{section:isomorphisms-class-group}
Assume again all the hypothesis on $M$, including normality. We are going to prove isomorphisms between three groups that we just described, namely $\Pic(W), \CaCl(W)$ and $\Cl(M)$. Recall that we have already proved in Proposition~\ref{proposition:isomorphism-CaCl-Pic} that $\CaCl(U)$ and $\Pic(U)$ are isomorphic for any open subset $U$ of $X$, and that $W$ is indeed open thanks to Proposition~\ref{proposition:W-open}. What is left, is to prove that $\Cl(M)$ is isomorphic to any of these two. Again, we follow \cite{flores2014picard} in order to prove it.

\begin{proposition}\label{proposition:CaCl-Cl}
    There is an isomorphism of groups $\CaCl(W)\cong\Cl(M)$.
\end{proposition}
\begin{proof}
    We start by proving that $\CaDiv(W)$ is a subgroup of $\Div(M)$. We can cover $W$ with the minimal open subsets that cover the prime ideals of height 1 defined in Proposition~\ref{proposition:minimal-open-set-prime-ideal}, so every Cartier divisor $D=\{D(f), \gamma_f\}$ defined on $W$ determines a Weil divisor, namely $\sum_{\p\in X^{(1)}}v_\p(\gamma_f)\p$.
    This correspondence is well defined, because $\gamma_f-\gamma_g$ is a unit defined on $D(f)\cap D(g)$, so the valuations have the same values. A principal Cartier divisor is mapped to a principal Weil divisor, because it is globally defined by an element of $\Gamma$, so we proved that there is a well-defined map $\CaDiv(W)\longrightarrow\Div(W)$ that sends principal divisors to principal divisors.
    This map is trivially injective. Let $D=\sum_{\p\in X^{(1)}}n_\p\p$ be a Weil divisor on $W$. For each $\p\in X^{(1)}$ let $D(f_\p)$ be the minimal open subset that covers it, namely $D(f_\p)=\{\p, \langle\infty\rangle\}$. Since $M$ is regular in codimension 1, we know that $M_\p\cong (\ZZ^l)^\infty\wedge \NN^\infty$, so let $t_\p$ be a generator of the principal maximal ideal $(M_\p)_+$ (namely, $t_\p$ generates $\NN^\infty$). The Cartier divisor associated to $D$ is then $\{(D(f_\p), n_\p t_\p)\}$. So, this map is also surjective. It is easy to see that it is also surjective on the principal divisors, namely every principal Weil divisor comes from a principal Cartier divisor, so when we look at the class groups we have the wanted isomorphism $\CaCl(W)\cong\Cl(M)$.
\end{proof}

So we just proved that for an integral, cancellative, torsion-free, positive, normal and regular binoid, there are isomorphisms $\Pic(W)\cong\CaCl(W)\cong\Cl(M)$.

\begin{remark}
    The Proposition above can be generalized to any regular open subset, and we will discuss this in Proposition~\ref{proposition:PicV-ClV}.
\end{remark}

\begin{example}\label{example:x+y=nz}
    Consider $M=(x, y, z\mid x+y=nz)$, for some $n\in \NN$. This is a slightly more general situation that what we already discussed in Example~\ref{example:pic-x+y=2z}. $X^{(1)}=\{\langle x, z\rangle, \langle y, z\rangle\}$, so $\Div(M)=\ZZ \langle x, z\rangle\oplus\ZZ\langle y, z\rangle$. In the localizations, the relations we have are
    \begin{align*}
    M_{\langle x, z\rangle}&=(x, y, z\mid x=nz-y)\\
    M_{\langle y, z\rangle}&=(x, y, z\mid y=nz-x)
    \end{align*}
    Let us focus on the first ideal, and have a look at the valuation $v_{\langle x, z\rangle}$
    \begin{equation*}
    \begin{tikzcd}[baseline=(current  bounding  box.center), cramped, row sep = 0ex,
    /tikz/column 1/.append style={anchor=base east},
    /tikz/column 3/.append style={anchor=base west}]
    M_{\langle x, z\rangle} \rar["\sim"] & \ZZ^\infty\wedge \NN^\infty\rar & \NN^\infty\\
    x = nz-y\rar[mapsto]& (-1, n)\rar[mapsto]& n\\
    y\rar[mapsto]& (1, 0)\rar[mapsto]& 0\\
    z\rar[mapsto]& (0, 1)\rar[mapsto]& 1
    \end{tikzcd}
    \end{equation*}
    And a similar thing goes for $v_{\langle y, z\rangle}$, where the roles of $x$ and $y$ are inverted.
    
    So, now consider $M$ as a subset of $\Gamma\cong (\ZZ^2)^\infty$, where the generators of this last group are exactly the inverted $x$ and $z$. The image of $\varphi=(v_{\langle x, z\rangle}, v_{\langle y, z\rangle})$ will be generated, as a subgroup, by the images of the three generators $x, y$ and $z$. So
    \begin{align*}
    \Prin(M)&=\langle\varphi(x), \varphi(y), \varphi(z) \rangle\\
    &=\langle(v_{\langle x, z\rangle}(x), v_{\langle y, z\rangle}(x)), (v_{\langle x, z\rangle}(y), v_{\langle y, z\rangle}(y)), (v_{\langle x, z\rangle}(z), v_{\langle y, z\rangle}(z))\rangle\\
    &=\langle(n, 0), (0, n), (1, 1)\rangle
    \end{align*}
    and the divisor class group will then be
    \begin{align*}
    \Cl(M)=\faktor{\Div(X)}{\Prin(X)}\cong\faktor{\ZZ^2}{\langle(n, 0), (0, n), (1, 1)\rangle}\cong\ZZ_n.
    \end{align*}
    Thanks to the Theorems above, we also know that this is the local Picard group of our binoid. Moreover, recall from Remark~\ref{remark:CaDiv-ideals} that the elements of this group all have a description as ideals, namely
    \[
    \CaCl(M)=\{\langle x, iz\rangle\mid i=0, \dots, n-1\}
    \]
    and the operation is just their sum.
\end{example}

\begin{example}
    Let our binoid be $M=(x, y, z, w\mid x+y=z+w)$. We have seen in Example~\ref{example:xyzw} that $X^{(1)}=\{\langle x, z\rangle, \langle x, w\rangle, \langle y, z\rangle, \langle y, w\rangle\}$.
    
    Then, for example, $M_{\langle x,z\rangle}\cong M_{y+w}$ and the relation here becomes $x=z+w-y$, so $M_{\langle x,z\rangle}\cong (\ZZ^{2})^\infty\wedge \NN^\infty$. By symmetry, we have the same for all the other ideals of height $1$, so we have that this binoid is regular in codimension $1$. Like in the previous example, it is easy to see the map, that simply goes to the coefficients (the $\NN$, again, refers to the coefficient of $z$, that is not invertible).
    \begin{equation*}
    \begin{tikzcd}[baseline=(current  bounding  box.center), cramped, row sep = 0ex,
    /tikz/column 1/.append style={anchor=base east},
    /tikz/column 3/.append style={anchor=base west}]
    M_{\langle x, z\rangle}^\bullet \rar["\sim"] & (\ZZ^2)^\infty\wedge \NN^\infty\rar & \NN^\infty\\
    x = z+w-y\rar[mapsto]& (-1, 1, 1)\rar[mapsto]& 1\\
    y\rar[mapsto]& (1, 0, 0)\rar[mapsto]& 0\\
    z\rar[mapsto]& (0, 0, 1)\rar[mapsto]& 1\\
    w\rar[mapsto]& (0, 1, 0)\rar[mapsto]& 0
    \end{tikzcd}
    \end{equation*}
    
    By symmetry arguments, again, we can see that the subgroup of principal divisors is generated by the $4$ vectors in $\ZZ^4$
    \[
    (1, 0, 1, 0), (1, 0, 0, 1), (0, 1, 1, 0), (0, 1, 0, 1).
    \]
    If $\{e_1, e_2, e_3, e_4\}$ are the generators of $\ZZ^4$, when we quotient by the principal divisors we get the relations
    \[
    e_1\sim -e_3\qquad e_1\sim -e_3 \qquad e_1\sim -e_4\qquad e_2\sim -e_4
    \]
    So the divisor class group of $M$ is
    \[
    \Cl(M)\cong\faktor{\ZZ^4}{\langle (1, 0, 1, 0), (1, 0, 0, 1), (0, 1, 1, 0), (0, 1, 0, 1) \rangle}\cong\ZZ.\qedhere
    \]
\end{example}

\subsection{Relations between \texorpdfstring{$\Pic(V)$}{PicV} and \texorpdfstring{$\Cl(V)$}{ClV}}

The first generalization that we want to do is to an arbitrary open subset $V$ of $\Spec M$, and the next Definition goes exactly in this direction.
Later on, we will generalize the results of Proposition~\ref{proposition:CaCl-Cl} to a more general open subset of the Spectrum of $M$.
Let $M$ be a integral, cancellative, regular in codimension 1 binoid and let again $W=X^{(1)}\cup X^{(0)}$ be the open subset of prime ideals of height at most 1.

\begin{definition}
    Let $V$ be an open subset of $\Spec M$. The group of Weil divisors of $V$ is
    \[
    \gls{weildivisorsV}=\bigoplus_{\begin{subarray}{c}
        \p\in V\\
        \htnew \p=1
        \end{subarray}
    }\ZZ\p.
    \]
    Thanks to Remark~\ref{remark:prin-not-positive} the map $\varphi$ is still injective when restricted to these components, so we call its image $\gls{principalweildivisorsV}$. The \emph{Divisor Class Group of $V$} is
    \[
    \gls{weildivisorsclassV}=\faktor{\Div(V)}{\Prin(V)}.
    \]    
\end{definition}

\begin{remark}
    If $W\subseteq V$ then $\Cl(V)=\Cl(W)=\Cl(M)$.
\end{remark}

The following results are preliminary to a counterexample that follows, where we show that, even in very nice situations, we might have $\Pic(U)\neq \Cl(U)$, where $U$ will be the punctured spectrum of a binoid.

\begin{lemma}
    Let $G$ be a group, then
    \[
    \Cl(M\wedge G^\infty)=\Cl(M).
    \]
\end{lemma}
\begin{proof}
    From \cite[Corollary 2.2.12]{Boettger} we know that $\Spec M\wedge G^\infty\cong\Spec M\times\Spec G^\infty$, but $\Spec G^\infty=\{\langle \infty\rangle\}$, so $\Div(M\wedge G^\infty)\cong\Div(M)$.
    
    Moreover, thanks to Remark~\ref{remark:prin-not-positive}, we know that $G$ is entirely contained in the kernel of $\varphi$, so its image in $\Div(M\wedge G^\infty)$ is trivial and it does not contribute to $\Prin(M\wedge G^\infty)$, so
    \[
    \Cl(M\wedge G^\infty)=\Cl(M).\qedhere
    \]
\end{proof}

\begin{lemma}\label{lemma:regular-ClM0}
    If $M$ is regular, then $\Cl M=0$.
\end{lemma}
\begin{proof}
    Since $M$ is regular, we know that there is an isomorphism $M=G^\infty\wedge(\NN^r)^\infty$, with $G$ a group. In particular, the prime of height one are of the form $\langle e_i\rangle$ with $e_i=(0, \dots, 1, \dots, 0)$ generator of $(\NN^r)^\infty$ for $i=1, \dots, r$. Then the map $\varphi$
    \begin{equation*}
    \begin{tikzcd}[baseline=(current  bounding  box.center), cramped, row sep = 0ex,
    /tikz/column 1/.append style={anchor=base east},
    /tikz/column 2/.append style={anchor=base west}]
    G\times \ZZ^r\cong \Gamma^\bullet \rar["\varphi"] & \Div(M)\cong\ZZ^r
    \end{tikzcd}
    \end{equation*}
    is surjective, so its cokernel is trivial.
\end{proof}

\begin{lemma}[{\cite[Corollary 5.5]{flores2014picard}}]\label{lemma:Cl-M-wedge-N}
    For every binoid scheme
    \[
    \Cl(X\times \AA^1)\cong \Cl(X).
    \]
    In particular, if $X=\Spec M$ then
    \[
    \Cl(M\wedge \NN^\infty)\cong \Cl(M).
    \]
\end{lemma}

\begin{proposition}\label{proposition:injection-PicU-PicU'}
    Let $M$ be a normal toric binoid. Let $U'\subseteq U$ be any two open subsets of $\Spec M$ such that $U^{(1)}=U'^{(1)}$. Then there is an injection
    \[
    \begin{tikzcd}[baseline=(current  bounding  box.center)]
    \Pic(U)\rar& \Pic(U').
    \end{tikzcd}
    \]
\end{proposition}
\begin{proof}
    Let $\U=\{D(f_i)\}$ be an affine covering of $U$.
    Then we can compute $\Pic(U)$ via \v{C}ech cohomology on these affine open subsets.
    Since $M$ is normal
    \[
    \Gamma(U, \O_M)=\displaystyle\bigcap_{\begin{subarray}{c}
        \htnew\p=1\\
        \p\in U
        \end{subarray}}M_\p.
    \]
    Since $U'$ and $U$ have the same points of height one, we have that $\Gamma(U', \O_M)=\Gamma(U, \O_M)$ and the same is true for intersections of $U'$ with the given covering of $U$, so
    \[
    \Gamma(U'\cap D(f_i), \O_M)=\Gamma(U\cap D(f_i), \O_M)= \Gamma(D(f_i), \O_M)=M_{f_i}.
    \]
    In particular, then, the \v{C}ech complexes to compute cohomology on $U$ and on $U'$ are the same.
    
    Recall that, in general, there is an injection from the first \v{C}ech cohomology on any covering to the first sheaf cohomology, because if we have a non-trivial cohomology class in $\vH^1$ and a covering that realizes it, then this non trivial class will also be non trivial in $\varinjlim\vH^1=\H^1$.
    
    So, whenever we have a non trivial cohomology class in $\Pic(U)$, it is computed by the \v{C}ech complex on $\U$. This \v{C}ech complex is the same as the one on the covering $\U\cap U'$ of $U'$, so it will be non trivial there, and so it will be non trivial in $\Pic(U')$.
\end{proof}

\begin{proposition}\label{proposition:normal-dimension-2}
    Let $M$ be a normal binoid of dimension 2. Then
    \[
    \Pic^{\loc}(M\wedge \NN^\infty)=0.
    \]
\end{proposition}
\begin{proof}
    Let $t$ be the generator of $\NN^\infty$. From \cite[Corollary 2.2.12]{Boettger} we know that
    \[
    \Spec (M\wedge \NN^\infty)\cong\Spec M\times\Spec \NN^\infty\cong \Spec M \times \{\langle \infty\rangle, \langle t\rangle\}\cong D(t) \uplus V(t).
    \]
    On one hand, $V(t)\cong \Spec M$ and, on the other hand, $D(t)\cong \Spec(M\wedge \NN^\infty)_t\cong \Spec(M\wedge \ZZ^\infty)\cong \Spec M$, so topologically $\Spec M\wedge \NN^\infty\cong \Spec M\uplus \Spec M$.
    
    Let $M_+^{\mathrm{ext}}$ be the extension of $M_+$ to $M\wedge \NN^\infty$. Then $V(M_+^{\mathrm{ext}})\cong \Spec \NN^\infty$.

    Let $U'=\Spec^\bullet M$, $U=\Spec^\bullet (M\wedge \NN^\infty)$ and $V=D(M_+^{\mathrm{ext}})\cong U'\times \AA^1$, i.e.\ $V$ is homeomorphic to two copies of $U'$, as we can see in the picture below.
    
    Let $\mathcal{L}$ be a line bundle defined on $U$. Then $\mathcal{L}\restriction_V\in\Pic(V)$. $M$ is regular at any point in $U'$, so $M\wedge\NN^\infty$ is regular at any point in $V$ and, thanks to Lemma~\ref{lemma:Cl-M-wedge-N}, we get $\Pic(V)\cong\Cl(V)\cong\Cl(U'\times \AA^1)\cong\Cl(U')\cong \Cl(M)$, so $\mathcal{L}\restriction_V\in\Cl(M)$. Let $D\in\Cl(M)$ be the divisor such that $D=\mathcal{L}\restriction_V$.
    
    Let $(V(t))^\bullet=V(t)\smallsetminus (M\wedge\NN^\infty)_+\cong U'$. Again from Lemma~\ref{lemma:Cl-M-wedge-N} we get that $\Cl((V(t))^\bullet)\cong\Cl(U')\cong \Cl(U'\times \AA^1)= \Cl(V)$.
    
    We have a diagram of inclusions
    \[
    \begin{tikzcd}[baseline=(current  bounding  box.center)]
    & V(t) \drar[hook] &\\
    (V(t))^\bullet \drar[hook] \urar[hook]& & U\\
    & V \urar[hook]&
    \end{tikzcd}
    \]
    
    Clearly, since $(V(t))^\bullet\subseteq V$ and $\mathcal{L}\restriction_V=D$, we have that $\mathcal{L}\restriction_{(V(t))^\bullet}=D$. On the other side, $V(t)\cong\Spec M$ is affine, so every line bundle on it is trivial. In particular, $\mathcal{L}\restriction_{V(t)}=0$. Since, lastly, $V(t)^\bullet\subseteq V(t)$, we have that $D=0$.
    
    Now, $W=U^{(1)}\cup U^{(0)}\subseteq V$, so in particular $V^{(1)}=U^{(1)}$ and we can apply Proposition~\ref{proposition:injection-PicU-PicU'} above to obtain an injection $\Pic(U)\hookrightarrow\Pic(V)$, but we just proved that the latter is trivial, so our claim follows.
\end{proof}

\begin{remark}
    The following picture of $\Spec(M\wedge \NN^\infty)$ might help in understanding what is going on in the proof of the Proposition above. For clarity, we represent $\Spec M$ with a plane, although it is just a finite number of points, and $V(M_+^{\mathrm{ext}})$ is represented by a line, although it is just the two points that lie on the planes.
    
    
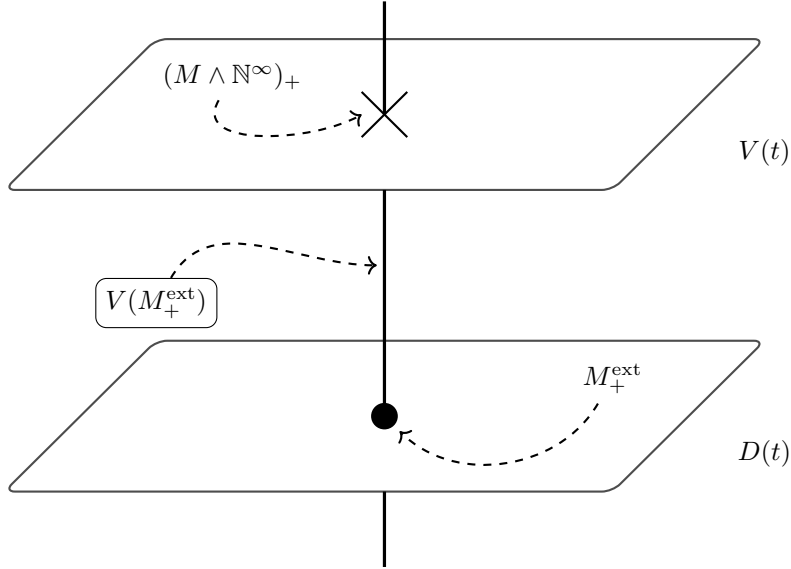
\vspace{1ex}
    
    \hspace{\mathindent}\begin{tikzpicture}[scale=1]
    \fill [white, rounded corners] (-4,2) -- (4,2) -- (6, 4) --(-2,4) -- cycle;
    \draw [grey1, thick, rounded corners] (-4,2) -- (4,2) -- (6, 4) --(-2,4) -- cycle;
    
    \node (Vt) at (6, 2.5){\small$V(t)$};

    \fill [white, rounded corners] (-4,-2) -- (4,-2) -- (6, 0) --(-2,0) -- cycle;
    \draw [grey1, thick, rounded corners] (-4,-2) -- (4,-2) -- (6, 0) --(-2,0) -- cycle;
    
    \node (Dt) at (6, -1.5){\small$D(t)$};
    
    \draw [black, very thick] (1,-1) -- (1,2);
    \draw [black, very thick] (1,-3) -- (1,-2);
    \draw [black, very thick] (1, 3) -- (1,4.5);
    
    \node[fill=white, draw=black, rounded corners] (VM+ext) at (-2, .5){\small$V(M_+^{\mathrm{ext}})$};
    \draw [black, dashed, thick, ->] (VM+ext) to [out=60,in=180] (1-0.1, 1);

    \fill[black](1, -1) circle[radius=5pt];
    \node (M+ext) at (4, -.5){\small$M_+^{\mathrm{ext}}$};
    \draw [black, dashed, thick, ->] (M+ext) to [out=-120,in=-45] (1+0.2, -1-0.2);
    
    \draw[black,thick] (1-.3,3-.3)--(1+.3,3+.3);
    \draw[black,thick] (1-.3,3+.3)--(1+.3,3-.3);
    \node (MN+) at (-1, 3.5){\small$(M\wedge \NN^\infty)_+$};                
    \draw [black, dashed, thick, ->] (MN+) to [out=-120,in=210] (1-0.3, 3);
    \end{tikzpicture}
    
    \vspace{1ex}
\end{remark}

\begin{example}
    Let $M=(x, y, z\mid x+y=nz)$. We know that it is regular in codimension 1 and we showed in Example~\ref{example:x+y=nz} that $\Pic(W)=\Cl(M\wedge \NN^\infty)=\Cl(M)=\ZZ_n$. On the other side, we just proved that $\Pic^{\loc}(M\wedge \NN^\infty)=0$. Morally, the problem is that $M$ is not regular at $M_+$, that has height 2, so $M\wedge \NN^\infty$ is not regular at $M_+^{\mathrm{ext}}\in D(t)\subseteq U$. A more geometric way to interpret this result comes from the fact that $M\wedge \NN^\infty$ is not an isolated singularity.
\end{example}

\begin{remark}
    This is not in contradiction with \cite[Corollary 5.5]{flores2014picard}, that states that for every binoid scheme $\Pic(X)\cong\Pic(X\times \AA^1)$, because\footnote{It is worth noting that this is true for any binoid scheme, without other hypothesis, unlike in the ring case.}
    \[
    \Spec^\bullet M \times \AA^1\ncong (\Spec M\times\AA^1)^\bullet=\Spec^\bullet(M\wedge \NN^1).
    \]
\end{remark}

\begin{remark}
    Proposition~\ref{proposition:injection-PicU-PicU'} gives us a way to extend the proof of Proposition~\ref{proposition:normal-dimension-2} to any binoid scheme, provided that there exists an open subset that contains all its prime ideals of height 1 that has trivial Picard group.
\end{remark}

The next Proposition shows that, under some regularity conditions, Picard group and Divisor Class group agree on subsets bigger than $W$.
\begin{proposition}\label{proposition:PicV-ClV}
    Let $V$ be an open subset of $\Spec M$ such that $M_\p$ is regular for all $\p\in V$. Then $\Pic(V)\cong\Cl(V)$.
\end{proposition}
\begin{proof}
    Thanks to Propositions~\ref{proposition:isomorphism-CaCl-Pic}, \ref{proposition:CaCl-Cl} and~\ref{proposition:injection-PicU-PicU'}, we know that there is an injective map
    \[
    \begin{tikzcd}[baseline=(current  bounding  box.center)]
    \Pic(V)\rar& \Cl(V)=\Pic(V\cap W).
    \end{tikzcd}
    \]
    To prove that it is surjective, we extend a line bundle on the right to a line bundle on the left.
    Let $\p_1, \dots, \p_r$ be the maximal prime ideals belonging to $V$ and let $D(f_i)$ be the minimal open subset that covers $\p_i$, so clearly $V\supseteq W\cup D(f_i)$.
    For any of these primes, let $N_i=M_{\p_i}$. By hypothesis, $N_i$ is regular, so by Lemma~\ref{lemma:regular-ClM0} we know that $\Cl(N_i)=0$. 
    Let $V_0=V\cap W$ and $V_i=V_{i-1}\cup D(f_i)$.
    If all the maximal prime ideals have height 1, then $V=V\cap W=V_0$ and then there is nothing to prove. So, assume that there exists at least one maximal prime ideal with height bigger than $1$, $\htnew(\p_i)\geq 2$, and assume that the Proposition is true for $V_{i-1}$.
    Let $\mathcal{L}'$ be a line bundle defined on $\Spec N_i\cap V_{i-1}$. Then we know that $\mathcal{L}'\cong 0$, and $\mathcal{L}'$ extends trivially to $\Spec N_i$.
    
    Let $\mathcal{L}_{i-1}$ be a line bundle defined on $V_{i-1}$, then $\mathcal{L}_{i-1}$ extends to a line bundle $\mathcal{L}_i$ on $V_i$, since $\mathcal{L}_{i-1}\restriction_{V_i}$ is trivial. So,by induction, every line bundle defined on $V_0=V\cap W$ extends to a line bundle on $W_r=V$. So the map above is surjective, and $\Pic(V)\cong\Cl(V)$.
\end{proof}

\begin{remark}
    The main point in the previous Proposition is whether we can extend a line bundle defined on $V\cap W$ to the whole $V$ or not, and in the regularity assumptions of the hypothesis, indeed we can.
\end{remark}

\begin{corollary}\label{corollary:isolated-singularity}
    Let $M$ be an isolated singularity of dimension at least 2. Then
    \[
    \Cl(M)\cong\Pic^{\loc}(M)
    \]
\end{corollary}
\begin{proof}
    Since it is an isolated singularity, its spectrum is regular outside the maximal point $M_+$, so we can apply the Proposition above with $V=\Spec^\bullet M$ to obtain our claim.
\end{proof}

Recall from Proposition~\ref{proposition:isomorphism-CaCl-Pic} that $\CaCl(V)\cong\Pic(V)$ for any open subset $V\subseteq \Spec M$.
If we don't have all the hypothesis that we assumed so far, Cartier divisors and Weil divisors might not agree in general, so we define the divisor class group of a general binoid as the Picard group of $W$.

\begin{definition}\label{definition:divisor-class-group-general}
    Let $M$ be any binoid and let $W$ be the set of its prime ideals of height at most one. We define its \emph{Divisor Class Group} $\Cl(M)$ to be $\Pic(W)$.
\end{definition}

\begin{example}
    Let $M_\triangle=(x,y,z\mid x+y+z=\infty)$. Its punctured spectrum $X=\Spec^\bullet M_\triangle$ is made of only prime ideals of height zero or one, so its Divisor Class Group is exactly its local Picard group, $\Cl(M)=\Pic^{\loc}(M)$. This is a special case of a more general property that holds for graphs, that we will exploit in Section~\ref{section:special-cases-simplicialcomplexes}.
\end{example}

\afterpage{\null\newpage}

\chapter{Simplicial Binoids}\label{chapter:simplcial-binoids}

In this chapter we will concentrate on the case of binoids arising from simplicial complexes, namely simplicial binoids, through the study of the relations between the subsets of the combinatorial spectrum and the simplicial complex from whom the binoid comes.
Then we will look at the sheaf of groups $\O^*_{M_\triangle}$ restricted to the quasi-affine case and, by exploiting some properties of the  \v{C}ech-Picard complex introduced in Definition~\ref{definition:cech-picard-complex}, we will provide explicit formulas for the computation of $\H^i(\Spec^\bullet M_\triangle, \O^*_{M_\triangle})$.
The last part of the Chapter is devoted to the computation of the divisor class group (as defined in Definition~\ref{definition:divisor-class-group-general}) of a simplicial binoid.

\section{The Spectrum of a Simplicial Binoid}

Recall that a \emph{simplicial complex} is a subset $\gls{simplicial-complex}$ of the power set of the finite \emph{vertex set} $V$ that is closed under taking subsets, i.e.\ $G\in\triangle$ and $F\subseteq G$ implies $F\in\triangle$. Its elements are called \emph{faces} and the maximal faces (under inclusion) are called \emph{facets}. The \emph{dimension} of a face is the number of vertices in it minus 1 and the dimension of $\triangle$ is the maximal dimension of its faces. A \emph{simplicial subcomplex} $\triangle'$ of $\triangle$ is a subset of $\triangle$ that is again a simplicial complex. If $W\subseteq V$ is a subset of the vertices the \emph{restriction of $\triangle$ to $W$} is $\gls{simplicial-complex-restriction}=\{F\in\triangle\mid F\subseteq W\}$. When we say that $\triangle$ is a simplicial complex on $V$, we assume, unless otherwise specified, that the singletons are faces, so $\{v\}\in\triangle$, for every $v\in V$.\footnote{Refer to \cite{BrunsHerzog} for a starting point about simplicial complexes and their role in algebra and geometry.}

\begin{definition}
    Let $\triangle$ be a simplicial complex on $V=[n]$. Its \emph{simplicial binoid} is the binoid with presentation
    \[
    \gls{simplicial-binoid}=\left(x_1, \dots, x_n\midd x_{i_1}+\dots +x_{i_j}=\infty \text{ where } \{i_1, \dots, i_j\}\in \mathcal{P}([n])\smallsetminus \triangle\right).
    \]
\end{definition}

\begin{definition}\label{definition:simplicial-binoid-emptyset}
    A special situation arises when $\triangle=\varnothing$, in which case we obtain the zero binoid    
    $\gls{simplicial-binoid-empty-set}:=\gls{zero-binoid}$.
\end{definition}

Simone Böttger, in \cite[Corollary 6.5.13]{Boettger}, proved that there exists an order-reversing correspondence between faces of the simplicial complex and prime ideals of the binoid. In particular, the minimal prime ideals correspond to the (complements of) facets.

\begin{remark}[{\cite[Corollary 6.5.13]{Boettger}}]\label{remark:boettger:correspondence-faces-prime-ideals}
    Let $\triangle$ be a simplicial complex on the vertex set $V$. There exists an inclusion-reversing semibinoid isomorphism\footnote{A semibinoid is a binoid in which there might not be a neutral element.}
    \begin{equation}
    \begin{tikzcd}[baseline=(current  bounding  box.center), cramped, row sep = 0ex,
    /tikz/column 1/.append style={anchor=base east},
    /tikz/column 2/.append style={anchor=base west}]
    (\triangle, \cap, \varnothing) \arrow[r, leftrightarrow, "\sim"] & (\Spec M_\triangle, \cup, M_{\triangle,+})\\
    F         \arrow[r, mapsto] & \langle x_i\mid i\in V\smallsetminus F\rangle
    \end{tikzcd}
    \end{equation}
\end{remark}

\begin{example}\label{example:spectrum-simplicial-complex}
    Let $V=[4]$ and let
    
    \begin{minipage}[c]{0.5\textwidth}
        \[
        \triangle = \left\{\begin{aligned}
        &\varnothing,\{1\},\{2\},\{3\},\{4\},\\
        &\{1, 2\}, \{1, 3\}, \{2, 3\}, \{3, 4\}\\
        &\{1, 2, 3\}
        \end{aligned}\right\}
        \]\vfill
    \end{minipage}\begin{minipage}[c]{0.5\textwidth}
    
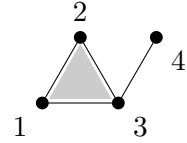
\hspace{6em}\begin{tikzpicture}
    \tikzstyle{point}=[circle,thick,draw=black,fill=black,inner sep=0pt,minimum width=4pt,minimum height=4pt]
    \node (v1)[point, label={[label distance=0cm]-135:$1$}] at (0,0) {};
    \node (v2)[point, label={[label distance=0cm]90:$2$}] at (0.5,0.87) {};
    \node (v3)[point, label={[label distance=0cm]-45:$3$}] at (1,0) {};
    \node (v4)[point, label={[label distance=0cm]-45:$4$}] at (1.5,0.87) {};
    
    \draw (v1.center) -- (v2.center);
    \draw (v1.center) -- (v3.center);
    \draw (v2.center) -- (v3.center);
    \draw (v3.center) -- (v4.center);
    
    \draw[color=black!20, style=fill, outer sep = 20pt] (0.1,0.06) -- (0.5,0.77) -- (0.9,0.06) -- cycle;
    \end{tikzpicture}
\end{minipage}

Its associated simplicial binoid is 
\[
M_\triangle=\left(x_1, x_2, x_3 ,x_4\mid x_1+x_4=\infty, x_2+x_4=\infty\right)
\]
whose spectrum is
\[
\Spec M_\triangle=\left\{\begin{aligned}
&\langle x_4\rangle, \langle x_1, x_4\rangle, \langle x_2, x_4\rangle, \langle x_3, x_4\rangle, \langle x_1, x_2\rangle,\\
&\langle x_1, x_2, x_3\rangle, \langle x_1, x_2, x_4\rangle, \langle x_1, x_3, x_4\rangle, \langle x_2, x_3, x_4\rangle\\
&\langle x_1, x_2, x_3, x_4\rangle
\end{aligned}\right\}
\]
that we can represent as a $\subseteq$-poset as follows

\vspace{1ex}

\hspace{\mathindent}$\Spec M_\triangle=$
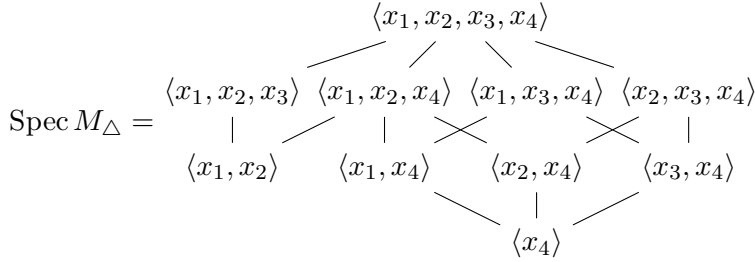
\begin{tikzpicture}[baseline={(current bounding box.center)}]
\node (1234) at (0,0){$\langle x_1,x_2,x_3, x_4\rangle$};
\node (123) at (-3,-1){$\langle x_1, x_2, x_3\rangle$};
\node (124) at (-1,-1){$\langle x_1, x_2, x_4\rangle$};
\node (134) at (1,-1){$\langle x_1, x_3, x_4\rangle$};
\node (234) at (3,-1){$\langle x_2, x_3, x_4\rangle$};
\node (12) at (-3,-2){$\langle x_1, x_2\rangle$};
\node (14) at (-1,-2){$\langle x_1, x_4\rangle$};
\node (24) at (1,-2){$\langle x_2, x_4\rangle$};
\node (34) at (3,-2){$\langle x_3, x_4\rangle$};
\node (4) at (1,-3){$\langle x_4\rangle$};

\draw[-](1234)--(123);
\draw[-](1234)--(124);
\draw[-](1234)--(134);
\draw[-](1234)--(234);
\draw[-](123)--(12);
\draw[-](124)--(12);
\draw[-](124)--(14);
\draw[-](124)--(24);
\draw[-](134)--(14);
\draw[-](134)--(34);
\draw[-](234)--(24);
\draw[-](234)--(34);
\draw[-](14)--(4);
\draw[-](24)--(4);
\draw[-](34)--(4);
\end{tikzpicture}
\vspace{1ex}

and see that it is the same as the $\supseteq$-poset defined by $\triangle$

\vspace{1ex}

\hspace{\mathindent}\hspace{2.38em}$\triangle_\supseteq=$
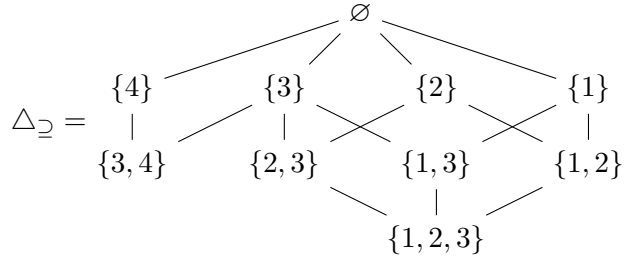
\begin{tikzpicture}[baseline={(current bounding box.center)}]
\node (1234) at (0,0){$\varnothing$};
\node (123) at (-3,-1){$\{4\}$};
\node (124) at (-1,-1){$\{3\}$};
\node (134) at (1,-1){$\{2\}$};
\node (234) at (3,-1){$\{1\}$};
\node (12) at (-3,-2){$\{3, 4\}$};
\node (14) at (-1,-2){$\{2, 3\}$};
\node (24) at (1,-2){$\{1, 3\}$};
\node (34) at (3,-2){$\{1, 2\}$};
\node (4) at (1,-3){$\{1, 2, 3\}$};

\draw[-](1234)--(123);
\draw[-](1234)--(124);
\draw[-](1234)--(134);
\draw[-](1234)--(234);
\draw[-](123)--(12);
\draw[-](124)--(12);
\draw[-](124)--(14);
\draw[-](124)--(24);
\draw[-](134)--(14);
\draw[-](134)--(34);
\draw[-](234)--(24);
\draw[-](234)--(34);
\draw[-](14)--(4);
\draw[-](24)--(4);
\draw[-](34)--(4);
\end{tikzpicture}

\vspace{1ex}
\end{example}

\subsection{Closed and Open Subsets of \texorpdfstring{$\Spec M_\triangle$}{spectrum of a simplicial binoid}}

We want to study the subsets of $\Spec M$ in the simplicial case. Recall from Proposition~\ref{remark:boettger:closedsubset-closedsuperset} that a subset of the spectrum of any binoid is closed if and only if it is superset-closed. This Proposition and the correspondence in Remark~\ref{remark:boettger:correspondence-faces-prime-ideals} induce a correspondence between simplicial subcomplexes of $\triangle$ and closed subsets of $\Spec M_\triangle$ that goes as follows
\begin{equation*}
\begin{tikzcd}[baseline=(current  bounding  box.center), row sep = -1em,
/tikz/column 1/.append style={anchor=base east},
/tikz/column 2/.append style={anchor=base west},
/tikz/column 3/.append style={anchor=base west},
cramped]
\begin{array}{c}
\text{Closed subsets}\\
\text{of }\Spec M_\triangle
\end{array}\rar[leftrightarrow] &\begin{array}{c}
\text{Radical ideals of }M_\triangle
\end{array}\rar[leftrightarrow]  &
\begin{array}{c}
\text{Subcomplexes of }\triangle
\end{array}\\
V\rar[leftrightarrow] & \begin{array}{c}
I = \p_1\cap\dots\cap\p_r\\
\p_i \text{ minimal prime in } V
\end{array}\rar[leftrightarrow]  &\{F_1, \dots, F_r\}
\end{tikzcd}
\end{equation*}
Where $\{F_1, \dots, F_r\}$ are the facets of the corresponding simplicial subcomplex.
In particular, $V(x_i)$ corresponds to the subsimplicial complex $\triangle'=\triangle_{[n]\smallsetminus \{i\}}$, restriction of $\triangle$ to $[n]\smallsetminus \{i\}$.

\begin{example}\label{example:spectrum-simplicial-binoid-2}
    Going back to the previous Example~\ref{example:spectrum-simplicial-complex}, we can easily see, for example, that $V(x_1)$ corresponds to $\triangle_{\{2, 3, 4\}}$
    
    \vspace{1ex}
    
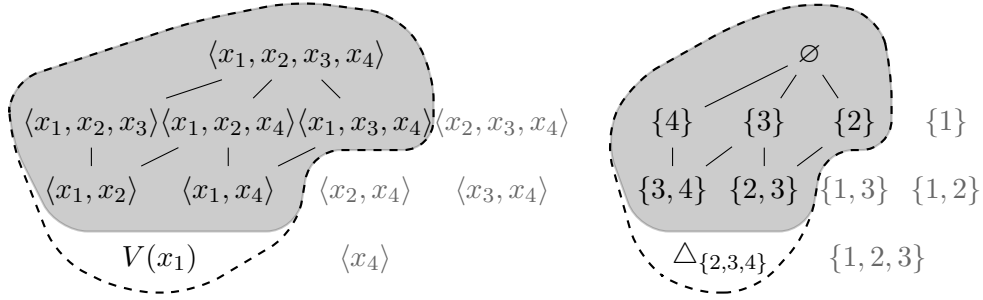
    \begin{minipage}[c]{.6\textwidth}
        \hspace{\mathindent}\begin{tikzpicture}[scale=0.9,baseline={(current bounding box.center)}]
        \node (1234) at (0,0){$\langle x_1,x_2,x_3, x_4\rangle$};
        \node (123) at (-3,-1){$\langle x_1, x_2, x_3\rangle$};
        \node (124) at (-1,-1){$\langle x_1, x_2, x_4\rangle$};
        \node (134) at (1,-1){$\langle x_1, x_3, x_4\rangle$};
        \node[opacity=.6] (234) at (3,-1){$\langle x_2, x_3, x_4\rangle$};
        \node (12) at (-3,-2){$\langle x_1, x_2\rangle$};
        \node (14) at (-1,-2){$\langle x_1, x_4\rangle$};
        \node[opacity=.6] (24) at (1,-2){$\langle x_2, x_4\rangle$};
        \node[opacity=.6] (34) at (3,-2){$\langle x_3, x_4\rangle$};
        \node[opacity=.6] (4) at (1,-3){$\langle x_4\rangle$};
        \node (Vx1) at (-2, -3){$V(x_1)$};

        \draw[-](1234)--(123);
        \draw[-](1234)--(124);
        \draw[-](1234)--(134);
        \draw[-](123)--(12);
        \draw[-](124)--(12);
        \draw[-](124)--(14);
        \draw[-](134)--(14);
        
        \draw[black, rounded corners=4mm, thick, fill, opacity=0.2](-3.5, -2.6)--(-4.2, -1.2)--(-4.2, -.4)--(-1.5, .5)--(0,.8)--(1.5,.6)--(2,-.2)--(2,-1.4)--(.2,-1.4)--(0,-2.6)--cycle;
        \draw[black, rounded corners=4mm, thick, dashed](-1.5, -3.5)--(-2.5,-3.5)--(-3.5, -3)--(-4.2, -1.2)--(-4.2, -.4)--(-1.5, .5)--(0,.8)--(1.5,.6)--(2,-.2)--(2,-1.4)--(.2,-1.4)--(0,-2.6)--cycle;
        \end{tikzpicture}
    \end{minipage}
    \begin{minipage}[c]{.4\textwidth}
        \begin{tikzpicture}[scale=0.9,baseline={(current bounding box.center)}]
        \node (1234) at (0,0){$\varnothing$};
        \node (123) at (-2,-1){$\{4\}$};
        \node (124) at (-0.66,-1){$\{3\}$};
        \node (134) at (0.66,-1){$\{2\}$};
        \node[opacity=.6] (234) at (2,-1){$\{1\}$};
        \node (12) at (-2,-2){$\{3, 4\}$};
        \node (14) at (-0.66,-2){$\{2, 3\}$};
        \node[opacity=.6] (24) at (0.66,-2){$\{1, 3\}$};
        \node[opacity=.6] (34) at (2,-2){$\{1, 2\}$};
        \node[opacity=.6] (4) at (1,-3){$\{1, 2, 3\}$};
        \node (Vx1) at (-1.3, -3){$\triangle_{\{2, 3, 4\}}$};
        
        \draw[-](1234)--(123);
        \draw[-](1234)--(124);
        \draw[-](1234)--(134);
        \draw[-](123)--(12);
        \draw[-](124)--(12);
        \draw[-](124)--(14);
        \draw[-](134)--(14);
        
        \draw[black, rounded corners=4mm, thick, fill, opacity=0.2](-2.5, -2.6)--(-3, -1.2)--(-2.6, -.4)--(-1, .5)--(0,.8)--(1,.6)--(1.2,-.2)--(1.2,-1.4)--(.2,-1.4)--(0,-2.6)--cycle;
        \draw[black, rounded corners=4mm, thick, dashed](-0.8, -3.5)--(-1.8,-3.5)--(-2.5, -3)--(-3, -1.2)--(-2.6, -.4)--(-1, .5)--(0,.8)--(1,.6)--(1.2,-.2)--(1.2,-1.4)--(.2,-1.4)--(0,-2.6)--cycle;
        \end{tikzpicture}
    \end{minipage}
    \vspace{1ex}
\end{example}

Since our goal is to discuss sheaves and, in particular, to build the \v{C}ech complex of the sheaf of units on the punctured spectrum, we are interested in study open subsets of $\Spec M_\triangle$.

Thanks to the correspondence between closed subsets and simplicial subcomplexes, we know that an open subset corresponds to the complement of a simplicial subcomplex, but we would like to have something more. We begin by considering the fundamental affine open subsets, that we define in Definition~\ref{definition:zariski-topology-combinatorial-spectrum} as
\[
D(f):=\left\{\p\in\Spec M\mid f\notin \p\right\}.
\]
In the simplicial case that we are considering now, this might be exploited better as follows.
We start by recalling that a simplicial binoid is semifree and reduced. See Definition~\ref{definition:semifree} and~\ref{definition:reduced} for the definitions of semifree and reduced respectively.

\begin{theorem}[{\cite[Theorem 6.5.8]{Boettger}}]Let $\triangle$ be a simplicial complex on $V$. The binoid $M_\triangle$ is finitely generated by $\#V=n$ elements, semifree and reduced. Conversely, every commutative binoid $M$ satisfying these properties is a simplicial binoid. More precisely, $M$ is isomorphic to $M_\triangle$, with $\triangle=\{F\subseteq W\mid\sum_{w\in F}w\neq\infty\}$ for a minimal generating set $W$.
\end{theorem}

\begin{definition}\label{definition:reduction-of-element}
    Let $f\in \left(M_\triangle\right)^\bullet$ be a non-infinity element in a simplicial binoid.
    
    Its \textbf{reduction $\gls{red}(f)$} is \[
    \displaystyle\red(f):=\sum_{x_i\in\supp(f)}x_i.
    \]
\end{definition}

\begin{remark}\label{remark:Df=Dintersectionsupport} It is the same to look at the affine fundamental open set defined by $f$ or by $\red(f)$, since we have
    \[
    D(f)=D(\red(f))=\bigcap_{x_i\in\supp(f)}(D(x_i)),
    \]
    that follows from three facts:\\
    \hspace{\mathindent} 1. $M_\triangle$ is semifree,\\
    \hspace{\mathindent} 2. $\supp(f)=\supp(\red(f))$,\\
    \hspace{\mathindent} 3. $f\notin \p=\langle x_{i_1}, \dots, x_{i_k}\rangle$ if and only if $x_{i_j}\notin\supp(f)$ for every $i_j$
\end{remark}

Thanks to this remark, we can not only restrict to consider open affine subsets defined by reduced elements but, moreover, to consider intersections of the fundamental open subsets defined by the variables $x_i$'s.

\begin{example}\label{example:spectrum-simplicial-binoid-3}
    We go back to Example~\ref{example:spectrum-simplicial-binoid-2} and we try to understand $D(x_1+x_3)$. This is, indeed, the intersection of $D(x_1)$ and $D(x_3)$
    
    
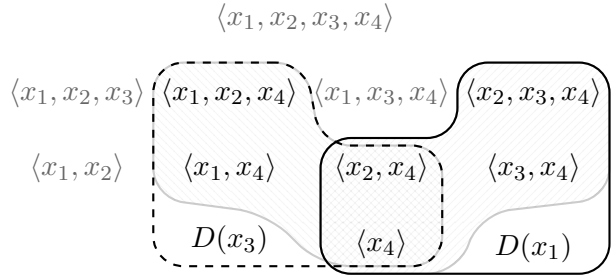
\vspace{1ex}
    
    \hspace{\mathindent}
    \begin{tikzpicture}[baseline={(current bounding box.center)}]
    \node[opacity=0.6] (1234) at (0,0){$\langle x_1,x_2,x_3, x_4\rangle$};
    \node[opacity=0.6] (123) at (-3,-1){$\langle x_1, x_2, x_3\rangle$};
    \node (124) at (-1,-1){$\langle x_1, x_2, x_4\rangle$};
    \node[opacity=0.6] (134) at (1,-1){$\langle x_1, x_3, x_4\rangle$};
    \node (234) at (3,-1){$\langle x_2, x_3, x_4\rangle$};
    \node[opacity=0.6] (12) at (-3,-2){$\langle x_1, x_2\rangle$};
    \node (14) at (-1,-2){$\langle x_1, x_4\rangle$};
    \node (24) at (1,-2){$\langle x_2, x_4\rangle$};
    \node (34) at (3,-2){$\langle x_3, x_4\rangle$};
    \node (4) at (1,-3){$\langle x_4\rangle$};
    \node (Dx1) at (3, -3.05){$D(x_1)$};
    \node (Dx3) at (-1, -2.95){$D(x_3)$};
    
    
    \draw[black, rounded corners=4mm, thick, pattern=north east lines, pattern color=black, opacity=0.2](4, -.6)--(2,-.6)--(2,-1.6)--(.2,-1.6)--(.2,-3.4)--(2, -3.4)--(2.3, -2.6)--(4, -2.4)--cycle;
    \draw[black, rounded corners=4mm, thick](4, -.6)--(2,-.6)--(2,-1.6)--(.2,-1.6)--(.2,-3.4)--(2, -3.4)--(4, -3.4)--cycle;
    
    \draw[black, rounded corners=4mm, thick, pattern=north west lines, pattern color=black, opacity=0.2](0, -.6)--(-2,-.6)--(-2,-2.4)--(-.5, -2.6)--(.2,-3.3)--(1.8, -3.3)--(1.8, -1.7)--(.2,-1.7)--cycle;
    \draw[black, rounded corners=4mm, thick, dashed](0, -.6)--(-2,-.6)--(-2,-3.3)--(1.8, -3.3)--(1.8, -1.7)--(.2,-1.7)--cycle;
    \end{tikzpicture}
    
    \vspace{3ex}
    
    So $D(x_1+x_3)$ is $\{\langle x_2, x_4\rangle, \langle x_4\rangle\}$. Its complement corresponds to the simplicial complex $\triangle'=\triangle\smallsetminus\{\{1, 3\}, \{1, 2, 3\}\}$.
\end{example}

\subsection{Covering \texorpdfstring{$\Spec^\bullet{M_\triangle}$}{spectrum of a simplicial binoid}}

As we have seen in Proposition~\ref{proposition:covering-punctured-spectrum-binoid} we can cover the punctured spectrum of any binoid with the  affine fundamental open subsets $\{D(x_i)\}$. While in general it might be the case that we don't need all the variables, for simplicial binoids we do.\\
With $\langle x_1, \dots, \widehat{x_i}, \dots, x_n\rangle$ we denote the prime ideal generated by all the variables except $x_i$, in a way similar to the standard notation for sets.

\begin{lemma}
    Let $\triangle$ be a simplicial complex on $V=[n]$. Then $\langle x_1, \dots, \widehat{x_i}, \dots, x_n\rangle\in\Spec^\bullet M_\triangle$ for every $i\in V$.
\end{lemma}
\begin{proof}
    Thanks to the correspondence stated in Remark~\ref{remark:boettger:correspondence-faces-prime-ideals}, it is enough to notice that $\{i\}$ is a face of $\triangle$ for every $i\in V$.
\end{proof}

\begin{proposition}\label{proposition:minimal-covering}
    $\{D(x_i)\}$ is a covering by affine subsets of $\Spec^\bullet M_\triangle$ that is minimal among all the possible affine coverings.
\end{proposition}
\begin{proof}
    We already know that this is a covering. To prove that it is minimal, it is enough to observe three things.\\
    First, thanks to the lemma above, we need to cover $\langle x_1, \dots, \widehat{x_i}, \dots, x_n\rangle$.\\
    Second, $D(x_i)$ is the only affine open subset here that covers $\langle x_1, \ldots, \widehat{x_i}, \ldots, x_n\rangle$.\\
    Third, thanks to Remark~\ref{remark:Df=Dintersectionsupport}, any other affine open subset that covers $\langle x_1, \ldots, \widehat{x_i}, \ldots, x_n\rangle$ needs to come from an element $f$ that has the same support of $x_i$, i.e.\ $f=mx_i$, for some $m\in\NN$, and $D(f)=D(x_i)$.
\end{proof}

\begin{proposition}\label{proposition:intersection-open-subsets}
    Let $\{D(x_i)\}$ be the covering of $\Spec^\bullet M_\triangle$ as above. Then $D(x_{i_1})\cap\dots\cap D(x_{i_j})\neq \varnothing$ if and only if $\{i_1, \dots, i_j\}$ is a face of $\triangle$.
\end{proposition}
\begin{proof}
    Thanks to the correspondence between prime ideals and faces, we can describe $D(x_{i_1})\cap\dots\cap D(x_{i_j})$ in term of faces, in the following way. $D(x_i)=\{\p\in\Spec M_\triangle\mid x_i\notin\p\}=\{F\in\triangle\mid i\in F\}$. Then $D(x_{i_1})\cap\dots\cap D(x_{i_j}) = \{F\in\triangle\mid \{i_1, \dots, i_j\}\subseteq F\}$.\\
    If $\{i_1, \dots, i_j\}$ is a face of $\triangle$ then this set is not empty. On the other hand, if this set is non empty, then there exists a face $G\in\triangle$ such that $\{i_1, \dots, i_j\}\subseteq G$. Since a subset of a face is a face, $\{i_1, \dots, i_j\}$ is itself a face.
\end{proof}
\begin{definition}Thanks to the previous Proposition, we can introduce the following notation. Let $F=\{i_1, \dots, i_j\}\in\triangle$, then 
    \[
    \gls{D-of-face}:=D(x_{i_1}+ \dots+ x_{i_j})=\bigcap_{i\in F}D(x_i).
    \]
\end{definition}

\begin{definition}\label{definition:nerve-of-covering}
    Let $\{U_i\}_{i\in I}$ be a finite collection of open subsets of a topological space $X$. The \emph{nerve of $\{U_i\}$} is the simplicial complex, written $\gls{nerve-U}$, defined on vertex set $I$ as follows
    \begin{itemize}
        \item $\varnothing\in\nerve(\{U_i\})$,
        \item for any $J\subseteq I$, $J\in \nerve(\{U_i\})$ if and only if $\displaystyle\bigcap_{j\in J}U_j\neq \varnothing$.
    \end{itemize}
\end{definition}

\begin{corollary}\label{corollary:cech-covering-nerve}
    The nerve of the covering of $\Spec^\bullet M_\triangle$ given by $\{D(x_i)\}$ is the simplicial complex itself.
\end{corollary}

The above Definition can be restated using directly $I$ as vertex set, since $V$ and $I$ have the same cardinality, but in this form is more clear that we are defining a new simplicial complex which, unless we are in the case of the Corollary, will be different from what we started with.

The following Lemma generalizes Proposition~\ref{proposition:minimal-covering}.

\begin{lemma}\label{lemma:minimal-covering}
    Let $U\subseteq\Spec^\bullet(M_\triangle)$ be an open subset. Then $U$ can be minimally covered by $\mathscr{V}=\{D(F_1), \dots, D(F_j)\}$ for some $F_1, \dots, F_j\in\triangle$.
\end{lemma}
\begin{proof}
    Let $\p_1, \dots, \p_k$ be the maximal prime ideals in $U$. Then $F_1, \dots, F_j$ are the faces corresponding to these prime ideals via the map in Remark~\ref{remark:boettger:correspondence-faces-prime-ideals}. Minimality can be proved as in Proposition~\ref{proposition:minimal-covering}.
\end{proof}

\subsection{Constant sheaves}
In this Section we explore the relation between \v{C}ech cohomology on the covering $\{D(x_i)\}$ of a constant sheaf of abelian groups on $\Spec^\bullet M_\triangle$ and simplicial cohomology of $\triangle$ with coefficients in that group.

\begin{remark}
    Let $G$ be a constant sheaf of abelian groups on a subset $U$ of $\Spec M$. We already know that
    \[
    \Gamma(U, G)=\left\{\begin{aligned}
    &G^{\#\{\text{connected components of }U\}},    & \text{ if } U\neq\varnothing,\\
    &0,                                             & \text{ otherwise}.
    \end{aligned}
    \right.
    \]
    Moreover, if $\{D(x_i)\}$ is the usual covering of the punctured spectrum given by the combinatorial open subsets, we can use it to compute cohomology via \v{C}ech cohomology, and we can explicitly write the groups in the \v{C}ech complex as
    \[
    \vC\left(D(X_{i_0})\cap \dots\cap D(x_{i_k}), G\right)=\left\{\begin{aligned}
    &G,    & \text{ if } D(X_{i_0})\cap \dots\cap D(x_{i_k})\neq\varnothing,\\
    &0,                                             & \text{ otherwise},
    \end{aligned}
    \right.
    \]
    because this intersection is either empty or connected.
\end{remark}

In the proof of the next Proposition, we will compare some \v{C}ech complexes with the complex used by Miller and Sturmfels in \cite[Section 1.3]{miller2005combinatorial}. There, they study reduced simplicial cohomology with coefficients in a field $\KK$ and explicitly write the maps between the groups. They study reduced simplicial cohomology because they talk about Betti numbers and Hochster formula. Here, instead, we are interested in the usual simplicial cohomology , so we just cut their simplicial complex and keep only the parts with non negative degree.

\begin{theorem}\label{thm:simplicial-cohomology}
    Let $\triangle$ be a simplicial complex on $V=[n]$. Let $\{D(x_i), 1\leq i\leq n\}$ be the usual acyclic covering of $\Spec^\bullet M_\triangle$ and let $\ZZ$ be the constant sheaf on this space. Then \v{C}ech cohomology and simplicial cohomology are described by the same chain complexes
    \[
    \glslink{simplicial-chain-cpx}{\C^\bullet(\triangle, \ZZ)}=\glslink{cech-chain-cpx}{\vC^\bullet(\{D(x_i)\}, \ZZ)}.
    \]
    In particular, the cohomology groups are the same
    \[
    \glslink{simplicial-cohomology}{\H^i(\triangle, \ZZ)}=\vH^i(\{D(x_i)\}, \ZZ)
    \]for all $i\geq 0$.
\end{theorem}
\begin{proof}
    We understand what we have on the right. Thanks to the previous Lemmata, $\vC^j(\{D(x_i)\}, \ZZ)=\ZZ^{\triangle_j}$, i.e. we have a $\ZZ$ for any face in $\triangle$ of dimension $j$. The maps are the usual maps of the \v{C}ech complex.
    
    On the left hand side, we can follow a reasoning similar to \cite[Section 1.3]{miller2005combinatorial}.
    
    The first thing to note is that the only difference between the complex for simplicial cohomology stated here and the one stated by them is the degree $-1$, since they are considering reduced simplicial cohomology.
    
    We can now easily see that the complex $\C^\bullet(\triangle, \ZZ)$, dual to the homology complex $\C_\bullet(\triangle, \ZZ)$, has the same groups as $\vC^\bullet(\{D(x_i)\}, \ZZ)$: in every degree $j\geq0$ this group is $\ZZ^{\triangle_j}$.
    
    As for the maps, it is again easy to see that the map for vector spaces that Miller and Sturmfels describe in their book, can be written instead for $\ZZ$ and, when we restrict their complex to the non negative degrees, it is exactly the map of the \v{C}ech complex described above.
\end{proof}

\begin{corollary}
    The same holds for any abelian group $G$, since $\Hom_\ZZ(\ZZ, G)\cong G\cong \ZZ\otimes_\ZZ G$ and so $\Hom(\C^\bullet, G)\cong \C^\bullet\otimes _\ZZ G$. In particular, we will use this fact in Chapter~\ref{chapter:SR-rings}, where the constant sheaf $\KK^*$ will contribute to the cohomology of the units of a ring.
\end{corollary}

\begin{corollary}\label{corollary:sheaf-cohomology}
    Since $\{D(x_i)\}$ is an acyclic covering of $\Spec^\bullet M_\triangle$ for every sheaf of abelian groups, the cohomology in the theorem above is also equal to the sheaf cohomology $\H^i(\Spec^\bullet M_\triangle, \ZZ)$.
\end{corollary}
The previous Corollary relates sheaf cohomology, \v{C}ech cohomology and simplicial cohomology in the case of a simplicial binoid. The next one, extends these results to any open subset of the spectrum.

\begin{corollary}\label{corollary:cech-covering-nerve-simplicial-cohomology}
    Let $U\subseteq\Spec^\bullet(M_\triangle)$ be an open subset minimally covered by the covering $\mathscr{V}=\{D(F_1), \dots, D(F_j)\}$ for some $F_1, \dots, F_j\in\triangle$. We have
    \[
    \H^i(U, \ZZ)\cong{\vH^i({\mathscr{V}}, \ZZ)}\cong{\H^i(\nerve(\mathscr{V}), \ZZ)}.
    \]
\end{corollary}
\begin{proof}
    The first isomorphism is easy because $D(F)$ is affine and $\ZZ$ is a sheaf of abelian groups, hence acyclic on the covering $\mathscr{V}$. Moreover, thanks to the previous Theorem, we know that $\H^i(\nerve(\mathscr{V}), \ZZ) = \H^i(\Spec^\bullet M_{\nerve(\mathscr{V})}, \ZZ)$. It is enough to show that $\Spec^\bullet M_{\nerve(\mathscr{V})}\cong U$ as topological spaces. This is easily done thanks to the correspondences above between prime ideals and faces of the simplicial complex.
\end{proof}

\begin{example}
    Let us consider again \\    
    \begin{minipage}[c]{0.5\textwidth}
        \[
        \triangle = \left\{\begin{aligned}
        &\varnothing,\{1\},\{2\},\{3\},\{4\},\\
        &\{1, 2\}, \{1, 3\}, \{2, 3\}, \{3, 4\}\\
        &\{1, 2, 3\}
        \end{aligned}\right\}
        \]\vfill
    \end{minipage}\begin{minipage}[c]{0.5\textwidth}
    \hspace{6em}\begin{tikzpicture}
    \tikzstyle{point}=[circle,thick,draw=black,fill=black,inner sep=0pt,minimum width=4pt,minimum height=4pt]
    \node (v1)[point, label={[label distance=0cm]-135:$1$}] at (0,0) {};
    \node (v2)[point, label={[label distance=0cm]90:$2$}] at (0.5,0.87) {};
    \node (v3)[point, label={[label distance=0cm]-45:$3$}] at (1,0) {};
    \node (v4)[point, label={[label distance=0cm]-45:$4$}] at (1.5,0.87) {};
    
    \draw (v1.center) -- (v2.center);
    \draw (v1.center) -- (v3.center);
    \draw (v2.center) -- (v3.center);
    \draw (v3.center) -- (v4.center);
    
    \draw[color=black!20, style=fill, outer sep = 20pt] (0.1,0.06) -- (0.5,0.77) -- (0.9,0.06) -- cycle;
    \end{tikzpicture}
\end{minipage}

The pointed spectrum of the associated binoid is
\[
\Spec^\bullet M_\triangle=\left\{\begin{aligned}
&\langle x_4\rangle, \langle x_1, x_4\rangle, \langle x_2, x_4\rangle, \langle x_3, x_4\rangle, \langle x_1, x_2\rangle,\\
&\langle x_1, x_2, x_3\rangle, \langle x_1, x_2, x_4\rangle, \langle x_1, x_3, x_4\rangle, \langle x_2, x_3, x_4\rangle
\end{aligned}\right\}
\]
as we have seen in Example~\ref{example:spectrum-simplicial-complex}. Consider the open subset \[U=\{\langle x_4\rangle, \langle x_1, x_2\rangle, \langle x_1, x_4\rangle, \langle x_2, x_4\rangle, \langle x_3, x_4\rangle,\langle x_2, x_3, x_4\rangle\}\] as shown in picture

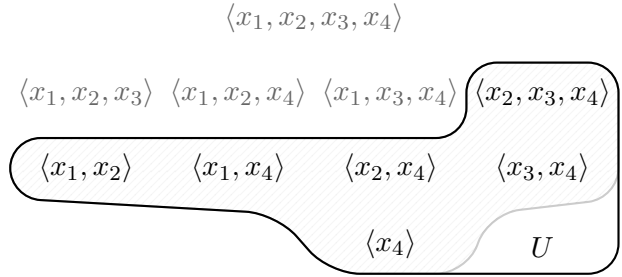
\vspace{1ex}

\hspace{\mathindent}
\begin{tikzpicture}[baseline={(current bounding box.center)}]
\node[opacity=0.6] (1234) at (0,0){$\langle x_1,x_2,x_3, x_4\rangle$};
\node[opacity=0.6] (123) at (-3,-1){$\langle x_1, x_2, x_3\rangle$};
\node[opacity=0.6] (124) at (-1,-1){$\langle x_1, x_2, x_4\rangle$};
\node[opacity=0.6] (134) at (1,-1){$\langle x_1, x_3, x_4\rangle$};
\node (234) at (3,-1){$\langle x_2, x_3, x_4\rangle$};
\node (12) at (-3,-2){$\langle x_1, x_2\rangle$};
\node (14) at (-1,-2){$\langle x_1, x_4\rangle$};
\node (24) at (1,-2){$\langle x_2, x_4\rangle$};
\node (34) at (3,-2){$\langle x_3, x_4\rangle$};
\node (4) at (1,-3){$\langle x_4\rangle$};
\node (Dx1) at (3, -3.05){$U$};


\draw[black, rounded corners=4mm, thick, pattern=north east lines, pattern color=black, opacity=0.2](4, -.6)--(2,-.6)--(2,-1.6)--(-4,-1.6)--(-4,-2.4)--(-.5, -2.6)--(.2,-3.4)--(2, -3.4)--(2.3, -2.6)--(4, -2.4)--cycle;
\draw[black, rounded corners=4mm, thick](4, -.6)--(2,-.6)--(2,-1.6)--(-4,-1.6)--(-4,-2.4)--(-.5, -2.6)--(.2,-3.4)--(4, -3.4)--cycle;
\end{tikzpicture}

\vspace{3ex}

This can be minimally covered by the affine covering
\[\mathscr{V}=\{D(x_3+ x_4), D(x_2+ x_3), D(x_1)\}=\{D(\{3, 4\}), D(\{2, 3\}), D(\{1\})\}\]
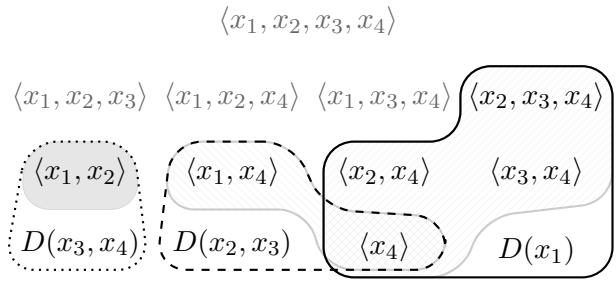
\hspace{\mathindent}
\begin{tikzpicture}[baseline={(current bounding box.center)}]
\node[opacity=0.6] (1234) at (0,0){$\langle x_1,x_2,x_3, x_4\rangle$};
\node[opacity=0.6] (123) at (-3,-1){$\langle x_1, x_2, x_3\rangle$};
\node[opacity=0.6] (124) at (-1,-1){$\langle x_1, x_2, x_4\rangle$};
\node[opacity=0.6] (134) at (1,-1){$\langle x_1, x_3, x_4\rangle$};
\node (234) at (3,-1){$\langle x_2, x_3, x_4\rangle$};
\node (12) at (-3,-2){$\langle x_1, x_2\rangle$};
\node (14) at (-1,-2){$\langle x_1, x_4\rangle$};
\node (24) at (1,-2){$\langle x_2, x_4\rangle$};
\node (34) at (3,-2){$\langle x_3, x_4\rangle$};
\node (4) at (1,-3){$\langle x_4\rangle$};
\node (Dx1) at (3, -3.05){$D(x_1)$};
\node (Dx23) at (-1, -2.95){$D(x_2, x_3)$};
\node (Dx34) at (-3, -2.95){$D(x_3, x_4)$};


\draw[black, rounded corners=4mm, thick, pattern=north east lines, pattern color=black, opacity=0.2](4, -.6)--(2,-.6)--(2,-1.6)--(.2,-1.6)--(.2,-3.4)--(2, -3.4)--(2.3, -2.6)--(4, -2.4)--cycle;
\draw[black, rounded corners=4mm, thick](4, -.6)--(2,-.6)--(2,-1.6)--(.2,-1.6)--(.2,-3.4)--(2, -3.4)--(4, -3.4)--cycle;

\draw[black, rounded corners=4mm, thick, pattern=north west lines, pattern color=black, opacity=0.2](-.2,-1.6)--(-1.8,-1.6)--(-1.9,-2.5)--(0, -2.5)--(.2, -3.3)--(1.8, -3.3)--(1.8, -2.6)--(.2,-2.5)--cycle;
\draw[black, rounded corners=4mm, thick, dashed](-.2,-1.6)--(-1.8,-1.6)--(-2,-3.3)--(1.8, -3.3)--(1.8, -2.6)--(.2,-2.5)--cycle;

\draw[black, rounded corners=4mm, pattern color=black, opacity=0.1, fill](-3.7,-1.6)--(-3.8,-2.5)--(-2.2,-2.5)--(-2.3, -1.6)--cycle;;
\draw[black, rounded corners=4mm, thick, dotted](-3.7,-1.6)--(-4,-3.3)--(-2.1,-3.3)--(-2.3, -1.6)--cycle;
\end{tikzpicture}

\vspace{3ex}

As we have seen in Theorem~\ref{theorem:vanishing-combinatorial-cohomology-affine}, we can use this covering to compute the cohomology of the any sheaf of abelian groups on $U$.

Let $U_1, U_2$ and $U_3$ be $D(x_3+ x_4), D(x_2+ x_3)$ and $D(x_1)$ respectively. Then
\begin{align*}
U_1\cap U_2 &= U_1\cap U_3 = \varnothing\\
U_2\cap U_3 &= \{\langle x_4\rangle\}
\end{align*}
So the nerve of this covering is the simplicial complex on $V=[3]$
\[\nerve(\mathscr{V}) = \{\varnothing, \{1\}, \{2\}, \{3\}, \{2,3\}\}\]

Consider the constant sheaf $\ZZ$. We know already that $\H^i(U, \ZZ)\cong\vH^i({\mathscr{V}}, \ZZ)$, and we can compute the latter by meaning of the \v{C}ech complex
\begin{equation*}
\begin{tikzcd}[baseline=(current  bounding  box.center), row sep = 0, 
/tikz/column 1/.append style={anchor=base east},
/tikz/column 3/.append style={anchor=base west}
]
\vC^\bullet: \vC^0_1\oplus \vC^0_2\oplus \vC^0_3\arrow[r, "\partial^0"] & \vC^1_{2, 3}\arrow[r, "\partial^1"] & 0\\
\ZZ\oplus \ZZ\oplus \ZZ \rar & \ZZ\rar & 0\\
(\ \alpha\ ,\ \beta\ ,\ \gamma\ )\arrow[r, mapsto]& (\gamma-\beta)
\end{tikzcd}
\end{equation*}
where $\vC^{j}_{i_0, \dots, i_j}=\Gamma(U_{i_0}\cap \dots\cap U_{i_j}, \ZZ)$, that is either $0$ or $\ZZ$. This is exactly the cochain complex needed to compute simplicial cohomology of $\nerve(\mathscr{V})$.
\end{example}

\subsection{The Link Complex and its Spectrum}
We want to study $M_{\lk_\triangle(F)}$ and relate it to $M_\triangle$, for a any face $F\in\triangle$.

\begin{definition}\label{definition:link-complex}
    Given a simplicial complex $\triangle$ and a face $F\in\triangle$, the link of $F$ in $\triangle$ is
    \[
    \gls{link}=\lk_\triangle(F)=\{G\in\triangle\mid F\cup G\in\triangle, F\cap G=\varnothing\}
    \]
    or, equivalently,
    \[
    \gls{link}=\lk_\triangle(F)=\{G\smallsetminus F \mid F\subseteq G\in\triangle\}.
    \]
\end{definition}

For a subset $I$ of $[n]$, denote by $x_I$ the set of variables $\{x_i\mid i\in I\}$. We can write $M_\triangle$ as
\[
M_\triangle = \left\{ x_{[n]} \midd \sum_{g\in G} x_g=\infty, \text{ for all } G\notin \triangle \right\}.
\]
For a face $F\in\triangle$, we denote by $\gls{xF}$ the set of variables $\{x_i\mid i\in F\}$ and by $(M_\triangle)_{x_F}=\left(M_\triangle\right)_{\sum_{i\in F} x_i}$ the binoid localized at the variables corresponding to the elements of $F$.

\begin{lemma}\label{lemma:generators-relations-link}
    Let $M_\triangle=(x_{[n]}\mid \R)$, where $\R$ is the set of minimal relations of $M_\triangle$, i.e.\
    \[
    \R=\left\{\sum_{g\in G} x_g = \infty, \forall G\text{ minimal non-face of }\triangle\right\}.
    \]
    Then $M_{\lk_\triangle(i)}=(x_J\mid \R')$, 
    where $J$ is the set of elements that share a face with $i$, i.e.\ 
    \[J=\left\{j\in[n]\midd\{j, i\}\in\triangle\right\}\] and
    \[
    \R'=\left\{\sum_{k\in K} x_k=\infty\midd \sum_{k\in K\cup \{i\}} x_k=\infty \in\R\right\}\bigcup\left\{\sum_{k\in K} x_k=\infty\in\R\midd i\notin K\right\}
    \]
\end{lemma}
\begin{proof}
    For the generators, this follows almost straightforward from the definition of link complex above, since $F\in\lk_\triangle(i)$ if and only if $F\cup\{i\}\in\triangle$. So the generators of $M_{\lk_\triangle(i)}$ are indexed by the vertices $j\in[n]$ with exactly this property.
    
    For the relations, the first step is to prove that the relations in $\R$ that do not involve $x_i$ will still be valid as relations in $\R'$. Let $x_{j_1}+\dots+x_{j_k}=\infty\in\R$ be a relation such that $i\notin\{j_1, \dots, j_k\}$. So $\{j_1, \dots, j_k\}$ is a (minimal) non-face, and it will also not be a face in the link, so the relation will still be valid.\\
    The second step is to see what happens to relations of the type $x_{j_1}+\dots+x_i+\dots+x_{j_k}=\infty\in\R$. All the maximal proper subsets of $\{x_{j_1},\dots,x_i, \ldots, x_{j_k}\}$ are faces of $\triangle$ and all but one (namely $\{x_{j_1},\dots,\widehat{x_i}, \ldots, x_{j_k}\}$) give rise to faces in $\lk_\triangle(i)$. The only one that gives rise to a non-face is exactly the one we are interested in, and so we set $x_{j_1}+\dots+\widehat{x_i}+\dots+ x_{j_k}=\infty\in\R'$. This relation is also minimal, because each of the proper subsets of $\{x_{j_1},\dots,\widehat{x_i}, \ldots, x_{j_k}\}$ is contained in another $(k-1)$-subset of $\{x_{j_1},\dots,x_i, \ldots, x_{j_k}\}$, that defines a face in the link complex, and so all its subsets are faces themselves.
\end{proof}

\begin{notation}
    Let $F\in\triangle$ and let $A$ be an abelian group. We denote as usual by $A^F$ the set morphisms from $F$ to $A$.
    It is easy to prove that $A^F\cong A^{|F|}$.
\end{notation}

\begin{proposition}\label{proposition:map_Zface-localization}
    Let $F$ be a face of $\triangle$. There exists an injective binoid morphism
    \[
    \begin{tikzcd}[baseline=(current  bounding  box.center), cramped, row sep = 0ex,
    /tikz/column 1/.append style={anchor=base east},
    /tikz/column 2/.append style={anchor=base west}]
    {(\ZZ^F)}^\infty \rar["\psi"] & {(M_\triangle)}_{x_F}\\
    \displaystyle\sum_{v\in F} n_v v \rar[mapsto] & \displaystyle \sum_{v\in F} n_v x_v
    \end{tikzcd}
    \]
\end{proposition}
\begin{proof}
    Since $x_v$ is a unit on the right for any $v\in F$, this is a well-defined morphism. Moreover, apart from $\infty$ this maps is clearly injective, because different vertices on the left are map to different vertices on the right.
    About $\infty$, since $F$ is a face in $\triangle$ all its subsets are faces and so all the sums $\displaystyle\sum_{v\in F} n_v v $ are different from $\infty$.
\end{proof}

We now recall some definitions about morphisms involving simplicial complexes from \cite{Boettger}.

\begin{definition}[{\cite[Definition 6.2.2]{Boettger}}]
    Let $\triangle$ and $\triangle'$ be two simplicial complexes on vertex sets $V$ and $V'$ respectively, and let $\lambda:V\longrightarrow V'$ be a set map. We say that $\lambda$ is 
    \begin{description}
        \item[\normalfont{\emph{simplicial}}] if every image of a face is a face;
        \item[\normalfont{\emph{$\alpha$-simplicial}}] if it satisfies one of the two equivalent conditions
        \begin{enumerate}
            \item every image of a non-face is a non-face,
            \item every preimage of a face is a face;
        \end{enumerate}
        \item[\normalfont{\emph{$\beta$-simplicial}}] if every preimage of a non-face is a non-face.
    \end{description}
\end{definition}

\begin{notation}
    Since the simplicial properties of a set map depend on the simplicial complexes involved, we follow the abuse of notation in \cite{Boettger} and denote the map by
    \[
    \begin{tikzcd}[baseline=(current  bounding  box.center), cramped, row sep = 0ex,
    /tikz/column 1/.append style={anchor=base east},
    /tikz/column 2/.append style={anchor=base west}]
    \lambda:(V, \triangle) \rar & (V', \triangle')\\
    v \rar[mapsto] & \lambda(v)
    \end{tikzcd}
    \]
    thus including the simplicial complexes in its definition.
\end{notation}

\begin{lemma}\label{lemma:beta-simplicial}
    The map 
    \[
    \begin{tikzcd}[baseline=(current  bounding  box.center), cramped, row sep = 0ex,
    /tikz/column 1/.append style={anchor=base east},
    /tikz/column 2/.append style={anchor=base west}]
    \iota:(V\smallsetminus F, \lk_\triangle(F)) \rar & (V, \triangle)\\
    v\rar[mapsto]& v
    \end{tikzcd}
    \]
    is simplicial and $\beta$-simplicial.
\end{lemma}
\begin{proof}
    In order to prove that the map is simplicial, we have to show that the image of a face is a face.
    Let $G$ be a face of $\lk_\triangle(F)$. Then $G\cup F\in\triangle$, as well as all its subsets. In particular, $\lambda(G)=G\subseteq G\cup F$, so $G\in\triangle$ and the map is simplicial.\\
    In order to show that it is $\beta$-simplicial, we have to prove that the preimage of a non-face is a non-face. Let $H$ be a non-face of $\triangle$. Then $\lambda^{-1}(H)=H\cap (V\smallsetminus F)$. Assume that the latter is a face of $\lk_\triangle(F)$. This is the case if and only if $(H\cap(V\smallsetminus F))\cup F\in\triangle$ and $(H\cap(V\smallsetminus F))\cap F=\varnothing$. In particular,
    \[
    (H\cap(V\smallsetminus F))\cup F=(H\cup F)\cap((V\smallsetminus F)\cup F)=(H\cup F)\cap V=H\cup F
    \]
    In particular, $H\subseteq H\cup F\in\triangle$, $H\in\triangle$, so $\lambda^{-1}(H)$ has to be a non-face of $\lk_\triangle(F)$.
\end{proof}

\begin{remark}\label{remark:not-alpha-simplicial}
    The map above is not $\alpha$-simplicial. Let $\triangle$ be the complex on $V=[5]$ with facets $\{1, 2, 3\}$ and $\{3, 4, 5\}$, i.e.\ the \lq butterfly''
    
    
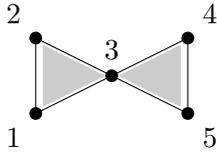
\vspace{1ex}
    
    \hspace{2em}\begin{tikzpicture}
    \tikzstyle{point}=[circle,thick,draw=black,fill=black,inner sep=0pt,minimum width=4pt,minimum height=4pt]
    \node (v1)[point, label={[label distance=0cm]-135:$1$}] at (0,0) {};
    \node (v2)[point, label={[label distance=0cm]135:$2$}] at (0,1) {};
    \node (v3)[point, label={[label distance=0cm]90:$3$}] at (1, .5) {};
    \node (v4)[point, label={[label distance=0cm]45:$4$}] at (2,1) {};
    \node (v5)[point, label={[label distance=0cm]-45:$5$}] at (2,0) {};
    
    \draw (v1.center) -- (v2.center);
    \draw (v1.center) -- (v3.center);
    \draw (v2.center) -- (v3.center);
    \draw (v3.center) -- (v4.center);
    \draw (v3.center) -- (v5.center);
    \draw (v4.center) -- (v5.center);
    
    \draw[color=black!20, style=fill, outer sep = 20pt] (0.1,0.1) -- (0.1,0.9) -- (0.9, 0.5) -- cycle;
    \draw[color=black!20, style=fill, outer sep = 20pt] (1.9,0.1) -- (1.9,0.9) -- (1.1, 0.5) -- cycle;
    \end{tikzpicture}
    
    \vspace{1ex}

    Consider $F=\{1\}$, so $\lk_\triangle(F)=\{\varnothing, \{2\}, \{3\}, \{2, 3\}\}$. The preimage of $\{3, 4, 5\}$ under the map above is $\{3, 4, 5\}\cap V\smallsetminus \{1\}=\{3, 4, 5\}\cap \{2, 3, 4, 5\}=\{3, 4, 5\}$ that it is clearly not a face in $\lk_\triangle(\{1\})$.
\end{remark}

\begin{lemma}[{\cite[Example 6.2.11]{Boettger}}]\label{lemma:alpha-simplicial}
    Let $F$ be a face of the simplicial complex $\triangle$ on vertex set $V$. $\lk_\triangle(F)$ can be seen as a simplicial complex on $V$ and the identity is an $\alpha$-simplicial map:
    \[
    \begin{tikzcd}[baseline=(current  bounding  box.center), cramped, row sep = 0ex,
    /tikz/column 1/.append style={anchor=base east},
    /tikz/column 2/.append style={anchor=base west}]
    \mathrm{Id}:(V, \triangle)\rar & (V, \lk_\triangle(F))\\
    v\rar[mapsto]& v
    \end{tikzcd}
    \]
    This map is not simplicial.
\end{lemma}

Our goal is to prove in the next section that $M_{\lk_\triangle(F)}\wedge {(\ZZ^F)}^\infty \cong {(M_\triangle)}_{x_F}$.

In order to use the universal property of the smash product (see \cite[Proposition 1.8.10]{Boettger}) to get a map $M_{\lk_\triangle(F)}\wedge {(\ZZ^F)}^\infty \longrightarrow {(M_\triangle)}_{x_F}$ in the following diagram,
\begin{equation}\label{equation:diagram-smash}
\begin{tikzcd}[baseline=(current  bounding  box.center), cramped]
M_{\lk_\triangle(F)}\wedge {(\ZZ^F)}^\infty & {(\ZZ^F)}^\infty \lar["i", swap] \dar["\psi"]\\
M_{\lk_\triangle(F)} \uar["j"] \rar["\varphi", swap] & {(M_\triangle)}_{x_F}
\end{tikzcd}
\end{equation}
we first need all the other maps. $i$ and $j$ are the maps that come from the smash product, so we are just missing $\varphi$ and $\psi$. We gave already $\psi$ above in Proposition~\ref{proposition:map_Zface-localization}, and we can give $\varphi$ as follows, although it does not come from simplicial properties (they would go the other way around, see \cite[Corollary 6.5.18]{Boettger}).

\begin{proposition}
    Let $F$ be a face of $\triangle$. There exists a binoid homomorphism
    \[
    \begin{tikzcd}[baseline=(current  bounding  box.center), cramped, row sep = 0ex,
    /tikz/column 1/.append style={anchor=base east},
    /tikz/column 2/.append style={anchor=base west}]
    M_{\lk_\triangle(F)} \rar & {(M_\triangle)}_{x_F}\\
    \displaystyle\sum_{w\in G} n_w x_w \rar[mapsto] & \displaystyle\sum_{w\in G} n_w x_w
    \end{tikzcd}
    \]
\end{proposition}
\begin{proof}
    From \cite[6.5.16]{Boettger}, for any binoid $N$, any $N$-point $\rho:V\smallsetminus F\longrightarrow N$ factors through $M_{\lk_\triangle(F)}$. In particular, the map
    \[
    \begin{tikzcd}[baseline=(current  bounding  box.center), cramped, row sep = 0ex,
    /tikz/column 1/.append style={anchor=base east},
    /tikz/column 2/.append style={anchor=base west}]
    \rho:V\smallsetminus F\rar & (M_\triangle)_{x_F}\\
    w\rar[mapsto]& w
    \end{tikzcd}
    \]
    is a $(V\smallsetminus F)$-point, since it respects the condition $\displaystyle \sum_{i\in I} w_i=\infty$ if $I\notin \lk_\triangle(F)$, so it factors through $M_{\lk_\triangle(F)}$, thus yielding us the map in the statement.
\end{proof}

We now have all the maps that we need in the diagram above, and we are ready to prove the isomorphism in the next section.

\begin{proposition}\label{proposition:dxi-spec-link}
    $\Spec M_{\lk_\triangle(i)}$ is homeomorphic to $D(x_i)$ via the injection
    \[
    \begin{tikzcd}[baseline=(current  bounding  box.center), row sep = 0,
    /tikz/column 1/.append style={anchor=base east},
    /tikz/column 2/.append style={anchor=base west}]
    \Spec M_{\lk_\triangle(i)}\rar[hook, "j"]&\Spec M_\triangle\\
    \p\rar[mapsto] & \p
    \end{tikzcd}
    \]
\end{proposition}
\begin{proof}
    From the definitions and Remark~\ref{remark:boettger:correspondence-faces-prime-ideals}, it is just an easy computation
    \begin{align*}
    \Spec M_{\lk_\triangle(i)} &= \left\{\p\in\Spec M_{\lk_\triangle (i)}\right\}\\
    &=\left\{([n]\smallsetminus\left\{i\right\})\smallsetminus G\mid G\in\lk_\triangle(i)\right\}\\
    &=\left\{[n]\smallsetminus(G\cup\left\{i\right\})\mid G\cup\left\{i\right\}\in\triangle, i\notin G\right\}\\
    &=\left\{[n]\smallsetminus F\mid i\in F\in\triangle\right\}\\
    &= \left\{\p\in\Spec M_\triangle\mid x_i\notin \p\right\} = D(x_i)\qedhere
    \end{align*}
\end{proof}

\begin{example}\label{example:spectrum-simplicial-binoid-4}
    We go on with our favourite Example~\ref{example:spectrum-simplicial-binoid-3} and we describe its link complexes. Recall
    
    \begin{minipage}[c]{0.5\textwidth}
        \[
        \triangle = \left\{\begin{aligned}
        &\varnothing,\{1\},\{2\},\{3\},\{4\},\\
        &\{1, 2\}, \{1, 3\}, \{2, 3\}, \{3, 4\}\\
        &\{1, 2, 3\}\end{aligned}\right\}
        \]\vfill
    \end{minipage}\begin{minipage}[c]{0.5\textwidth}
    \hspace{6em}\begin{tikzpicture}
    \tikzstyle{point}=[circle,thick,draw=black,fill=black,inner sep=0pt,minimum width=4pt,minimum height=4pt]
    \node (v1)[point, label={[label distance=0cm]-135:$1$}] at (0,0) {};
    \node (v2)[point, label={[label distance=0cm]90:$2$}] at (0.5,0.87) {};
    \node (v3)[point, label={[label distance=0cm]-45:$3$}] at (1,0) {};
    \node (v4)[point, label={[label distance=0cm]-45:$4$}] at (1.5,0.87) {};
    
    \draw (v1.center) -- (v2.center);
    \draw (v1.center) -- (v3.center);
    \draw (v2.center) -- (v3.center);
    \draw (v3.center) -- (v4.center);
    
    \draw[color=black!20, style=fill, outer sep = 20pt] (0.1,0.06) -- (0.5,0.77) -- (0.9,0.06) -- cycle;
    \end{tikzpicture}
\end{minipage}

and its simplicial binoid is
\[
M_\triangle=\{x_1, x_2, x_3, x_4\mid x_1+x_4=\infty, x_2+x_4=\infty\}.
\]
Then the link of the vertices in the simplicial complex and their simplicial binoids are

\vspace{1ex}

\begin{minipage}[c]{0.4\textwidth}
    \[
    \lk_\triangle(1) = \left\{\begin{aligned}
    \varnothing,&\{2\},\{3\},\\
    &\{2, 3\}\end{aligned}\right\}
    \]\vfill
\end{minipage}\begin{minipage}[c]{0.2\textwidth}
\hspace{2em}\begin{tikzpicture}
\tikzstyle{point}=[circle,thick,draw=black,fill=black,inner sep=0pt,minimum width=4pt,minimum height=4pt]
\node (v2)[point, label={[label distance=0cm]90:$2$}] at (0.5,0.87) {};
\node (v3)[point, label={[label distance=0cm]-45:$3$}] at (1,0) {};

\draw (v2.center) -- (v3.center);

\end{tikzpicture}
\end{minipage}\begin{minipage}[c]{0.4\textwidth}
\[
M_{\lk_\triangle(1)} = (x_2, x_3\mid \varnothing)
\]\vfill
\end{minipage}

\vspace{1ex}

\begin{minipage}[c]{0.4\textwidth}
    \[
    \lk_\triangle(2) = \left\{\begin{aligned}
    \varnothing,&\{1\},\{3\},\\
    &\{1, 3\}\end{aligned}\right\}
    \]\vfill
\end{minipage}\begin{minipage}[c]{0.2\textwidth}
\begin{tikzpicture}[baseline=(current  bounding  box.center)]
\tikzstyle{point}=[circle,thick,draw=black,fill=black,inner sep=0pt,minimum width=4pt,minimum height=4pt]
\node (v1)[point, label={[label distance=0cm]-135:$1$}] at (0,0) {};
\node (v3)[point, label={[label distance=0cm]-45:$3$}] at (1,0) {};

\draw (v1.center) -- (v3.center);

\end{tikzpicture}
\end{minipage}\begin{minipage}[c]{0.4\textwidth}
\[
M_{\lk_\triangle(2)} = (x_1, x_3\mid \varnothing)
\]\vfill
\end{minipage}

\vspace{1ex}

\begin{minipage}[c]{0.4\textwidth}
    \[
    \lk_\triangle(3) = \left\{\begin{aligned}
    &\varnothing,\{1\},\{2\},\\
    &\{4\},\{1, 2\}\end{aligned}\right\}
    \]\vfill
\end{minipage}\begin{minipage}[c]{0.2\textwidth}
\begin{tikzpicture}
\hspace{.5em}\tikzstyle{point}=[circle,thick,draw=black,fill=black,inner sep=0pt,minimum width=4pt,minimum height=4pt]
\node (v1)[point, label={[label distance=0cm]-135:$1$}] at (0,0) {};
\node (v2)[point, label={[label distance=0cm]90:$2$}] at (0.5,0.87) {};
\node (v4)[point, label={[label distance=0cm]-45:$4$}] at (1.5,0.87) {};

\draw (v1.center) -- (v2.center);

\end{tikzpicture}
\end{minipage}\begin{minipage}[c]{0.4\textwidth}
\[
M_{\lk_\triangle(3)} = \left(\begin{aligned}& x_1, x_2, x_4\\\midrule & x_1+x_4=\infty,\\ & x_2+x_4=\infty\end{aligned}\right)
\]\vfill
\end{minipage}

\vspace{1ex}

\begin{minipage}[c]{0.4\textwidth}
    \[
    \lk_\triangle(4) = \left\{\begin{aligned}
    \varnothing,\{3\}\end{aligned}\right\}
    \]\vfill
\end{minipage}\begin{minipage}[c]{0.2\textwidth}
\begin{tikzpicture}[baseline=(current  bounding  box.center)]
\hspace{4.5em}\tikzstyle{point}=[circle,thick,draw=black,fill=black,inner sep=0pt,minimum width=4pt,minimum height=4pt]
\node (v3)[point, label={[label distance=0cm]-45:$3$}] at (1,0) {};


\end{tikzpicture}
\end{minipage}\begin{minipage}[c]{0.4\textwidth}
\[
M_{\lk_\triangle(4)} = (x_3\mid \varnothing).\qedhere
\]
\end{minipage}

\end{example}

\subsection{\texorpdfstring{$\O_{M_\triangle}(D(x_i))$}{OM localization}}
From the definition of the structure sheaf of a binoid, Definition~\ref{definition:affine-scheme-binoid}, we know that in general $\O_M(D(f))=M_f$, the localization of $M$ at $f$. 

\begin{theorem}\label{theorem:localization-simplicial-binoid-multiple}
    For any face $F\in \triangle$ there is an isomorphism
    \begin{equation}\label{equation:localization-simpicial-binoid-wedge}
    (M_\triangle)_{x_F}\cong M_{\triangle'}\wedge(\ZZ^F)^\infty
    \end{equation}
    where $\triangle'=\lk_\triangle(F)$.
\end{theorem}
\begin{proof}
    We just have to prove that the map $M_{\lk_\triangle(F)}\wedge {(\ZZ^F)}^\infty \longrightarrow {(M_\triangle)}_{x_F}$, that we get from diagram~\eqref{equation:diagram-smash}, is an isomorphism.    
    In order to do so, we give an explicit description of it. In the diagram
    \[
    \begin{tikzcd}[baseline=(current  bounding  box.center), cramped]
    M_{\lk_\triangle(F)}\wedge {(\ZZ^F)}^\infty\drar["\zeta"] & {(\ZZ^F)}^\infty \lar["i", swap] \dar["\psi"]\\
    M_{\lk_\triangle(F)} \uar["j"] \rar["\varphi", swap] & {(M_\triangle)}_{x_F}
    \end{tikzcd}
    \]
    we know the maps $\varphi$ and $\psi$, we know that $i$ and $j$ are the inclusions in the smash product and we know that the triangles commute, so $\zeta i=\psi$ and $\zeta j=\varphi$. We can then explicitly describe $\zeta$ on the generators of $M_{\lk_\triangle(F)}\wedge {(\ZZ^F)}^\infty$ as
    \[
    \begin{tikzcd}[baseline=(current  bounding  box.center), row sep = 0,
    /tikz/column 1/.append style={anchor=base east},
    /tikz/column 2/.append style={anchor=base west}]
    M_{\lk_\triangle(F)}\wedge {(\ZZ^F)}^\infty\rar["\zeta"]&{(M_\triangle)}_{x_F}\\
    w\wedge 0\rar[mapsto]&\varphi(w)=w\\
    0\wedge v\rar[mapsto]&\psi(v)=v
    \end{tikzcd}
    \]
    more in general,
    \[
    \zeta\left(\sum_{w\in G} n_w x_w\wedge \sum_{v\in F} m_vx_v\right)=\sum_{w\in G} n_w x_w + \sum_{v\in F} m_vx_v
    \]
    and this is not $\infty$ because $G\cup F\in\triangle$, since $G\in\lk_\triangle(F)$.
    Apart from $\infty$, this map is injective because the maps $\psi$ and $\varphi$ are injective themselves.
    
    Moreover, it is surjective because every element $f\in (M_\triangle)_{x_F}$ has a unique description (the binoid is semifree) with respect to the semibasis
    \[
    f=\sum_{w\in V\smallsetminus F} n_w x_w + \sum_{v\in F} n_v x_v.
    \]
    
    If $f=\infty$, since $x_v$ is a unit for any $v\in F$, we have that
    \[
    \sum_{w\in V\smallsetminus F} n_w x_w =\infty,
    \]
    that is true if and only if $\{x_w\mid n_w\neq 0\}\notin\lk_\triangle(F)$ so, in particular, $f=\zeta(\infty\wedge 0)$ (and, by the properties of the smash product, this is also $\zeta(0\wedge \infty)$).
    
    If $f\neq \infty$, then it is clear that $f=\zeta(\sum_{w\in G} n_w x_w \wedge \sum_{v\in F} n_v x_v)$ for $G=\{x_w\mid n_w\neq 0\}\in\lk_\triangle(F)$.
    
    So, $\zeta$ is an isomorphism and we get our result.
\end{proof}

The following Example shows in full details that the localization at a face $F$ of $\triangle$ yields a free part $\ZZ^{F}$ of rank equal to the cardinality of $F$.

\begin{example}\label{example:spectrum-simplicial-binoid-5}
    Let's have another look at our favourite Example~\ref{example:spectrum-simplicial-binoid-4}, whose simplicial binoid is
    \[
    M_\triangle=\{x_1, x_2, x_3, x_4\mid x_1+x_4=\infty, x_2+x_4=\infty\}.
    \]
    Let us compute its localization at $x_1$
    \begin{align*}
    M_{x_1}=\bigl(x_1, -x_1, x_2, x_3, x_4\mid x_1+x_4=\infty, x_2+x_4=\infty, x_1+(-x_1)=0\bigr)
    \end{align*}
    
    But $x_4=(-x_1+x_1)+x_4=-x_1+(x_1+x_4)=-x_1+\infty=\infty$, and the second relation becomes trivial, $x_2+\infty=\infty$. So the minimal generators and relations are now
    \begin{align*}
    M_{x_1} &=\bigl(x_1, -x_1, x_2, x_3, x_4\mid x_4=\infty, x_1+(-x_1)=0\bigr)\\
    &=\bigl(x_1, -x_1, x_2, x_3\mid x_1+(-x_1)=0\bigr)\\
    &=\bigl(x_2, x_3\mid\varnothing\bigr)\wedge\bigl(x_1, -x_1\mid x_1+(-x_1)=0\bigr)\\
    &=M_{\lk_\triangle(1)}\wedge\ZZ^\infty.
    \end{align*}
    By chance, here we can eliminate the dependence from the link, since $M_{\lk_\triangle(1)}=M_{\mathscr{P}(\{2, 3\})}\cong {(\NN^2)}^\infty$, and obtain
    \begin{align*}
    M_{x_1} &\cong {(\NN^2)}^\infty\wedge \ZZ^\infty.
    \end{align*}
    $M_{x_2}$ is exactly the symmetric situation, so $M_{x_2}=M_{\lk_\triangle(2)}\wedge\ZZ^\infty\cong {(\NN^2)}^\infty\wedge \ZZ^\infty$.\\
    $x_3$ is not involved in any of the relations, so these do not change when localizing and they stay exactly as they are, giving us $M_{x_3}=M_{\lk_\triangle(3)}\wedge\ZZ^\infty$.\\
    The next case is again interesting
    \begin{align*}
    M_{x_4} &=\bigl(x_1, x_2, x_3, x_4, -x_4 \mid x_1=\infty, x_2=\infty, x_4+(-x_4)=0\bigr)\\
    &=\bigl(x_3\mid\varnothing\bigr)\wedge\bigl(x_4, -x_4\mid x_4+(-x_4)=0\bigr)\\
    &=M_{\lk_\triangle(4)}\wedge\ZZ^\infty\cong\NN^\infty\wedge\ZZ^\infty
    \end{align*}
    When we localize two times we end up with the following binoids
    \begin{align*}
    M_{x_1+ x_2} &=\left(M_{x_1}\right)_{x_2}=\bigl(x_2, -x_2, x_3\mid x_2+(-x_2)=0\bigr)\wedge\bigl(x_1, -x_1\mid x_1+(-x_1)=0\bigr)\\
    &=\bigl(x_3\mid\varnothing\bigr)\wedge\bigl(x_2, -x_2\mid x_2+(-x_2)=0\bigr)\wedge\bigl(x_1, -x_1\mid x_1+(-x_1)=0\bigr)\\
    &=M_{\lk_\triangle\left(\{1, 2\}\right)}\wedge\ZZ^\infty\wedge\ZZ^\infty=M_{\lk_\triangle\left(\{1, 2\}\right)}\wedge(\ZZ^2)^\infty\cong \NN^\infty \wedge(\ZZ^2)^\infty
    \end{align*}
    \begin{align*}
    M_{x_1+ x_3} &=\bigl(x_2\mid\varnothing\bigr)\wedge\bigl(x_3, -x_3\mid x_3+(-x_3)=0\bigr)\wedge\bigl(x_1, -x_1\mid x_1+(-x_1)=0\bigr)\\
    &=M_{\lk_\triangle\left(\{1, 3\}\right)}\wedge(\ZZ^2)^\infty\cong \NN^\infty \wedge(\ZZ^2)^\infty
    \end{align*}
    \begin{align*}
    M_{x_2+ x_3} &=\bigl(x_1\mid\varnothing\bigr)\wedge\bigl(x_2, -x_2\mid x_2+(-x_2)=0\bigr)\wedge\bigl(x_3, -x_3\mid x_3+(-x_3)=0\bigr)\\
    &=M_{\lk_\triangle\left(\{2, 3\}\right)}\wedge(\ZZ^2)^\infty\cong \NN^\infty \wedge(\ZZ^2)^\infty
    \end{align*}
    \begin{align*}
    M_{x_3+ x_4} &=\bigl(x_3, -x_3\mid x_3+(-x_3)=0\bigr)\wedge\bigl(x_4, -x_4\mid x_4+(-x_4)=0\bigr)\\
    &=M_{\lk_\triangle\left(\{3, 4\}\right)}\wedge(\ZZ^2)^\infty=M_{\{\varnothing\}}\wedge(\ZZ^2)^\infty=(\ZZ^2)^\infty
    \end{align*}
    
    It is worth pointing out that inverting $\{x_1, x_4\}$ -- or any other non-face in $\triangle$ -- would yield the degenerate case $M_{x_1, x_4}=\infty$, because we are inverting the infinity.\footnote{Recall from Definition~\ref{definition:simplicial-binoid-emptyset} that $M_{\{\varnothing\}}=\{0, \infty\}\neq \infty$}
    
    The last localization that gives us a non-degenerate situation is $M_{x_1, x_2, x_3}\cong{(\ZZ^3)}^\infty$.
\end{example}

\section{The punctured \texorpdfstring{\v{C}}{C}ech-Picard Complex}

Our goal is to compute the cohomology of the sheaf of units $\O^*$ in the simplicial case. In order to do so, we should look for an appropriate \v{C}ech covering of the punctured spectrum. We proved in Theorem~\ref{theorem:vanishing-combinatorial-cohomology-affine} that the cohomology of $\O^*$ vanishes on affine pieces, since it is a sheaf of abelian groups, so in particular we can use $\{D(x_i)\}$ as an affine open cover to compute cohomology of $\O^*_{M_\triangle}$ on $\Spec^\bullet M_\triangle$ via \v{C}ech cohomology.

In what follows, assume that $M$ is a positive semifree binoid with semibasis $\{x_1, \dots, x_n\}$.\footnote{See Definition~\ref{definition:semifree} for the definition of semifree binoid.}

\begin{definition}\label{definition:extension-by-zeros}
    Let $\ZZ$ be the constant sheaf on $D(x_i)$ and let $j_i:D(x_i)\longrightarrow \Spec^\bullet M$ be the usual open embedding. Denote by $\O^*_{x_i}$ the extension by zero of $\ZZ$ along $j$, that is the sheafification of the presheaf on $\Spec^\bullet M$\footnote{This sheaf would usually be denoted by ${j_i}_!\ZZ$, but we simplify the notation since there is no confusion. See \cite[Exercise II.1.19.b]{hartshorne1977algebraic} for the definition of the operation $!$.}
    \[
    \begin{tikzcd}[baseline=(current  bounding  box.center), cramped, row sep = 0ex,
    /tikz/column 1/.append style={anchor=base east},
    /tikz/column 2/.append style={anchor=base west},
    ampersand replacement=\&]
    \G:U \rar[mapsto] \& \left\{\begin{aligned}
    &\ZZ, && \text{ if } U\subseteq D(x_i),\\
    &0, && \text{ otherwise.}
    \end{aligned}\right.
    \end{tikzcd}
    \]
\end{definition}

\begin{notation}
    Since $x_i$ is a unit of $\O_M$ on $D(x_i)$, we would think of $\ZZ$ on $D(x_i)$ as multiples of $x_i$, thus denoting this group with $\ZZ x_i$.
\end{notation}

\begin{remark}
    We can easily describe the stalk of $\O^*_{x_i}$ at $\p$ as follows
    \[
    \left(\O^*_{x_i}\right)_\p=\varinjlim_{\p\in U}\O^*_{x_i}(U)=\left\{\begin{aligned}
    & \ZZ x_i, && \text{ if } \p\in D(x_i),\\
    &0, &&\text{ otherwise.}
    \end{aligned}\right.
    \]
\end{remark}

\begin{proposition}
    There exists a morphism of sheaves
    \[
    \begin{tikzcd}[baseline=(current  bounding  box.center), cramped, row sep = 0ex,
    /tikz/column 1/.append style={anchor=base east},
    /tikz/column 2/.append style={anchor=base west}]
    \O^*_{x_i} \rar & \O^*_{M}.
    \end{tikzcd}
    \]
\end{proposition}
\begin{proof}
    There is a morphism of presheaves
    \[
    \begin{tikzcd}[baseline=(current  bounding  box.center), cramped, row sep = 0ex,
    /tikz/column 1/.append style={anchor=base east},
    /tikz/column 2/.append style={anchor=base west}]
    \G \rar & \O^*_{M}
    \end{tikzcd}
    \]
    because if $U\subseteq D(x_i)$ then $\G(U)=\ZZ$ and $x_i$ is a unit of $\O_M(U)$, so we just send $1\mapsto x_i$. If $U\nsubseteq D(x_i)$ then the value of the presheaf is $0$, so we send it to $0$ in $\O_M^*(U)$.
    Thanks to the universal property of the sheafification, we have the following diagram
    \[
    \begin{tikzcd}[baseline=(current  bounding  box.center), cramped]
    \G \rar\dar & \O^*_{M}\\
    \O^*_{x_i}\urar
    \end{tikzcd}
    \]
    that yields the desired morphism $\O^*_{x_i} \longrightarrow \O^*_{M}$.
\end{proof}

\begin{example}
    Consider the binoid $M=(x, y\mid x+y=\infty)$. Its punctured spectrum is $U=\{\langle x\rangle, \langle y\rangle\}$, that we can cover with $D(x)$ and $D(y)$, that have empty intersection. Indeed, $\O^*_x(D(x))=\ZZ$ and $\O^*_x(D(y))=0$, so $\O^*_x(U)=\ZZ$.
\end{example}

\begin{theorem}\label{theorem:semifree-decomposition-of-sheaf}
    There exists an isomorphism of sheaves
    \[
    \begin{tikzcd}[baseline=(current  bounding  box.center), cramped, row sep = 0ex,
    /tikz/column 1/.append style={anchor=base east},
    /tikz/column 2/.append style={anchor=base west}]
    \displaystyle\bigoplus_{i=1}^n\O^*_{x_i} \rar & \O^*_{M}.
    \end{tikzcd}
    \]
\end{theorem}
\begin{proof}
    This map exists because it is induced component-wise by the maps obtained by the previous Proposition applied to the different $x_i$'s.
    
    In order to show that it is an isomorphism, recall from \cite[Exercise II.1.2]{hartshorne1977algebraic} that a morphism $\varphi:\F\longrightarrow \G$ between two sheaves on a topological space $X$ is an isomorphism if and only if it is an isomorphism for the stalks of these sheaves.
    
    Let $\p\in\Spec M$. Thanks to Lemma~\ref{proposition:minimal-open-set-prime-ideal}, there exists a unique minimal open subset that contains $\p$, namely the fundamental open subset $D\left(\sum_{x_j\notin \p} x_j\right)$
    
    On the right hand side,
    \[
    \O^*_{M, \p} = \O^*_{M}\left(D\left(\sum_{x_j\notin \p} x_j\right)\right)\cong\ZZ^r,
    \]
    where $r$ is the cardinality of $\{x_j\mid x_j\notin \p\}$.
    
    On the other hand,
    \[
    \left(\displaystyle\bigoplus_{i=1}^n\O^*_{x_i}\right)_\p\cong \displaystyle\bigoplus_{i=1}^n\left(\O^*_{x_i}\right)_\p\cong \displaystyle\bigoplus_{i=1}^n\O^*_{x_i}\left(D\left(\sum_{x_j\notin \p} x_j\right)\right)\cong\ZZ^s
    \]
    where $s$ is the cardinality of $\{i\in[n]\mid {x_i\notin \p}\}$, so $s=r$, and the isomorphism between them is the identity.
\end{proof}

Thanks to this Theorem, we already know that we can decompose the \v{C}ech complex for $\O^*$ on the covering $\{D(x_i)\}$ in the direct sum of the \v{C}ech subcomplexes of this decomposition. In this section, we are going to give explicit descriptions of these complexes.

\begin{example}
    We illustrate with a counter example why we use the extension by zeros and not the pushforward of $\ZZ$.
    
    Consider the simplicial binoid $M=(x_1, x_2, x_3\mid x_1+x_2+x_3=\infty)$. Its simplicial complex is the open triangle, with 3 sides and no $2$-dimensional face. The simplicial binoid of the link of $1$ is $M_{\lk_\triangle(1)} = (x_2, x_3\mid x_2+x_3=\infty)$, whose spectrum is homeomorphic to $D(x_1)$.
    When we consider $U=D(x_3)$, we have a non-empty intersection $D(x_1)\cap U=\{\{x_2\}\}$, and then, when we try to compare the evaluations of $\O^*_{x_1}$, ${j_i}_*\widetilde{\ZZ}$ and ${j_i}_!\ZZ$, we get
    \begin{align*}
    \O^*_{x_1}(D(x_3))&=\O^*_{x_1}(D(x_3)\cap D(x_1))=0\\
    {j_i}_!\ZZ(D(x_3))&=0\\
    {j_i}_*\ZZ(D(x_3))&={j_i}_*\ZZ(D(x_3)\cap D(x_1))=\ZZ,
    \end{align*}
    thus showing that the pushforward is not the right sheaf to consider here.
\end{example}

\subsection{The Groups}
We will now investigate the groups involved in the punctured \v{C}ech-Picard complex
\[
\vC^\bullet = \vC^\bullet(\{D(x_i)\}, \O^*),
\]
i.e.\ the \v{C}ech complex of the sheaf of units w.r.t.\ the combinatorial affine covering of $\Spec^\bullet M$.
\begin{theorem}\label{theorem:value-sheaf}
    Let $M_\triangle=(x_1, \dots, x_n\mid\R)$ be the simplicial binoid associated to the simplicial complex $\triangle$ on the vertex set $V=[n]$ and let $F\subseteq[n]$. Then
    \begin{equation}
    \O^*_{M_\triangle}\left(\displaystyle\bigcap_{i\in F}D(x_i)\right)\cong\left\{\begin{aligned}
    \ZZ^F && \text{ if } F\in\triangle,\\
    0 && \text{ otherwise.}
    \end{aligned}\right.
    \end{equation}
\end{theorem}
\begin{proof} Thanks to Proposition~\ref{proposition:intersection-open-subsets} and Proposition~\ref{proposition:structure-sheaf-on-empty-set} we have
    \[
    \O_{M_\triangle}\left(\displaystyle\bigcap_{i\in F}D(x_i)\right)\cong 
    \left\{\begin{aligned}
    (M_\triangle)_{x_F} && \text{ if } F\in\triangle,\\
    0 && \text{ otherwise.}
    \end{aligned}\right.
    \]
    Recall from Theorem~\ref{theorem:localization-simplicial-binoid-multiple} the Equation~\eqref{equation:localization-simpicial-binoid-wedge} that states
    \[
    (M_\triangle)_{x_F}\cong M_{\lk_\triangle(F)}\wedge\ZZ^F.
    \]
    Lastly, after recalling that any simplicial binoid is positive, we obtain our result.
\end{proof}

\begin{remark}\label{remark:inverses-simplicial-binoid-vertices}
    This group $\ZZ^F$ is generated by inverting the $j$ variables in $x_F$ that have no relations between themselves, thanks to the fact that they correspond to vertices contained in a face.
\end{remark}

\begin{remark}
    Another proof of this Theorem comes from Theorem~\ref{theorem:semifree-decomposition-of-sheaf} and Definition~\ref{definition:extension-by-zeros}. Since $\displaystyle\bigcap_{i\in F}D(x_i)\subseteq D(x_i)$ for any $i\in F$ and $\O^*=\displaystyle\bigoplus_{i\in[n]}\O^*_{x_i}$, we get a $\ZZ$ for any $i$ in $F$, thus giving us the thesis. 
\end{remark}

\begin{example}\label{example:spectrum-simplicial-binoid-6}
    Let's go back again at our Example~\ref{example:spectrum-simplicial-binoid-5}, whose simplicial binoid is
    \[
    M_\triangle=\{x_1, x_2, x_3, x_4\mid x_1+x_4=\infty, x_2+x_4=\infty\}.
    \]
    We already computed the localizations, and so we are able to describe the values of the sheaf of units $\O^*_{M_\triangle}$ on these open subsets
    \begin{align*}
    \O^*_{M_\triangle}(D(x_1))=M_{x_1}^* &\cong \bigl(x_1, -x_1, x_2, x_3\mid x_1+(-x_1)=0\bigr)^*\\
    &=\bigl(x_1, -x_1\mid x_1+(-x_1)=0\bigr)^*\cong \ZZ,\\
    \O^*_{M_\triangle}(D(x_2))\cong M_{x_2}^*&\cong \ZZ,\\
    \O^*_{M_\triangle}(D(x_3))\cong M_{x_3}^*&\cong \ZZ,\\
    \O^*_{M_\triangle}(D(x_4))\cong M_{x_4}^*&\cong \ZZ,
    \end{align*}
    on intersections of two of them
    \begin{align*}
    \O^*_{M_\triangle}(D(x_1)\cap D(x_2))&=M_{x_1,x_2}^* \cong \left(M_{\lk_\triangle\left(\{1, 2\}\right)}\wedge(\ZZ^2)^\infty\right)^*\cong \ZZ^2,\\
    \O^*_{M_\triangle}(D(x_1)\cap D(x_3))&=M_{x_1,x_3}^* \cong \ZZ^2,\\
    \O^*_{M_\triangle}(D(x_1)\cap D(x_4))&\cong 0,\\
    \O^*_{M_\triangle}(D(x_2)\cap D(x_3))&=M_{x_2,x_3}^* \cong \ZZ^2,\\
    \O^*_{M_\triangle}(D(x_2)\cap D(x_4))&\cong 0,\\
    \O^*_{M_\triangle}(D(x_3)\cap D(x_4))&=M_{x_3,x_4}^* \cong \ZZ^2,
    \end{align*}
    and intersections of three (or more) of them
    \begin{align*}
    \O^*_{M_\triangle}(D(x_1)\cap D(x_2)\cap D(x_3)) &=M_{x_1,x_2, x_3}^* \cong \ZZ^3,\\
    \O^*_{M_\triangle}(D(x_1)\cap D(x_2)\cap D(x_4)) &\cong 0,\\
    \O^*_{M_\triangle}(D(x_1)\cap D(x_3)\cap D(x_4)) &\cong 0,\\
    \O^*_{M_\triangle}(D(x_2)\cap D(x_3)\cap D(x_4)) &\cong 0,\\
    \O^*_{M_\triangle}(D(x_1)\cap D(x_2)\cap D(x_3)\cap D(x_4)) &\cong 0.\qedhere
    \end{align*}
\end{example}

So we can explicitly describe the groups of the \v{C}ech complex for the sheaf of units on the combinatorial affine covering $\{D(x_i)\}$, defined in general in Definition~\ref{definition:cech-picard-complex}, as\footnote{Recall that $\triangle_j$ is the set of faces of $\triangle$ of dimension $j$.}
\begin{equation}
\vC^j=\bigoplus_{\begin{subarray}{c}F\in\triangle\\\left|{F}\right|=j+1\end{subarray}}\ZZ^{j+1}=\bigoplus_{F\in\gls{triangle-j}}\ZZ^{F}.
\end{equation}

We have indeed an even nicer description of the groups, where we can index them on the single vertices, using the results from Theorem~\ref{theorem:semifree-decomposition-of-sheaf} above.

\begin{theorem}\label{theorem:split-cech-picard-groups}
    There exist groups $\vC^j_i$ indexed by the vertices $i\in V$ such that $\vC^j$ is a direct sum of them
    \begin{equation}
    \vC^j=\bigoplus_{i\in V}\vC^j_{i}.
    \end{equation}
\end{theorem}
\begin{proof}
    $M_\triangle$ is semifree, so Theorem~\ref{theorem:semifree-decomposition-of-sheaf} applies, and we can decompose
    \[
    \O^*_{M_\triangle} = \bigoplus_{i\in V}\O^*_{x_i}.\qedhere
    \]
\end{proof}

\begin{remark}
    Consider a  face $F$ of dimension $j$. The evaluation $\O^*_{x_i}(D(F))$ is non zero if and only if $F$ contains $i$, so we can define
    \[
    \vC^j_{i} = \bigoplus_{\left\{F\in\triangle_j\midd i\in F\right\}}\O^*_{x_i}(D(F)) = \bigoplus_{\left\{F\in\triangle_j\midd i\in F\right\}} \ZZ.
    \]
    Taking the sum over the vertices, we get an explicit description
    \begin{equation*}
    \vC^j=\bigoplus_{i\in V}\biggl(\bigoplus_{\left\{F\in\triangle_j\midd i\in F\right\}}\ZZ\biggr) = \bigoplus_{i\in V}\vC^j_{i}.\qedhere
    \end{equation*}
\end{remark}

So we can write explicitly any group in the \v{C}ech-Picard complex as a sum of groups indexed by the vertices. 
In what follows it will be convenient to use an uncommon notation, so we will take some time now to explain it, then we will do an Example to understand it better.

\begin{notation}\label{notation:simplicial-binoid-faces}
    Let $F=\{i_0, \dots, i_j\}$ be a face in the simplicial complex $\triangle$. Thanks to what we just proved, $\O^*_{M_\triangle}(D(F))=\ZZ^F=\ZZ^{j+1}$. When we have a section of $\O^*(D(F))$, we would like to be able to recover $F$, so we denote such a section with
    \[
    \sigma^F\in\O^*_{M_\triangle}(D(F)).
    \]
    This $\sigma^F$ is indeed a $j$-tuple of integers, where $j$ is the cardinality of $F$. We want to be able to address its components. To do so, we place indices in the usual lower right corner, making our general section a $j$-uple indexed by elements of $F$, that looks like
    \[
    \sigma^F=\left(\sigma^F_{i_1}, \sigma^F_{i_2}, \dots, \sigma^F_{i_j}\right) = \left(\sigma^F_{i} \midd i\in F\right).
    \]
    
    A particular case is when $F$ is a vertex, say $F=\{i\}$, so that $\sigma^F=\left(\sigma^F_{i_1}\right)$. In this case we will often omit the subscripted index.
\end{notation}

\begin{example}\label{example:spectrum-simplicial-binoid-7}
    To better explain this notation, let's have a look again at our usual Example~\ref{example:spectrum-simplicial-binoid-6}, whose simplicial binoid is
    \[
    M_\triangle=\{x_1, x_2, x_3, x_4\mid x_1+x_4=\infty, x_2+x_4=\infty\}.
    \]
    We computed the localizations, and we computed the images of $\O^*$ on the intersections of the coordinate affine fundamental open subsets. 
    
    Let us describe $\vC^1$. We know by definition that
    \begin{align*}
    \vC^1=\bigoplus_{1\leq i < j\leq 4}\O^*_{M_\triangle}(D(x_i)\cap D(x_j))
    \end{align*}
    Thanks to what we already did, we know that we can rewrite this as
    \begin{align*}
    \vC^1\cong\bigoplus_{F\in\triangle_1}\ZZ^2
    \end{align*}
    Let us fix for example $F=\{1, 2\}$. Then $\ZZ^2$ describes the invertible elements of the form $\alpha x_1+\beta x_2$, with $(\alpha, \beta)\in\ZZ^2$. So we denote the groups in $\vC^1$ with the face they correspond to
    \[
    \vC^1\cong \ZZ^{\{1,2\}}\oplus\ZZ^{\{1,3\}}\oplus\ZZ^{\{2,3\}}\oplus\ZZ^{\{3,4\}}
    \]
    and we can write a general element of $\sigma\in\vC^1$ as
    \begin{align*}
    \sigma&=\left(\sigma^{\{1,2\}}, \sigma^{\{1,3\}}, \sigma^{\{2,3\}}, \sigma^{\{3,4\}}\right)\\
    &=\left((\sigma^{\{1,2\}}_{1}, \sigma^{\{1,2\}}_{2}), (\sigma^{\{1,3\}}_{1},\sigma^{\{1,3\}}_{3}), (\sigma^{\{2,3\}}_{2}, \sigma^{\{2,3\}}_{3}), (\sigma^{\{3,4\}}_{3}, \sigma^{\{3,4\}}_{4})\right).
    \end{align*}
    In this way we know that, for example, $\sigma^{\{1,3\}}_{3}$ is the coefficient of $x_3$ in $\O^*(D(\{x_1, x_3\}))$ and $\sigma^{\{3,4\}}_{3}$ is the coefficient of $x_3$ in $\O^*(D(\{x_3, x_4\}))$.
    
    In our re-indexing of the groups, what we are doing is actually grouping together the coefficients related to the same coordinate. So, the $\vC^1_i$'s will be
    \begin{align*}
    \vC^1_1 = \ZZ^2 &=\left\{\left(\sigma^{\{1,2\}}_{1}, \sigma^{\{1,3\}}_{1}\right)\right\}\\
    \vC^1_2 = \ZZ^2 &=\left\{\left(\sigma^{\{1,2\}}_{2}, \sigma^{\{2,3\}}_{2}\right)\right\}\\
    \vC^1_3 = \ZZ^3 &=\left\{\left(\sigma^{\{1,3\}}_{3}, \sigma^{\{2,3\}}_{3}, \sigma^{\{3,4\}}_{3}\right)\right\}\\
    \vC^1_4 = \ZZ &=\left\{\left(\sigma^{\{3, 4\}}_{4}\right)\right\}\qedhere
    \end{align*}
\end{example}

We now understand the groups in the \v{C}ech-Picard complex, and we have a way to describe them in terms of membership of a vertex to a face. What about the maps? Can we split them too?

\subsection{The Maps}

An explicit description of the coboundary map is rather easy. Consider two faces $G\neq F$ such that $G=\{i_0, \ldots, i_{j-1}\}$ and $F=G\cup\{i_l\}$, and an element $\sigma^{G}\in\O_{M_\triangle}^*(D(G))$. Then $\sigma^{G}\in\ZZ^j$ is
\[
\sigma^{G}=\left(\sigma^{G}_{i_0}, \sigma^{G}_{i_1}, \ldots, \sigma^{G}_{i_{j-1}}\right)
\]
and we can then restrict it to $D(F)$ to obtain 
\[
\sigma^{G}\restriction_{F}=\left(\sigma^{G}_{i_0}, \sigma^{G}_{i_1}, \ldots, 0_{i_l}, \dots, \sigma^{G}_{i_{j-1}}\right).
\]
Indeed, due to this observation, the image of the $j$-th coboundary map (as defined in Definition~\ref{definition:cech-picard-complex} following \cite[Section III.4]{hartshorne1977algebraic})
\begin{equation}
\begin{tikzcd}[baseline=(current  bounding  box.center), cramped, row sep = 0ex,
/tikz/column 1/.append style={anchor=base east},
/tikz/column 2/.append style={anchor=base west}]
\partial^{j}:\vC^j\arrow[r] & \vC^{j+1}\\
\sigma  \arrow[r, mapsto] & \tau:=\partial^j(\sigma)
\end{tikzcd}
\end{equation}
can be described component-wise.

To make the situation a bit more clear, we are working with the following groups

\begin{equation*}
\begin{tikzcd}[baseline=(current  bounding  box.center), cramped, row sep = 0, column sep=1.3em,
/tikz/column 1/.append style={anchor=base west},
/tikz/column 2/.append style={anchor=base west}]
\hspace{.5em}\displaystyle\bigoplus_{i_0<\ldots< i_j}\Gamma\left(D(x_{i_0})\cap\dots\cap D(x_{i_j}), \O^*\right) \arrow[r, "\partial"] & \hspace{.5em}\displaystyle\bigoplus_{i_0<\ldots< i_{j+1}}\Gamma\left(D(x_{i_0})\cap\dots\cap D(x_{i_{j+1}}), \O^*\right)\\
\hspace{1.7em}\rotatebox[origin=c]{-90}{$=$} & \hspace{2.1em}\rotatebox[origin=c]{-90}{$=$}\\
\displaystyle\bigoplus_{\begin{subarray}{c}
    i_0<\ldots< i_j\\
    \{i_0, \ldots, i_j\}\in\triangle
    \end{subarray}}\Gamma\left(D(x_{i_0})\cap\dots\cap D(x_{i_j}), \O^*\right) \arrow[r] & \displaystyle\bigoplus_{\begin{subarray}{c}
    i_0<\ldots< i_{j+1}\\
    \{i_0, \ldots, i_{j+1}\}\in\triangle
    \end{subarray}}\Gamma\left(D(x_{i_0})\cap\dots\cap D(x_{i_{j+1}}), \O^*\right)\\
\hspace{1.7em}\rotatebox[origin=c]{-90}{$=$} & \hspace{2.1em}\rotatebox[origin=c]{-90}{$=$}\\
\hspace{1em}\displaystyle\bigoplus_{G\in\triangle_j}\ZZ^G  \arrow[r] & \displaystyle\hspace{1em}\bigoplus_{F\in\triangle_{j+1}}\ZZ^F
\end{tikzcd}
\end{equation*}

\begin{theorem}\label{theorem:split-cech-picard-maps} Let $F=\{i_0, \ldots, i_j\}\in\triangle$. Then we can write the above map component-wise as
    \begin{equation}
    \tau^{F}_{i_k}=\sum_{l=0}^{j}(-1)^{l}\left(\sigma^{F\smallsetminus\{i_l\} }\restriction_F\right)_{i_k}
    \end{equation}
    for every $k\in\{0, 1, \ldots, j\}$.
\end{theorem}
\begin{proof}
    This is just a straightforward computation, since the description of the map in the complex is
    \begin{align*}
    \hspace{-2em}\left(\partial^j(\sigma)\right)^F:&=\sum_{l=1}^{j+1}(-1)^{l-1}\left(\sigma^{F\smallsetminus\{i_l\}}\restriction_F\right)\\
    &=\sum_{l=1}^{j+1}(-1)^{l-1}\left(\left(\sigma^{F\smallsetminus\{i_l\}}\restriction_F\right)_{i_1}, \dots, \left(\sigma^{F\smallsetminus\{i_l\}}\restriction_F\right)_{i_k}, \dots, \left(\sigma^{F\smallsetminus\{i_l\}}\restriction_F\right)_{i_{j+1}}\right)\\
    &=\left(\dots, \sum_{l=1}^{j+1}(-1)^{l-1}\left(\sigma^{F\smallsetminus\{i_l\}}\restriction_F\right)_{i_k}, \dots\right)\qedhere
    \end{align*}
\end{proof}

\begin{remark}\label{corollary:image-cohomology-simplicial-binoid}
    By the definition of the restriction we have  $\left(\sigma^{F\smallsetminus\{i\}}\restriction_F\right)_{i}=0$, while everything else is untouched. So we can rewrite the above formula as
    \begin{equation*}
    \tau^{F}_{i}=\sum_{i_l\in F\smallsetminus \{i\}}(-1)^{l-1}\left(\sigma^{F\smallsetminus\{i_l\} }\restriction_F\right)_{i}=\sum_{i_l\in F\smallsetminus \{i\}}(-1)^{l-1}\left(\sigma^{F\smallsetminus\{i_l\}}_{i}\right).
    \end{equation*}
\end{remark}

This means that the maps split on the vertex-wise indexing that we introduced above and the \v{C}ech-Picard complex splits in the following complexes indexed by the vertices
\begin{equation}
\begin{tikzcd}[baseline=(current  bounding  box.center), cramped, row sep = 0ex,]
\vC^\bullet_{i}:0\arrow[r] &  \vC_{i}^0\arrow[r, "\partial^0_{i}"] & \vC_{i}^1\arrow[r, "\partial^1_{i}"] & \dots\arrow[r, "\partial^{j-1}_{i}"] & \vC_{i}^j\arrow[r, "\partial^j_{i}"] & \dots\arrow[r] &  0
\end{tikzcd}
\end{equation}
where by $\partial^j_{i}$ we denote the map of the \v{C}ech complex that we described above. So we have a decomposition in smaller complexes
\[
\vC^\bullet=\bigoplus_{i\in V}\vC^\bullet_{i}
\]
where, in case they don't have all the same length, we extend these complexes on the right with enough $0$'s.

We will prove in the next Section that this decomposition is the same as the decomposition that we get from Theorem~\ref{theorem:semifree-decomposition-of-sheaf} for a general semifree binoid. In the case of a simplicial binoid, we will be able to relate this to the links of the vertices of the simplicial complex we started with.

\begin{example}\label{example:spectrum-simplicial-binoid-8}
    Let's go back to our usual Example~\ref{example:spectrum-simplicial-binoid-7}, whose simplicial binoid is
    \[
    M_\triangle=\{x_1, x_2, x_3, x_4\mid x_1+x_4=\infty, x_2+x_4=\infty\}.
    \]
    We computed the localizations, we computed the images of $\O^*$ on the intersections of the coordinate affine fundamental open subsets and we computed the groups $\vC^j_i$. Let's put together these results to build up the \v{C}ech-Picard complex.
    
    The complex is
    \begin{equation*}
    \begin{tikzcd}[baseline=(current  bounding  box.center), cramped]
    \vC^\bullet: \vC^0\arrow[r, "\partial^0"] & \vC^1\arrow[r, "\partial^1"] & \vC^2\arrow[r, "\partial^2"] & 0
    \end{tikzcd}
    \end{equation*}
    that we can explicitly write as
    
    \vspace{1ex}
    
    \begin{tikzcd}[baseline=(current  bounding  box.center), cramped,
        /tikz/column 2/.append style={anchor=base east},
        /tikz/column 3/.append style={anchor=base west}]
        \ZZ^{\{1\}}\oplus\ZZ^{\{2\}}\oplus\ZZ^{\{3\}}\oplus\ZZ^{\{4\}}\arrow[r, "\partial^0"] & ({\ZZ^2})^{\{1,2\}}\oplus ({\ZZ^2})^{\{1,3\}}\oplus ({\ZZ^2})^{\{2,3\}}\oplus({\ZZ^2})^{\{3,4\}}\arrow[out=0, in=180, looseness=5, overlay, d, pos=0.04, "\partial^1"]\\ & ({\ZZ^3})^{\{1,2, 3\}}\arrow[r, "\partial^2"] & 0
    \end{tikzcd}
    
    \vspace{3ex}
    
    and we know that it is the direct sum of the chain complexes
    \begin{equation*}
    \begin{tikzcd}[baseline=(current  bounding  box.center), cramped]
    \vC^\bullet: \vC_i^0\arrow[r, "\partial^0"] & \vC_i^1\arrow[r, "\partial^1"] & \vC_i^2\arrow[r, "\partial^2"] & 0
    \end{tikzcd}
    \end{equation*}
    with $i=1, \dots, 4$.
    
    The groups in these smaller complexes are\footnote{Recall from our choice of notation that, since there is not any ambiguity, we omit the subscript index when $F$ is a vertex.}
    \begin{align*}
    \vC^0_1 = \ZZ &=\left\{\left(\sigma^{\{1\}}\right)\right\}&
    \vC^0_2 = \ZZ &=\left\{\left(\sigma^{\{2\}}\right)\right\}\\
    \vC^0_3 = \ZZ &=\left\{\left(\sigma^{\{3\}}\right)\right\} &
    \vC^0_4 = \ZZ &=\left\{\left(\sigma^{\{4\}}\right)\right\}    
    \end{align*}
    
    \begin{align*}
    \vC^1_1 = \ZZ^2 &=\left\{\left(\sigma^{\{1,2\}}_{1}, \sigma^{\{1,3\}}_{1}\right)\right\}&
    \vC^1_2 = \ZZ^2 &=\left\{\left(\sigma^{\{1,2\}}_{2}, \sigma^{\{2,3\}}_{2}\right)\right\}\\
    \vC^1_3 = \ZZ^3 &=\left\{\left(\sigma^{\{1,3\}}_{3}, \sigma^{\{2,3\}}_{3}, \sigma^{\{3,4\}}_{3}\right)\right\}&
    \vC^1_4 = \ZZ &=\left\{\left(\sigma^{\{3, 4\}}_{4}\right)\right\}
    \end{align*}
    
    \begin{align*}
    \vC^2_1 = \ZZ &=\left\{\left(\sigma^{\{1,2,3\}}_{1}\right)\right\}\hspace{2em} &         \vC^2_2 = \ZZ &=\left\{\left(\sigma^{\{1,2,3\}}_{2}\right)\right\}\hspace{2em}\\
    \vC^2_3 = \ZZ &=\left\{\left(\sigma^{\{1,2,3\}}_{3}\right)\right\}&
    \vC^2_4 = 0  &
    \end{align*}
    where all the $\sigma_*^{\{*\}}$ are integer numbers.
    
    Now that we split the groups, we turn to the maps. We start with $\partial^0$
    
    \vspace{1ex}
    
    \begin{tikzcd}[baseline=(current  bounding  box.center), cramped, row sep = 0pt, column sep = 1.3em,
        /tikz/column 1/.append style={anchor=base east},
        /tikz/column 2/.append style={anchor=base west},]
        \left(\sigma^{\{1\}}, \sigma^{\{2\}}, \sigma^{\{3\}}, \sigma^{\{4\}}\right)\arrow[r, mapsto, "\partial^0"] & \left((-\sigma^{\{1\}}, \sigma^{\{2\}}), (-\sigma^{\{1\}}, \sigma^{\{3\}}), (-\sigma^{\{2\}}, \sigma^{\{3\}}), (-\sigma^{\{3\}},\sigma^{\{4\}})\right)\\
    \end{tikzcd}
    
    \vspace{3ex}
    
    that we can write as the direct sum of the maps $\partial^0_i$ of the complexes $\vC^\bullet_i$
    
    \vspace{1ex}
    \begin{minipage}{0.5\textwidth}
        \begin{tikzcd}[baseline=(current  bounding  box.center), cramped, row sep = 0pt,
            /tikz/column 1/.append style={anchor=base east},
            /tikz/column 2/.append style={anchor=base west},]
            \ZZ^{\{1\}}\arrow[r, "\partial^0_1"] & \ZZ^{\{1,2\}}_{1}\oplus \ZZ^{\{1,3\}}_{1}\\
            \left(\sigma^{\{1\}}\right)\arrow[r, mapsto] & \left(-\sigma^{\{1\}}, -\sigma^{\{1\}}\right)\\
        \end{tikzcd}
    \end{minipage}\begin{minipage}{0.5\textwidth}        
    \begin{tikzcd}[baseline=(current  bounding  box.center), cramped, row sep =0pt,
        /tikz/column 1/.append style={anchor=base east},
        /tikz/column 2/.append style={anchor=base west},]
        \ZZ^{\{2\}}\arrow[r, "\partial^0_2"] & \ZZ^{\{1,2\}}_{2}\oplus \ZZ^{\{2,3\}}_{2}\\
        \left(\sigma^{\{2\}}\right)\arrow[r, mapsto] & \left(\sigma^{\{2\}}, -\sigma^{\{2\}}\right)\\
    \end{tikzcd}
\end{minipage}

\vspace{4ex}

\begin{minipage}{0.5\textwidth}
    \begin{tikzcd}[baseline=(current  bounding  box.center), cramped, row sep = 0pt,
        /tikz/column 1/.append style={anchor=base east},
        /tikz/column 2/.append style={anchor=base west},]
        \ZZ^{\{3\}}\arrow[r, "\partial^0_3"] & \ZZ^{\{1,3\}}_{3}\oplus \ZZ^{\{2,3\}}_{3}\oplus \ZZ^{\{3,4\}}_{3}\\
        \left(\sigma^{\{3\}}\right)\arrow[r, mapsto] & \left(\sigma^{\{3\}},\sigma^{\{3\}}, -\sigma^{\{3\}}\right)\\
    \end{tikzcd}
\end{minipage}\begin{minipage}{0.5\textwidth}        
\begin{tikzcd}[baseline=(current  bounding  box.center), cramped, row sep = 0pt,
    /tikz/column 1/.append style={anchor=base east},
    /tikz/column 2/.append style={anchor=base west},]
    \ZZ^{\{4\}}\arrow[r, "\partial^0_4"] & \ZZ^{\{3,4\}}_{4}\\
    \left(\sigma^{\{4\}}\right)\arrow[r, mapsto] & \left(\sigma^{\{4\}}\right)\\
\end{tikzcd}
\end{minipage}

\vspace{4ex}

We deal now with $\partial^1$

\vspace{1ex}

\begin{tikzcd}[baseline=(current  bounding  box.center), cramped,
    /tikz/column 1/.append style={anchor=base east}, column sep = -130pt]
    \left((\sigma^{\{1, 2\}}_{1}, \sigma^{\{1, 2\}}_{2}), (\sigma^{\{1, 3\}}_{1}, \sigma^{\{1, 3\}}_{3}), (\sigma^{\{2, 3\}}_{2}, \sigma^{\{2, 3\}}_{3}), (\sigma^{\{3, 4\}}_{3},\sigma^{\{3, 4\}}_{4})\right)\arrow[out=0, in=180, looseness=3, overlay, dr, pos=0.05, "\partial^1", mapsto]\\
    & \left(-\sigma^{\{1, 3\}}_{1}+\sigma^{\{1, 2\}}_{1}, \sigma^{\{2, 3\}}_{2}+\sigma^{\{1, 2\}}_{2}, \sigma^{\{2, 3\}}_{3}-\sigma^{\{1, 3\}}_{3}, 0\right)
\end{tikzcd}

\vspace{4ex}

That again we can split as the sum of

\vspace{2ex}

\begin{minipage}{0.5\textwidth}
    \begin{tikzcd}[baseline=(current  bounding  box.center), cramped, row sep = 0pt,
        /tikz/column 1/.append style={anchor=base east},
        /tikz/column 2/.append style={anchor=base west},]
        \ZZ^{\{1,2\}}_{1}\oplus \ZZ^{\{1,3\}}_{1}\arrow[r, "\partial^1_1"] & \ZZ^{\{1, 2, 3\}}_{1}\\
        \left(\sigma^{\{1, 2\}}_{1}, \sigma^{\{1, 3\}}_{1}\right)\arrow[r, mapsto] & \left(-\sigma^{\{1, 3\}}_{1}+\sigma^{\{1, 2\}}_{1}\right)
    \end{tikzcd}
\end{minipage}\begin{minipage}{0.5\textwidth}\flushright    
\begin{tikzcd}[baseline=(current  bounding  box.center), cramped, row sep =0pt,
    /tikz/column 1/.append style={anchor=base east},
    /tikz/column 2/.append style={anchor=base west},]
    \ZZ^{\{1,2\}}_{2}\oplus \ZZ^{\{2,3\}}_{2}\arrow[r, "\partial^1_2"] & \ZZ^{\{1, 2, 3\}}_{2}\\
    \left(\sigma^{\{1, 2\}}_{2}, \sigma^{\{2, 3\}}_{2}\right)\arrow[r, mapsto] & \left(\sigma^{\{2, 3\}}_{2}+\sigma^{\{1, 2\}}_{2}\right)
\end{tikzcd}
\end{minipage}

\vspace{4ex}

\begin{minipage}{0.65\textwidth}
    \begin{tikzcd}[baseline=(current  bounding  box.center), cramped, row sep = 0pt,
        /tikz/column 1/.append style={anchor=base east},
        /tikz/column 2/.append style={anchor=base west},]
        \ZZ^{\{1,3\}}_{3}\oplus \ZZ^{\{2,3\}}_{3}\oplus \ZZ^{\{3,4\}}_{3}\arrow[r, "\partial^1_3"] & \ZZ^{\{1, 2, 3\}}_{3}\\
        \left(\sigma^{\{1, 3\}}_{3}, \sigma^{\{2, 3\}}_{3}, \sigma^{\{3, 4\}}_{3}\right)\arrow[r, mapsto] & \left(\sigma^{\{2, 3\}}_{3}-\sigma^{\{1, 3\}}_{3}\right)
    \end{tikzcd}
\end{minipage}\begin{minipage}{0.35\textwidth}        
\begin{tikzcd}[baseline=(current  bounding  box.center), cramped, row sep = 0pt,
    /tikz/column 1/.append style={anchor=base east},
    /tikz/column 2/.append style={anchor=base west},]
    \ZZ^{\{3, 4\}}_{4}\arrow[r, "\partial^1_4"] & 0\\
    \left(\sigma^{\{3, 4\}}_{4}\right)\arrow[r, mapsto] & 0
\end{tikzcd}
\end{minipage}

\vspace{4ex}

So we recover that we can describe the original complex as a direct sum of the smaller ones, like we proved in general above.
\end{example}

\subsection{The Link Complex and its Chain Complex}
In the previous Section we described the \v{C}ech-Picard complex for the covering $D(x_i)$ of the punctured spectrum of a simplicial binoid completely in terms of the combinatorics of the simplicial complex itself.
In this Section, we will relate it to the chain complex used to compute simplicial cohomology with coefficients in $\ZZ$ of the link of a vertex in our simplicial complex.

\begin{lemma}\label{lemma:commentD}We have\footnote{Recall that $\widetilde{\C}^\bullet$ is the extended complex that computes reduced simplicial cohomology.}
    \[
    \vC^\bullet_{i}=\widetilde{\C}^{\bullet-1}(\lk_{\triangle}(i), \ZZ)
    \]
    as cochain complexes.
\end{lemma}
\begin{proof}
    From Definition~\ref{definition:extension-by-zeros} we have the isomorphism $\O^*_{x_i}\cong {j_i}_!\ZZ$, induced by the open embedding.
    Let $\m_i$ be the maximal ideal (w.r.t.\ inclusion) of $D(x_i)$, i.e.\ $\m_i =\langle x_j\mid j\neq i\rangle$. Let $D^\bullet(x_i)$ be the punctured spectrum of the link, i.e.\ $D^\bullet(x_i)=D(x_i)\smallsetminus\{\m_i\}\cong \Spec^\bullet M_{\lk_\triangle(i)}$.
    
    We can cover $j_i(\Spec^\bullet M_{\lk_\triangle(i)})$ with the affine open subsets $D(x_i+x_j)=D(x_i)\cap D(x_j)$, with $i\neq j$, on which we know that $\ZZ$ is an acyclic sheaf, because it is a sheaf of abelian groups on an affine scheme of binoids.
    
    $D(x_i)$ is the only coordinate affine open subset of $\Spec^\bullet M_\triangle$ that covers $\m_i$, and this is the only point not covered by another coordinate affine open subset, i.e.\ $\m_i\notin D(x_j)$, for any $j\neq i$.
    
    When we look for cohomology, we can use the following complexes to compute the simplicial cohomology of the link $\lk_\triangle(i)$, the cohomology of the sheaf $\ZZ$ on $D^\bullet(x_i)$ and the cohomology of $\O^*_{x_i}$ on $\Spec^\bullet M_\triangle$, respectively
    \[
    \C^\bullet(\lk_\triangle(i), \ZZ), \qquad \vC^\bullet(\{D(x_i)\cap D(x_j)\}_{j\neq i}, \ZZ), \qquad \vC^\bullet(\Spec^\bullet M_\triangle, \O^*_{x_i}).
    \]
    We have that the first two are the same, thanks to Theorem~\ref{thm:simplicial-cohomology}. The third is different, because there is one more open subset (namely $D(x_i)$) on which this sheaf evaluation is $\ZZ$.
    
    What we have is indeed that, for any $k\geq 1$,
    \begin{align*}
    \vC^k(\Spec^\bullet M_\triangle, \O^*_{x_i}) &= \bigoplus_{\begin{subarray}{c}
        F\in\triangle_k
        \end{subarray}}\Gamma(D(F), \O^*_{x_i})\\
    &=\bigoplus_{\begin{subarray}{c}
        i\in F\in\triangle_k
        \end{subarray}}\Gamma(D(F), \O^*_{x_i})\oplus \bigoplus_{\begin{subarray}{c}
        i\notin G\in\triangle_k
        \end{subarray}}\Gamma(D(G), \O^*_{x_i})\\
    &=\bigoplus_{\begin{subarray}{c}
        i\in F\in\triangle_k
        \end{subarray}}\Gamma(D(F), \O^*_{x_i})\oplus 0\\
    &=\bigoplus_{\begin{subarray}{c}
        i\in F\in\triangle_k
        \end{subarray}}\Gamma(D(F\smallsetminus \{i\})\cap D(x_i), \O^*_{x_i})\\
    &=\bigoplus_{\begin{subarray}{c}
        i\in F\in\triangle_k\\
        G=F\smallsetminus\{i\}
        \end{subarray}}\Gamma(D(G)\cap D(x_i), {j_i}_!\ZZ)\\
    &=\bigoplus_{\begin{subarray}{c}
        i\in F\in\triangle_k\\
        G=F\smallsetminus\{i\}
        \end{subarray}}\Gamma(D(G)\cap D(\{i\}), \ZZ)\\
    &=\bigoplus_{\begin{subarray}{c}
        G\in\left(\lk_\triangle(i)\right)_{k-1}
        \end{subarray}}\Gamma(D(G), \ZZ)\\
    &=\C^{k-1}(\lk_\triangle(i), \ZZ).
    \end{align*}
    So we just proved that, assuming $k\geq 1$,
    \[
    \vC^k(\Spec^\bullet M_\triangle, \O^*_{x_i})=\vC^{k-1}(\{D(x_i)\cap D(x_j)\}_{j\neq i}, \ZZ) = \C^{k-1}(\lk_\triangle(i), \ZZ).
    \]
    For $k= 0$ these groups are trivially the same
    \[
    \vC^0(\Spec^\bullet M, \O^*_{x_i}) = \ZZ \cong \widetilde{\C}^{-1}(\lk_\triangle(i), \ZZ).
    \]
    Applying Theorem~\ref{theorem:split-cech-picard-groups} we then get our thesis
    \[
    \vC^\bullet_i=\vC^\bullet(\Spec^\bullet M, \O^*_{x_i}) = \widetilde{\C}^{\bullet-1}(\lk_\triangle(i), \ZZ).\qedhere
    \]
\end{proof}

Patching together what we proved in the previous sections we get the following result, that makes the final step towards the next section, where we will finally relate sheaf and simplicial cohomology.

\begin{theorem}\label{theorem:sum-cech-picard-link-groups} We can rewrite the punctured \v{C}ech-Picard complex as direct sum of the reduced simplicial cochain complexes of the links of vertices in $\triangle$
    \[
    \vC^{\bullet}=\bigoplus_{i\in V}\widetilde{\C}^{\bullet-1}(\lk_\triangle(i), \ZZ).
    \]
\end{theorem}
\begin{proof}
    This follows from the previous Lemma and from the sheaf decomposition
    \[
    \O^*_{M_\triangle}=\O^*_{x_1}\oplus\dots \oplus\O^*_{x_n},
    \]
    already exploited in Theorem~\ref{theorem:split-cech-picard-groups}.
\end{proof}

\section{Cohomology}
Summing up what we did until now, we can produce the following Theorem that allows us to compute sheaf cohomology n term of reduced simplicial cohomology.
\begin{theorem}\label{theorem:cohomology-simplicial-complex}
    Let $\triangle$ be a simplicial complex on the finite vertex set $V = [n]$. We have the following explicit formula for the computation of the cohomology groups of its \v{C}ech-Picard complex
    \begin{equation}\label{formula:cohomology-simplicial-binoid}
    {\H^j\left(\Spec^\bullet M_\triangle, \O^*_{M_\triangle}\right)}\cong\bigoplus_{i\in V}\glslink{reduced-simplicial-cohomology}{\widetilde{\H}^{j-1}\left(\lk_\triangle(i), \ZZ\right)}
    \end{equation}
    for $j\geq0$.
\end{theorem}
\begin{proof}
    Recall from the results in this chapter that we can use the open subsets defined by the variables $\{D(x_i)\}$ as a \v{C}ech covering for $\Spec^\bullet M_\triangle$ (Proposition~\ref{proposition:minimal-covering}) and we can associate to any non empty intersection in this covering a face of the simplicial complex (Proposition ~\ref{proposition:intersection-open-subsets}).
    This allows us to index the components of the sections with faces and their elements (Notation on page~\pageref{notation:simplicial-binoid-faces}), making us able to index the groups appearing in the \v{C}ech-Picard complex with the vertices in the simplicial complex and finally allowing us to split the cochain complex into smaller ones (Theorem~\ref{theorem:split-cech-picard-groups} for the groups and Theorem~\ref{theorem:split-cech-picard-maps} for the maps), and obtain
    \begin{equation*}
    \vC^\bullet=\bigoplus_{i\in V}\vC^\bullet_{i}.
    \end{equation*}
    On the other hand, in Theorem~\ref{theorem:semifree-decomposition-of-sheaf} we proved that there exists an isomorphism of sheaves
    \[
    \O^*_{M_\triangle}=\O^*_{x_1}\oplus \O^*_{x_2}\oplus \dots \oplus \O^*_{x_n}
    \]
    where $\O^*_{x_i}$ is isomorphic to ${j_i}_!\ZZ$, where $j_i$ is the open embedding $D(x_i)\longrightarrow \Spec^\bullet M_\triangle$.
    
    In Lemma~\ref{lemma:commentD} we observed that
    \[
    \vC^\bullet\left(\Spec^\bullet M, \O^*_{x_i}\right) = \widetilde{\C}^{\bullet-1}\left(\lk_\triangle(i), \ZZ\right).
    \]
    By putting all these things together, we obtain that
    \begin{equation}
    \begin{aligned}
    \H^j\left(\Spec^\bullet M_\triangle, \O^*_{M_\triangle}\right)&=\H^j\left(\Spec^\bullet M_\triangle, \bigoplus_{i\in V}\O^*_{x_i}\right)\\
    &=\bigoplus_{i\in V}\H^j\left(\Spec^\bullet M_\triangle, \O^*_{x_i}\right)\\
    &=\bigoplus_{i\in V}\widetilde{\H}^{j-1}\left(\lk_\triangle(i), \ZZ\right),
    \end{aligned}
    \end{equation}
    where $\widetilde{\H}$ is the usual reduced simplicial cohomology.
\end{proof}

\begin{corollary}\label{corollary:cohomology-simplicial-binoid}
    The $0$-th and $1$-th cohomology groups are always free and they have the following form
    \begin{align*}
    \H^0(\Spec^\bullet M_\triangle, \O^*)&=\ZZ^{\#\{0-\dim\text{ facets of $\triangle$}\}}\\
    \H^1(\Spec^\bullet M_\triangle, \O^*)&=\ZZ^r
    \end{align*}
    where
    \begin{align*}
    r&=\sum_{v\in V}\rk({\widetilde{\H}^0(\lk_\triangle(v), \ZZ)})\\
    &=\sum_{v\in V} \rk({\H^0(\lk_\triangle(v), \ZZ)}) -\#\{\text{$0$-$\dim$ non-facets of $\triangle$}\}
    \end{align*}
\end{corollary}

\begin{corollary}
    $\H^j\left(\Spec^\bullet M_\triangle, \O^*_{M_\triangle}\right)=0$, for $j\geq \dim\triangle+1$.
\end{corollary}
\begin{proof}
    Let $m$ be the dimension of $\triangle$, so the facets of maximal dimension have $m+1$ vertices. Each intersection of $m+2$ subsets in $\{D(x_i)\}$ will then be empty, because there are no faces with $m+2$ vertices, so $\vC^{m+1}\left(\{D(x_i)\}, \O^*_{M_\triangle}\right)=0$ and the results follows trivially.
\end{proof}

\begin{remark}
    Another proof of the previous Corollary can be achieved through the dimension of the link complexes, that have dimension at most $m-1$, so their reduced cohomology is trivial in degree higher than that.
\end{remark}

\begin{remark}\label{remark:grothendieck-theorem}
    The previous Corollary is an expected result, particularly in view of the known Theorem of Grothendieck in the famous Tohoku paper, \cite[Theorem 3.6.5]{grothendieck1957surquelques}, where he shows it for any abelian sheaf on a topological space $X$ and degree $j \geq n+1$, where $n$ is the combinatorial dimension of $X$, i.e\ the maximum length of strictly decreasing chains of closed subsets.
\end{remark}

\section{Special cases}\label{section:special-cases-simplicialcomplexes}
In this section we will study some special cases and examples, for which we carry on explicit computations of the cohomology groups.

\subsection{The \texorpdfstring{$0$}{0}-dimensional case}

\begin{corollary}[$\dim\triangle = 0$]\label{proposition:cohmology-0-dim}
    For a $0$-dimensional simplicial complex on $V=[n]$ the only non trivial cohomology is $\H^0(\Spec^\bullet M_\triangle, \O^*) = \ZZ^n$.
\end{corollary}
\begin{proof}
    We observe that the complex is just $n$ isolated vertices, then we apply the previous Corollary~\ref{corollary:cohomology-simplicial-binoid} and we get our thesis.
\end{proof}

\subsection{The affine space}

\begin{corollary}[The affine space]
    If $\triangle$ is an $(n-1)$-dimensional simplex, then $M_\triangle\cong{(\NN^n)}^\infty$. If $n\geq 2$, we have that $\lk_\triangle(i)\cong\Delta_{n-2}$, the $(n-2)$-dimensional simplex, whose reduced cohomology is trivial, so 
    \[
    \H^k(\Spec^\bullet M_\triangle, \O^*)=0
    \]
    for all $k\geq 0$.
    On the other hand, if $n=1$ then $\triangle$ is a point, the $0$-dimensional simplex, so $\lk_\triangle(i)=\{\varnothing\}$ and we can either use the result above or the property of reduced simplicial cohomology to get that $\H^0(\Spec^\bullet M_\triangle, \O^*) = \ZZ$, $\H^k(\Spec M_\triangle, \O^*) = 0$ for $k\geq 1$.
\end{corollary}

\subsubsection{The affine \texorpdfstring{${(\NN^3)}^\infty$}{aff3space}}\label{subsection:3-vertices}
The simplicial complex $\triangle$ is

\begin{minipage}[c]{0.5\textwidth}
    \[
    \triangle = \left\{\begin{aligned}
    &\varnothing,\{1\},\{2\}, \{3\},\\
    &\{1, 2\}, \{1, 3\}, \{2, 3\},\\
    &\{1, 2, 3\}\end{aligned}\right\}
    \]\vfill
\end{minipage}\begin{minipage}[c]{0.2\textwidth}
\begin{tikzpicture}[baseline=(current  bounding  box.center)]
\tikzstyle{point}=[circle,thick,draw=black,fill=black,inner sep=0pt,minimum width=4pt,minimum height=4pt]
\node (v1)[point, label={[label distance=0cm]-135:$1$}] at (0,0) {};
\node (v2)[point, label={[label distance=0cm]90:$2$}] at (0.5,0.87) {};
\node (v3)[point, label={[label distance=0cm]-45:$3$}] at (1,0) {};
\draw (v1.center) -- (v2.center);
\draw (v1.center) -- (v3.center);
\draw (v2.center) -- (v3.center);
\draw[color=black!20, style=fill, outer sep = 20pt] (0.1,0.06) -- (0.5,0.77) -- (0.9,0.06) -- cycle;
\end{tikzpicture}
\end{minipage}\begin{minipage}[c]{0.3\textwidth}
\vfill
$M_{\triangle} = (x_1, x_2, x_3\mid \varnothing)$\\
\vfill
\end{minipage}

Its spectrum is
\[
\Spec^\bullet  M_\triangle=\bigl\{\infty, \{x_1\}, \{x_2\}, \{x_3\}, \{x_1, x_2\}, \{x_1, x_3\}, \{x_2, x_3\}\bigr\}
\]
Thanks to the symmetry of this complex, we can describe any of the $\O^*_{x_i}$ and then use it for all the others.
The link of $x_1$ is the simplex on $\{x_2, x_3\}$, so we can see that the complex for simplicial cohomology of this link is
\begin{equation*}
\begin{tikzcd}[baseline=(current  bounding  box.center), cramped, row sep = 0pt,
/tikz/column 1/.append style={anchor=base east},
/tikz/column 2/.append style={anchor=base west}, column sep = 2em]
\widetilde{\C}_{\lk_\triangle(1)}: & \ZZ^\varnothing \arrow[r] & \ZZ^{\{2\}}\oplus\ZZ^{\{3\}} \arrow[r] & \ZZ^{\{2, 3\}} \arrow[r] & 0\\
& \alpha \arrow[r, mapsto]  & (\alpha, \alpha)                       & &\\
&   & (\alpha_2, \alpha_3) \arrow[r, mapsto]                        & -\alpha_2+\alpha_3 &\\
\end{tikzcd}
\end{equation*}
and has trivial cohomology groups, so the cohomology $\H^i(\Spec^\bullet M_\triangle, \O^*)$ is $0$ in every degree.
The details of the computations of the complex of the link are carried over in Subsection~\ref{subsection:computing-for-favourite-example}, where we have a similar situation.

\subsection{Graphs}

\begin{corollary}[$\dim\triangle = 1$]\label{corollary:graphs-general}
    Let $\triangle = (V, E)$ be a simple graph. Then $\H^0(\vC^\bullet)=\ZZ^s$, where $s=\#\{v\in V\mid \nexists e\in E, v\in e\}$ is the number of isolated vertices of the graph and $\H^1(\vC^\bullet)= \ZZ^r$, where $r=\displaystyle\sum_{\begin{subarray}{c}
        v\in V\\
        v\text{ not isolated}
        \end{subarray}}\left(\#\{e\in E\mid v\in e\}-1\right)$, the sum over all non-isolated vertices of the number of edges containing them, minus 1. All higher cohomologies vanish.
\end{corollary}

\begin{remark}
    It is important to highlight that, in this case with this very explicit description, again we are able to compute cohomology via indexing on the vertices, but we are not able to transform this in an index that depends on the edges. Morally, vertices give generators and edges give relations between them, although not in a completely straight-forward way.
\end{remark}

\begin{definition}
    Let $(V, E)$ be a graph. The degree of a vertex $v\in V$ is the number of neighbours of $v$, i.e.\ the number of edges that contain $v$. A graph is $k$-regular if every vertex has the same degree $k$.
\end{definition}

\begin{remark}
    Let $k_v$ be the degree of $v\in V$ in $\triangle$. The Corollary above for the $\H^1$ of a graph can be rewritten as $\H^1(\vC^\bullet)= \ZZ^r$, where $r=\displaystyle\sum_{\begin{subarray}{c}
        v\in V\\
        v\text{ not isolated}
        \end{subarray}}\left(k_v-1\right).$
    
\end{remark}

\begin{remark}[$k$-regular graph]
    Let $V=[n]$. In the case $k=0$, we are looking at the case in Corollary~\ref{proposition:cohmology-0-dim}. Let $k>0$, so any vertex $i$ is contained in $k$ edges. There are no isolated vertices, so $\H^0(\vC^\bullet)=0$. About the first cohomology, the link of each $i$ is exactly $k$ disjoint vertices, so
    \[
    \H^1(\vC^\bullet) = \ZZ^{n(k-1)}.
    \]
\end{remark}
\begin{remark}[Complete graph]
    A particular case is the complete graph $K_n$, that is $(n-1)$-regular and for which we get
    \[
    \H^1(\vC^\bullet) = \ZZ^{n(n-2)}.
    \]
\end{remark}

\begin{remark}[Cycle graph]
    The cycle graph on $V=[n]$ is the graph with edges $E = \left\{\{i, i+1\}\mid i\in [n-1]\right\}\cup\{1, n\}$. This graphs is $2$-regular, so
    \[
    \H^1(\vC^\bullet) = \ZZ^n.
    \]
\end{remark}

\begin{remark}[Path graph]
    The path graph on $V=[n]$ is the acyclic graph with edges $E = \left\{\{i, i+1\}\mid i\in [n-1]\right\}$. In this case all but two vertices ($1$ and $n$, the extremal ones) have degree $2$, and the two extremal have degree 1, so
    \[
    \H^1(\vC^\bullet) = \ZZ^{n-2}.
    \]
\end{remark}

During the presentation of a poster in October 2015 in Osnabrück, professor Gunnar Fløystad asked us what is the relation, if there is any, between this local Picard group of a simplicial binoid of a graph and the graph-theoretical Picard group of the graph, also knows as Jacobian group, Sandpile group and Critical group. The latter is defined to be the torsion part of $\faktor{\ZZ^V}{\mathrm{im}(L)}$, there $L$ is the adiacency matrix of the graph. One can prove that this group is finite, see for example \cite[Proposition 32.2]{biggs1997algebraic} and \cite[Theorem 7.3]{biggs1999chip}. So, as far as we know, there is no relation between our (free) groups and the graph-theoretical ones.

\subsection{An example with isolated vertices}\label{secondexample}
Let $\triangle$ be the simplicial complex on $V=\{1, 2, 3, 4\}$ with facets $\{1, 2\}$, $\{3\}$ and $\{4\}$.
\begin{minipage}[c]{0.5\textwidth}
    \[
    \triangle = \left\{\begin{aligned}
    &\varnothing,\{1\}, \{2\}, \{3\}, \{4\}, \\
    &\{1, 2\}
    \end{aligned}\right\}
    \]\vfill
\end{minipage}
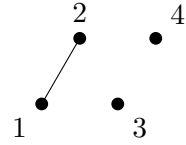
\begin{minipage}[c]{0.5\textwidth}
    \hspace{6em}\begin{tikzpicture}
    \tikzstyle{point}=[circle,thick,draw=black,fill=black,inner sep=0pt,minimum width=4pt,minimum height=4pt]
    \node (v1)[point, label={[label distance=0cm]-135:$1$}] at (0,0) {};
    \node (v2)[point, label={[label distance=0cm]90:$2$}] at (0.5,0.87) {};
    \node (v3)[point, label={[label distance=0cm]-45:$3$}] at (1,0) {};
    \node (v4)[point, label={[label distance=0cm]+45:$4$}] at (1.5,0.87) {};
    
    \draw (v1.center) -- (v2.center);
    \end{tikzpicture}
\end{minipage}

with simplicial binoid 
\[
M_\triangle=\left\{x_1, x_2, x_3, x_4\midd \begin{aligned}
x_1+x_3=\infty, x_1+x_4=\infty\\
x_2+x_3=\infty, x_2+x_4=\infty\\
x_3+x_4=\infty
\end{aligned}\right\}.
\]

The \v{C}ech-Picard complex on the punctured spectrum looks like
\begin{equation*}
\begin{tikzcd}[cramped, row sep = 0pt]
\vC: 0\rar &\displaystyle \bigoplus_{\{i\}\in\triangle}\ZZ\rar["\partial^0"]&\displaystyle \bigoplus_{\{i, j\}\in\triangle}\ZZ^2\rar & 0\\
&(\alpha_1, \alpha_2, \alpha_3, \alpha_4)\rar[mapsto] & (-\alpha_1, \alpha_2)
\end{tikzcd}
\end{equation*}
So we have $\H^0(\vC)=\ZZ^2, \H^1(\vC)=0$ and these are the results that we were expecting from Corollary~\ref{corollary:graphs-general}.

\subsection{The case \texorpdfstring{$x_1+x_2+x_3=\infty$}{x1+x2+x3=inf}}\label{subsection:x+y+z=infty}
The (one-dimensional) simplicial complex $\triangle$ is

\begin{minipage}[c]{0.45\textwidth}
    \[
    \triangle = \left\{\begin{aligned}
    &\varnothing,\{1\},\{2\}, \{3\},\\
    &\{1, 2\}, \{1, 3\}, \{2, 3\}\end{aligned}\right\}
    \]\vfill
\end{minipage}\begin{minipage}[c]{0.18\textwidth}
\begin{tikzpicture}[baseline=(current  bounding  box.center)]
\tikzstyle{point}=[circle,thick,draw=black,fill=black,inner sep=0pt,minimum width=4pt,minimum height=4pt]
\node (v1)[point, label={[label distance=0cm]-135:$1$}] at (0,0) {};
\node (v2)[point, label={[label distance=0cm]90:$2$}] at (0.5,0.87) {};
\node (v3)[point, label={[label distance=0cm]-45:$3$}] at (1,0) {};
\draw (v1.center) -- (v2.center);
\draw (v1.center) -- (v3.center);
\draw (v2.center) -- (v3.center);
\end{tikzpicture}
\end{minipage}\begin{minipage}[c]{0.37\textwidth}
\vfill
$(x_1, x_2, x_3\mid x_1+x_2+x_3=\infty)$\\
\vfill
\end{minipage}

and its spectrum is
\[
\Spec^\bullet M_\triangle=\bigl\{\{x_1\}, \{x_2\}, \{x_3\}, \{x_1, x_2\}, \{x_1, x_3\}, \{x_2, x_3\}\bigr\}
\]
From what we did above, we know that the cohomology groups are $\H^0=0$ and $\H^1=\ZZ^3$. In this case we have a nice explicit description of the elements of $\H^1$ with $M_\triangle$-sets that are invertible on the punctured spectrum.\footnote{Recall that an invertible $M$-set $S$ on an open subset $U$ is such that $\widetilde{S}$ is invertible as a $\O_{M_\triangle}$-sheaf on $U$, see Definition~\ref{definition:invertible-sheaf} for the latter.} Indeed, they have a global description as $M_\triangle$-sets of the form
\[
S=\left(e_1, e_2, e_3\midd \begin{aligned}
e_1+ax_2&=e_2+bx_1,\\
e_1+cx_3&=e_3+dx_1,\\
e_2+ex_3&=e_3+fx_2
\end{aligned}\right)
\]
with $a, b, c, d, e$ and $f$ natural strictly positive numbers (so they yield global objects). For all $x_i$'s we have $S_{x_i}=\langle e_i\rangle_{{M_\triangle}_{x_i}}$, so for every choice of $(a,b,c,d,e,f)$, the $M_\triangle$-set $S$ is invertible on the punctured spectrum.
\\
To study the isomorphism classes, we focus on $x_1$ and $x_2$ and then other cases will follow similarly, thanks to the symmetry of the problem.

We have isomorphisms $M_{x_1}\cong S_{x_1}$, where $e_2=e_1+ax_2-bx_1$, $e_3=e_1+cx_3-dx_1$, so it is clear that we need $b$ and $d$ strictly positive in order to compute these localisations. Similarly we have that $M_{x_2}\cong S_{x_2}$, that gives us other positivity conditions on the coefficients.
When we intersect, we have that the second localizations are equal
\[
{(S_{x_1})}_{x_2}=S_{x_1+x_2}=S_{x_2+x_1}={(S_{x_2})}_{x_1}
\]and we can draw the following diagram
\begin{equation*}
\begin{tikzcd}[baseline=(current  bounding  box.center),
/tikz/column 1/.append style={anchor=base east},
/tikz/column 2/.append style={anchor=base east}, 
/tikz/column 4/.append style={anchor=base west},
/tikz/column 5/.append style={anchor=base west}, column sep = 2em]
M_{x_1}\rar["\sim"] \arrow[dr]& S_{x_1}\arrow[dr] & & S_{x_2}\arrow[dl]& \lar[swap, "\sim"] M_{x_2}\arrow[dl]\\
&M_{x_1+x_2}\arrow[r, "\sim"] \arrow[rr, bend right=30, "\sim", swap]& S_{x_1+x_2} & \arrow[swap, l,"\sim"] M_{x_2+x_1}&\\[-1.5em]
&0\arrow[rr, bend right=30, mapsto, start anchor = east, end anchor = west]& & ax_2-bx_1&
\end{tikzcd}
\end{equation*}

By symmetry we now know that locally, on the double intersections, we have the relations

\vspace{2ex}

\hspace{\mathindent}\begin{tabular}{l|c}
    $D(x_1)\cap D(x_2)$ & $e_1=e_2+bx_1-ax_2$\\
    $D(x_1)\cap D(x_3)$ & $e_1=e_3+dx_1-cx_3$\\
    $D(x_2)\cap D(x_3)$ & $e_2=e_3+fx_2-ex_3$\\
\end{tabular}

\vspace{2ex}

Now, since every cocycle in $\H^1$ represents a local isomorphism $\varphi$ of the vector bundles, we can try to express these relations by studying the cohomology.

Thanks to symmetry, we can concentrate on one particular variable, say $x_1$, and study $\O^*_{x_1}$. The link of $1$ in the complex is $\{\varnothing, \{2\}, \{3\}\}$. This gives us the chain complex for reduced simplicial cohomology\footnote{Here $\varnothing$ on $\ZZ$ is an index, the object $\ZZ^\varnothing$ is not the set of maps, that would be empty.}
\begin{equation*}
\begin{tikzcd}[baseline=(current  bounding  box.center), row sep = 0pt,
/tikz/column 1/.append style={anchor=base east},
/tikz/column 2/.append style={anchor=base west}, column sep = 2em]
\widetilde{\C}_{\lk_\triangle(1)}: & \ZZ^\varnothing \arrow[r, "\partial^{-1}"] & \ZZ^{\{2\}}\oplus\ZZ^{\{3\}} \arrow[r, "\partial^0"] & 0\\
& \alpha \arrow[r, mapsto]  & (\alpha, \alpha)                       & &
\end{tikzcd}
\end{equation*}

Thus inducing the relation in $\H^1$
\[
(\beta_1, \beta_2)\sim (\beta'_1, \beta'_2) \Longleftrightarrow \beta_1-\beta_2 = \beta'_1 - \beta'_2
\]

When reporting this back to the original simplicial complex, this means that two $M_\triangle$-sets $S$ and $S'$ are in the same isomorphism class if the coefficients of $x_1$ in the face $\{1, 2\}$ and $\{1, 3\}$ have the same difference.

So, such class $\varphi\in\H^1$ corresponds to the $6$-uple $(b, -a, d, -c, f, -e)$ of coefficients up to the equivalence relation given by the subgroup $(\alpha, \beta, \alpha, \gamma, -\beta, \gamma)$.
\\
So, all the possible $M_\triangle$-sets invertible on $\Spec^\bullet M_\triangle$ are of the same form as $S$, varying the parameters $(a, b, c, d, e, f)$ with isomorphism given by the relations stated above.
To explain more clearly, consider another invertible $M_\triangle$-set 
\[
S'=\left(e_1, e_2, e_3\midd \begin{aligned}
e_1+a'x_2&=e_2+b'x_1,\\
e_1+c'x_3&=e_3+d'x_1,\\
e_2+e'x_3&=e_3+f'x_2
\end{aligned}\right)
\]
Then $S'\cong S$ if and only if
\begin{align*}
b-d=b'-d'&&a-f=a'-f'&&c-e=c'-e'.
\end{align*}

We can see the line bundles that correspond to the vertex $x_1$ in the cohomology of $\O^*_{x_1}$, $\H^1(\O^*_{x_1})\cong \ZZ$, as the subgroup of $\faktor{\ZZ^6}{\ZZ^3}$ generated by $(1, b, 1, d, 1, 1)$. Similarly for $x_2$ the subgroup is generated by $(a, 1, 1, 1, 1, f)$ and for $x_3$ by $(1, 1, c, 1, e, 1)$.

\subsection{The star graph}\label{example:the-star-graph}

Consider the simplicial complex on $V=[n]$ with facets $\{1, k\}$ for any $k=2, \dots, n$, and no other facets

\begin{minipage}[c]{0.5\textwidth}
    \[
    \triangle = \left\{\begin{aligned}
    &\varnothing,\{1\},\{2\}, \{3\},\dots, \{n\}\\
    &\{1, 2\}, \{1, 3\}, \dots, \{1, n\}\end{aligned}\right\}
    \]\vfill
\end{minipage}
\begin{minipage}[c]{0.5\textwidth}
    \begin{tikzpicture}[baseline=(current  bounding  box.center)]
    \tikzstyle{point}=[circle,thick,draw=black,fill=black,inner sep=0pt,minimum width=4pt,minimum height=4pt]
    \node (v1)[point, label={[label distance=0cm]-135:$1$}] at (0,0) {};
    \node (v2)[point, label={[label distance=0cm]45:$2$}] at (1,0.87) {};
    \node (v3)[point, label={[label distance=0cm]0:$3$}] at (1,0) {};
    \node (v4)[] at (1,-0.87) {$\vdots$};
    \node (vk)[point, label={[label distance=0cm]-45:$n$}] at (1,-2) {};
    
    \draw (v1.center) -- (v2.center);
    \draw (v1.center) -- (v3.center);
    \draw (v1.center) -- (vk.center);
    \end{tikzpicture}
\end{minipage}

From above we know that $\H^0=0$ and $\H^1=\ZZ^{n-2}$. We can use an argument similar to the previous example to show that the line bundles can be globally represented by $M$-sets of the form
\[
S=\left(e_1, e_2, e_3, \dots, e_n\midd \begin{aligned}
e_1+x_2&=e_2+a_2x_1,\\
e_1+x_3&=e_3+a_3x_1,\\
&\dots\\
e_1+x_n&=e_n+a_nx_1
\end{aligned}\right),
\]
with $a_k$ a strictly positive integer. It is easy to check that on any $D(x_i)$, $S_{x_i}$ is generated by $e_i$. The intersections $D(x_i)\cap D(x_j)$ are empty if and only if at least one of $i$ and $j$ is different from one. Assume that $i=1$, then over $D(x_1)\cap D(x_j)=D(x_1+x_j)$ the cohomology class is given by $a_jx_1-x_j$. In particular, we can easily see how
\[
\H^1\cong \ZZ^{n-2}\cong \faktor{\ZZ^{n-1}}{\Delta \ZZ},
\]
where $\Delta\ZZ$ is the diagonal. We can understand this if we think of adding $(\lambda, \lambda, \dots, \lambda)$ as a scaling of the variables by the same factor, i.e.\ $S$, defined by $(a_2, \dots, a_n)$ and $S'$, defined by $(b_2, \dots, b_n)$, are isomorphic if and only if $(a_2, \dots, a_n)-(b_2, \dots, b_n)=(\lambda, \dots, \lambda)\in\ZZ^{n-1}$.

So, the cohomology classes are generated by the elements $(0, \dots, 0, 1, 0, \dots, 0)$, but if we pick $a_i=1$ and $a_j=0$ for $j\neq i$, the $M$-set
\[
S=\left(e_1, e_2, e_3, \dots, e_n\midd \begin{aligned}
e_1+x_2&=e_2,\\
e_1+x_3&=e_3,\\
&\dots\\
e_1+x_i&=e_i+x_1\\
&\dots\\
e_1+x_n&=e_n
\end{aligned}\right)\cong \left(e_1, e_i \midd e_1+x_i=e_i+x_1\right)
\]
does not define a line bundle because, for example, it is not trivial on $D(x_2)$. So, a representation of the generators is given by $(1, 1, \dots, 2, \dots, 1)$, that is clearly in relation with $(0, 0,\dots, 1,\dots, 0)$, but the latter does not define a line bundle.

Moreover, since smashing line bundles corresponds to adding the classes of the vectors $(a_2, \dots, a_n)$ in $\H^1$, we get that $S\wedge S'$, restricted to the punctured spectrum, is given by $(a_2+b_2, \dots, a_n+b_n)$, where all this $a_i$'s and $b_i$'s are strictly positive.

\subsection{Our favourite Example}\label{subsection:computing-for-favourite-example}
Let us go back one more time to our favourite Example~\ref{example:spectrum-simplicial-binoid-8} 

\begin{minipage}[c]{0.5\textwidth}
    \[
    \triangle = \left\{\begin{aligned}
    &\varnothing,\{1\},\{2\},\{3\},\{4\},\\
    &\{1, 2\}, \{1, 3\}, \{2, 3\}, \{3, 4\}\\
    &\{1, 2, 3\}\end{aligned}\right\}
    \]\vfill
\end{minipage}\begin{minipage}[c]{0.5\textwidth}
\hspace{6em}\begin{tikzpicture}
\tikzstyle{point}=[circle,thick,draw=black,fill=black,inner sep=0pt,minimum width=4pt,minimum height=4pt]
\node (v1)[point, label={[label distance=0cm]-135:$1$}] at (0,0) {};
\node (v2)[point, label={[label distance=0cm]90:$2$}] at (0.5,0.87) {};
\node (v3)[point, label={[label distance=0cm]-45:$3$}] at (1,0) {};
\node (v4)[point, label={[label distance=0cm]-45:$4$}] at (1.5,0.87) {};

\draw (v1.center) -- (v2.center);
\draw (v1.center) -- (v3.center);
\draw (v2.center) -- (v3.center);
\draw (v3.center) -- (v4.center);

\draw[color=black!20, style=fill, outer sep = 20pt] (0.1,0.06) -- (0.5,0.77) -- (0.9,0.06) -- cycle;
\end{tikzpicture}
\end{minipage}

with simplicial binoid
\[
M_\triangle=\{x_1, x_2, x_3, x_4\mid x_1+x_4=\infty, x_2+x_4=\infty\}.
\]
of which we computed the simplicial binoids of the links of the vertices in Example~\ref{example:spectrum-simplicial-binoid-4}

\vspace{1ex}

\begin{minipage}[c]{0.4\textwidth}
    \[
    \lk_\triangle(1) = \left\{\begin{aligned}
    \varnothing,&\{2\},\{3\},\\
    &\{2, 3\}\end{aligned}\right\}
    \]\vfill
\end{minipage}\begin{minipage}[c]{0.2\textwidth}
\hspace{2em}\begin{tikzpicture}
\tikzstyle{point}=[circle,thick,draw=black,fill=black,inner sep=0pt,minimum width=4pt,minimum height=4pt]
\node (v2)[point, label={[label distance=0cm]90:$2$}] at (0.5,0.87) {};
\node (v3)[point, label={[label distance=0cm]-45:$3$}] at (1,0) {};
\draw (v2.center) -- (v3.center);
\end{tikzpicture}
\end{minipage}\begin{minipage}[c]{0.4\textwidth}
\[
M_{\lk_\triangle(1)} = (x_2, x_3\mid \varnothing)
\]\vfill
\end{minipage}

\vspace{1ex}

\begin{minipage}[c]{0.4\textwidth}
    \[
    \lk_\triangle(2) = \left\{\begin{aligned}
    \varnothing,&\{1\},\{3\},\\
    &\{1, 3\}\end{aligned}\right\}
    \]\vfill
\end{minipage}\begin{minipage}[c]{0.2\textwidth}
\begin{tikzpicture}[baseline=(current  bounding  box.center)]
\tikzstyle{point}=[circle,thick,draw=black,fill=black,inner sep=0pt,minimum width=4pt,minimum height=4pt]
\node (v1)[point, label={[label distance=0cm]-135:$1$}] at (0,0) {};
\node (v3)[point, label={[label distance=0cm]-45:$3$}] at (1,0) {};
\draw (v1.center) -- (v3.center);
\end{tikzpicture}
\end{minipage}\begin{minipage}[c]{0.4\textwidth}
\[
M_{\lk_\triangle(2)} = (x_1, x_3\mid \varnothing)
\]\vfill
\end{minipage}

\vspace{1ex}

\begin{minipage}[c]{0.4\textwidth}
    \[
    \lk_\triangle(3) = \left\{\begin{aligned}
    &\varnothing,\{1\},\{2\},\\
    &\{4\},\{1, 2\}\end{aligned}\right\}
    \]\vfill
\end{minipage}\begin{minipage}[c]{0.2\textwidth}
\begin{tikzpicture}
\hspace{.5em}\tikzstyle{point}=[circle,thick,draw=black,fill=black,inner sep=0pt,minimum width=4pt,minimum height=4pt]
\node (v1)[point, label={[label distance=0cm]-135:$1$}] at (0,0) {};
\node (v2)[point, label={[label distance=0cm]90:$2$}] at (0.5,0.87) {};
\node (v4)[point, label={[label distance=0cm]-45:$4$}] at (1.5,0.87) {};
\draw (v1.center) -- (v2.center);
\end{tikzpicture}
\end{minipage}\begin{minipage}[c]{0.4\textwidth}
\[
M_{\lk_\triangle(3)} = \left(\begin{aligned}& x_1, x_2, x_4\\\midrule & x_1+x_4=\infty,\\ & x_2+x_4=\infty\end{aligned}\right)
\]\vfill
\end{minipage}

\vspace{1ex}

\begin{minipage}[c]{0.4\textwidth}
    \[
    \lk_\triangle(4) = \left\{\begin{aligned}
    \varnothing,\{3\}\end{aligned}\right\}
    \]\vfill
\end{minipage}\begin{minipage}[c]{0.2\textwidth}
\begin{tikzpicture}[baseline=(current  bounding  box.center)]
\hspace{4.5em}\tikzstyle{point}=[circle,thick,draw=black,fill=black,inner sep=0pt,minimum width=4pt,minimum height=4pt]
\node (v3)[point, label={[label distance=0cm]-45:$3$}] at (1,0) {};
\end{tikzpicture}
\end{minipage}\begin{minipage}[c]{0.4\textwidth}
\[
M_{\lk_\triangle(4)} = (x_3\mid \varnothing)
\]\vfill
\end{minipage}

Their reduced simplicial cohomology complexes look like
\begin{equation*}\begin{tikzcd}[baseline=(current  bounding  box.center), cramped, row sep = 0pt, /tikz/column 1/.append style={anchor=base east},
/tikz/column 2/.append style={anchor=base west}, column sep = 2em]
\widetilde{\C}_{\lk_\triangle(1)}: & \ZZ^\varnothing \arrow[r] & \ZZ^{\{2\}}\oplus\ZZ^{\{3\}} \arrow[r] & \ZZ^{\{2, 3\}} \arrow[r] & 0\\
& \alpha \arrow[r, mapsto]  & (\alpha, \alpha)                       & &\\
&   & (\alpha_2, \alpha_3) \arrow[r, mapsto]                        & -\alpha_2+\alpha_3 &\\
\end{tikzcd}
\end{equation*}

\begin{equation*}\begin{tikzcd}[baseline=(current  bounding  box.center), cramped, row sep = 0pt, /tikz/column 1/.append style={anchor=base east},
/tikz/column 2/.append style={anchor=base west}, column sep = 2em]        
\widetilde{\C}_{\lk_\triangle(2)}: & \ZZ^\varnothing \arrow[r] & \ZZ^{\{1\}}\oplus\ZZ^{\{3\}} \arrow[r] & \ZZ^{\{1, 3\}} \arrow[r]  & 0\\
& \alpha \arrow[r, mapsto]   & (\alpha, \alpha)                       & &\\
&   & (\alpha_1, \alpha_3)  \arrow[r, mapsto]                       & -\alpha_1+\alpha_3 &\\        
\end{tikzcd}
\end{equation*}        

\begin{equation*}\begin{tikzcd}[baseline=(current  bounding  box.center), cramped, row sep = 0pt, /tikz/column 1/.append style={anchor=base east},
/tikz/column 2/.append style={anchor=base west}, column sep = 2em]        
\widetilde{\C}_{\lk_\triangle(3)}: & \ZZ^\varnothing \arrow[r] & \ZZ^{\{1\}}\oplus\ZZ^{\{2\}}\oplus\ZZ^{\{4\}} \arrow[r] & \ZZ^{\{1, 2\}} \arrow[r] & 0\\
& \alpha  \arrow[r, mapsto]  & (\alpha, \alpha, \alpha)                       & &\\
&   & (\alpha_1, \alpha_2, \alpha_4)  \arrow[r, mapsto]                  & -\alpha_1+\alpha_2 &\\
\end{tikzcd}
\end{equation*}        

\begin{equation*}\begin{tikzcd}[baseline=(current  bounding  box.center), cramped, row sep = 0pt, /tikz/column 1/.append style={anchor=base east},
/tikz/column 2/.append style={anchor=base west}, column sep = 2em]        
\widetilde{\C}_{\lk_\triangle(4)}: & \ZZ^\varnothing \arrow[r] & \ZZ^{\{3\}} \arrow[r] & 0\\
& \alpha  \arrow[r, mapsto]  & \alpha                        &
\end{tikzcd}
\end{equation*}
So we have the cohomology groups\\\vspace{1ex}

\hspace{\mathindent}\begin{tabular}{l|cccc}
    & $\widetilde{\C}_{\lk_\triangle(1)}$ & $\widetilde{\C}_{\lk_\triangle(2)}$ & $\widetilde{\C}_{\lk_\triangle(3)}$ & $\widetilde{\C}_{\lk_\triangle(4)}$\\
    $\H^{-1}(\cdot)$ & $0$ & $0$ & $0$ &  $0$ \\
    $\H^0(\cdot)$ & $\ZZ$ & $\ZZ$ & $\ZZ^2$ &  $\ZZ$ \\
    $\H^1(\cdot)$ & $0$ & $0$ & $0$ &  $0$ 
\end{tabular}\\\vspace{1ex}
and we can apply the formula in Theorem~\ref{theorem:cohomology-simplicial-complex} and its specialization of Corollary~\ref{corollary:cohomology-simplicial-binoid} to obtain the cohomology groups for the \v{C}ech-Picard chain complex
\begin{align*}
\H^0(\Spec^\bullet M_\triangle, \O^*)&=\ZZ^{0}=0\\
\H^1(\Spec^\bullet M_\triangle, \O^*)&=\ZZ^{(1+1+2+1)-4}=\ZZ^1\\
\H^j(\Spec^\bullet M_\triangle, \O^*)&=0 \quad \text{ for } j\geq 2
\end{align*}

A generator of the group $\H^1$ is given by the ideal $I=\langle x_1, x_2, x_4\rangle$, and we can easily see that locally $I_{x_i}\cong M_{x_i}$ for any variable.

The inverse of $I$ in $\Pic^{\loc}M_\triangle$ is its dual $I^\vee$, as in general for ideals.

\subsection{An example with torsion cohomology}\label{thirdexample}
We know already that $\H^0$ and $\H^1$ are always free groups. Consider the following well-known minimal triangulation of $\PP^2_\RR$
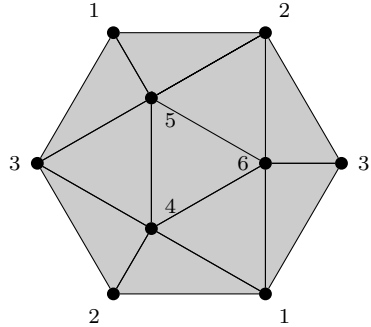
\begin{center}
    \begin{tikzpicture}
    \tikzstyle{point}=[circle,thick,draw=black,fill=black,inner sep=0pt,minimum width=4pt,minimum height=4pt]
    \fill[fill=black!20] (3,0)--(1,0)--(0,1.732)--(1,3.464)--(3,3.464)--(4,1.732)--(3,0)--cycle;
    
    \node (1a)[point, label={-60:\scriptsize $1$}] at (3,0) {};
    \node (2a)[point, label={240:\scriptsize $2$}] at (1,0) {};
    \node (3a)[point, label={180:\scriptsize $3$}] at (0,1.732) {};
    \node (1b)[point, label={120:\scriptsize $1$}] at (1,3.464) {};
    \node (2b)[point, label={60:\scriptsize $2$}] at (3,3.464) {};
    \node (3b)[point, label={0:\scriptsize $3$}] at (4,1.732) {};
    \node (4)[point, label={60:\scriptsize $4$}] at (1.5,0.866) {};
    \node (5)[point, label={-60:\scriptsize $5$}] at (1.5,2.598) {};
    \node (6)[point, label={180:\scriptsize $6$}] at (3,1.732) {};
    
    \draw (1a)--(2a)--(4)--(1a);
    \draw (2a)--(3a)--(4)--(2a);
    \draw (3a)--(4)--(5)--(3a);
    \draw (1b)--(3a)--(5)--(1b);
    \draw (1b)--(2b)--(5)--(1b);
    \draw (4)--(5)--(6)--(4);
    \draw (2b)--(5)--(5)--(2b);
    \draw (2b)--(3b)--(6)--(2b);
    \draw (3b)--(6)--(1a)--(3b);
    \draw (1a)--(4)--(6)--(1a);
    
    \end{tikzpicture}
\end{center}
with 6 vertices, 15 edges and 10 faces. Now we extend it to obtain an example with torsion cohomology. To do so, we add one new vertex to every face of the previous triangulation and then we close it under the subset operation. We obtain the following complex
\[
\triangle = \left\{\begin{aligned}
&\varnothing,\{1\}, \{2\}, \{3\}, \{4\}, \{5\}, \{6\}, \{7\},\\
&\{ 1, 2 \}, \{ 1, 3 \}, \{ 1, 4 \}, \{ 1, 5 \}, \{ 1, 6 \}, \{ 1, 7 \}, \{ 2, 3 \}, \{ 2, 4 \}, \{ 2, 5 \}, \{ 2, 6 \}, \\
&\{ 2, 7 \}, \{ 3, 4 \}, \{ 3, 5 \}, \{ 3, 6 \}, \{ 3, 7 \}, \{ 4, 5 \}, \{ 4, 6 \}, \{ 4, 7 \}, \{ 5, 6 \}, \{ 5, 7 \}, \{ 6, 7 \},\\ 
&\{ 1, 2, 4 \}, \{ 1, 2, 5 \}, \{ 1, 2, 7 \}, \{ 1, 3, 5 \}, \{ 1, 3, 6 \}, \{ 1, 3, 7 \}, \{ 1, 4, 6 \}, \{ 1, 4, 7 \},\\
&\{ 1, 5, 7 \}, \{ 1, 6, 7 \}, \{ 2, 3, 4 \}, \{ 2, 3, 6 \}, \{ 2, 3, 7 \}, \{ 2, 4, 7 \}, \{ 2, 5, 6 \}, \{ 2, 5, 7 \}, \\
&\{ 2, 6, 7 \}, \{ 3, 4, 5 \}, \{ 3, 4, 7 \}, \{ 3, 5, 7 \}, \{ 3, 6, 7 \}, \{ 4, 5, 6 \}, \{ 4, 5, 7 \}, \{ 4, 6, 7 \}, \{ 5, 6, 7 \},\\
&\{1, 2, 4, 7\}, \{1, 2, 5, 7\}, \{1, 3, 5, 7\}, \{1, 3, 6, 7\}, \{1, 4, 6, 7\},\\
&\{2, 3, 4, 7\}, \{2, 3, 6, 7\}, \{2, 5, 6, 7\}, \{3, 4, 5, 7\}, \{4, 5, 6, 7\}
\end{aligned}\right.
\]
that, from a geometrical point of view, corresponds to (a triangulation of) a cone over $\PP^2_\RR$.\\
We have 7 simplicial complexes given by the links of the points. The first 6 are isomorphic to each other thanks to the properties of the projective space (they are triangulations of fundamental affine open subsets of this cone) while the 7th is nothing else than the triangulation of the projective space we started with.
\begin{align*}
\lk_\triangle(1)&=\left\{\begin{aligned}
&\varnothing, \{ 2 \}, \{ 3 \}, \{ 4 \}, \{ 5 \}, \{ 6 \}, \{ 7 \},\\
&\{ 2, 4 \}, \{ 2, 5 \}, \{ 2, 7 \}, \{ 3, 5 \}, \{ 3, 6 \}, \{ 3, 7 \}, \{ 4, 6 \}, \{ 4, 7 \}, \{ 5, 7 \}, \{ 6, 7 \},\\
&\{ 2, 4, 7 \}, \{ 2, 5, 7 \}, \{ 3, 5, 7 \}, \{ 3, 6, 7 \}, \{ 4, 6, 7 \}
\end{aligned}\right.\end{align*}
\begin{align*}
\lk_\triangle(2)&=\left\{\begin{aligned}
&\varnothing, \{ 1 \}, \{ 3 \}, \{ 4 \}, \{ 5 \}, \{ 6 \}, \{ 7 \},\\
&\{ 1, 4 \}, \{ 1, 5 \}, \{ 1, 7 \}, \{ 3, 4 \}, \{ 3, 6 \}, \{ 3, 7 \}, \{ 4, 7 \}, \{ 5, 6 \}, \{ 5, 7 \}, \{ 6, 7 \},\\
&\{ 1, 4, 7 \}, \{ 1, 5, 7 \}, \{ 3, 4, 7 \}, \{ 3, 6, 7 \}, \{ 5, 6, 7 \}
\end{aligned}\right.\end{align*}
\begin{align*}
\lk_\triangle(3)&=\left\{\begin{aligned}
&\varnothing, \{ 1 \}, \{ 2 \}, \{ 4 \}, \{ 5 \}, \{ 6 \}, \{ 7 \},\\
&\{ 1, 5 \}, \{ 1, 6 \}, \{ 1, 7 \}, \{ 2, 4 \}, \{ 2, 6 \}, \{ 2, 7 \}, \{ 4, 5 \}, \{ 4, 7 \}, \{ 5, 7 \}, \{ 6, 7 \},\\
&\{ 1, 5, 7 \}, \{ 1, 6, 7 \}, \{ 2, 4, 7 \}, \{ 2, 6, 7 \}, \{ 4, 5, 7 \}
\end{aligned}\right.\end{align*}
\begin{align*}
\lk_\triangle(4)&=\left\{\begin{aligned}
&\varnothing, \{ 1 \}, \{ 2 \}, \{ 3 \}, \{ 5 \}, \{ 6 \}, \{ 7 \},\\
&\{ 1, 2 \}, \{ 1, 6 \}, \{ 1, 7 \}, \{ 2, 3 \}, \{ 2, 7 \}, \{ 3, 5 \}, \{ 3, 7 \}, \{ 5, 6 \}, \{ 5, 7 \}, \{ 6, 7 \},\\
&\{ 1, 2, 7 \}, \{ 1, 6, 7 \}, \{ 2, 3, 7 \}, \{ 3, 5, 7 \}, \{ 5, 6, 7 \}
\end{aligned}\right.\end{align*}
\begin{align*}
\lk_\triangle(5)&=\left\{\begin{aligned}
&\varnothing, \{ 1 \}, \{ 2 \}, \{ 3 \}, \{ 4 \}, \{ 6 \}, \{ 7 \},\\
&\{ 1, 2 \}, \{ 1, 3 \}, \{ 1, 7 \}, \{ 2, 6 \}, \{ 2, 7 \}, \{ 3, 4 \}, \{ 3, 7 \}, \{ 4, 6 \}, \{ 4, 7 \}, \{ 6, 7 \},\\
&\{ 1, 2, 7 \}, \{ 1, 3, 7 \}, \{ 2, 6, 7 \}, \{ 3, 4, 7 \}, \{ 4, 6, 7 \}
\end{aligned}\right.\end{align*}
\begin{align*}
\lk_\triangle(6)&=\left\{\begin{aligned}
&\varnothing, \{ 1 \}, \{ 2 \}, \{ 3 \}, \{ 4 \}, \{ 5 \}, \{ 7 \},\\
&\{ 1, 3 \}, \{ 1, 4 \}, \{ 1, 7 \}, \{ 2, 3 \}, \{ 2, 5 \}, \{ 2, 7 \}, \{ 3, 7 \}, \{ 4, 5 \}, \{ 4, 7 \}, \{ 5, 7 \},\\
&\{ 1, 3, 7 \}, \{ 1, 4, 7 \}, \{ 2, 3, 7 \}, \{ 2, 5, 7 \}, \{ 4, 5, 7 \}
\end{aligned}\right.\end{align*}
\begin{align*}
\lk_\triangle(7)&=\left\{\begin{aligned}
&\varnothing, \{ 1 \}, \{ 2 \}, \{ 3 \}, \{ 4 \}, \{ 5 \}, \{ 6 \},\\
&\{ 1, 2 \}, \{ 1, 3 \}, \{ 1, 4 \}, \{ 1, 5 \}, \{ 1, 6 \}, \{ 2, 3 \}, \{ 2, 4 \}, \{ 2, 5 \}, \{ 2, 6 \}, \{ 3, 4 \},\\
&\{ 3, 5 \}, \{ 3, 6 \}, \{ 4, 5 \}, \{ 4, 6 \}, \{ 5, 6 \}\\
&\{ 1, 2, 4 \}, \{ 1, 2, 5 \}, \{ 1, 3, 5 \}, \{ 1, 3, 6 \}, \{ 1, 4, 6 \},\\
&\{ 2, 3, 4 \}, \{ 2, 3, 6 \}, \{ 2, 5, 6 \}, \{ 3, 4, 5 \}, \{ 4, 5, 6 \}
\end{aligned}\right.
\end{align*}
In particular, it is really easy to compute the cohomology of the \v{C}ech-Picard chain complex of $\triangle$, since reduced simplicial cohomology of the above smaller complexes is well known

\vspace{1ex}

\hspace{\mathindent}\begin{tabular}{c|cc}
    & $\lk_\triangle(i)$ & $\lk_\triangle(7)$\\
    $\widetilde{\H}^{-1}(\cdot)$ & $0$                & $0$\\
    $\widetilde{\H}^0(\cdot)$ & $0$                & $0$\\
    $\widetilde{\H}^1(\cdot)$ & $0$                & $0$\\
    $\widetilde{\H}^2(\cdot)$ & $0$                & $\faktor{\ZZ}{2\ZZ}$
\end{tabular}

\vspace{2ex}

So we have that $\H^0(\vC_\triangle)=\H^1(\vC_\triangle)=\H^2(\vC_\triangle)=0$ but $\H^3(\vC_\triangle)=\faktor{\ZZ}{2\ZZ}$ and we obtained torsion cohomology, as expected.

\section{The Picard group of an open subset of \texorpdfstring{$\Spec M_\triangle$}{specM}}
Recall from Definition~\ref{definition:divisor-class-group-general} that we defined the divisor class group of a general binoid to be $\Pic(W)$, where $W$ is the set of prime ideals of height at most one. Obviously if $\dim \triangle\leq 1$ then $W=\Spec^\bullet M_\triangle$, so we already computed it above.

In what follows, we consider an arbitrary dimension of the simplicial complex, and we will see that without strong assumptions, like in the next Proposition, we have little hope to find an easy description of $\Pic(V)$ for some open subset $V$ of $\Spec M_\triangle$ and even of the simpler $\Cl(M_\triangle)$.

\begin{proposition}\label{proposition:cover-W-pure-smp-cpx}
    Let $\triangle$ be a pure simplicial complex of dimension $m$.\footnote{Recall that a simplicial complex is pure if all its facets have the same dimension.} Let $W$ be the subset of $\Spec M$ of prime ideals of height at most 1. Then $W$ can be covered by $\{D(G)\}_{G\in\triangle_{m-1}}$.
\end{proposition}
\begin{proof}
    Thanks to the correspondence between prime ideals and faces, we know that the prime ideals of height zero correspond exactly to the facets, all of dimension $m$ thanks to the pureness of the simplicial complex, and so the prime ideals of height $1$ correspond to all the faces of dimension $m-1$. Thanks to Proposition~\ref{proposition:minimal-open-set-prime-ideal} we know that each of this prime ideals is minimally covered by only one open subset. Let $x_1, \dots, x_n$ be the generators of $M_\triangle$, let $\p$ be a prime of height 1 and assume without loss of generality that it corresponds to the face $G=\{x_1, \dots, x_m\}\in\triangle_{m-1}$. The minimal open subset covering it is then $D(x_1+\dots +x_m)=D(G)$.
\end{proof}

Thanks to the correspondences proven in this Chapter, it is easy to prove the next result.

\begin{proposition}
    Let $F$ be a face of dimension $k$ in a simplicial complex $\triangle$ of dimension $m$ and let $\p$ be its corresponding prime ideal. 
    Let $h$ be the maximal dimension of a facet that contains $F$. Then $\gls{height-prime}=h-k$.
\end{proposition}
\begin{proof}
    An increasing chain of prime ideals that has $\p$ as maximal element, corresponds to a decreasing chain of faces that has $F$ as minimal element. Clearly the facet with the maximal dimension that contains $F$ gives the longest chain, that has length exactly $h-k+1$, so $\htnew\p=h-k$.
\end{proof}

\begin{example}\label{example:picW-favourite}
    Let us continue with our favourite Example~\ref{example:spectrum-simplicial-binoid-8}
    
    \begin{minipage}[c]{0.5\textwidth}
        \[
        \triangle = \left\{\begin{aligned}
        &\varnothing,\{1\},\{2\},\{3\},\{4\},\\
        &\{1, 2\}, \{1, 3\}, \{2, 3\}, \{3, 4\}\\
        &\{1, 2, 3\}\end{aligned}\right\}
        \]\vfill
    \end{minipage}\begin{minipage}[c]{0.5\textwidth}
    \hspace{6em}\begin{tikzpicture}
    \tikzstyle{point}=[circle,thick,draw=black,fill=black,inner sep=0pt,minimum width=4pt,minimum height=4pt]
    \node (v1)[point, label={[label distance=0cm]-135:$1$}] at (0,0) {};
    \node (v2)[point, label={[label distance=0cm]90:$2$}] at (0.5,0.87) {};
    \node (v3)[point, label={[label distance=0cm]-45:$3$}] at (1,0) {};
    \node (v4)[point, label={[label distance=0cm]-45:$4$}] at (1.5,0.87) {};
    
    \draw (v1.center) -- (v2.center);
    \draw (v1.center) -- (v3.center);
    \draw (v2.center) -- (v3.center);
    \draw (v3.center) -- (v4.center);
    
    \draw[color=black!20, style=fill, outer sep = 20pt] (0.1,0.06) -- (0.5,0.77) -- (0.9,0.06) -- cycle;
    \end{tikzpicture}
\end{minipage}

and recall that its spectrum is the poset 

\vspace{1ex}

\hspace{\mathindent}$\Spec M_\triangle=$
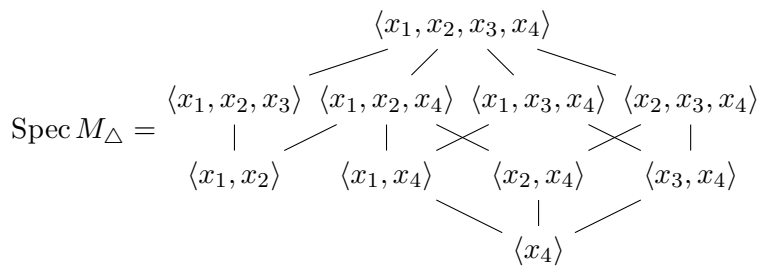
\begin{tikzpicture}[baseline={(current bounding box.center)}]
\node (1234) at (0,0){$\langle x_1,x_2,x_3, x_4\rangle$};
\node (123) at (-3,-1){$\langle x_1, x_2, x_3\rangle$};
\node (124) at (-1,-1){$\langle x_1, x_2, x_4\rangle$};
\node (134) at (1,-1){$\langle x_1, x_3, x_4\rangle$};
\node (234) at (3,-1){$\langle x_2, x_3, x_4\rangle$};
\node (12) at (-3,-2){$\langle x_1, x_2\rangle$};
\node (14) at (-1,-2){$\langle x_1, x_4\rangle$};
\node (24) at (1,-2){$\langle x_2, x_4\rangle$};
\node (34) at (3,-2){$\langle x_3, x_4\rangle$};
\node (4) at (1,-3){$\langle x_4\rangle$};

\draw[-](1234)--(123);
\draw[-](1234)--(124);
\draw[-](1234)--(134);
\draw[-](1234)--(234);
\draw[-](123)--(12);
\draw[-](124)--(12);
\draw[-](124)--(14);
\draw[-](124)--(24);
\draw[-](134)--(14);
\draw[-](134)--(34);
\draw[-](234)--(24);
\draw[-](234)--(34);
\draw[-](14)--(4);
\draw[-](24)--(4);
\draw[-](34)--(4);
\end{tikzpicture}
\vspace{1ex}

The prime ideals of height 0 correspond to the facets of the simplicial complex, so they are
\[
X^{(0)}=\{\langle x_4\rangle, \langle x_1, x_2\rangle\}
\]
and the prime ideals of height 1 are
\[
X^{(1)}=\{ \langle x_1, x_2, x_3\rangle, \langle x_1, x_4\rangle, \langle x_2, x_4\rangle, \langle x_3, x_4\rangle \}
\]
so the minimal cover of $W$ is
\[
\mathscr{V}=\{ D(x_4), D(x_2+x_3), D(x_1+x_3), D(x_1+x_2) \}.
\]
The intersections of these open subsets are
\begin{align*}
D(x_4)\cap D(x_i+x_j)&=\varnothing, \ \forall i, j\neq 4\\
D(x_i+x_j)\cap D(x_j+x_k) &= \{\langle x_4\rangle\}, \ \forall i\neq k\\
D(x_2+x_3)\cap D(x_1+x_3)\cap D(x_1+x_2) &=\{\langle x_4 \rangle\}
\end{align*}
so the \v{C}ech complex that we have to consider is
\[
\begin{tikzcd}[baseline=(current  bounding  box.center), cramped,
/tikz/column 1/.append style={anchor=base east},
/tikz/column 2/.append style={anchor=base west}]
\ZZ^{\{4\}}\oplus \ZZ^{\{2,3\}}\oplus \ZZ^{\{1,3\}}\oplus \ZZ^{\{1,2\}} \arrow[out=0, in=180, looseness=2, overlay, d, pos=0.04, "\delta^0"]\\
\ZZ^{\{1,2,3\}}\oplus \ZZ^{\{1,2,3\}}\oplus \ZZ^{\{1,2,3\}} \arrow[r, "\delta^1"] & \ZZ^{\{1,2, 3\}}\arrow[r, "\delta^2"] & 0
\end{tikzcd}
\]
Since we know that we can decompose $\O^*_{M_\triangle}$ in the direct sum of $\O^*_{x_i}$'s, we can decompose this chain complex into the direct sum of the chain complexes
\begin{equation*}
\begin{tikzcd}[baseline=(current  bounding  box.center), cramped]
\vC^\bullet: \vC_i^0\arrow[r, "\partial^0"] & \vC_i^1\arrow[r, "\partial^1"] & \vC_i^2\arrow[r, "\partial^2"] & 0
\end{tikzcd}
\end{equation*}
with $i=1, \dots, 4$.

The groups in these smaller complexes are
\begin{align*}
\vC^0_1 = \ZZ^2 &=\left\{\left(\sigma^{\{1,2\}}_{1}, \sigma^{\{1,3\}}_{1}\right)\right\}&
\vC^0_2 = \ZZ^2 &=\left\{\left(\sigma^{\{1,2\}}_{2}, \sigma^{\{2,3\}}_{2}\right)\right\}\\
\vC^0_3 = \ZZ^3 &=\left\{\left(\sigma^{\{1,3\}}_{3}, \sigma^{\{2,3\}}_{3}\right)\right\}&
\vC^0_4 = \ZZ &=\left\{\left(\sigma^{\{4\}}_{4}\right)\right\}
\end{align*}
but for higher degree we have to be careful, because any face con be represented many times. Here $\{1, 2, 3\}$ has to be taken in account 3 times for $\vC^1$, so for clarity we index it with a subscript $a$, $b$ or $c$, that correspond to the intersections $D(\{1,2\})\cap D(\{1,3\})$, $D(\{1,2\})\cap D(\{2,3\})$ and $D(\{1,3\})\cap D(\{2,3\})$ respectively.
\begin{align*}
\vC^1_1 &= \ZZ^3 =\left\{\left(\sigma^{\{1,2,3\}}_{1, a}, \sigma^{\{1,2,3\}}_{1, b}, \sigma^{\{1,2,3\}}_{1, c}\right)\right\}\\
\vC^1_2 &= \ZZ^3 =\left\{\left(\sigma^{\{1,2,3\}}_{2, a}, \sigma^{\{1,2,3\}}_{2, b}, \sigma^{\{1,2,3\}}_{2, c}\right)\right\}\\
\vC^1_3 &= \ZZ^3 =\left\{\left(\sigma^{\{1,2,3\}}_{3, a}, \sigma^{\{1,2,3\}}_{3, b}, \sigma^{\{1,2,3\}}_{3, c}\right)\right\}\\
\vC^1_4 &= 0
\end{align*}
for $\vC^2$, the face $\{1, 2, 3\}$ has to be taken in account only once, so we get
\begin{align*}
\vC^2_1 &= \ZZ =\left\{\left(\sigma^{\{1,2,3\}}_{1}\right)\right\}\\
\vC^2_2 &= \ZZ =\left\{\left(\sigma^{\{1,2,3\}}_{2}\right)\right\}\\
\vC^2_3 &= \ZZ =\left\{\left(\sigma^{\{1,2,3\}}_{3}\right)\right\}\\
\vC^2_4 &= 0
\end{align*}
where all the $\sigma_*^{\{*\}}$ are integer numbers.

Now that we split the groups, we turn to the maps. We start with $\partial^0$
\[
\begin{tikzcd}[baseline=(current  bounding  box.center), cramped,
/tikz/column 1/.append style={anchor=base east}, column sep = -18em,
ampersand replacement=\&]
\left((\sigma^{\{1, 2\}}_{1}, \sigma^{\{1, 2\}}_{2}), (\sigma^{\{1, 3\}}_{1}, \sigma^{\{1, 3\}}_{3}), (\sigma^{\{2, 3\}}_{2}, \sigma^{\{2, 3\}}_{3}), (\sigma^{\{4\}}_{4})\right)\arrow[out=0, in=180, looseness=3, overlay, dr, pos=0.05, "\partial^0", mapsto, end anchor=170]\\
\& \begin{aligned}\left(\sigma^{\{1, 3\}}_{1}\right.-\sigma^{\{1, 2\}}_{1}&,&\ -\sigma^{\{1,2\}}_2&,& \sigma^{\{1, 3\}}_{3},\hspace{1em}\\
-\sigma^{\{1, 2\}}_{1}&,&\ \sigma^{\{2, 3\}}_2-\sigma^{\{1,2\}}_2&,& \sigma^{\{2, 3\}}_{3},\hspace{1em}\\
-\sigma^{\{1, 3\}}_{1}&,&\ \sigma^{\{2, 3\}}_2&,& \sigma^{\{2, 3\}}_{3}\left.-\sigma^{\{1,3\}}_3\right)
\end{aligned}
\end{tikzcd}
\]
That we can split as the sum of
\[
\begin{tikzcd}[baseline=(current  bounding  box.center), cramped, row sep = 0pt,
/tikz/column 1/.append style={anchor=base east},
/tikz/column 2/.append style={anchor=base west},]
\ZZ^{\{1,2\}}_{1}\oplus \ZZ^{\{1,3\}}_{1}\arrow[r, "\partial^0_1"] & \ZZ^{\{1, 2, 3\}}_{1}\oplus \ZZ^{\{1, 2, 3\}}_{1}\oplus \ZZ^{\{1, 2, 3\}}_{1}\\
\left(\sigma^{\{1, 2\}}_{1}, \sigma^{\{1, 3\}}_{1}\right)\arrow[r, mapsto] & \left(\sigma^{\{1, 3\}}_{1}-\sigma^{\{1, 2\}}_{1}, -\sigma^{\{1, 2\}}_{1}, -\sigma^{\{1, 3\}}_{1}\right)
\end{tikzcd}
\]
and the equivalent diagrams for the vertices $2$ and $3$.

We look at the map $\partial^1$ only in the simpler case when we concentrate on the vertex $1$.
This map looks like
\[
\begin{tikzcd}[baseline=(current  bounding  box.center), cramped, row sep = 0pt,
/tikz/column 1/.append style={anchor=base east},
/tikz/column 2/.append style={anchor=base west},]
\ZZ^{\{1, 2, 3\}}_{1, a}\oplus \ZZ^{\{1, 2, 3\}}_{1, b}\oplus \ZZ^{\{1, 2, 3\}}_{1, c}\arrow[r, "\partial^1_1"] & \ZZ^{\{1, 2, 3\}}_1\\
\left(\sigma^{\{1, 2, 3\}}_{1, a}, \sigma^{\{1, 2, 3\}}_{1, b}, \sigma^{\{1, 2, 3\}}_{1, a}\right)\arrow[r, mapsto] & \left(\sigma^{\{1, 2, 3\}}_{1, a}-\sigma^{\{1, 2, 3\}}_{1, b}+\sigma^{\{1, 2, 3\}}_{1, c}\right)
\end{tikzcd}
\]

So we recover that we can describe the original complex as a direct sum of the smaller ones, like we proved in general above.
\end{example}

\begin{definition}
    Let $\triangle$ be a simplicial complex and let $\{G_i\}_{i\in I}\subseteq\triangle$ be a subset of the faces. The \emph{crosscut complex of $\{G_i\}_{i\in I}$ in $\triangle$} is the simplicial complex $\gls{crosscut-complex}$ on vertex set $I$ such that 
    \begin{itemize}
        \item $\varnothing\in\ccc(\{G_i\}_{i\in I}, \triangle)$,
        \item for any $J\subseteq I$, $J\in\ccc(\{G_i\}_{i\in I}, \triangle)$ if and only if $\displaystyle\bigcup_{j\in J} G_j\in\triangle$.
    \end{itemize}
\end{definition}

This definition is a special case of the more general crosscut complex of a join-semi\-lattice, where we see $\{G_i\}_{i\in I}$ as generating a sub-lattice of $\triangle_\cup^\infty=(\triangle \cup \{V\}, \cup, \varnothing, V)$. The interested reader might refer to \cite[Section 6.4]{Boettger} for more details on $\triangle_\cup^\infty$ and to \cite[Section 10]{bjorner1995topological} for more applications of the crosscut complex.\\
It is worth noting the similarity between the definition of the crosscut complex and the nerve of a covering, given in Definition~\ref{definition:nerve-of-covering}, where we use the intersection instead of the union.

\begin{remark}
    Clearly if $\{G_i\}$ is the set of vertices $\{v\}_{v\in V}$ of $\triangle$, then
    \[
    \ccc(\{G_i\}_{i\in I}, \triangle)=\triangle.
    \]
\end{remark}

The following Proposition generalizes Corollary~\ref{corollary:cech-covering-nerve}.

\begin{proposition}\label{proposition:nerve-crosscut}
    Let $\{D(G_i)\}$ be a collection of fundamental open subsets of $\Spec M_\triangle$. Then $\nerve(\{D(G_i)\})=\ccc(\{G_i\}_{i\in I}, \triangle)$.
\end{proposition}
\begin{proof}
    From Proposition~\ref{proposition:intersection-open-subsets} and Definition~\ref{definition:nerve-of-covering} we know that $D\left(\displaystyle\bigcup_{j\in J}G_j\right)\neq\varnothing$ if and only if $\displaystyle\bigcup_{j\in J}G_j\in\triangle$, and the latter is true if and only if $J\in\ccc(\{G_i\}_{i\in I}, \triangle)$ by definition.
\end{proof}

Recall from Lemma~\ref{lemma:minimal-covering} that any open subset $V$ of $\Spec^\bullet M$ can be minimally covered by some $D(F_j)$ corresponding to maximal primes in $V$.

\begin{theorem}Let $V$ be an open subset of $\Spec^\bullet M_{\triangle}$ with maximal primes $\p_0, \dots, \p_r$. We can compute the cohomology of $\O^*_{M_\triangle}\restriction_V$ through \v{C}ech cohomology using the minimal open covering $\mathscr{V}=\{D(F_0), \dots, D(F_r)\}$ defined in Lemma~\ref{lemma:minimal-covering}. In particular we have that
    \[
    \vC^k(\mathscr{V}, \O^*_{M_\triangle})\cong\bigoplus_{J\in\left(\ccc(\{F_i\}, \triangle)\right)_k}\ZZ^{\bigcup_{j\in J} F_j}.
    \]
\end{theorem}
\begin{proof}
    Since $\mathscr{V}$ is a covering of $V$ with affine open subsets, it is acyclic for any sheaf of abelian groups, in particular for $\O^*_{M_\triangle}$, so we can use it to compute its cohomology.
    
    Thanks to Proposition~\ref{proposition:nerve-crosscut},
    \[
    \bigcap_{j\in J} D(G_j)\neq\varnothing\text{ if and only if } \bigcup_{j\in J} G_j\in\triangle \text{ if and only if } J\in\ccc(\{F_i\}, \triangle)
    \] and, thanks to the results in Theorem~\ref{theorem:value-sheaf}, for $J\in \ccc(\{F_i\}, \triangle)$,
    \[
    \O^*_{M_\triangle}\left(\bigcap_{j\in J} D(G_j)\right)=\O^*_{M_\triangle}\left(D\left(\bigcup_{j\in J} G_j\right)\right)\cong\ZZ^{\bigcup_{j\in J} G_j},
    \]
    where we use again the notation introduced at page~\pageref{notation:simplicial-binoid-faces}.
\end{proof}

\begin{corollary}\label{corollary:pure-dimensional}
    If $\triangle$ is pure of dimension $m$, $\triangle_{m-1}=\{F_1, \dots, F_r\}$ and $W$ is the set of primes of height at most one, then
    \begin{align*}
    \vC^0(W, \O^*_M)&\cong\bigoplus_{F\in\triangle_{m-1}}\ZZ^F\cong\bigoplus_{F\in\triangle_{m-1}}\ZZ^{m}\\
    \vC^j(W, \O^*_M)&\cong\bigoplus_{F_0\cup \dots\cup F_j\in \triangle_m}\ZZ^{F_0\cup \dots\cup F_j}\cong\bigoplus_{F_0\cup \dots\cup F_j\in \triangle_m}\ZZ^{m+1}
    \end{align*}
    for $1\leq j\leq m$ and $\vC^j(W, \O^*_M)=0$ for any $j\geq m+1$.
\end{corollary}
\begin{proof}
    The first part is clear from the Theorem above, so we have just to prove that $\vC^j(W, \O^*_M)=0$ for any $j\geq m+1$. Let $G$ be a facet of dimension $m$, so its cardinality is $m+1$. There exist $m+1$ faces in $G$ that define primes of height one, namely we remove one vertex at a time from $G$ to obtain them, call them $\{F_i\}$. These define open subsets that cover part of $W$, namely $V(G)\cap W$, so in the chain of intersections between these open subsets that give rise to the cochain complex above, we have that $\vC^j\restriction_{V(G)\cap W}$ is determined by the intersection of $j+1$ of these $\{F_i\}$. In particular, $\vC^m$ is determined by the intersection of all the $F_i$'s and $\vC^j=0$ for every $j\geq m+1$ because we have nothing to intersect more.
\end{proof}

\begin{proposition}Let $\triangle$ be any simplicial complex of dimension $m$. Then
    \[
    \vC^j(W, \O^*_M)=0
    \] for any $j\geq m+1$.
\end{proposition}
\begin{proof}
    We can use the same proof as for the Corollary above, where we select $G$ to be a facet of maximal dimension.
\end{proof}

\begin{remark}
    Since we are still talking about units of localizations of $M_\triangle$, it is easy to see that again these groups can be split, together with the maps, into groups regarding one vertex at a time, and so also the \v{C}ech complex can be again split into smaller ones.
\end{remark}

\begin{example}
    Let $\triangle$ be the pure simplicial complex of dimension 2 on 5 vertices
    \begin{minipage}[c]{0.5\textwidth}
        \[
        \triangle=\left\{\begin{aligned}
        &\varnothing, \{1\}, \{2\}, \{3\}, \{4\}, \{5\},\\
        &\{1, 2\}, \{1, 3\}, \{2, 3\},\\
        & \{1, 4\}, \{1, 5\}, \{4, 5\},\\
        &\{1, 2, 3\}, \{1, 4, 5\}
        \end{aligned}
        \right\}
        \]\vfill
    \end{minipage}\begin{minipage}[c]{0.5\textwidth}
    \hspace{6em}\begin{tikzpicture}
    \tikzstyle{point}=[circle,thick,draw=black,fill=black,inner sep=0pt,minimum width=4pt,minimum height=4pt]
    \node (v1)[point, label={[label distance=0cm]-90:$1$}] at (0,0) {};
    \node (v2)[point, label={[label distance=0cm]-135:$2$}] at (-0.87, -0.5) {};
    \node (v3)[point, label={[label distance=0cm]135:$3$}] at (-0.87, 0.5) {};
    \node (v4)[point, label={[label distance=0cm]-45:$4$}] at (0.87, -0.5) {};
    \node (v5)[point, label={[label distance=0cm]45:$5$}] at (0.87, 0.5) {};
    
    \draw (v1.center) -- (v2.center);
    \draw (v1.center) -- (v3.center);
    \draw (v2.center) -- (v3.center);
    \draw (v1.center) -- (v4.center);
    \draw (v1.center) -- (v5.center);
    \draw (v4.center) -- (v5.center);
    
    \draw[color=black!20, style=fill, outer sep = 20pt] (0.1, 0) -- (0.8, -0.4) -- (0.8, 0.4) -- cycle;
    \draw[color=black!20, style=fill, outer sep = 20pt] (-0.1, 0) -- (-0.8, -0.4) -- (-0.8, 0.4) -- cycle;
    \end{tikzpicture}
\end{minipage}

Then the primes of height at most one are
\[
W^{(0)}=\left\{\langle 4, 5\rangle, \langle 2, 3\rangle\right\}
\qquad
W^{(1)}=\left\{\begin{aligned}
&\langle 3, 4, 5\rangle, \langle 2, 4, 5\rangle, \langle 1, 4, 5\rangle,\\
&\langle 2, 3, 5\rangle, \langle 2, 3, 4\rangle, \langle 1, 2, 3\rangle
\end{aligned}
\right\}.
\]
$W$ is then covered by
\[
\mathscr{W}=\left\{\begin{aligned}
&D(\{1, 2\}), D(\{1, 3\}), D(\{2, 3\}),\\
&D(\{1, 4\}), D(\{1, 5\}), D(\{4, 5\})
\end{aligned}
\right\}
\]
and the chain complex looks like
\begin{align*}
\vC^0(W, \O^*_M)&\cong\bigoplus_{F\in\triangle_1}\ZZ^F\\
&\cong \ZZ^{\{1, 2\}}\oplus \ZZ^{\{1, 3\}}\oplus\ZZ^{\{2, 3\}}\\
&\hspace{2em}\oplus\ZZ^{\{1, 4\}}\oplus\ZZ^{\{1, 5\}}\oplus\ZZ^{\{4, 5\}}
\end{align*}
\begin{align*}
\vC^1(W, \O^*_M)&\cong\bigoplus_{F_0\cup F_1\in\triangle_2}\ZZ^{F_0\cup F_1}\\
&\cong (\ZZ)^{\{1, 2\}\cup\{1, 3\}}\oplus(\ZZ)^{\{1, 2\}\cup\{2, 3\}}\oplus(\ZZ)^{\{1, 3\}\cup\{2, 3\}}\\
&\hspace{2em}\oplus(\ZZ)^{\{1, 4\}\cup\{1, 5\}}\oplus(\ZZ)^{\{1, 4\}\cup\{4, 5\}}\oplus(\ZZ)^{\{1, 5\}\cup\{4, 5\}}\\
&\cong (\ZZ)^{\{1, 2, 3\}}\oplus(\ZZ)^{\{1, 2, 3\}}\oplus(\ZZ)^{\{1, 2, 3\}}\\
&\hspace{2em}\oplus(\ZZ)^{\{1, 4, 5\}}\oplus(\ZZ)^{\{1, 4, 5\}}\oplus(\ZZ)^{\{1, 4, 5\}}
\end{align*}
\begin{align*}
\vC^2(W, \O^*_M)&\cong\bigoplus_{F_0\cup F_1\cup F_2\in\triangle_2}(\ZZ^2)^{F_0\cup F_1\cup F_2}\\
&\cong (\ZZ)^{\{1, 2\}\cup\{1, 3\}\cup\{2, 3\}}\oplus(\ZZ)^{\{1, 4\}\cup\{1, 5\}\cup\{4, 5\}}\\
&\cong(\ZZ)^{\{1, 2, 3\}}\oplus(\ZZ)^{\{1, 4, 5\}}
\end{align*}
and then, whenever we try to unite 4 faces, we end up outside of the facets, so the higher groups are all trivial.

As we said, since we are looking at the restriction of the units sheaf of $M_\triangle$, we can concentrate on the single vertices and look in more details at the complex $\vC^\bullet_i(W, \O^*_M)$. Thanks to the symmetry of the problem, we can treat $2, 3, 4$ and $5$ similarly and then it will suffice to add the results.
$1$ is a particular case, and we start from it.
\begin{align*}
\vC^0_1(W, \O^*_M)&\cong (\ZZ)^{\{1, 2\}}_1\oplus (\ZZ)^{\{1, 3\}}_1\oplus(\ZZ)^{\{1, 4\}}_1\oplus(\ZZ)^{\{1, 5\}}_1
\end{align*}
\begin{align*}
\vC^1_1(W, \O^*_M)&\cong (\ZZ)_1^{\{1, 2\}\cup\{1, 3\}}\oplus(\ZZ)_1^{\{1, 2\}\cup\{2, 3\}}\oplus(\ZZ)_1^{\{1, 3\}\cup\{2, 3\}}\\
&\hspace{2em}\oplus(\ZZ)_1^{\{1, 4\}\cup\{1, 5\}}\oplus(\ZZ)_1^{\{1, 4\}\cup\{4, 5\}}\oplus(\ZZ)_1^{\{1, 5\}\cup\{4, 5\}}
\end{align*}
\begin{align*}
\vC^2_1(W, \O^*_M)&\cong (\ZZ)_1^{\{1, 2\}\cup\{1, 3\}\cup\{2, 3\}}\oplus(\ZZ)_1^{\{1, 4\}\cup\{1, 5\}\cup\{4, 5\}}
\end{align*}

Let us start from $\partial^0_1:\vC^0_1\longrightarrow\vC^1_1$. Let $(\alpha_1, \dots, \alpha_4)\in\vC^0_1(W, \O^*_M)$. Then
\[
\partial^0_1(\alpha_1, \dots, \alpha_4)=(-\alpha_1+\alpha_2, -\alpha_1, -\alpha_2, -\alpha_3+\alpha_4, -\alpha_3, -\alpha_4)
\]
that is clearly injective, so $\H^0_1(W, \O^*_M)=0$, and has image
\[
\im(\partial^0_1)=\{(\beta_1, \dots, \beta_6)\in\vC^1_1(W, \O^*_M)\mid \beta_2=\beta_1+\beta_3, \beta_5=\beta_4+\beta_6 \}.
\]

What about $\partial^1_1:\vC^1_1\longrightarrow\vC^2_1$? Let $(\beta_1, \dots, \beta_6)\in\vC^1_1(W, \O^*_M)$. Then
\[
\partial^1_1(\beta_1, \dots, \beta_6)=(\beta_1-\beta_2+\beta_3, \beta_4-\beta_5+\beta_6)
\]
that is clearly surjective, so $\H^2_1(W, \O^*_M)=0$, and has kernel
\[
\ker(\partial^1_1)=\{(\beta_1, \dots, \beta_6)\in\vC^1_1(W, \O^*_M)\mid \beta_2=\beta_1+\beta_3, \beta_5=\beta_4+\beta_6 \}
\]

Since $\ker\partial^1_1=\im\partial^0_1$ we have that $\H^1_1(W, \O^*_M)=0$.

When we concentrate on another vertex, say $2$, we get a situation similar to the vertex $2$ in Example~\ref{example:picW-favourite}, for which we already showed the maps. It is easy to see that also in this case $\H^1_2(W, \O^*_M)=0$ so, by the symmetry properties of the problem, we can conclude that $\Pic(W)=0$.
\end{example}

Morally, what happens in the previous example is that the intersection between simplexes happens in codimension 2 (at a point), and this does not provide elements in the Picard group.

\begin{proposition}\label{proposition:cohomology-W-simplicial-cohomology}
    Let $\triangle$ be a pure dimensional simplicial complex of dimension $m$ with $(m-1)$-dimensional faces $\{F_i\}$. Then we can again relate sheaf cohomology on $W$ and simplicial cohomology, and prove this vanishing result
    \[
    {\H^j(W, \O^*_{M_\triangle})}=\H^j(\ccc(\{F_i\}, \triangle), \ZZ^{m+1})
    \]
    for all $j\geq 2$.
\end{proposition}
\begin{proof}
    Thanks to Corollary~\ref{corollary:pure-dimensional} we know that the image of the sheaf on $D(F_{i_0})\cap \dots\cap D(F_{i_j})$ is the constant sheaf $\ZZ^{m+1}$, provided $j\geq 1$. Thanks to Proposition~\ref{proposition:nerve-crosscut}, $D(F_{i_0})\cap \dots\cap D(F_{i_j})\neq \varnothing$ if and only if $\{v_{F_{i_0}}, \dots, v_{F_{i_j}}\}$ is a face in the crosscut complex, so the groups in the cochain complex are the same on the right and on the left for any $j\geq 1$.
    Moreover, thanks to Corollary~\ref{corollary:cech-covering-nerve}, we can relate the covering of $W$ to the covering of $\Spec^\bullet (M_{\ccc(\{F_i\}, \triangle)})$ in a very easy way, since
    \[
    \nerve(\{D(F_i)\})\cong \ccc(\{F_i\}, \triangle)\cong\nerve{\{v_{F_i}\}}.
    \]
    
    Finally, thanks to Theorem~\ref{thm:simplicial-cohomology}, we have that
    \[
    \vC^\bullet(\Spec^\bullet (M_{\ccc(\{F_i\}, \triangle)}), \ZZ^{m+1})\cong\vC^\bullet(\ccc(\{F_i\}, \triangle), \ZZ^{m+1})
    \] for any $j\geq 0$. When we take cohomology in degree higher of equal $2$, we get the isomorphism.
\end{proof}

The previous proposition cannot be extended to degree $0$ or $1$, because $\vH^0=\ker \partial^0$, $\vH^1=\faktor{\ker \partial^1}{\im(\partial^0)}$ but $\C^0$ is the product of some $\ZZ^m$'s, not $\ZZ^{m+1}$'s, so the map $\partial^0$ is not part of the complex used to compute simplicial cohomology with coefficients in $\ZZ^{m+1}$, and this leaves out exactly the first two cohomologies.
The next results show, indeed, that these are actually the only non trivial ones.

\begin{lemma}
    Let $V$ be any subset of $\Spec M_\triangle$ of dimension $k$. Then
    \[
    \H^j(V, \O^*_{M_\triangle})=0,
    \]
    for all $j\geq k+1$.
\end{lemma}
\begin{proof}
    By the vanishing Theorem of Grothendieck, \cite[Theorem 3.6.5]{grothendieck1957surquelques}, we obtain our thesis.\footnote{Recall from Remark~\ref{remark:grothendieck-theorem} that the combinatorial dimension is the maximum length of strictly decreasing chains of closed subsets.}
\end{proof}

\begin{corollary}
    For all $j\geq 2$,
    \[
    \H^j(W, \O^*_{M_\triangle})=0.
    \]
\end{corollary}
\begin{proof}
    Since $W$ is the set of primes of height at most 1, it has dimension 1 and from the previous Lemma we get the result.
\end{proof}

\begin{remark}
    The previous Corollary leaves out $j=0, 1$, and it shows clearly that it might be hard to compute the divisor class group, even in simple cases.
\end{remark}

\begin{example}
    Let $\triangle$ be the simplex on $V=[n+1]$. Then $M_\triangle\cong(\NN^{n+1})^\infty$ and $\Cl(M_\triangle)=0$.
\end{example}

\begin{proposition}
    Let $\triangle$ be a pure simplicial complex of dimension $m$ whose $(m-1)$-dimensional faces are $\triangle_{m-1}=\{F_0, \dots, F_r\}$, and let $\triangle'$ be the crosscut complex of these faces in $\triangle$, $\triangle'=\ccc(\{F_i\}_{i=0, \dots, r}, \triangle)$. In order to compute $\H^0(W, \O^*_{M_\triangle})$ and $\H^1(W, \O^*_{M_\triangle})$ it is enough to consider the cohomology of the smaller chain complex
    \[
    \begin{tikzcd}[baseline=(current  bounding  box.center), cramped]
    0\rar&\displaystyle\bigoplus_{v\in \triangle'_0}\ZZ^{m}\arrow[r, "\partial^0"] & \displaystyle\bigoplus_{\{v_0, v_1\}\in \triangle'_1}\ZZ^{m+1} \arrow[r, "\partial^1"] & \displaystyle\bigoplus_{\{v_0, v_1, v_2\}\in \triangle'_2}\ZZ^{m+1}\arrow[r, "\partial^2"] & 0
    \end{tikzcd}
    \]
    where the maps are the usual maps that arise when restricting the units to smaller subsets in $\Spec M_\triangle$.
\end{proposition}
\begin{proof}
    This result is trivial thanks to Corollary~\ref{corollary:pure-dimensional} and the definition of the crosscut complex.
\end{proof}

\afterpage{\null\newpage}

\chapter{From Combinatorics to Algebra}\label{chapter:injections}

In this Chapter we investigate some relations between the local Picard group of binoids and the local Picard group of binoid $\KK$-algebras. We will address the problem by trying to understand how the sheaf of units behave on the combinatorial side and on the algebraic side, and if there is any relation between $\O^*_{M}(U)$ and $\O^*_{\KK[M]}(V)$ for some $U\subseteq \Spec M$ and $V\subseteq\Spec \KK[M]$ that have a common description.
We are going to recall some functors between binoids and rings, $M$-sets and $\KK[M]$-modules, and sheaves. In general the sheaves will be different, because the algebraic side will depend on $\KK$, and so the cohomologies can be different too. Maybe some will vanish on one hand but not on the other, maybe it depends on the field, maybe not. We will address some of these behaviours in this and in the next Chapter.\\
In what follows, when we say scheme, we always refer to a subscheme of an affine scheme $\Spec R$ of some noetherian and finitely generated $\KK$-algebra $R$. This will save us to repeat every time some hypothesis that we would assume anyway, like separateness and noetherianity.\\
For simplicity, with $\KK$ we denote any field and we always consider $\KK$-algebras, but many results presented here are true for a wider class of base rings.

\section{Injections between \texorpdfstring{$\Spec M$}{Spec M} and \texorpdfstring{$\Spec \KK[M]$}{Spec K M}}

The faithful functor\footnote{Recall that a functor $\mathcal{F}:\C\longrightarrow \mathcal{D}$ is faithful (resp.\ full) if for any two objects $X, Y\in \C$, the induced map $\mathcal{F}_{X, Y}\Hom_\C(X, Y)\longrightarrow \Hom_\mathcal{D}(\C(X), \C(Y))$ is injective (resp.\ surjective).\\
    $\mathcal{F}$ is essentially injective (resp.\ essentially surjective) if it is injective (resp.\ surjective) on the objects.}
\begin{equation}\label{functor:binoids-rings}
\begin{tikzcd}[baseline=(current  bounding  box.center),
/tikz/column 1/.append style={anchor=base east},
/tikz/column 2/.append style={anchor=base west}, 
row sep = 0pt]
\gls{functor-KK}: \mathrm{Binoids}\rar & \mathrm{Rings} \\
M\rar[mapsto] & \gls{binoid-algebra}
\end{tikzcd}
\end{equation}
induces other functors of spectra, sheaves and then cohomology groups, that we are going to exploit in what follows.
We note that this functor is only faithful. It is neither full, essentially injective nor surjective.

The second functor that we want to consider is between finitely generated $M$-sets and finitely generated $\KK[M]$-modules:
\begin{equation}\label{functor:mset-kmmodules}
\begin{tikzcd}[baseline=(current  bounding  box.center),
/tikz/column 1/.append style={anchor=base east},
/tikz/column 2/.append style={anchor=base west}, 
row sep = 0pt]
\KK[\hspace{1em}]: M\text{-}\mathrm{Sets}\rar & \KK[M]\text{-}\mathrm{Modules} \\
S \hspace{1em\rar[mapsto]}& \KK[S]
\end{tikzcd}
\end{equation}
where $\KK[S]$ is the free $\KK$-module on $S\smallsetminus\{p\}$, where $p$ is the special point of $S$, together with the natural action of $\KK[M]$.
It is easy to see that this functor is again faithful.

\begin{lemma}[{\cite[Corollary  3.2.8]{Boettger}}]\label{lemma:functor-respects-localizations}
    The functor $\KK[\hspace{1em}]$ respects localizations, i.e.\ let $E$ be an additive subset of $M$, that defines the multiplicative subset $\overline{E}=\{T^a\mid a\in E\}$ of $\KK[M]$. Then $\KK[M_E]=\overline{E}^{-1}\KK[M]$.
\end{lemma}

From now on, assume that $M$ is \emph{torsion-free up to nilpotence} and \emph{cancellative}.

\begin{lemma}
    If $M$ is a torsion-free up to nilpotence, integral and cancellative binoid, then $\KK[M]$ is an integral domain.
\end{lemma}
\begin{proof}
    Under this conditions, the map $M\rightarrow\Gamma$ is an injection and this injection reflects on the rings, $\KK[M]\hookrightarrow\KK[\Gamma]$.\footnote{See Definition~\ref{definition:gamma} for a definition of $\Gamma$.} Moreover, $\Gamma\cong(\ZZ^r)^\infty$ for some $r$. Since $\KK[(\ZZ^r)^\infty]$ is an integral domain, also $\KK[M]$ is such.
\end{proof}

\begin{lemma}\label{lemma:prime-in-M-iff-prime-in-KM}
    $\p$ is a prime ideal of $M$ if and only if $\mathfrak{P}=\KK[\p]$ is a prime ideal of $\KK[M]$.
\end{lemma}
\begin{proof}
    $\Longleftarrow$ trivial, since $\KK[\p]\cap M=\p$.\\
    $\Longrightarrow$ $\mathfrak{P}$ is prime if and only if $\faktor{\KK[M]}{\mathfrak{P}}$ is an integral domain. We know that $\KK\left[\faktor{M}{I}\right]\cong\faktor{\KK[M]}{\KK[I]}$ for any ideal, and $\faktor{M}{\p}$ is integral because $\p$ is a prime ideal. We can apply the Lemma above and get the result.
\end{proof}

The functor \eqref{functor:mset-kmmodules}, together with Lemma~\ref{lemma:prime-in-M-iff-prime-in-KM}, gives rise to an injection of sets
\begin{equation}\label{injection:specm-speckm}
\begin{tikzcd}[baseline=(current  bounding  box.center),
/tikz/column 1/.append style={anchor=base east},
/tikz/column 2/.append style={anchor=base west}, 
row sep = 0pt]
i: \Spec M\rar[hook] & \Spec \KK[M]\\
\p \hspace{1em}\rar[mapsto] & \KK[\p]
\end{tikzcd}
\end{equation}

\begin{lemma}\label{lemma:i-continuous}
    $i$ is a continuous map between the two spaces equipped with the respective Zariski topologies.
\end{lemma}
\begin{proof}
    It suffices to prove that the preimage of a fundamental open subset $D(F)$ is an open subset in $\Spec M$. Indeed, $i^{-1}(D(F))=\{\p\in\Spec M\mid F\notin\KK[\p]\}$.
    \\
    Let $F=\sum_{\mu\in M}\alpha_\mu T^\mu$. Then $F\in\KK[\p]$ if and only if $\supp F\subseteq\p$. So $F\notin\KK[\p]$ if and only if there exists $\mu\in\supp F$ such that $\mu\notin\p$. This happens if and only if $\p\in D(\mu)$. So
    \[
    i^{-1}(D(F))=\bigcup_{\mu\in\supp F}D(\mu).\qedhere
    \]
\end{proof}
\begin{remark}
    $i$ is not a closed injection.
\end{remark}

\begin{lemma}\label{lemma:U-intersect-SpecM}
    For any non empty open subset $\varnothing\neq U\subseteq\Spec\KK[M]$ the intersection $U\cap i(\Spec M)$ is non empty.
\end{lemma}
\begin{proof}
    If $M$ is integral then $\langle0\rangle=i(\langle\infty\rangle)\in U$. If $M$ is non integral, consider a minimal prime ideal $\mathfrak{P}\in\Spec \KK[M]$. Then $\mathfrak{P}=i(\p)$, see \cite[Corollary 3.3.5]{Boettger}. Thanks to the correspondence between prime ideals in $\KK[M]$ that contain $\mathfrak{P}$ and prime ideals in $\faktor{\KK[M]}{\mathfrak{P}}$ and thanks to the fact that $\faktor{\KK[M]}{\mathfrak{P}}$ is integral, we can apply the previous case and obtain our result.
\end{proof}

What about sheaves? Since the map in \eqref{injection:specm-speckm} is continuous, we can pushforward a sheaf from the combinatorial spectrum to the algebraic spectrum
\begin{equation}\label{functor:sheavesM-sheavesKM}
\begin{tikzcd}[baseline=(current  bounding  box.center),
/tikz/column 1/.append style={anchor=base east},
/tikz/column 2/.append style={anchor=base west}, 
row sep = 0pt]
\KK[\hspace{1em}]: \mathfrak{Sheaves}_{\Spec M}\rar & \mathfrak{Sheaves}_{\Spec \KK[M]}\\
\F\rar[mapsto] & \KK[\mathcal{\F}]:=i_*\F
\end{tikzcd}
\end{equation}
\begin{remark}
    If $\F$ is a sheaf of $M$-sets (resp.\ sets, abelian groups, \dots) on $\Spec M$, then $\KK[\F]$ is a sheaf of $M$-sets (resp.\ sets, abelian groups, \dots) on $\Spec \KK[M]$, because $\KK[\F](U)=\F(i^{-1}(U))$ by definition of the pushforward $i_*$.
\end{remark}

This restricts in a natural way to the punctured situation we are mainly interested in.

The next step is to say something more explicit about the relation between the corresponding topologies. The injection \eqref{injection:specm-speckm} induces a functor between the two topologies%
\begin{equation}\label{functor:topspecm-topspeckm}
\begin{tikzcd}[baseline=(current  bounding  box.center),
/tikz/column 1/.append style={anchor=base east},
/tikz/column 2/.append style={anchor=base west}, 
row sep = 0pt]
\KK[\hspace{1em}]: \Top_{\Spec M}\rar & \Top_{\Spec \KK[M]}\\
D(I) \hspace{1em}\rar[mapsto]& D(\KK[I])\\
\end{tikzcd}
\end{equation}
Thus this functor maps affine open subsets of the form $D(I)$ to affine open subsets of the same form, namely $D(\KK[I])$, inducing another topology on $\Spec\KK[M]$, coarser than the Zariski topology, that we exploit in detail below.

\begin{definition}
    The covering $D(X_i)$ of $\Spec^\bullet \KK[M]$ is called affine \emph{combinatorial covering}.
\end{definition}

Indeed, the image of this functor gives rise to a new topological space
\[
(\Spec \KK[M], \KK[\Top_{\Spec  M}]),
\] subspace of $(\Spec  \KK[M],\Top_{\Spec  \KK[M]})$. In some situations, like in the Stanley-Reisner rings in the next Chapter, we are able to show that this last topology is fine enough for us, in the sense that these open sets are acyclic for $\O^*_{\KK[M]}$ and so we can compute cohomology of this sheaf through \v{C}ech cohomology on a nice finite combinatorial covering induced by this combinatorial topology $\KK[\Top_{\Spec  M}]$.

Via $i$ and $\KK[\hspace{1em}]$ we obtain also an identification of the finite topological spaces
\[
(\Spec  M, \Top_{\Spec  M})\longleftrightarrow(i(\Spec  M), \KK[\Top_{\Spec  M}]\restriction_{i(\Spec  M)}).
\]
The topology on the right is the same as the one induced on these points by the restriction of $\Top_{\Spec  \KK[M]}$, so we have that
\[
(i(\Spec  M), \KK[\Top_{\Spec  M}]\restriction_{i(\Spec  M)})=(i(\Spec  M), \Top_{\Spec  \KK[M]}\restriction_{i(\Spec  M)})
\]
is a subspace of $(\Spec  \KK[M], \Top_{\Spec  \KK[M]})$.

\begin{definition}
    Let $\widehat{I}$ be an ideal in $\KK[M]$. We say that $\widehat{I}$ is \emph{combinatorial} if $\widehat{I}=\KK[I]$ for some $I$ ideal of $M$.
\end{definition}
\begin{remark}
    Any combinatorial ideal is monomial.
\end{remark}

\begin{definition}
    Given a prime ideal $\mathfrak{P}\in\KK[M]$ we denote by $\mathfrak{P}^{\mathrm{mon}}$ the ideal of $\KK[M]$ generated by the monomials in $\mathfrak{P}$, and by $\mathfrak{P}^{\comb}$ the ideal in $M$ such that $\mathfrak{P}^{\mathrm{mon}}=\KK[\mathfrak{P}^{\comb}]$.
\end{definition}
\begin{remark}
    The ideal $\mathfrak{P}^{\mathrm{mon}}$ is clearly combinatorial.
\end{remark}

\begin{remark}\label{remark:abuse-notation-Pcomb}
    Every monomial $G\in\mathfrak{P}^{\mathrm{mon}}$ corresponds uniquely to an element $g$ in $\mathfrak{P}^{\comb}$, so we will abuse the notation and say that $\mathfrak{P}^{\comb}=\mathfrak{P}\cap M$. Moreover, $\mathfrak{P}^{\comb}$ is a prime ideal in $M$.
\end{remark}

\begin{lemma}
    Let $P=\sum \alpha_jP_j\in\KK[M]$ with $P_j$ monomials. Then $i^{-1}(D(P))=\bigcup D(p_j)$, where $p_j$ is the element in $M$ corresponding to the monomial $P_j$ (i.e.\ the sum of exponents).
\end{lemma}
\begin{example}
    Let $M=(x, y\mid x+y=\infty)$. Let $P=2X+4Y^2\in\KK[M]$. Then $D(P)=\{\mathfrak{P}\mid P\notin \mathfrak{P}\}$ and $D(P)\cap i(\Spec M)=\{\langle X\rangle, \langle Y\rangle\}=i(D(y)\cup D(x))$.
\end{example}

All the above restricts naturally to $\Spec^\bullet\KK[M]$, the punctured situation we are interested in.

\begin{definition}
    We denote by $\gls{punctured-spectrum-KM}$ the \emph{punctured spectrum of $\KK[M]$}, i.e.\
    \[
    \Spec^\bullet\KK[M]:=\Spec \KK[M]\smallsetminus \{\KK[M_+]\}.
    \]
\end{definition}
In local cohomology one defines the punctured spectrum for a local ring, when we remove the maximal ideal, for example in \cite[Definition 15.3]{iyengar2007twenty}. Here we don't have in general a local ring, but the maximal ideal $\KK[M_+]$ is special enough to consider it differently from the others. If the binoid ring is $\NN$-graded, $\KK[M_+]$ is the irrelevant ideal, and this will often be the case in our work.

We will be interested in computing the cohomology of some sheaves, both on $\Spec \KK[M]$ and on its punctured version.
It is a known result that $\H^1(X, \O^*_X)$ is the Picard group of $X$, i.e.\ the group of invertible $\O_X$-sheaves on $X$, for any ringed space $(X, \O^*_X)$, see for example \cite[Exercise III.4.5]{hartshorne1977algebraic}.
\begin{definition}
    Let $\KK[M]$ be a binoid ring. Its \emph{local Picard group} is the Picard group of the scheme
    \[
    (\Spec^\bullet \KK[M], \O_{\KK[M]}\restriction_{\Spec^\bullet \KK[M]})
    \]
    and it is denoted by $\gls{PiclocKM}$.
\end{definition}

The adjective \emph{local} that we are using is motivated because what we are doing is looking at that happens around the special point $\KK[M_+]$, and then we might think of the long exact sequence of cohomology groups in \cite[Exercise III.2.3.(e)]{hartshorne1977algebraic} that relates cohomology with support in the point to the cohomology of the punctured spectrum. The hope is to understand something about $\Spec^\bullet R$ in order to better understand what happens in the special point.

In general, $\KK[M_+]$ will be a singularity, and it goes a while back the idea of study a singularity by removing it and see what happens around it. Trivially, $\KK[M_+]$ is at most an isolated singularity if and only if $\Spec^\bullet \KK[M]$ is be smooth. Another classical example where we remove a point in order to study it is the famous theorem published by Mumford in 1961.
\begin{theorem}[\cite{mumford1961topology}]
    Let $P$ be a point lying on a normal complex projective variety $X$ of dimension 2. $P$ is smooth if and only if $\pi_1(X\smallsetminus\{P\})=0$.
\end{theorem}

\section{The Combinatorial Topology}

In this Section we describe a new topology on $\Spec \KK[M]$ that we call combinatorial, and we will prove that in some cases this topology is sufficient to compute cohomology of the sheaves defined in the Zariski topology.

\begin{proposition}
    The collection of sets $\{D(\mathfrak{A})\}$ with $\mathfrak{A}$ a combinatorial ideal is a topology on $\Spec\KK[M]$.
\end{proposition}
\begin{proof}
    $D(\langle 0\rangle)=\varnothing$ and $D(\langle 1\rangle)=\Spec\KK[M]$. $D(\mathfrak{A})\cap D(\mathfrak{B})=D(\mathfrak{A}\cdot\mathfrak{B})$, and the latter is monomial (so combinatorial) again. Finally, $\bigcup_iD(\mathfrak{A}_i)=D(\sum_i \mathfrak{A}_i)$ and the latter is again monomial.
\end{proof}

\begin{definition}\label{definition:combinatorial-topology}
    The topology $\{D(\mathfrak{A})\}$ is called \emph{combinatorial topology} of $\Spec \KK[M]$ and it is denoted by $\Top_{\Spec \KK[M]}^{\comb}$.
\end{definition}

\begin{corollary}\label{corollary:basis-combinatorial-topology}
    The collection of sets $\{D(P)\}$, with $P$ a monomial in $\KK[M]$, is a basis for the combinatorial topology.
\end{corollary}
\begin{proof}
    These sets form a basis for a topology on $\Spec \KK[M]$ because $D(P_1)\cap D(P_2)=D(P_1\cdot P_2)$. To prove that they generate the combinatorial topology, we first observe that $D(P)=D(\langle P\rangle)$ was already part of the basis of the combinatorial topology. On the other side, $\displaystyle D(\mathfrak{A})=\bigcup_{P\in\mathfrak{A}^{\comb}}D(P)$, so we can generate the basis before with these new open subsets.
\end{proof}

\begin{remark}
    If $P=\alpha\prod X_i^{r_i}$ with all $r_i$'s non zero, then $D(P)=\bigcap D(X_i)$, so $\{D(X_i)\}$ generate this topology via union and intersections, but it is not a basis because the intersection of any two of them does not contain one of the same type.
\end{remark}

\begin{example}
    Let $M=\NN^\infty$, so $\KK[M]=\KK[X]$. The combinatorial topology of $\Spec \KK[M]$ is given by $D(0)=\varnothing$, $D(1)=\Spec \KK[M]$ and $D(X)$.
\end{example}

\begin{remark}
    The combinatorial topology is, in general, a really coarse topology, since $\Spec\KK[M]$ equipped with it is not even $T_0$. In fact, two points $\mathfrak{P}$ and $\mathfrak{Q}$ have exactly the same combinatorial neighbourhoods if and only if they have the same set of monomials, $\mathfrak{P}^{\mathrm{mon}}=\mathfrak{Q}^{\mathrm{mon}}$.
\end{remark}

\begin{remark}
    The combinatorial topology is a topology for $\Spec \KK[M]$, that is a subtopology of the Zariski topology, so we can restrict Zariski sheaves to combinatorial open subsets, and they are still sheaves for the new space.
\end{remark}

\begin{remark}
    We have indeed this diagram
    \[
    \begin{tikzcd}
    \Spec M\rar["i"] \drar["j", swap]& \Spec \KK[M]_{\mathrm{Zar}}\dar["\lambda"]\\
    &\Spec \KK[M]_{\comb}
    \end{tikzcd}
    \]
    where $\lambda$ is the identity, $i$ is the injection that we proved to be continuous in Lemma~\ref{lemma:i-continuous} and $j$ is the embedding to the space with the combinatorial topology, that is again obviously continuous.
    
    The restriction in the previous remark is nothing else than the pushforward along $\lambda$. Given any sheaf $\F$ on $\Spec\KK[M]_{\mathrm{Zar}}$, we get a sheaf $\lambda_*\F=\F\restriction_{\Top_{\comb}}$ and in particular also $\lambda_*(i_*\O^*_M)=j_*\O^*_M$.
\end{remark}

\begin{definition}
    Let $\F$ be a sheaf on $\Spec \KK[M]$ equipped with the Zariski topology. We denote by $\F^{\comb}$ the restriction of this sheaf to the combinatorial topology.
\end{definition}


\begin{lemma}
    For any sheaf of abelian groups $\F$ on $\Spec\KK[M]$ equipped with the combinatorial topology and any $\mathfrak{P}$ prime ideal of $\KK[M]$ we have that
    \[
    \F_{\mathfrak{P}}\cong \F_{\KK[\mathfrak{P}^{\comb}]}.
    \]
\end{lemma}
\begin{proof}
    $\mathfrak{P}$ and $\KK[\mathfrak{P}^{\comb}]$ have the same combinatorial neighbourhoods, i.e.\ $\mathfrak{P}\in U\in\Top^{\comb}$ if and only if $\KK[\mathfrak{P}^{\comb}]\in U\in\Top^{\comb}$.
\end{proof}

\begin{lemma}\label{lemma:units-group-algebra}
    $\KK\left[(\ZZ^l)^\infty\right]^*\cong \KK^*\oplus \ZZ^l$.
\end{lemma}
\begin{proof}
    The algebra on the left is the same as the monoid algebra $\KK[\ZZ^l]$, for which the result is known.
\end{proof}

\begin{lemma}\label{lemma:units-minimal-prime}
    If $\p$ is a minimal prime ideal of $M$, then
    \[
    \KK\left[M_\p\right]^*\cong \KK\left[(\ZZ^l)^\infty\right]^*
    \]
    for some $l$.
\end{lemma}

\begin{lemma}\label{lemma:torsion-free-cancellative-reduced-algebra}
    Let $M$ be a reduced, torsion-free, cancellative binoid. Then $\KK[M]$ is reduced.
\end{lemma}
\begin{proof}
    Since $M$ is torsion-free and cancellative, there is a correspondence between minimal prime ideals of $M$ and minimal prime ideals of $\KK[M]$. Since the nilradical is the intersection of minimal prime ideals and it is trivial in the binoid because it is reduced, it is trivial in the algebra, thus proving that the latter is also reduced.
\end{proof}

\begin{proposition}\label{proposition:split-algebra}
    Let $M$ be a reduced, torsion-free, cancellative binoid and let $\KK[M]$ be its binoid algebra. Then
    \[
    \left(\KK[M]\right)^*=\KK^*\oplus M^*.
    \]
\end{proposition}
\begin{proof}
    What we have to prove is that, under these hypothesis, any unit is a product of a monomial and a unit in the field.
    On the binoid side, since by definition $M_+=M\smallsetminus M^*$, there is an isomorphism
    \[
    M_*^\infty \cong \faktor{M}{M_+}.
    \]
    Let $\p$ be a minimal prime ideal of $M$. Since $\p\subseteq M_+$, there are maps
    \[
    \begin{tikzcd}[cramped]
    M\rar["\pi_\p"]& \faktor{M}{\p}\rar["\pi_{M_+}"]&\faktor{M}{M_+}=(M^*)^\infty
    \end{tikzcd}
    \]
    and, since $(M^*)^\infty\subseteq M$, we have a map $\sigma$ going the other way
    \[
    \begin{tikzcd}[cramped]
    M\rar["\pi_\p"]& \faktor{M}{\p}\rar["\pi_{M_+}"]&(M^*)^\infty\arrow[ll, out=-10, in=190, overlay, pos=0.07, "\sigma"]
    \end{tikzcd}
    \]
    such that the composition $\pi_{M_+}\circ \pi_\p\circ \sigma$ is the identity of $(M^*)^\infty$.
    Thanks to the functor from binoids to rings, we get maps for the rings
    \[
    \begin{tikzcd}[cramped]
    \KK[M]\rar& \KK\left[\faktor{M}{\p}\right]\rar&\KK\left[(M^*)^\infty\right]\arrow[ll, out=-10, in=190, overlay, pos=0.07]
    \end{tikzcd}
    \]
    that again compose to the identity on the right and that induce maps of groups
    \[
    \begin{tikzcd}[cramped]
    \KK[M]^*\rar& \KK\left[\faktor{M}{\p}\right]^*\rar&\KK\left[(M^*)^\infty\right]^*\arrow[ll, out=-10, in=190, overlay, pos=0.07]
    \end{tikzcd}
    \]
    that again compose to the identity on the right.
    Let $P$ be a unit in $\KK[M]$. Then
    \[
    P= \lambda_\nu X^\nu +\sum_{\mu\in M_+}\lambda_{\mu}T^\mu
    \]
    with $\nu\in M^*$, since $\KK[M^*]\cong \faktor{\KK[M]}{\KK[M_+]}$ and the statement is true for $\KK[M^*]$ thanks to Lemma~\ref{lemma:units-group-algebra}.
    If we apply the first map $\pi_\p$ to $P$, we get
    \[
    \faktor{\lambda_\nu X^\nu +\sum_{\mu\in M_+}\lambda_{\mu}T^\mu}{\KK[\p]}\in \KK\left[\faktor{M}{\p}\right]^*.
    \]
    We can apply the previous Lemma~\ref{lemma:units-minimal-prime} to obtain that this has to be a monomial. So, in particular, $\sum_{\mu\in M_+}\lambda_{\mu}T^\mu\in\KK[\p]$ for all minimal prime ideals $\p$. This means that
    \[
    \sum_{\mu\in M_+}\lambda_{\mu}T^\mu \in \bigcap_{\begin{subarray}{c}
        \p \text{ minimal}\\
        \text{prime of } M
        \end{subarray}} \KK[\p]=\nil(\KK[M]).
    \]
    Since $\KK[M]$ is reduced thanks to Lemma~\ref{lemma:torsion-free-cancellative-reduced-algebra}, its nilradical is trivial, so
    \[
    \sum_{\mu\in M_+}\lambda_{\mu}T^\mu=0
    \]
    and $P$ is a monomial
    \[
    P=\lambda_\nu X^\nu\in\KK^*\times M^*.\qedhere
    \]
\end{proof}

\begin{remark}\label{remark:units-non-reduced}
    If $M$ is torsion-free and cancellative but not reduced, then the algebra is not reduced. Still, we can split its units as
    \[
    \left(\KK[M]\right)^*=\KK^*\oplus M^*\oplus(1+\n)
    \]
    where $\n$ is the nilradical of $\KK[M]$. Indeed, in the above proof, we would just have that
    \[
    N=\sum_{\mu\in M_+}\lambda_{\mu}T^\mu
    \]
    is nilpotent. So in particular $1+N$ is a unit, as well as $1+\frac{N}{\lambda_\nu X^\nu}$. let $(1+M)$ be the inverse of $1+\frac{N}{\lambda_\nu X^\nu}$.
    So the inverse of $P=\lambda_\nu X^\nu +N$ is then $(\frac{1+M}{\lambda_\nu X^\nu})$.
\end{remark}

\begin{proposition}\label{proposition:split-sheaves-comb-top}
    Let $M$ be a reduced, torsion-free and cancellative binoid and let $\KK[M]$ be its binoid algebra. Then
    \[
    (\O_{\KK[M]}^*)^{\comb}\cong (\KK^*)^{\comb}\oplus (i_*\O^*_M)^{\comb},
    \]
    where $\KK^*$ is the constant sheaf, as usual.
\end{proposition}
\begin{proof}
    We have a natural sheaf homomorphism in the Zariski topology
    \[
    \begin{tikzcd}[cramped]
    \KK^*\oplus i_*\O^*_M\rar & \O_{\KK[M]}^*
    \end{tikzcd}
    \]
    because every element in $\KK^*(U)\oplus i_*\O^*_M(U)$ is trivially a unit in $\O^*_{\KK[M]}(U)$ for any $U$, and in particular in the combinatorial topology the map
    \[
    \begin{tikzcd}[cramped]
    (\KK^*)^{\comb}\oplus (i_*\O^*_M)^{\comb}\rar & (\O_{\KK[M]}^*)^{\comb}
    \end{tikzcd}
    \]
    is an isomorphism. It is enough to show the isomorphism of combinatorial affine open subsets $D(P)$, with $P$ a monomial. Indeed,
    \begin{align*}
    \Gamma\left(D(P), (\O_{\KK[M]}^*)^{\comb}\right)&=\Gamma\left(D(P), \O_{\KK[M]}^*\right)\\
    &=\Gamma\left(\Spec\KK[M_P], \O_{\KK[M]}^*\right)=\KK[M_P]^*
    \end{align*}
    Since $M_P$ is again reduced, torsion-free and cancellative, we can apply the proposition above and obtain a decomposition
    \begin{align*}
    \Gamma\left(D(P), (\O_{\KK[M]}^*)^{\comb}\right)&=\KK[M_P]^*=\KK^*\oplus (M_P)^*\\
    &=\Gamma\left(D(P), (\KK^*)^{\comb}\right)\oplus \Gamma\left(D(P), (i_*\O^*)^{\comb}\right)\\
    &=\Gamma\left(D(P), (\KK^*)^{\comb}\oplus (i_*\O^*)^{\comb}\right)\qedhere
    \end{align*}
\end{proof}

\begin{notation} For any combinatorial open subset $U$ of $\Spec\KK[M]$ and Zariski sheaf $\F$ on $U$, we use the usual notation $\H^i_{\comb}(U, \F)$ to denote the cohomology of the sheaf $\F^{\comb}$ on $U$, i.e.\ the cohomology of $\F$ in the combinatorial topology.
\end{notation}

\begin{proposition}\label{proposition:hzar-hcomb}
    If $U=D(\mathfrak{A})=\bigcup D(P)$ is a combinatorial open subset of $\Spec \KK[M]$ and $\{D(P)\}$, with $P$ monomial, is an acyclic covering for the Zariski sheaf $\F$ on $U$, then
    \[
    \H^j_{\mathrm{Zar}}(U, \F)=\H^j_{\comb}(U, \F),
    \]
    for all $j\geq 0$.
\end{proposition}
\begin{proof}
    Since $\{D(P)\}$ is an acyclic covering for $\F$ in the Zariski topology, we can use it to compute Zariski cohomology via \v{C}ech cohomology on this covering, and the complexes look the same for the Zariski topology and for the combinatorial topology
    \[
    \vC(\{D(P)\}, \F)=\vC(\{D(P)\}, \F^{\comb})
    \]
    so in particular
    \[
    \H^i_{\mathrm{Zar}}(U, \F)=\H^i_{\comb}(U, \F).\qedhere
    \]
\end{proof}

\begin{remark}
    For the toric case in Remark~\ref{remark:toric-case-sheaf-units} and for the Stanley-Reisner case in the next Chapter, we proved that the \v{C}ech complex for $\O^*$ can be split in the sum of the \v{C}ech complexes for $\KK^*$ and for $i_*\O^*$. The advantage of the combinatorial topology is that this decomposition comes already at a sheaf level, which is not true for the Zariski topology.
\end{remark}

\begin{remark}
    Since the combinatorial topology is a subtopology of the Zariski topology, we get a continuous map $X_{\mathrm{Zar}}\longrightarrow X_{\comb}$ that is not bicontinuous, so the Proposition above cannot be extended to arbitrary sheaves without considering the acyclic covering.
\end{remark}

\begin{remark}
    Let $\G$ be a Zariski sheaf on $U$ combinatorial open subset of $\Spec \KK[M]$, and let
    \[
    \G\longrightarrow\F_\bullet
    \]
    be a flasque resolution of it. We get naturally a sequence of flasque sheaves in the combinatorial topology via the restriction to it,
    \[
    \G^{\comb}\longrightarrow\F^{\comb}_{\bullet}.
    \]
    Indeed, each sheaf in $\F^{\comb}_\bullet$ is again flasque because the restriction maps
    \[
    \begin{tikzcd}[cramped]
    \F^{\comb}_i(D(\mathfrak{A}))\rar&\F^{\comb}_i(D(\mathfrak{B}))
    \end{tikzcd}
    \]
    are surjective, since they were surjective before.    
    Since, in general, $\F^{\comb}_{\mathfrak{P}}\neq \F_{\mathfrak{P}}$, this new flasque sequence of sheaves is not exact, so it is not a resolution and we cannot use it to compute cohomology. 
\end{remark}

\section{Pushforwards}

\begin{lemma}\label{lemma:cohomology_pushforward}
    Let $\widetilde{U}$ be a combinatorial open subset of $\Spec \KK[M]$, with a covering $\widetilde{\U}=\{\widetilde{U_j}\}_{j\in J}$ made of combinatorial affine open subsets.
    Let $U=i^{-1}(\widetilde{U})$ be the correspondent open subset of $\Spec M$, covered by $\U=i^{-1}(\widetilde{\U})=\{i^{-1}(\widetilde{U_j})\}_{j\in J}$ and let $\F$ be a sheaf of abelian groups on $U$. Then
    \begin{equation}\label{isomorphism:cohomologyM-cohomologyKM}
    \H^j(U, \F)\cong \vH^j(\widetilde{\U}, i_*\F)
    \end{equation}
    for all $j$.
\end{lemma}
\begin{proof}
    Let $U_j=i^{-1}(\widetilde{U_j})$. $\{U_j\}$ defines an acyclic covering of $U$ for $\F$, because they are affine open subsets of $\Spec M$, so its \v{C}ech cohomology computes the cohomology on the left. Moreover, $i_*\F(\widetilde{U_j})=\F(i^{-1}(\widetilde{U_j}))=\F(U_j)$, so the \v{C}ech complexes are the same, $\vC(\widetilde{\U}, i_*\F)=\vC(\U, \F)$, and we get our result.
\end{proof}

\begin{corollary}
    Let $\F$ be a sheaf of abelian groups on $\Spec^\bullet M$ and let $\U=\{D(X_k)\}$ be the combinatorial covering of $\Spec^\bullet \KK[M]$. Then
    \begin{equation}
    \H^j(\Spec^\bullet M, \F)\cong \vH^j(\U, i_*\F)
    \end{equation}
    for all $j$.
\end{corollary}

\begin{lemma}\label{lemma:trivial-cohomology-push-forward}
    Let $\F$ be a sheaf of abelian groups on $\Spec M$ and $\U$ any covering of $\Spec\KK[M]$. Then
    \begin{equation}
    \vH^j(\U, i_*\F)=0
    \end{equation}
    for all $j\geq 1$.
\end{lemma}
\begin{proof}
    The preimage of the covering $i^{-1}(\U)$ is a covering of $\Spec M$. In particular, since $i_*\F(U_j)=\F(i^{-1}(U_j))$ for all $U_j\in\U$, the \v{C}ech complexes are the same 
    \[
    \C(\U, i_*\F)=\C(i^{-1}(\U), \F).
    \]
    Finally, since $\Spec M$ is affine, we know that the cohomology of degree bigger than 0 of the combinatorial complex is zero, and so it is the one of the pushforward.
\end{proof} 

\begin{theorem}\label{theorem:H1pushforward_is_0}
    $\H^1(\Spec \KK[M], i_*\F)=0$
\end{theorem}
\begin{proof}
    From \cite[Exercise III.4.4]{hartshorne1977algebraic} we know that
    \[
    \H^1(\Spec \KK[M], i_*\F)=\varinjlim_\U\vH^1(\U, i_*\F),
    \]
    where the limit is taken over all the possible coverings of $X$.
    Assume that there is a non-zero cohomology class $[c]$ in $\H^1(\Spec \KK[M], i_*\F)$. Then there exists a covering that realizes it, i.e.\ $[c]\in\vH^1(\U, i_*\F)$. But this is impossible, thanks to Lemma~\ref{lemma:trivial-cohomology-push-forward}.
\end{proof}

\begin{lemma}
    $(i_*\F)_{\mathfrak{P}}\cong\F_{\mathfrak{P}^{\comb}}$.\footnote{Recall from Remark~\ref{remark:abuse-notation-Pcomb} our abuse of notation $\mathfrak{P}^{\comb}=\mathfrak{P}\cap M$.}
\end{lemma}
\begin{proof}
    We begin by investigating the stalk of the pushforward
    \begin{align*}
    (i_*\F)_{\mathfrak{P}}=\varinjlim_{\mathfrak{P}\in U}\F(i^{-1}(U))=\varinjlim_{P\notin \mathfrak{P}}\F(i^{-1}(D(P)))=\varinjlim_{P\notin \mathfrak{P}}\F(\cup(D(P_j)))
    \end{align*}
    where $P=\sum \alpha_jP_j$, $\alpha_j\neq 0$. Moreover, $P\notin \mathfrak{P}$ implies that there exists $j$ such that $P_j\notin\mathfrak{P}$, and this is true if and only if $P_j\notin \mathfrak{P}^{\comb}$.
    Consider the direct limit
    \[
    \varinjlim_{g\notin\mathfrak{P}^{\comb}}\F(D(g)).
    \]
    Since $\{g\notin\mathfrak{P}^{\comb}\}\subseteq \{P\notin\mathfrak{P}\}$, there is a natural map
    \[
    \varinjlim_{g\notin\mathfrak{P}^{\comb}}\F(D(g))\longrightarrow \varinjlim_{P\notin \mathfrak{P}}\F(\cup(D(P_j))).
    \]
    This map is surjective because, given a section in the stalk $s\in (i_*\F)_{\mathfrak{P}}$, there exists a polynomial $P=\sum\alpha_jP_j$ such that $s\in\F(\cup(D(P_j)))$. In particular, one of these $P_j$'s is not in $\mathfrak{P}$ and so not in $\mathfrak{P}^{\comb}$. Let $P_k$ be this monomial, so $s$ comes via the restriction $\F(\cup(D(P_j)))\longrightarrow \F(D(P_k))$ also from a section in $\F(D(P_k))$. As such, it comes from the left, so the map is surjective.
    \\
    This map is also injective because, given $s$ and $t$ in $\displaystyle\varinjlim_{g\notin\mathfrak{P}^{\comb}}\F(D(g))$, if their images are the same in the limit then in particular they are the same on some open subset $D(P_j)$ such that $P_j\notin\mathfrak{P}^{\comb}$, so they were already the same before. This proves that
    \[
    (i_*\F)_{\mathfrak{P}}\cong \varinjlim_{g\notin\mathfrak{P}^{\comb}}\F(D(g))=\F_{\mathfrak{P}^{\comb}}.\qedhere
    \]
\end{proof}

\begin{theorem}\label{theorem:push-forward-exact}
    The pushforward of a sheaf of abelian groups along $i$ is exact.
\end{theorem}
\begin{proof}
    Let $\F$ be a sheaf of abelian groups on $\Spec M$. Since exactness is a local property, it is enough to prove it on the stalks. We can apply the Lemma above to obtain
    \[
    (i_*\F)_{\mathfrak{P}}\cong\F_{\mathfrak{P}^{\comb}},
    \]
    that is exactly our desired result.
\end{proof}

\begin{remark}
    When we have an exact sequence of sheaves on $\Spec M$
    \[
    \begin{tikzcd}[cramped, row sep = 0pt]
    0\rar & \F\rar & \G\rar & \mathscr{H}\rar & 0
    \end{tikzcd}
    \]
    and we pushforward it, in general we get an exact sequence
    \[
    \begin{tikzcd}[cramped, row sep = 0pt]
    0\rar & i_*\F\rar & i_*\G\rar & \faktor{i_*\G}{i_*\F}\rar & 0.
    \end{tikzcd}
    \]
    The Proposition above proves that we can see the last quotient as the pushforward of the original quotient on $\Spec M$.
\end{remark}

\section[\texorpdfstring{$\H^\bullet(\MakeLowercase{i}_*\O^*_M)$}{Cohomology of pushforward} and \texorpdfstring{$\H^\bullet(\O^*)$}{sheaf of units} in the combinatorial topology]{\texorpdfstring{$\H^\bullet(\MakeLowercase{i}_*\O^*_M)$}{Cohomology of pushforward} and \texorpdfstring{$\H^\bullet(\O^*)$}{sheaf of units} in the combinatorial topology
    \sectionmark{\texorpdfstring{$\H^\bullet_{\comb}(\MakeLowercase{i}_*\O^*_M)$}{Cohomology of pushforward} and \texorpdfstring{$\H^\bullet_{\comb}(\O^*)$}{sheaf of units}}
}

\sectionmark{\texorpdfstring{$\H^\bullet_{\comb}(\MakeLowercase{i}_*\O^*_M)$}{Cohomology of pushforward} and \texorpdfstring{$\H^\bullet_{\comb}(\O^*)$}{sheaf of units}}

The idea behind what follows is to try to describe the local Picard group of a binoid $\KK$-algebra in terms of some combinatorial properties coming from the binoid itself.

\begin{proposition}\label{proposition:cohomology-pushforward-affine} For any sheaf of abelian groups $\F$ on $\Spec M$, the cohomology of the pushforward vanishes
    \[
    \H^j(\Spec \KK[M], i_*\F)=0
    \] for all $j\geq 1$.
\end{proposition}
\begin{proof}
    We use induction on $j\geq 1$. For $j=1$ this is true, for all sheaves of abelian groups, thanks to Theorem~\ref{theorem:H1pushforward_is_0}. Let $j\geq 1$. We can embed $\F$ in a flasque sheaf $\G$ on $\Spec M$ and build the exact sequence
    \[
    \begin{tikzcd}[cramped, row sep = 0pt]
    0\rar & \F\rar & \G\rar & \mathcal{Q}\rar & 0
    \end{tikzcd}
    \]
    on $\Spec M$, where $\mathcal{Q}=\faktor{\G}{\F}$. Then we pushforward this sequence along $i$ and, thanks to Theorem~\ref{theorem:push-forward-exact}, we get an exact sequence on $\Spec \KK[M]$
    \[
    \begin{tikzcd}[cramped, row sep = 0pt]
    0\rar & i_*\F\rar & i_*\G\rar & i_*\mathcal{Q}\rar & 0
    \end{tikzcd}
    \]
    that yields a long exact sequence in cohomology (we omit the topological space for ease of notation)
    \[
    \begin{tikzcd}[cramped, row sep = 2ex,
    /tikz/column 3/.append style={anchor=base west}]
    0\rar & \H^0(i_*\F)\rar & \H^0(i_*\G)\rar & \H^0(i_*\mathcal{Q})\arrow[out=-5, in=175, looseness=1.5, overlay, dll]\\
    & \H^1( i_*\F)\rar & \H^1(i_*\G)\rar & \H^1(i_*\mathcal{Q})\arrow[out=-5, in=175, looseness=1.5, overlay, dll]\\
    & \H^2( i_*\F)\rar & \dots
    \end{tikzcd}
    \]
    Thanks to \cite[Exercise II.1.16.(d)]{hartshorne1977algebraic}, we know that $i_*\G$ is again flasque, so
    \[
    \H^j(\Spec \KK[M],i_*\G)=0
    \]
    for all $j\geq 1$, and we get isomorphisms
    \[
    \H^j(i_*\mathcal{Q})\cong \H^{j+1}( i_*\F).
    \]
    By the inductive hypothesis, the left hand side is $0$, and so is the right hand side.
\end{proof}

\begin{remark}
    The previous result is somehow expected, in view of Theorem~\ref{theorem:vanishing-combinatorial-cohomology-affine}. Nevertheless, it is an interesting result that will find applications in computing the cohomology of the sheaf of units on $\Spec^\bullet\KK[M]$.
\end{remark}

\begin{remark}
    The cohomology of any sheaf $i_*\F$ on any open subset $U$ of $\Spec\KK[M]$ can be computed by \v{C}ech cohomology, using the affine combinatorial covering of $U$, that is the cover given by the fundamental open subsets $\{D(P)\}$, with $P$ monomials. This is true because $D(P)\cong \Spec\KK[M_P]$ and $(i_*\F)\restriction_{D(P)}=i_*(\F\restriction_{D(P)})$ and from Proposition~\ref{proposition:cohomology-pushforward-affine} this cover is acyclic.
\end{remark}

\begin{remark}
    The Proposition is true in particular for the cohomology of the sheaf $i_*\O^*_M$, that in general is a subsheaf of $\O^*_{\KK[M]}$, so we can compute the cohomology of $i_*\O^*$ on the punctured spectrum using the acyclic covering given by the coordinates $\{D(X_i)\}$.
\end{remark}

\begin{definition}
    $i_*\O^*_M$ is called the sheaf of \emph{combinatorial units} of $\KK[M]$.
\end{definition}

\begin{corollary}
    $\H^i(\Spec^\bullet\KK[M], i_*\O^*_M)\cong \H^i(\Spec^\bullet M, \O^*_M)$
\end{corollary}
\begin{proof}
    We can use $\{D(X_i)\}$ as a \v{C}ech covering on the left, and we obtain the same \v{C}ech cochain complex that we have on the right when covering $\Spec^\bullet M$ with $\{D(x_i)\}$.
\end{proof}

\begin{corollary}
    If we have $\O^*_{\KK[M]}\cong i_*\O^*_M\oplus\F$ for some sheaf of abelian groups $\F$ in the combinatorial topology then $\Pic^{\loc}(M)\neq 0$ implies $\Pic^{\loc}(\KK[M])\neq 0$.
\end{corollary}
\begin{proof}
    Since cohomology commutes with direct sums, we have that
    \[
    \H^j(\Spec^\bullet\KK[M], \O^*_{\KK[M]})=\H^j(\Spec^\bullet\KK[M], i_*\O^*_M)\oplus \H^j(\Spec^\bullet\KK[M], \F).
    \]
    From the previous Corollary,
    \[
    \H^j(\Spec^\bullet\KK[M], i_*\O^*_M)\cong \H^j(\Spec^\bullet M, \O^*_M)
    \]
    so, in particular,
    \begin{align*}
    \Pic^{\loc}(\KK[M])&=\H^1(\Spec^\bullet\KK[M], \O^*_{\KK[M]})\\
    &=\H^1(\Spec^\bullet\KK[M], i_*\O^*_M)\oplus \H^1(\Spec^\bullet\KK[M], \F)\\
    &=\Pic^{\loc}(M)\oplus \H^1(\Spec^\bullet\KK[M], \F)
    \end{align*}
    and if $\Pic^{\loc}(M)\neq 0$ then $\Pic^{\loc}(\KK[M])\neq 0$.
\end{proof}

\begin{remark}
    In general, $i_*\O^*_M\subsetneq \O^*_{\KK[M]}$. In particular, at least $\KK^*$ contributes to the units of the algebra, but not to the ones of the binoid. There are cases, though, in which the cohomology is the same, and we will see some examples of it in the next Chapter.
\end{remark}

\begin{remark}\label{remark:toric-case-sheaf-units}
    A special situation in which the cohomology of the sheaf of units of the binoid and the cohomology of the sheaf of units of the binoid algebra with the combinatorial topology agree, is the divisor class group in the normal toric case, and higher cohomology of the sheaf of units.
    
    Let $\Spec\KK[M]$ be a normal toric variety and let $Q=X_1\dots X_m$ be the product of the variables in the ring. Then $D(Q)$ is the open torus inside the variety. Clearly $\KK[M]_Q^*\cong \KK^*\oplus\ZZ^m$. Since any localization at $D(P)$ of a toric variety is again a toric variety for any $P$ monomial in $\KK[M]$, and these form a basis for the combinatorial topology, we can easily see that
    \[
    \O^*_{\KK[M]}\cong\KK^*\oplus i_*\O^*_M.
    \]
    The cohomology of degree bigger than 0 of $\KK^*$ vanishes because the ring is an integral domain, so we have that
    \[
    \H^j_{\comb}(\O^*_{\KK[M]})=\H^j_{\comb}(i_*\O^*_M)
    \]
    for all $j\geq 1$.   
    
    In more details, when we look at the divisor class group of $\KK[M]$, we can deduce the following sequence 
    from \cite[Proposition II.6.5]{hartshorne1977algebraic}
    \[
    \KK[M]^*_Q\stackrel{\phi}{\longrightarrow} \Div(V(Q))\longrightarrow\Cl(\KK[M])\longrightarrow \Cl(\KK[M]_Q)\longrightarrow 0
    \]
    where $\Div(V(Q))\cong\ZZ^n$, where $n$ is the number of facets of the cone, and $\Cl(\KK[M]_Q)=0$ because it is the torus. The map $\phi$ send the elements of the field to $0$ and between the $\ZZ$'s is the same map between divisors in Proposition~\ref{proposition:morphism-divisors}, so the groups are the same. More on this property can be found in \cite[Section 4.1]{cox2011toric}.
\end{remark}

Unlike for the toric case above and the Stanley Reisner ring that we will cover in the next Chapter, in general it is not true that $\O^*_{\KK[M]}\cong \KK^*\oplus i_*\O^*_M$, even when regarded as presheaves on the cover $\{D(X_j)\}$ of $\Spec^\bullet \KK[M]$, and so moreover in the combinatorial topology. The following are examples of this behaviour, because they involve nilpotents.

\begin{example}
    Consider the non cancellative and torsion binoid $M=(x, y \mid 2x=x+y, 2y=x+y)$, whose ring is $R=\faktor{\KK[X, Y]}{(X^2-XY, Y^2-XY)}$. The element $X-Y$ is nilpotent in $R$, since $(X-Y)^2=X^2-2XY+Y^2$, but does not come from a nilpotent element in $M$, since $M$ is reduced. So $1+X-Y$ is algebraically invertible but it is not the product of a combinatorially invertible element and a field unit, and this shows that the sheaf of units of the ring is not a direct sum of the sheaf of units of the binoid and the units of the field.
\end{example}

The next example show that this behaviour can be not only global, but also local on the punctured spectum.

\begin{example}
    Consider the ring $R=\faktor{\KK[X, Y, Z]}{(XY(Z^2-1))}$, coming from the torsion free but non-cancellative binoid $M=(x, y, z \mid x+y+2z=x+y)$.\\
    Then the $\KK$-spectrum of $M$ is the union of the four planes $X=0, Y=0$ and $Z=\pm1$, represented in Figure~\ref{fig:spec(xyz2-xy)} for $\KK=\RR$. The punctured spectrum $\Spec^\bullet \KK[M]$ is covered by $D(X)$ and $D(Y)$.
    
    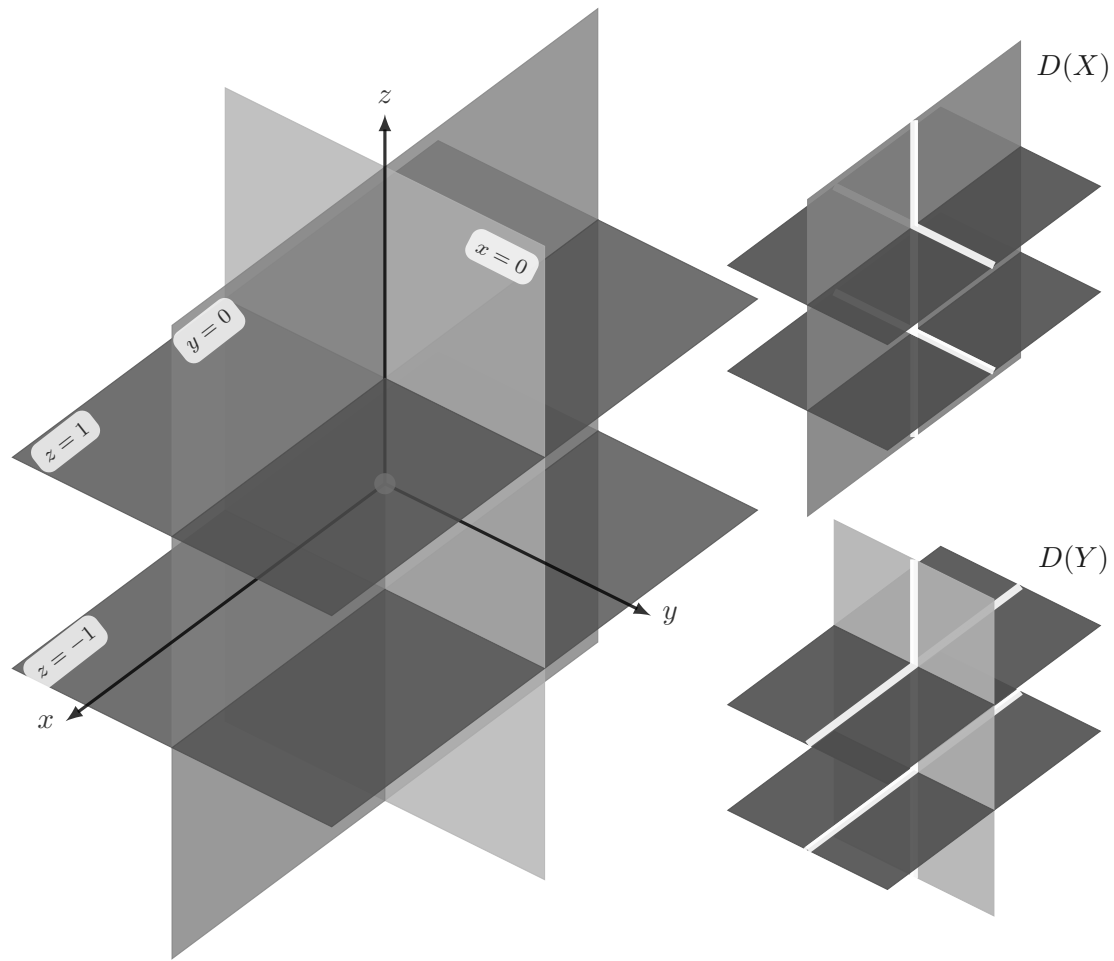
\begin{figure}[ht]
        \centering
        \begin{subfigure}[b]{0.64\linewidth}
            \begin{tikzpicture}[opacity=0.8, scale=0.7]
            \filldraw [grey1] (0,-4) -- (-3,-2.5) -- (1,0.5) -- (4, -1) -- cycle;
            \filldraw [grey2] (0,-4) -- (0,-8) -- (4, -5) -- (4, -1) -- cycle;
            \filldraw [grey2] (0,-4) -- (0,0) -- (4, 3) -- (4, -1) -- cycle;
            \filldraw [grey1] (0,-4) -- (3,-5.5) -- (7,-2.5) -- (4, -1) -- cycle;
            \filldraw [grey3] (0,-4) -- (0,-8) -- (-3, -6.5) -- (-3,-2.5) -- cycle;
            \filldraw [grey3] (0,-4) -- (0,0) -- (-3, 1.5) -- (-3,-2.5) -- cycle;
            \filldraw [grey3] (0,-4) -- (0,0) -- (3, -1.5) -- (3,-5.5) -- cycle;
            \filldraw [grey3] (0,-4) -- (0,-8) -- (3, -9.5) -- (3,-5.5) -- cycle;
            \filldraw [grey1] (0,-4) -- (-3,-2.5) -- (-7,-5.5) -- (-4, -7) -- cycle;
            \filldraw [grey2] (0,-4) -- (0,-8) -- (-4, -11) -- (-4, -7) -- cycle;
            \filldraw [grey2] (0,-4) -- (0,0) -- (-4, -3) -- (-4, -7) -- cycle;
            \filldraw [grey1] (0,-4) -- (3,-5.5) -- (-1,-8.5) -- (-4, -7) -- cycle;
            \filldraw [grey1] (0,0) -- (-3,1.5) -- (1,4.5) -- (4, 3) -- cycle;
            \filldraw [grey2] (0,0) -- (0,4) -- (4, 7) -- (4, 3) -- cycle;
            \filldraw [grey1] (0,0) -- (3,-1.5) -- (7,1.5) -- (4, 3) -- cycle;
            \filldraw [grey3] (0,0) -- (0,4) -- (-3, 5.5) -- (-3,1.5) -- cycle;
            \filldraw [grey3] (0,0) -- (0,4) -- (3, 2.5) -- (3,-1.5) -- cycle;
            \filldraw [grey1] (0,0) -- (-3,1.5) -- (-7,-1.5) -- (-4, -3) -- cycle;
            \filldraw [grey2] (0,0) -- (0,4) -- (-4, 1) -- (-4, -3) -- cycle;
            \draw[->, very thick, >=latex] (0,-2) -- (0,5) node [above] {$z$};
            \draw[->, very thick, >=latex] (0,-2) -- (-6,-6.5) node [left] {$x$};
            \draw[->, very thick, >=latex] (0,-2) -- (5,-4.5) node [right] {$y$};
            \fill[white] (0,-2) circle[radius=2mm];
            \filldraw [grey1] (0,0) -- (3,-1.5) -- (-1,-4.5) -- (-4, -3) -- cycle;
            \node[fill=white, rounded corners, rotate=38] (z=-1) at (-6, -5.2){\scriptsize$z=-1$};
            \node[fill=white, rounded corners, rotate=38] (z=1) at (-6, -1.2){\scriptsize$z=1$}; \node[fill=white, rounded corners, rotate=-27] (x=0) at (2.2, 2.3){\scriptsize$x=0$};
            \node[fill=white, rounded corners, rotate=38] (y=0) at (-3.3, 0.9){\scriptsize$y=0$};
            \end{tikzpicture}
        \end{subfigure}
        \begin{subfigure}[b]{0.35\linewidth}
            \begin{tikzpicture}[opacity=0.9, scale=0.35]
            \filldraw [grey1] (0,-4) -- (-3,-2.5) -- (1,0.5) -- (4, -1) -- cycle;
            \filldraw [grey2] (0,-4) -- (0,-8) -- (4, -5) -- (4, -1) -- cycle;
            \filldraw [grey2] (0,-4) -- (0,0) -- (4, 3) -- (4, -1) -- cycle;
            \filldraw [grey1] (0,-4) -- (3,-5.5) -- (7,-2.5) -- (4, -1) -- cycle;
            \filldraw [grey1] (0,-4) -- (-3,-2.5) -- (-7,-5.5) -- (-4, -7) -- cycle;
            \draw [white, line width=1mm] (-3,-2.5) -- (0,-4) -- (0,-8);
            \filldraw [grey2] (0,-4) -- (0,-8) -- (-4, -11) -- (-4, -7) -- cycle;
            \filldraw [grey2] (0,-4) -- (0,0) -- (-4, -3) -- (-4, -7) -- cycle;
            \draw [white, line width=1mm] (0,-8) -- (0,-4) -- (3,-5.5);
            \filldraw [grey1] (0,-4) -- (3,-5.5) -- (-1,-8.5) -- (-4, -7) -- cycle;
            \draw [white, line width=1mm] (3,-5.5) -- (0,-4) -- (0,0) -- (3, -1.5);
            \filldraw [grey1] (0,0) -- (-3,1.5) -- (1,4.5) -- (4, 3) -- cycle;
            \filldraw [grey2] (0,0) -- (0,4) -- (4, 7) -- (4, 3) -- cycle;
            \filldraw [grey1] (0,0) -- (3,-1.5) -- (7,1.5) -- (4, 3) -- cycle;
            \filldraw [grey1] (0,0) -- (-3,1.5) -- (-7,-1.5) -- (-4, -3) -- cycle;
            \draw [white, line width=1mm] (-3,-2.5) -- (0,-4) -- (0,0) -- (-3, 1.5);
            \draw [white, line width=1mm] (-3,1.5) -- (0,0) -- (0,4);
            \filldraw [grey2] (0,0) -- (0,4) -- (-4, 1) -- (-4, -3) -- cycle;
            \filldraw [grey1] (0,0) -- (3,-1.5) -- (-1,-4.5) -- (-4, -3) -- cycle;
            \draw [white, line width=1mm] (3,-1.5) -- (0,0) -- (0,4);
            \node (DX) at (6, 6){$D(X)$};
            \end{tikzpicture}\\
            \begin{tikzpicture}[opacity=0.9, scale=0.35]
            \filldraw [grey1] (0,-4) -- (-3,-2.5) -- (1,0.5) -- (4, -1) -- cycle;
            \draw [white, line width=1mm] (0,-4) -- (0,0) -- (4, 3) -- (4, -1) -- cycle;
            \filldraw [grey1] (0,-4) -- (3,-5.5) -- (7,-2.5) -- (4, -1) -- cycle;
            \filldraw [grey3] (0,-4) -- (0,-8) -- (-3, -6.5) -- (-3,-2.5) -- cycle;
            \filldraw [grey3] (0,-4) -- (0,0) -- (-3, 1.5) -- (-3,-2.5) -- cycle;
            \draw [white, line width=1mm] (0,-8) -- (0,-4) -- (4, -1);
            \filldraw [grey3] (0,-4) -- (0,0) -- (3, -1.5) -- (3,-5.5) -- cycle;
            \filldraw [grey3] (0,-4) -- (0,-8) -- (3, -9.5) -- (3,-5.5) -- cycle;
            \filldraw [grey1] (0,-4) -- (-3,-2.5) -- (-7,-5.5) -- (-4, -7) -- cycle;
            \draw [white, line width=1mm] (0,-4) -- (0,-8) -- (-4, -11) -- (-4, -7) -- cycle;
            \filldraw [grey1] (0,-4) -- (3,-5.5) -- (-1,-8.5) -- (-4, -7) -- cycle;
            \filldraw [grey1] (0,0) -- (-3,1.5) -- (1,4.5) -- (4, 3) -- cycle;
            \draw [white, line width=1mm] (-4, -7) -- (0,-4) -- (0,0);
            \filldraw [grey1] (0,0) -- (3,-1.5) -- (7,1.5) -- (4, 3) -- cycle;
            \filldraw [grey3] (0,0) -- (0,4) -- (-3, 5.5) -- (-3,1.5) -- cycle;
            \draw [white, line width=1mm] (4, 3) -- (0,0) -- (0,4);
            \filldraw [grey3] (0,0) -- (0,4) -- (3, 2.5) -- (3,-1.5) -- cycle;
            \filldraw [grey1] (0,0) -- (-3,1.5) -- (-7,-1.5) -- (-4, -3) -- cycle;
            \filldraw [grey1] (0,0) -- (3,-1.5) -- (-1,-4.5) -- (-4, -3) -- cycle;
            \draw [white, line width=1mm] (-4, -3) -- (0,0) -- (0,4);
            \node (DY) at (6, 4){$D(Y)$};
            \end{tikzpicture}
        \end{subfigure}
        \caption{Visual representation of $\RR-\Spec^\bullet(M)$ and its combinatorial covering}
        \label{fig:spec(xyz2-xy)}
    \end{figure}
    
    On these open subsets we can use a Mayer-Vietoris-like argument (as in the case of the Stanley-Reisner ring, developed in Chapter~\ref{chapter:SR-rings}) to prove that this covering is acyclic for the sheaf of units $\O^*_R$, thus we can use them as \v{C}ech covering to compute cohomology.
    \begin{align*}
    &\O(D(X))=R_X=\faktor{\KK[X, X^{-1}, Y, Z]}{(Y(Z^2-1))}\cong\faktor{\KK[Y, Z]}{(Y(Z^2-1))}[X, X^{-1}]\\
    &\O(D(Y))=R_Y=\faktor{\KK[X, Y, Y^{-1}, Z]}{(X(Z^2-1))}\cong\faktor{\KK[X, Z]}{(X(Z^2-1))}[Y, Y^{-1}]\\
    &\O(D(XY))=R_{XY}=\faktor{\KK[X, X^{-1}, Y, Y^{-1}, Z]}{(Z^2-1)}\\
    &\hspace{10em}\cong\faktor{\KK[Z]}{(Z^2-1)}[X, X^{-1}, Y, Y^{-1}]
    \end{align*}
    The groups of units are
    \begin{align*}
    &R_{X}^*=\{\alpha x^m \mid \alpha\in\KK^*, m\in\ZZ\}\cong \KK^*\oplus\ZZ\\
    &R_{Y}^*=\{\beta y^n \mid \alpha\in\KK^*, n\in\ZZ\}\cong\KK^*\oplus\ZZ\\
    &R_{XY}^*=\{\gamma x^m y^n (\delta z+\varepsilon)\mid \gamma\in\KK^*,\ m, n\in\ZZ,\ \delta, \varepsilon\in\KK,\ \delta^2\neq \varepsilon^2\}\ncong\KK^*\oplus\ZZ^2
    \end{align*}
    In $R_{XY}$ also the elements of the form $(\delta z+\varepsilon)$ are invertible, provided $\delta^2\neq \varepsilon^2$. These units are not combinatorial, they show up only at this level ant do not come from earlier in the \v{C}ech complex, so they will add to the combinatorial Picard group.
    
    Indeed, the \v{C}ech complex looks like
    \begin{equation}
    \begin{tikzcd}[baseline=(current  bounding  box.center),
    row sep = 0pt]
    0\rar & R_X^*\oplus R_Y^*\rar & R_{XY}^* \rar & 0\\
    &    (\alpha x^m, \beta y^n)\rar[mapsto] &\dfrac{\beta}{\alpha}x^{-m}y^{n}&\\
    &                            &\gamma x^my^n (\delta z+\varepsilon)\rar[mapsto] & 0
    \end{tikzcd}
    \end{equation}
    And the Picard group of the punctured spectrum (or local Picard group of $R$) is then exactly the group of these ``new'' invertible elements $\Pic(\Spec^\bullet(R))=\{(\delta z+\varepsilon)\mid \delta, \varepsilon\in\KK,\ \delta^2\neq \varepsilon^2\}$.
\end{example}

The previous Example shows that the combinatorial topology, although it is coarser than the Zariski topology and more similar to the topology on the binoid spectrum, does not automatically give us an isomorphism of the cohomologies.

\section{Cohomology of \texorpdfstring{$\O^*$}{pippo} in the Zariski topology}
In this Section we will present some known results about cohomology of $\O^*$ in the Zariski topology and try to relate them with the combinatorial topology.
We begin with a result on the vanishing of $\O^*$ for the affine space.

\subsection{The affine space}

\begin{notation}Let $R$ be a commutative ring and $\F$ a sheaf on $\Spec R$. We ease the notation by denoting
    $\H^i(R, \F):=\H^i(\Spec(R), \F)$ in the Zariski topology. If, moreover, $R=\KK[M]$ for some binoid $M$, we denote also
    $\H^i(\Spec^\bullet R, \F):=\H^i(\Spec(R)\smallsetminus \KK[M_+], \F\restriction_{\Spec R\smallsetminus \KK[M_+]})$.
\end{notation}

\begin{lemma}\label{Lemma:aciclicityforaffinespace}
    Let $V=\AA^n$ or $V=\AA^n\times(\AA^*)^m$. Then $\H^i(V, \O^*)=0$ for $i\geq 1$.
\end{lemma}
\begin{proof}
    Recall that the spaces that we are considering are $\AA^n=\Spec(\KK[x_1, \dots, x_n])$ and $\AA^n\times\AA^*=\Spec(\KK[x_1, \dots, x_n][y_1^{\pm1}, \dots, y_m^{\pm1}])$. Since the two schemes are integral, the sheaf $\mathcal{K}^*$ of non-zero rational functions is constant, and hence flasque, in both cases. This constant sheaf fits into the exact sequence     
    \begin{equation*}
    \begin{tikzcd}[baseline=(current  bounding  box.center), cramped, row sep =0pt]
    0 \rar & \O^* \rar & \mathcal{K}^* \rar & \faktor{\mathcal{K}^*}{\O^*} \rar & 0
    \end{tikzcd}
    \end{equation*}
    The sheaf of Cartier divisors is exactly the last sheaf
    \[
    \glslink{cadivX}{\mathcal{C}\mathrm{aDiv}}=\faktor{\mathcal{K}^*}{\O^*}
    \]
    and thanks to smoothness and local factoriality Cartier and Weil divisors agree, as well as their sheafifications, $\mathcal{C}\mathrm{aCl}$ and $\glslink{wdivX}{\mathcal{W}\mathrm{Div}}$ (see \cite[Section II.6]{hartshorne1977algebraic}).
    So $\faktor{\mathcal{K}^*}{\O^*}$ is flasque because $\mathcal{W}\mathrm{Div}$ is such, thanks to integrality (in this case every Weil divisor can be extended).
    \\
    When taking cohomology of this short exact sequence, we obtain the long exact sequence
    \begin{equation*}
    \begin{tikzcd}[baseline=(current  bounding  box.center), row sep=0pt,
    tikz/column 4/.append style={anchor=west}]
    0 \rar &  \H^0(\O^*) \rar &  \H^0(\mathcal{K}^*) \rar &  \displaystyle \H^0\left(\faktor{\mathcal{K}^*}{\O^*}\right)=\Div(R)\arrow[out=-5, in=175, looseness=1.5, overlay, dll]\\
    & \H^1(\O^*) \rar &  \H^1(\mathcal{K}^*) \rar &  \dots\hspace{4.5em}
    \end{tikzcd}
    \end{equation*}
    Thanks to flasqueness we have that $\H^i(\mathcal{K}^*) = \displaystyle\H^i\left(\faktor{\mathcal{K}^*}{\O^*}\right) = 0$ for every $i\geq 1$, so $\H^i(\O^*)=0$ for every $i\geq 2$. We still have to prove it for $i=1$.
    \\
    Both $\AA^n$ and $\AA^n\times(\AA^*)^m$ come from Unique Factorization Domains, so thanks to \cite[Proposition II.6.2 and Corollary II.6.16]{hartshorne1977algebraic} we have that
    \[
    \H^1(\O^*)\cong\faktor{\Div(R)}{\KK^*}=\Cl(R)=0
    \] and the proof is complete.
\end{proof}

\begin{lemma}\label{Lemma:aciclicityaffinevariety-overdim}
    Let $V$ be a variety. Then $\H^i(V, \O^*)=0$ for $i>\dim(V)$.
\end{lemma}
This is a well-known result by Grothendieck~\cite[Theorem 3.6.5]{grothendieck1957surquelques}, already cited in Remark~\ref{remark:grothendieck-theorem}.

\begin{remark}
    If $V$ is $W\times(\AA^*)^m$, then it is still affine because
    \[
    (\AA^*)^m=\Spec\KK[y_1^{\pm1},\dots, y_m^{\pm1}],
    \] and the result holds then for $i>\dim(W)+m$.
\end{remark}

\subsection{Toric varieties}
We state now two important results about the cohomology of $\O^*$ in the case of a toric variety, both presented in \cite{demeyer1992cohomological}. There, the authors use them to prove some results that relate usual sheaf cohomology in the Zariski topology and \'etale cohomology on toric varieties.

\begin{theorem}[{\cite[Lemma 5]{demeyer1992cohomological}}]
    1. If $X = U_\sigma$ is a normal affine toric variety, then
    \[
    \H^p_{\mathrm{Zar}}(X, \O^*) = 0,
    \] for all $p\geq 1$.\\
    2. If $X_\Sigma$ is the toric variety associated to the fan $\Sigma$ and $\U=\{U_{\sigma_j}\}_{j=1, \dots, r}$ is the usual covering of $X$ with affine open subsets associated to the maximal cones $\{\sigma_1, \dots, \sigma_r\}$ in $\Sigma$, then
    \[
    \vH^p(\U, \O^*)=\H_{\mathrm{Zar}}(X, \O^*).
    \]
\end{theorem}

\begin{remark}
    From \cite[Theorem 3.1.5]{cox2011toric}, if $\sigma$ is a maximal cone in $\Sigma$ we know that $U_\sigma$ is normal and separated.
\end{remark}

Let $\Spec\KK[M]=X_\Sigma$ be a normal affine separated toric variety and let $X=\Spec^\bullet \KK[M]$. $D(X_i)$ is a normal affine separated toric variety embedded in $X$, so $\{D(X_i)\}$ defines an acyclic covering for $\O^*$ on $X$. In particular, we can use it to compute Zariski sheaf cohomology via \v{C}ech cohomology on $X$. Moreover, since the $D(X_i)$'s are combinatorial, we can apply Proposition~\ref{proposition:hzar-hcomb} to obtain isomorphisms
\[
\H^p_{\mathrm{Zar}}(X, \O^*) = \H^p_{\comb}(X, \O^*) = \vH^p(\{D(X_i)\}, \O^*)
\]
Moreover, we can then apply Proposition~\ref{proposition:split-sheaves-comb-top} to split the sheaf $(\O^*)^{\comb}$ and obtain
\begin{align*}
\H^p_{\mathrm{Zar}}(X, \O^*) &= \H^p_{\comb}(X, \O^*) = \H^p_{\comb}(X, \KK^*)\oplus \H^p_{\comb}(X, i_*\O^*_M)\\
&= \vH^p(\{D(X_i)\}, \KK^*)\oplus \vH^p(\{D(X_i)\}, i_*\O^*_M)\\
& = \vH^p(\{D(X_i)\}, \KK^*)\oplus \H^p(\Spec^\bullet M, \O^*_M).
\end{align*}

\begin{remark} The first part of the theorem above show really that, when $\Spec\KK[M]$ is a normal separated toric variety,
    \[
    \H^p(\Spec\KK[M], \O^*)=\H^p(\Spec M, \O^*_M)=0,
    \]
    for all $p\geq 1$.
    If we drop the hypothesis of normality, this might be not true, as the following Example shows.
\end{remark}

\begin{example}
    Consider the Neil binoid $M=(x, y\mid 2x=3y)$. Its algebra is $\faktor{\KK[X, Y]}{\langle X^2-Y^3\rangle}$ and it defines the curve called Neil parabola, i.e.\
    \[
    C=\Spec\faktor{\KK[X, Y]}{\langle X^2-Y^3\rangle},
    \]
    that is a toric variety.\footnote{This curve is also known as semicubical curve, or simply cusp.}
    
    We already know, from Theorem~\ref{theorem:vanishing-combinatorial-cohomology-affine}, that $\H^p(\Spec M, \O^*_M)=0$ for all $p\geq 1$, because it is affine.
    
    On the other hand, this variety is not normal and, indeed
    \[
    \Pic(C)=\KK_+,
    \]
    where $\KK_+$ is $\KK$ seen as the additive group.\footnote{More on this result can be found in \cite[Exercise 11.15 and 11.16]{eisenbud2013commutative}.}
    
    So, in this case, we have that
    \[
    \KK_+=\H^1(\Spec\KK[M], \O^*)\neq \H^1(\Spec M, \O^*_M)=0.
    \]

    In particular, this shows that if we have a variety $X$ such that one of its $D(X_i)$'s (or intersection of them) is isomorphic to a product of the cusp with something else, we cannot use this covering to compute cohomology. For example, we cannot use $\{D(X_i)\}$ to compute cohomology of the sheaf of units on 
    \[
    V=\Spec^\bullet \faktor{\KK[X, Y, Z]}{\langle X^2-Y^3\rangle}.
    \]
    Indeed, $D(Z)\cong C\times \AA^*$, so $\H^1(D(Z), \O^*)\supseteq \KK_+\neq 0$, and the combinatorial covering $\{D(X), D(Y), D(Z)\}$ is not acyclic for $\O^*$.
\end{example}

\section{Non reduced case}
In the next Chapter we will discuss in detail the case of the Stanley-Reisner rings, reduced algebras that are quotients of polynomial rings by radical monomial ideals. In this Section we discuss the relations between the behaviour of the cohomology of the sheaf of units in the Zariski topology in the reduced and in the non reduced case.

In particular, one might expect some differences, because any nilpotent element $n\in R=\KK[M]$ defines a unit $1+n$, that is obviously non combinatoric (in the sense that it does not come from a unit is $M$), and so they make the computations more difficult (both in Zariski and combinatorial topology). We will see some examples where this behaviour is really explicit.

In what follows, unless otherwise specified, we consider $R$ to be a noetherian ring, not necessarily a binoid algebra, and $X=\Spec R=\Spec R_{\red}$.

\begin{definition}
    Let $R$ be any ring. $\nil(R)$ is the nilradical of $R$.
\end{definition}

The exact sequence
\begin{equation*}
\begin{tikzcd}[baseline=(current  bounding  box.center),
/tikz/column 1/.append style={anchor=base east},
/tikz/column 2/.append style={anchor=base west}, 
row sep = 0pt]
0\rar &\nil(R) \rar& R\rar &R_{\red}\rar&0
\end{tikzcd}
\end{equation*}

induces the exact sequence on the units
\begin{equation*}
\begin{tikzcd}[baseline=(current  bounding  box.center),
/tikz/column 1/.append style={anchor=base east},
/tikz/column 2/.append style={anchor=base west}, 
row sep = 0pt]
1\rar &1+\nil(R) \rar& R^*\rar &R^*_{\red}\rar&0,
\end{tikzcd}
\end{equation*}
where $1+\nil(R)$ is the group of units of $R$ of the form $1+n$ with $n\in\nil(R)$.

\begin{definition}
    The \emph{sheaf of nilpotents} is the sheaf of ideals $\N$ on $\Spec R$ such that
    \[
    \N(U)=\nil(\O_X(U)).
    \]
\end{definition}

\begin{proposition}
    $\N$ is a sheaf of ideals that fits into a short exact sequence
    \begin{equation*}
    \begin{tikzcd}[baseline=(current  bounding  box.center),
    /tikz/column 1/.append style={anchor=base east},
    /tikz/column 2/.append style={anchor=base west}, 
    row sep = 0pt]
    0\rar &\N \rar& \O_X\rar &\O_{X_{\red}}\rar&0.
    \end{tikzcd}
    \end{equation*}
\end{proposition}

\begin{definition}
    Let $\N$ be the sheaf of nilpotents. We denote by $1+\N$ the sheaf $(1+\N)(U)=1+\nil(\O_X(U))$ of the \emph{unipotent elements}.\footnote{Recall that an element $P\in R$ is unipotent if $P-1$ is nilpotent, see for example \cite[Section 6]{borel1956Groupes}.}
\end{definition}

\begin{proposition}\label{proposition:exact-sequence-units-1+nilpotents-reduction}
    $1+\N$ is a sheaf of abelian groups that fits into a short exact sequence
    \begin{equation*}
    \begin{tikzcd}[baseline=(current  bounding  box.center),
    /tikz/column 1/.append style={anchor=base east},
    /tikz/column 2/.append style={anchor=base west}, 
    row sep = 0pt]
    1\rar &1+\N \rar& \O^*_X\rar &\O^*_{X_{\red}}\rar&1
    \end{tikzcd}
    \end{equation*}
\end{proposition}
\begin{proof}
    Call $\pi_{\red}$ the map on the right. It is surjective because it is induced by the corresponding surjective map $\O_X\longrightarrow\O_{X_{\red}}$, where a unit on the right has a preimage in $R$, what is again a unit.
    
    Every section of $1+\N$ is a section of $\ker(\pi_{\red})$, so $1+\N\hookrightarrow\ker(\pi_{\red})$. Moreover, this inclusion is an isomorphism because, for every open subset $U$ of $X$, the map $\O_X(U)\rightarrow\O_{X_{\red}}(U)$ kills only the nilpotents, so there is a reversed inclusion $\ker(\pi_{\red})\hookrightarrow1+\N$.
\end{proof}

\begin{lemma}\label{lemma:1+N^m-1-coherent}
    Let $\nil(R)$ be the nilradical of a noetherian ring $R$. Then there exist $m\in \NN$ such that $\nil(R)^n=0$ for any $n\geq m$ and an isomorphism of sheaves of abelian groups
    \begin{equation*}
    \begin{tikzcd}[baseline=(current  bounding  box.center),
    /tikz/column 1/.append style={anchor=base east},
    /tikz/column 2/.append style={anchor=base west}, 
    row sep = 0pt]
    \N^{m-1} \rar["\sim"]& 1+\N^{m-1}
    \end{tikzcd}
    \end{equation*}
\end{lemma}
\begin{proof} The existence of $m$ that annihilates $\N$ is trivial, because for each section there exists a finite $m$ that annihilates it and our ring is noetherian, so we can just look at the generators of the nilradical. For the second claim, there exists an isomorphism of groups%
    \begin{equation*}
    \begin{tikzcd}[baseline=(current  bounding  box.center),
    /tikz/column 1/.append style={anchor=base east},
    /tikz/column 2/.append style={anchor=base west}, 
    row sep = 0pt]
    \nil(R)^{m-1} \rar["\sim"]& 1+\nil(R)^{m-1}\\
    f\rar[mapsto]& (1+f)
    \end{tikzcd}
    \end{equation*}%
    where the operation on the left is the sum and on the right is the multiplication. Indeed, for any two elements in $1+\nil(R)$ we have that $(1+f)(1+g)=(1+f+g+fg)$ but the latter element is $0$ because $fg\in\nil(R)^m=0$.\\
    This induces the wanted isomorphism of the sheaves $1+\N^{m-1}\cong \N^{m-1}$.
\end{proof}

\begin{remark}
    As sheaves of set $\N^k\cong 1+\N^k$ for all $k\geq 1$, via the map above. However, for $k\neq m-1$ this is not a map of sheaves of abelian groups, unless $\N^2=0$, because this is the only case in which $f\cdot g$ would be zero for all $f, g$ sections of $\N$.
\end{remark}

\begin{theorem}\label{theorem:1+N-acyclic}
    $\H^i(X, 1+\N)=0$ for all $i\geq 1$ and $X=\Spec R$.
\end{theorem}
\begin{proof}
    The first step is to observe again that
    \[
    \Spec R\cong \Spec\faktor{R}{\nil(R)^k}
    \]
    for all $k\geq 1$. So, in particular, for any sheaf it is the same to compute the cohomology on the space on left or on the space on the right.
    The surjective map
    \[
    \begin{tikzcd}[baseline=(current  bounding  box.center),
    /tikz/column 1/.append style={anchor=base east},
    /tikz/column 2/.append style={anchor=base west}, 
    row sep = 0pt]
    \O_R \rar["\pi"]& \O_{\faktor{R}{\nil(R)^{k+1}}}
    \end{tikzcd}
    \]
    has kernel exactly $\N^{k+1}=\widetilde{\nil(R)^{k+1}}$. Let $\mathcal{M}$ be the sheaf of nilpotents of $\O_{\faktor{R}{\nil(R)^{k+1}}}$. Clearly $\mathcal{M}$ is the image of $\N$ via this map. Thanks to the ideal properties, moreover, $\mathcal{M}^k$ is the image of $\N^k$, so $\mathcal{M}^k\cong \faktor{\N^k}{\N^{k+1}}$. The induced map between the groups of units
    \[
    \begin{tikzcd}[baseline=(current  bounding  box.center),
    /tikz/column 1/.append style={anchor=base east},
    /tikz/column 2/.append style={anchor=base west}, 
    row sep = 0pt]
    \O^*_R \rar& \O^*_{\faktor{R}{\nil(R)^{k+1}}}
    \end{tikzcd}
    \]
    has kernel $1+\N^{k+1}$, so this induces an injective map
    \[
    \begin{tikzcd}[baseline=(current  bounding  box.center),
    /tikz/column 1/.append style={anchor=base east},
    /tikz/column 2/.append style={anchor=base west}, 
    row sep = 0pt]
    \faktor{1+\N^k}{1+\N^{k+1}} \rar[hook]& \O^*_{\faktor{R}{\nil(R)^{k+1}}}.
    \end{tikzcd}
    \]
    On the other side, there is a trivial injection
    \[
    \begin{tikzcd}[baseline=(current  bounding  box.center),
    /tikz/column 1/.append style={anchor=base east},
    /tikz/column 2/.append style={anchor=base west}, 
    row sep = 0pt]
    1+\mathcal{M}^k \rar[hook]& \O^*_{\faktor{R}{\nil(R)^{k+1}}}.
    \end{tikzcd}
    \]
    Since $\pi(1+\N^k)\subseteq 1+\mathcal{M}^k$ and has kernel exactly $1+\N^{k+1}$, we get an injective map
    \[
    \begin{tikzcd}[baseline=(current  bounding  box.center),
    /tikz/column 1/.append style={anchor=base east},
    /tikz/column 2/.append style={anchor=base west}, 
    row sep = 0pt]
    \faktor{1+\N^k}{1+\N^{k+1}}\rar[hook] & 1+\mathcal{M}^k
    \end{tikzcd}
    \]
    that fits in the diagram
    \[
    \begin{tikzcd}[baseline=(current  bounding  box.center),
    /tikz/column 1/.append style={anchor=base east},
    /tikz/column 2/.append style={anchor=base west}]
    \faktor{1+\N^k}{1+\N^{k+1}} \arrow[hook]{rr}\drar[hook]&& \O^*_{\faktor{R}{\nil(R)^{k+1}}}\\
    &1+\mathcal{M}^k \urar[hook]&
    \end{tikzcd}
    \]
    
    This map is also surjective, since $\N^k$ maps surjectively on $\mathcal{M}^k$. Clearly $\mathcal{M}^{k+1}=0$, so by Lemma~\ref{lemma:1+N^m-1-coherent} we obtain that $1+\mathcal{M}^k\cong\mathcal{M}^k$. In particular then, for any $i\geq1$,
    \[
    \H^i\left(X, \faktor{1+\N^k}{1+\N^{k+1}}\right)\cong\H^i(X, 1+\mathcal{M}^k)\cong\H^i(X, \mathcal{M}^k)=0
    \]
    the latter because $\mathcal{M}^k$ is coherent on $\Spec\faktor{R}{\nil(R)^{k+1}}$, but we observed at the beginning that cohomology there and on $X$ agree.
    
    Consider now the short exact sequence of sheaves on $X$
    \[
    \begin{tikzcd}[cramped, row sep = 0pt]
    1\rar &  1+\N^{k+1} \rar &  1+\N^k \rar & \faktor{1+\N^k}{1+\N^{k+1}}\rar & 1.
    \end{tikzcd}
    \]
    We just proved that the cohomology of degree at least $1$ of the latter is trivial, so we have that
    \[
    \H^i(X, 1+\N^{k+1})\cong \H^i(X, 1+\N^{k})
    \]
    for all $i\geq 2$ and for all $k\geq 1$. If $i=1$ we get a surjective map
    \[
    \begin{tikzcd}[cramped, row sep = 0pt]
    \rar&  \H^1(1+\N^{k+1}) \rar &  \H^1(1+\N^k) \rar & 0.
    \end{tikzcd}
    \]
    Now we apply Lemma~\ref{lemma:1+N^m-1-coherent}, and for the $m$ such that $\N^m=0$, we obtain that its cohomology vanishes.
    We can then easily apply a descending induction on the power of this sheaf of ideals to get that
    \[
    \H^i(X, 1+\N^k)=0
    \]
    for all $i\geq 1$ and for all $k\geq 1$.
\end{proof}

\begin{remark}
    Another proof of the previous theorem can be deduced in characteristic 0 from \cite[Lemma 8.43]{bruns2009polytopes}, because in this situation $\nil(R)\cong 1+\nil(R)$ via exponentiation and logarithm, so the sheaf $1+\N$ becomes isomorphic to a sheaf of ideals, that is a coherent sheaf, thus proving that higher cohomology vanishes. It is worth noting that we do not prove that the sheaf is coherent in other characteristics, but just that cohomology in higher degrees vanishes.
\end{remark}

\begin{remark}
    From Proposition~\ref{proposition:exact-sequence-units-1+nilpotents-reduction} and the previous Theorem we obtain that, on an affine scheme $X$,
    \[
    \H^i(X, \O^*_X)\cong\H^i(X, \O^*_{X_{\red}}).
    \]
    In some cases, like when the reduction is a normal toric variety, or when we consider a multiple affine space, this is zero.
    
    This is not true, obviously, in non-affine situations, like general open subsets or the punctured spectrum that we will investigate further in the next Chapter.
\end{remark}

\begin{corollary}\label{corollary:combinatorial-covering-non-reduced}
    Let $M$ be a binoid. Assume that $\O^*_{\KK[M]_{\red}}$ is acyclic on the covering $\{D(X_j)\}$ of $\Spec^\bullet\KK[M]$. Then we can use the combinatorial covering $\{D(X_j)\}$ to compute the cohomology of the non reduced sheaf $\O^*_{\KK[M]}$ on the punctured spectrum $\Spec^\bullet\KK[M]$ through \v{C}ech cohomology on the same covering.
\end{corollary}

\begin{remark}
    In the next Chapter we compute explicitly the cohomology of $\O^*$ in the Stanley-Reisner case. At first we prove that $D(X_j)$ is an acyclic covering for this sheaf. The previous result implies that this covering is also acyclic for any other monomial ideal, since the radical of a monomial ideal is the Stanley-Reisner ideal of some simplicial complex $\triangle$, and
    \[
    \left(\faktor{\KK[X_1, \dots, X_n]}{I}\right)_{\red}\cong\faktor{\KK[X_1, \dots, X_n]}{\sqrt{I}}.
    \]
\end{remark}

\afterpage{\null\newpage}

\chapter{Stanley-Reisner rings}\label{chapter:SR-rings}
In this Chapter, we address the problem of computing the local Picard group of a Stanley-Reisner ring, through the study of the cohomology of the sheaf of units $\O^*$ in the Zariski topology. While doing so, we will also look at cohomology of higher degrees of the sheaf of units, and we will give some combinatorial description of this cohomology, in a fashion similar to what we did in Chapter~\ref{chapter:simplcial-binoids}.
In order to describe this cohomology group, we first prove that $\H^i(\KK[\triangle], \O^*)=0$ and $\H^i(\KK[\triangle][x, x^{-1}], \O^*)=0$ for $i\geq 1$. We use then this result in a Corollary to give a combinatorial \v{C}ech covering to compute cohomology of $\O^*$ on $\Spec^\bullet(\KK[\triangle])$ and this will ultimately allow us to give explicit formulas similar to the ones in the combinatorial case.
Lastly, we will look at the non-reduced monomial case and we will apply results from the last section of Chapter~\ref{chapter:injections} in order to get some explicit results.

\section{The Spectrum of a Stanley-Reisner ring}
\subsection{Irreducible components}
Let $\triangle$ denote a simplicial complex on the finite set $V$ of vertices.

\begin{lemma}\label{Lemma:correspondencefacessubsetsofspec}
    There is a bijective correspondence between the faces of the complex $\triangle$ and the reduced, irreducible linear coordinate subspaces contained in $\Spec \glslink{SRring}{\KK[\triangle]}$. This correspondence is dimension-preserving, in the sense that the dimension of the irreducible linear coordinate subspace is the dimension of the face associated to it plus one.
\end{lemma}
\begin{proof}
    Faces of the simplicial complex correspond bijectively to prime ideals of the simplicial binoid via the identification stated before $F\mapsto\p:=\langle V\smallsetminus F\rangle$ (in Remark~\ref{remark:boettger:correspondence-faces-prime-ideals}).
    Prime ideals of the binoid injects to the prime ideals of the algebra via $\p\mapsto \KK[\p]$, where $\KK[\p]$ is the ideal generated by the elements corresponding to the generators of $\p$, as stated in Lemma~\ref{lemma:prime-in-M-iff-prime-in-KM}.
    Clearly $\KK[\p]$ defines a reduced and irreducible linear coordinate subspace $\mathcal{V}(\p)$, being just generated by a subset of the variables.
    
    On the other side, consider a reduced and irreducible linear coordinate subspace of $\Spec(\KK[\triangle])$. Its defining ideal is prime and generated by some of the variables. We can map it back to a prime ideal in the simplicial binoid and, thanks again to the map of Remark~\ref{remark:boettger:correspondence-faces-prime-ideals}, we get a face $F\in\triangle$.
    
    The dimension-preserving part is trivial, since such a subspace of $\AA^n$ of dimension $d$ is the zero locus of a prime ideal generated by $n-d$ variables.
\end{proof}

\begin{lemma}\label{Lemma:correspondencefacetsirreduciblecomponents}
    The irreducible components of $\Spec(\KK[\triangle])$ correspond bijectively to the facets of $\triangle$. Indeed, this spectrum is a union of irreducible coordinate linear subspaces defined by the facets in the correspondence of Lemma~\ref{Lemma:correspondencefacessubsetsofspec}.
\end{lemma}

\begin{lemma}\label{Lemma:descriptionofcomponentoffacet}
    Let $F$ be a facet of $\triangle$. The irreducible component defined by $F$ is $\Spec(\KK[\glslink{closure-face}{\mathcal{P}(F)}])$, where $\mathcal{P}(F)$ is the power set of $F$, i.e.\ the $\subseteq$-closure of $F$ in $\triangle$. In particular, since $\mathcal{P}(F)$ is a simplex, this irreducible component is an affine space $\AA^{\#F}$.
\end{lemma}
\begin{proof}
    This follows from Lemma~\ref{Lemma:correspondencefacessubsetsofspec} applied to $\mathcal{P}(F)$.
\end{proof}

\begin{lemma}\label{Lemma:decompositionofspectrum}
    Let $F$ be a facet of $\triangle$ and $\triangle':=\triangle\smallsetminus\{F\}$. Then \[\Spec(\KK[\triangle])=\Spec(\KK[\triangle'])\cup \Spec(\KK[\mathcal{P}(F)]).\]
\end{lemma}
\begin{proof}
    This follows from Lemma~\ref{Lemma:correspondencefacetsirreduciblecomponents} and Lemma~\ref{Lemma:descriptionofcomponentoffacet}.
\end{proof}

\begin{lemma}\label{Lemma:intersectionofspecs}
    Let $F$ be a facet of $\triangle$, $\triangle'=\triangle\smallsetminus\{F\}$ and $\triangle''=\triangle'\cap\mathcal{P}(F)$. Then \[\Spec(\KK[\triangle''])=\Spec(\KK[\triangle'])\cap\Spec(\KK[\mathcal{P}(F)]).\]
\end{lemma}
\begin{proof}
    This follows from the correspondence of faces and linear subspaces, Lemma~\ref{Lemma:correspondencefacetsirreduciblecomponents} and Lemma~\ref{Lemma:descriptionofcomponentoffacet}.
\end{proof}

\begin{lemma}\label{Lemma:dimensionintersection}
    Under the hypothesis of the previous Lemma,
    \[
    \dim\left(\Spec\left(\KK[\triangle'']\right)\right)\leq\min\left\{\dim\left(\Spec\left(\KK[\triangle']\right)\right), \dim\left(\Spec\left(\KK[\mathcal{P}(F)]\right)\right)\right\}.
    \]
\end{lemma}
\begin{proof}
    Thanks to Lemma~\ref{Lemma:correspondencefacessubsetsofspec}, this is equivalent to $\dim(\triangle'')\leq\min\{\dim(\triangle'), \dim(F)\}$. This is trivial since $F$ is a facet of $\triangle$.
\end{proof}

\begin{corollary}
    Under the same hypothesis as above, equality holds if and only if $F$ is the unique facet of maximal dimension.
\end{corollary}
\begin{proof}
    $\Longleftarrow)$ If $F$ is the unique facet of maximal dimension, then $\left(\mathcal{P}(F)\smallsetminus \{F\}\right)\subseteq\triangle'$, $\dim(\triangle'')=\dim(\triangle')=\dim(F)-1$.
    \\
    $\Longrightarrow)$ We prove the contropositive. If $F$ is a facet but not of maximal dimension, then $\dim(\triangle')=\dim(\triangle)>\dim(F)$ and $\triangle'\cap \mathcal{P}(F)\subseteq(\mathcal{P}(F)\smallsetminus \{F\})$, so we have the following inequalities of the dimensions
    \[
    \dim(\triangle'')\leq\dim(\mathcal{P}(F)\smallsetminus\{F\})=\dim(F)-1<\dim(F)<\dim(\triangle')
    \]
    Assume now that $F$ is not unique, so there exists another facet $G$ of maximal dimension in $\triangle$. Then $G\in\triangle'$, $\dim(\triangle')=\dim(F)$, and $\dim(\triangle'')=\dim(F)-1$, so again equality does not hold.
\end{proof}

\begin{example}\label{example:spectrum-simplicial-ring-1}
    Let us go back to our favourite Example~\ref{example:spectrum-simplicial-binoid-8} 
    
    \begin{minipage}[c]{0.5\textwidth}
        \[
        \triangle = \left\{\begin{aligned}
        &\varnothing,\{1\},\{2\},\{3\},\{4\},\\
        &\{1, 2\}, \{1, 3\}, \{2, 3\}, \{3, 4\}\\
        &\{1, 2, 3\}\end{aligned}\right\}
        \]\vfill
    \end{minipage}\begin{minipage}[c]{0.5\textwidth}
    \hspace{6em}\begin{tikzpicture}
    \tikzstyle{point}=[circle,thick,draw=black,fill=black,inner sep=0pt,minimum width=4pt,minimum height=4pt]
    \node (v1)[point, label={[label distance=0cm]-135:$1$}] at (0,0) {};
    \node (v2)[point, label={[label distance=0cm]90:$2$}] at (0.5,0.87) {};
    \node (v3)[point, label={[label distance=0cm]-45:$3$}] at (1,0) {};
    \node (v4)[point, label={[label distance=0cm]-45:$4$}] at (1.5,0.87) {};
    
    \draw (v1.center) -- (v2.center);
    \draw (v1.center) -- (v3.center);
    \draw (v2.center) -- (v3.center);
    \draw (v3.center) -- (v4.center);
    
    \draw[color=black!20, style=fill, outer sep = 20pt] (0.1,0.06) -- (0.5,0.77) -- (0.9,0.06) -- cycle;
    \end{tikzpicture}
\end{minipage}

If $F=\{3, 4\}$ then $\triangle'$ is the simplex on $\{1, 2, 3\}$ and $\triangle''$ is just the vertex $\{3\}$, so its dimension is strictly smaller than both the other two.

But, if $F=\{1, 2, 3\}$, then its closure is the 2-dimensional simplex and $\triangle'$ is

\begin{minipage}[c]{0.5\textwidth}
    \[
    \triangle' = \left\{\begin{aligned}
    &\varnothing,\{1\},\{2\},\{3\},\{4\},\\
    &\{1, 2\}, \{1, 3\}, \{2, 3\}, \{3, 4\}\end{aligned}\right\}
    \]\vfill
\end{minipage}\begin{minipage}[c]{0.5\textwidth}
\hspace{6em}\begin{tikzpicture}
\tikzstyle{point}=[circle,thick,draw=black,fill=black,inner sep=0pt,minimum width=4pt,minimum height=4pt]
\node (v1)[point, label={[label distance=0cm]-135:$1$}] at (0,0) {};
\node (v2)[point, label={[label distance=0cm]90:$2$}] at (0.5,0.87) {};
\node (v3)[point, label={[label distance=0cm]-45:$3$}] at (1,0) {};
\node (v4)[point, label={[label distance=0cm]-45:$4$}] at (1.5,0.87) {};

\draw (v1.center) -- (v2.center);
\draw (v1.center) -- (v3.center);
\draw (v2.center) -- (v3.center);
\draw (v3.center) -- (v4.center);
\end{tikzpicture}
\end{minipage}

and $\triangle''$ is 

\begin{minipage}[c]{0.5\textwidth}
    \[
    \triangle'' = \left\{\begin{aligned}
    &\varnothing,\{1\},\{2\},\{3\},\\
    &\{1, 2\}, \{1, 3\}, \{2, 3\}\end{aligned}\right\}
    \]\vfill
\end{minipage}\begin{minipage}[c]{0.5\textwidth}
\hspace{6em}\begin{tikzpicture}
\tikzstyle{point}=[circle,thick,draw=black,fill=black,inner sep=0pt,minimum width=4pt,minimum height=4pt]
\node (v1)[point, label={[label distance=0cm]-135:$1$}] at (0,0) {};
\node (v2)[point, label={[label distance=0cm]90:$2$}] at (0.5,0.87) {};
\node (v3)[point, label={[label distance=0cm]-45:$3$}] at (1,0) {};

\draw (v1.center) -- (v2.center);
\draw (v1.center) -- (v3.center);
\draw (v2.center) -- (v3.center);
\end{tikzpicture}
\end{minipage}
\end{example}

\begin{corollary}
    If $\triangle$ is a simplex we have $\triangle''=\triangle'$ and, equivalently,
    \[
    \Spec\left(\KK[\triangle'']\right)=\Spec\left(\KK[\triangle']\right).
    \]
\end{corollary}
\begin{proof}
    If $\triangle$ is a simplex and $F$ is the facet, than $\triangle=\mathcal{P}(F)$, so $\triangle''=\triangle'$ and the stated equality follows.
\end{proof}

\begin{example}\label{example:spectrum-simplicial-ring-2}
    Let us start from
    
    \begin{minipage}[c]{0.5\textwidth}
        \[
        \triangle = \left\{\begin{aligned}
        &\varnothing,\{1\},\{2\},\{3\},\\
        &\{1, 2\}, \{1, 3\}, \{2, 3\}\end{aligned}\right\}
        \]\vfill
    \end{minipage}\begin{minipage}[c]{0.5\textwidth}
    
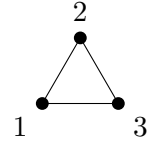
\hspace{6em}\begin{tikzpicture}
    \tikzstyle{point}=[circle,thick,draw=black,fill=black,inner sep=0pt,minimum width=4pt,minimum height=4pt]
    \node (v1)[point, label={[label distance=0cm]-135:$1$}] at (0,0) {};
    \node (v2)[point, label={[label distance=0cm]90:$2$}] at (0.5,0.87) {};
    \node (v3)[point, label={[label distance=0cm]-45:$3$}] at (1,0) {};
    
    \draw (v1.center) -- (v2.center);
    \draw (v1.center) -- (v3.center);
    \draw (v2.center) -- (v3.center);
    \end{tikzpicture}
\end{minipage}

We will use $x, y, z$ instead of $x_1, x_2, x_3$ for the generators of the simplicial binoid, for ease of notation.\\
The associated simplicial binoid is $M_\triangle = (x, y, z\mid x+y+z=\infty)$ and the associated Stanley-Reisner algebra is
\[
\KK[M_\triangle]\cong\KK[\triangle] = \faktor{\KK[X, Y, Z]}{(XYZ)}
\]
whose spectrum $\Spec\KK[\triangle]$ can be drawn as\footnote{For obvious drawing reasons, here we draw $\RR-\Spec \RR[\triangle]$.}
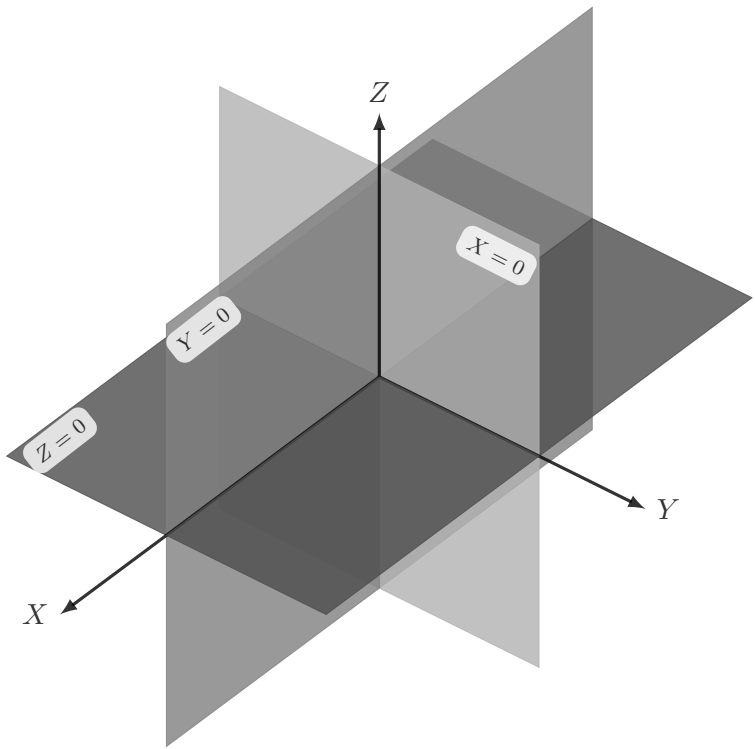
\begin{equation*}
\begin{tikzpicture}[opacity=0.8, scale=0.7]
\filldraw [grey2] (0,-4) -- (0,0) -- (4, 3) -- (4, -1) -- cycle;
\filldraw [grey3] (0,-4) -- (0,0) -- (-3, 1.5) -- (-3,-2.5) -- cycle;
\filldraw [grey3] (0,-4) -- (0,0) -- (3, -1.5) -- (3,-5.5) -- cycle;
\filldraw [grey2] (0,-4) -- (0,0) -- (-4, -3) -- (-4, -7) -- cycle;
\filldraw [grey1] (0,0) -- (-3,1.5) -- (1,4.5) -- (4, 3) -- cycle;
\filldraw [grey2] (0,0) -- (0,4) -- (4, 7) -- (4, 3) -- cycle;
\filldraw [grey1] (0,0) -- (3,-1.5) -- (7,1.5) -- (4, 3) -- cycle;
\filldraw [grey3] (0,0) -- (0,4) -- (-3, 5.5) -- (-3,1.5) -- cycle;
\filldraw [grey3] (0,0) -- (0,4) -- (3, 2.5) -- (3,-1.5) -- cycle;
\filldraw [grey1] (0,0) -- (-3,1.5) -- (-7,-1.5) -- (-4, -3) -- cycle;
\filldraw [grey2] (0,0) -- (0,4) -- (-4, 1) -- (-4, -3) -- cycle;
\draw[->, very thick, >=latex] (0,0) -- (0,5) node [above] {$Z$};
\draw[->, very thick, >=latex] (0,0) -- (-6,-4.5) node [left] {$X$};
\draw[->, very thick, >=latex] (0,0) -- (5,-2.5) node [right] {$Y$};
\filldraw [grey1] (0,0) -- (3,-1.5) -- (-1,-4.5) -- (-4, -3) -- cycle;
\node[fill=white, rounded corners, rotate=38] (z=0) at (-6, -1.2){\scriptsize$Z=0$}; 
\node[fill=white, rounded corners, rotate=-27] (x=0) at (2.2, 2.3){\scriptsize$X=0$};
\node[fill=white, rounded corners, rotate=38] (y=0) at (-3.3, 0.9){\scriptsize$Y=0$};
\end{tikzpicture}
\end{equation*}

Consider now the facet $F=\{2, 3\}$, whose complement in $\triangle$ is

\vspace{1ex}

\begin{minipage}[c]{0.5\textwidth}
    \[
    \triangle' = \left\{\begin{aligned}
    &\varnothing,\{1\},\{2\},\{3\},\\
    &\{1, 2\}, \{1, 3\}\end{aligned}\right\}
    \]\vfill
\end{minipage}\begin{minipage}[c]{0.5\textwidth}
\hspace{2em}\begin{tikzpicture}
\tikzstyle{point}=[circle,thick,draw=black,fill=black,inner sep=0pt,minimum width=4pt,minimum height=4pt]
\node (v1)[point, label={[label distance=0cm]-135:$1$}] at (0,0) {};
\node (v2)[point, label={[label distance=0cm]90:$2$}] at (0.5,0.87) {};
\node (v3)[point, label={[label distance=0cm]-45:$3$}] at (1,0) {};

\draw (v1.center) -- (v2.center);
\draw (v1.center) -- (v3.center);
\end{tikzpicture}
\end{minipage}

\vspace{1ex}

whose associated algebra is $\KK[\triangle']=\faktor{\KK[X, Y, X]}{(YZ)}$ with spectrum $\Spec\KK[\triangle']$

\vspace{1ex}

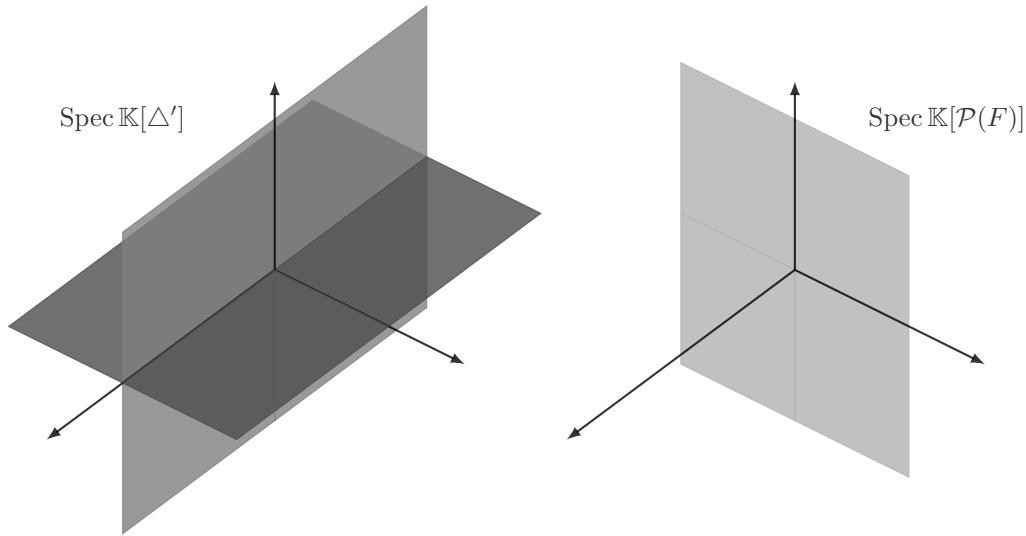
\begin{minipage}{0.5\textwidth}
    \begin{tikzpicture}[opacity=0.8, scale=0.5]
    \filldraw [grey2] (0,-4) -- (0,0) -- (4, 3) -- (4, -1) -- cycle;
    \filldraw [grey2] (0,-4) -- (0,0) -- (-4, -3) -- (-4, -7) -- cycle;
    \filldraw [grey1] (0,0) -- (-3,1.5) -- (1,4.5) -- (4, 3) -- cycle;
    \filldraw [grey2] (0,0) -- (0,4) -- (4, 7) -- (4, 3) -- cycle;
    \filldraw [grey1] (0,0) -- (3,-1.5) -- (7,1.5) -- (4, 3) -- cycle;
    \filldraw [grey1] (0,0) -- (-3,1.5) -- (-7,-1.5) -- (-4, -3) -- cycle;
    \filldraw [grey2] (0,0) -- (0,4) -- (-4, 1) -- (-4, -3) -- cycle;
    \draw[->, thick, >=latex] (0,0) -- (0,5);
    \draw[->, thick, >=latex] (0,0) -- (-6,-4.5);
    \draw[->, thick, >=latex] (0,0) -- (5,-2.5);
    \node (name) at (-4, 4){\small$\Spec \KK[\triangle']$};
    \filldraw [grey1] (0,0) -- (3,-1.5) -- (-1,-4.5) -- (-4, -3) -- cycle;
    \end{tikzpicture}
\end{minipage}
\begin{minipage}{0.5\textwidth}
    \begin{tikzpicture}[opacity=0.8, scale=0.5]
    \filldraw [grey3] (0,-4) -- (0,0) -- (-3, 1.5) -- (-3,-2.5) -- cycle;
    \filldraw [grey3] (0,-4) -- (0,0) -- (3, -1.5) -- (3,-5.5) -- cycle;
    \filldraw [grey3] (0,0) -- (0,4) -- (-3, 5.5) -- (-3,1.5) -- cycle;
    \filldraw [grey3] (0,0) -- (0,4) -- (3, 2.5) -- (3,-1.5) -- cycle;
    \draw[->, thick, >=latex] (0,0) -- (0,5);
    \draw[->, thick, >=latex] (0,0) -- (-6,-4.5);
    \draw[->, thick, >=latex] (0,0) -- (5,-2.5);
    \node (name) at (4, 4){\small$\Spec \KK[\mathcal{P}(F)]$};
    \end{tikzpicture}
\end{minipage}

\vspace{1ex}

and we can see that the algebraic spectrum of the intersection of simplicial complexes is the intersection of the spectra, i.e.\ two coordinate lines.

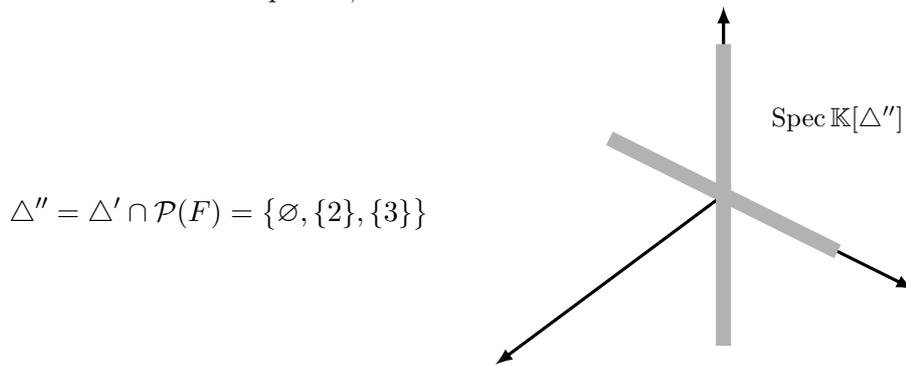
\begin{minipage}{0.5\textwidth}
    \[
    \triangle''=\triangle'\cap \mathcal{P}(F)=\bigl\{\varnothing, \{2\}, \{3\}\bigr\}
    \]
\end{minipage}
\begin{minipage}{0.5\textwidth}
    \begin{tikzpicture}[ scale=0.5]

    \draw[->, very thick, >=latex] (0,0) -- (0,5);
    \draw[->, very thick, >=latex] (0,0) -- (-6,-4.5);
    \draw[->, very thick, >=latex] (0,0) -- (5,-2.5);
    
    \draw [grey3, line width=2mm] (0,-4) -- (0,4);
    \draw [grey3, line width=2mm] (3, -1.5) -- (-3,1.5);
    
    \node (name) at (3, 2){\small$\Spec \KK[\triangle'']$};
    \end{tikzpicture}
\end{minipage}

\vspace{1ex}

and we can see that the dimension is strictly smaller, since we are not dealing with a simplex.
\end{example}

These previous Lemmata will play a key role in the rest of the Chapter.

\begin{proposition}
    By intersecting $\Spec\RR[\triangle]$ with the hyperplane $\{\sum X_i=1 \}$ and considering $X_i\geq 0$, we recover a geometric realization of the abstract simplicial complex we started with.
\end{proposition}
\begin{proof}
    Assume that the vertex set is $[n]$ and consider $e_j=(0, \dots, 0, 1, 0, \dots, 0)^t\in\RR^n$, where the only non-zero entry is in position $j$. Thanks to Lemma~\ref{Lemma:correspondencefacetsirreduciblecomponents}, $\mathrm{span}_\RR(\{ e_j\mid j\in J\})\subseteq\Spec\RR[\triangle]$ if and only $J\in\triangle$. Let $T_J$ be the convex hull of the points $\{e_j\mid j\in J\}$, so $T_J$ is a geometric simplex $\triangle_{|J|-1}$, defined by the equation $\sum_{j=1}^{|J|} x_j=1$ and all the inequalities $x_j\geq 0$. Then $T_J\subseteq\mathrm{span}_\RR(\{ e_j\mid j\in J\})$ and again $T_J\subseteq\Spec\RR[\triangle]$ if and only if $J\in\triangle$. The collection of these sets $\{T_J\}$ that are in $\Spec \RR[\triangle]$ is then the geometric realization of the simplex.
\end{proof}

\begin{example}
    In the previous example, when we intersect the spectrum with the plane $X+Y+Z=1$ and restrict to $X\geq0$, $Y\geq 0$ and $Z\geq 0$, we obtain the white triangle in figure (we do not draw the whole plane, for simplicity).
    
    
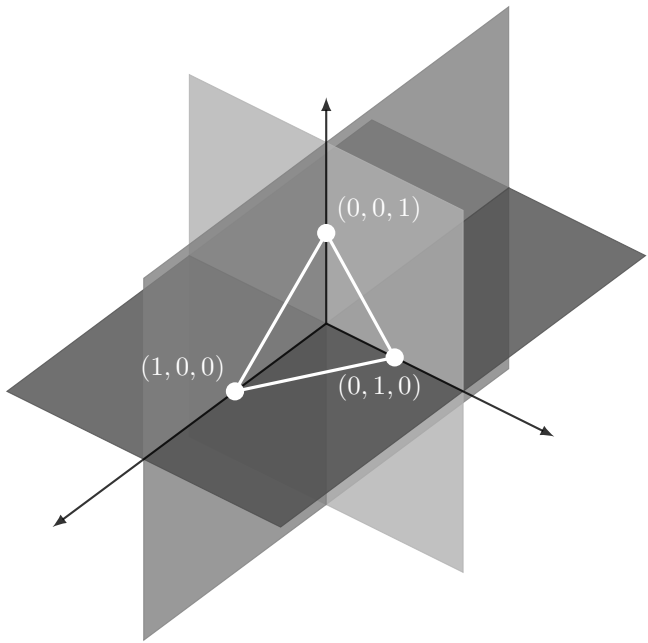
\begin{equation*}
    \begin{tikzpicture}[opacity=0.8, scale=0.6] 
    \filldraw [grey2] (0,-4) -- (0,0) -- (4, 3) -- (4, -1) -- cycle;
    \filldraw [grey3] (0,-4) -- (0,0) -- (-3, 1.5) -- (-3,-2.5) -- cycle;
    \filldraw [grey3] (0,-4) -- (0,0) -- (3, -1.5) -- (3,-5.5) -- cycle;
    \filldraw [grey2] (0,-4) -- (0,0) -- (-4, -3) -- (-4, -7) -- cycle;
    \filldraw [grey1] (0,0) -- (-3,1.5) -- (1,4.5) -- (4, 3) -- cycle;
    \filldraw [grey2] (0,0) -- (0,4) -- (4, 7) -- (4, 3) -- cycle;
    \filldraw [grey1] (0,0) -- (3,-1.5) -- (7,1.5) -- (4, 3) -- cycle;
    \filldraw [grey3] (0,0) -- (0,4) -- (-3, 5.5) -- (-3,1.5) -- cycle;
    \filldraw [grey3] (0,0) -- (0,4) -- (3, 2.5) -- (3,-1.5) -- cycle;
    \filldraw [grey1] (0,0) -- (-3,1.5) -- (-7,-1.5) -- (-4, -3) -- cycle;
    \filldraw [grey2] (0,0) -- (0,4) -- (-4, 1) -- (-4, -3) -- cycle;
    
    \filldraw [grey1] (0,0) -- (3,-1.5) -- (-1,-4.5) -- (-4, -3) -- cycle;
    
    \draw[->, thick, >=latex] (0,0) -- (0,5);
    \draw[->, thick, >=latex] (0,0) -- (-6,-4.5);
    \draw[->, thick, >=latex] (0,0) -- (5,-2.5);
    
    \fill[white, opacity=1] (-2,-1.5) circle[radius=2mm] node[above left] {\small $(1, 0, 0)$};
    \fill[white, opacity=1] (0,2) circle[radius=2mm] node[above right] {\small $(0, 0, 1)$};
    \fill[white, opacity=1] (1.5,-0.75) circle[radius=2mm] node[shift={(-0.2, -0.4)}] {\small $(0, 1, 0)$};
    \draw[white, opacity=1, very thick] (-2,-1.5) -- (0,2) -- (1.5,-0.75) --cycle;
    \end{tikzpicture}
    \end{equation*}
    
\end{example}

\subsection{Covering the punctured Spectrum}

Similarly to what we did in the combinatorial case, we can associate some combinatorial structure to the open combinatorial covering $\{D(X_i)\}$ of $\Spec^\bullet\KK[\triangle]$.

\begin{proposition}\label{proposition:covering-algebraic-spectrum}
    Let $F=\{i_1, \dots, i_j\}$ be a subset of $V$ and let $D(F)$ denote the affine open subset $D(X_{i_1}\cdots X_{i_j})=D(X_{i_1})\cap\dots\cap D(X_{i_j})$. Then $D(F)\neq\varnothing$ if and only if $F\in\triangle$.
\end{proposition}
\begin{proof}
    This is easy to see. If $F$ is a face in $\triangle$ then $(X_i\mid i\notin F)$ is a prime ideal in $\KK[\triangle]$ contained in $\displaystyle\bigcap_{i\in F} D(X_i)$.\\
    If $F$ is not a face, then $\prod X_i=0$, so $\displaystyle\bigcap_{i\in F} D(X_i)=\displaystyle D(\prod_{i\in F}X_i)=D(0)=\varnothing$.
\end{proof}

The previous Propositions shows that we have the same intersection pattern that we have in the combinatorial case.

\begin{corollary}\label{corollary:nerve-covering-stanley-reisner}
    The nerve of the covering $\{D(X_i)\}$ of $\Spec^\bullet\KK[\triangle]$ is the simplicial complex itself.
\end{corollary}

\begin{example}
    Consider again the empty triangle above with ring $\KK[\triangle]$. Intersecting $D(X)$ and $D(Y)$ in the spectrum, leaves only the line $Z=0$. If we try to intersect again with $D(Z)$, we end up empty.
\end{example}

\section{Cohomology of the sheaf of units}

In this section we are going to prove that the covering of the punctured spectrum of a Stanley-Reisner algebra given by the coordinate fundamental open subsets is an acyclic covering for the sheaf of units. In order to show this, we will use the fact that $D(X)\cong\Spec\KK[(M_\triangle)_x]\cong\Spec\KK[\triangle']\times \AA^*$ that we described above, and in particular we will prove that $\H^j(\Spec\KK[\triangle], \O^*)=0$ for $j\geq 1$.

\begin{lemma}
    Let $R$ be a local ring and $\mathfrak{A}$ an ideal of $R$. Then the map
    \[
    \begin{tikzcd}[cramped]
    R^*\rar&\left(\faktor{R}{\mathfrak{A}}\right)^*
    \end{tikzcd}
    \]
    is surjective.
\end{lemma}
\begin{proof}
    This is easily proved because in a local ring the group of units is the complement of the maximal ideal, and quotients of local rings by ideals are again local rings.
\end{proof}

\begin{proposition}\label{proposition:exact-sequence-units-ideals-trivial-intersection}
    Let $\mathfrak{A}$ and $\mathfrak{B}$ be ideals of a commutative ring $R$ such that $\mathfrak{A}\cap \mathfrak{B}=0$. Let $X=\Spec R$, $Y=\Spec \faktor{R}{\mathfrak{A}}$ and $Z=\Spec\faktor{R}{\mathfrak{B}}$. Then there exists a short exact sequence of sheaves
    \[
    \begin{tikzcd}[cramped]
    1\rar&\O^*_X\rar["\varphi"]&i_*\O^*_Y\oplus i_*\O^*_Z\rar["\psi"] & i_*\O^*_{Y\cap Z}\rar&1,
    \end{tikzcd}
    \]
    where $i$ are the inclusion maps and $\varphi(f)=(f\restriction_Y,  f\restriction_Z)$ and $\psi(g, h)=gh^{-1}$.\footnote{We use the same symbol for the inclusion of $Y$ and $Z$ because there is no ambiguity.}
    \begin{align*}
    \end{align*}
\end{proposition}
\begin{proof}
    Clearly $X=Y\cup Z$. These maps exist because they are induced by the taking the quotients of the involved rings, and the fact that this is a complex is clear.
    
    In order to prove the exactness of this sequence, we look at the stalks at a point $\mathfrak{P}$. The surjectivity of $\psi$ follows from the Lemma above, because the stalks are local rings and $Y\cap Z$ is defined by a quotient of the ring of $Y$ (and of the ring of $Z$). In order to prove injectivity of $\varphi$, we look at it on a stalk
    \[
    \begin{tikzcd}[cramped]
    \O^*_{X,\mathfrak{P}}\rar["\varphi_{\mathfrak{P}}"]&(i_*\O^*_Y)_{\mathfrak{P}}\oplus (i_*\O^*_Z)_{\mathfrak{P}}.
    \end{tikzcd}
    \]
    Since $Y=\Spec \faktor{R}{\mathfrak{A}}$ and $Z=\Spec \faktor{R}{\mathfrak{B}}$ we can rewrite this sequence as
    \[
    \begin{tikzcd}[cramped,
    /tikz/column 1/.append style={anchor=base east},
    /tikz/column 2/.append style={anchor=base west},
    row sep = 0ex]
    (R_\mathfrak{P})^*\rar["\varphi_{\mathfrak{P}}"]&\left(\faktor{R_\mathfrak{P}}{\mathfrak{A}}\right)^*\oplus \left(\faktor{R_\mathfrak{P}}{\mathfrak{B}}\right)^*\\
    f\rar[mapsto]&(f, f)
    \end{tikzcd}
    \]
    where, with $\faktor{R_\mathfrak{P}}{\mathfrak{A}}$ and $\faktor{R_\mathfrak{P}}{\mathfrak{B}}$ we mean the quotients via the extended ideals. Consider now $f\in (R_\mathfrak{P})^*$ such that $\varphi_\mathfrak{P}(f)=(1, 1)$. Then $f-1\in\mathfrak{A}$ and $f-1\in\mathfrak{B}$, so $f-1\in\mathfrak{A}\cap\mathfrak{B}=0$, so finally $f=1$ and this map is injective.
    
    In order to prove exactness in the middle, we have to show that if $\psi(g, h)=1$ then they both lie in the image of $\varphi$.    
    Recall that we have an exact sequence of rings
    \[
    \begin{tikzcd}[cramped]
    0\rar& R_\mathfrak{P}\rar &\faktor{R_\mathfrak{P}}{\mathfrak{A}} \oplus \faktor{R_\mathfrak{P}}{\mathfrak{B}}\rar["\psi"] & \faktor{R_\mathfrak{P}}{\mathfrak{A}+\mathfrak{B}}\rar&0.
    \end{tikzcd}
    \]    
    Let $g, h\in R$ such that $g$ is a unit on $Y$, $h$ is a unit on $Z$ and $\psi(g, h)=1$. This happens if and only if $g=h$ in $Y\cap Z$, because the map $\psi$ sends them to $gh^{-1}$.
    The same holds for $R_\mathfrak{P}$ and the quotients in the sequence above.
    So, there exists $f$ in $R_\mathfrak{P}$ such that $f=g+a=h+b$ in $R_\mathfrak{P}$, with $a\in\mathfrak{A}$ and $b\in\mathfrak{B}$ (where these are the extended ideals in $R_\mathfrak{P}$). What is left to prove, is that $f$ is a unit of $R_\mathfrak{P}$.
    Clearly $\faktor{f}{\mathfrak{A}}=g$ is invertible. Assume that $f$ is not, so it belongs to the maximal ideal $\mathfrak{P}R_\mathfrak{P}$, and if we now go modulo $\mathfrak{A}$, it belongs to $\mathfrak{P}\faktor{R_\mathfrak{P}}{\mathfrak{A}}$, that is again the maximal ideal, and so it would not be invertible.
    So $f$ is invertible and $g, h$ both come from before, thus that the sequence is also exact in the middle.
\end{proof}

\begin{lemma}\label{lemma:trivial-intersection-ideals-SR-ring}
    Let $\triangle$ be a simplicial complex and let $F$ be one of its facets. Let $X$ be $\Spec \KK[\triangle]$ and $Y$ be the maximal linear coordinate component of $X$ that corresponds to $F$. Let $Z=\overline{X\smallsetminus Y}$ be the union of all the other maximal linear coordinate components in $X$. Then we can apply the Proposition above on $Y$ and $Z$.
\end{lemma}
\begin{proof}
    We have to prove that the ideal that defines $Y$ and the ideal that defines $Z$ in $X$ have trivial intersection.
    $Y$ is an affine subspace $Y\cong\AA^k$ of $X$, so the ideal that defines it is radical.
    In the same way, $Z$ is a union of affine spaces, so again defined by a radical ideal.\footnote{In details, the ideal of $Z$ can be seen as the colon of two other ideals, see \cite[Corollary 4.8]{cox2006using} for a reference.}
    
    We can apply a basic fact about radical ideals and algebraic sets and see that $Y\cup Z$ is defined by the intersection of the two ideals. In particular, since $Y\cup Z=X$, we have that the two ideals have trivial intersection in $\KK[\triangle]$.
\end{proof}

\begin{remark}
    In the lemma above, $Z$ corresponds to a simplicial complex\footnote{Like before in Lemma~\ref{Lemma:descriptionofcomponentoffacet}, $\mathcal{P}(F)$ is the subset-closure of $F$, i.e.\ its power set.}
    \[
    \triangle'=\overline{\left(\triangle\smallsetminus \mathcal{P}(F)\right)}_{\subseteq},
    \]
    that is the subset-closure of the subset of $\triangle$ obtained by removing $F$ and all its subsets from $\triangle$.
\end{remark}

\begin{example}
    Let $\triangle$ be the simplicial complex on $V=[5]$ with facets $\{1, 2, 3\}$ and $\{3, 4, 5\}$, that we have already seen in Remark~\ref{remark:not-alpha-simplicial}
    
    \begin{minipage}[c]{0.5\textwidth}
        \[
        \triangle = \left\{\begin{aligned}
        &\varnothing, \{1\}, \{2\}, \{3\}, \{4\}, \{5\},\\
        &\{1, 2\}, \{1, 3\}, \{2, 3\}, \\
        &\{3, 4\}, \{3, 5\}, \{4, 5\},\\
        &\{1, 2, 3\}, \{3, 4, 5\}\end{aligned}\right\}
        \]\vfill
    \end{minipage}\begin{minipage}[c]{0.5\textwidth}
    
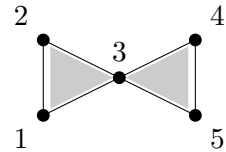
\hspace{6em}\begin{tikzpicture}
    \tikzstyle{point}=[circle,thick,draw=black,fill=black,inner sep=0pt,minimum width=4pt,minimum height=4pt]
    \node (v1)[point, label={[label distance=0cm]-135:$1$}] at (0,0) {};
    \node (v2)[point, label={[label distance=0cm]135:$2$}] at (0,1) {};
    \node (v3)[point, label={[label distance=0cm]90:$3$}] at (1, .5) {};
    \node (v4)[point, label={[label distance=0cm]45:$4$}] at (2,1) {};
    \node (v5)[point, label={[label distance=0cm]-45:$5$}] at (2,0) {};
    
    \draw (v1.center) -- (v2.center);
    \draw (v1.center) -- (v3.center);
    \draw (v2.center) -- (v3.center);
    \draw (v3.center) -- (v4.center);
    \draw (v3.center) -- (v5.center);
    \draw (v4.center) -- (v5.center);
    
    \draw[color=black!20, style=fill, outer sep = 20pt] (0.1,0.1) -- (0.1,0.9) -- (0.9, 0.5) -- cycle;
    \draw[color=black!20, style=fill, outer sep = 20pt] (1.9,0.1) -- (1.9,0.9) -- (1.1, 0.5) -- cycle;
    \end{tikzpicture}
\end{minipage}

Let $F=\{3,4,5\}$. Then $\mathcal{P}(F)$ is the simplex on $\{3, 4, 5\}$,
\[        
\mathcal{P}(F)=\{\varnothing, \{3\}, \{4\}, \{5\}, \{3, 4\}, \{3, 5\}, \{4, 5\},  \{3, 4, 5\}\}
\]
So
\[
\triangle\smallsetminus \mathcal{P}(F)=\{\{1\}, \{2\}, \{1, 2\}, \{1, 3\}, \{2, 3\}, \{1, 2, 3\}\}
\]
and

\begin{minipage}[c]{0.5\textwidth}
    \begin{align*}
    \triangle' & = \overline{\left(\triangle\smallsetminus \mathcal{P}(F)\right)}_{\subseteq}\\
    & = \left\{\begin{aligned}
    &\varnothing, \{1\}, \{2\}, \{3\},\\
    &\{1, 2\}, \{1, 3\}, \{2, 3\}, \\
    &\{1, 2, 3\}\end{aligned}\right\}
    \end{align*}\vfill
\end{minipage}\begin{minipage}[c]{0.5\textwidth}
\hspace{6em}\begin{tikzpicture}
\tikzstyle{point}=[circle,thick,draw=black,fill=black,inner sep=0pt,minimum width=4pt,minimum height=4pt]
\node (v1)[point, label={[label distance=0cm]-135:$1$}] at (0,0) {};
\node (v2)[point, label={[label distance=0cm]135:$2$}] at (0,1) {};
\node (v3)[point, label={[label distance=0cm]90:$3$}] at (1, .5) {};

\draw (v1.center) -- (v2.center);
\draw (v1.center) -- (v3.center);
\draw (v2.center) -- (v3.center);

\draw[color=black!20, style=fill, outer sep = 20pt] (0.1,0.1) -- (0.1,0.9) -- (0.9, 0.5) -- cycle;
\end{tikzpicture}
\end{minipage}\qedhere
\end{example}

\begin{theorem}\label{Thm:vanishingspecstanleyreisner}
    $\H^j(\KK[\triangle], \O^*)=0$ for all $j\geq 1$.
\end{theorem}
\begin{proof}
    Let $X=\Spec\KK[\triangle]$. We prove the claim by induction on the number of facets of $\triangle$, that correspond to the number of maximal coordinate linear subspaces of $X$. If $\triangle$ has only one facet, then it is a simplex and $X\cong\AA^n$ for some $n$, so we proved in Lemma~\ref{Lemma:aciclicityforaffinespace} that $\H^i(X, \O^*)=0$ for all $j\geq 1$.
    
    Let now $\triangle$ be any simplicial complex. Consider $Y\cong\AA^m$ a subset of $X$ associated to a facet $F$ of $\triangle$, so $Y$ is a maximal coordinate linear subset of $X$. Let $Z$ be the closure of the complement of $Y$ in $X$, i.e.\ $Z=\overline{X\smallsetminus Y}$. Clearly $Z\cong\KK[\triangle']$ for some simplicial complex $\triangle'$, where $\triangle'=\overline{\left(\triangle\smallsetminus \mathcal{P}(F)\right)}_{\subseteq}$, the subset-closure of the subset of $\triangle$ obtained by removing $F$ and all its subsets from $\triangle$, as we showed in the previous remark and example. Clearly, $\triangle'$ has a facet less than $\triangle$, namely $F$.
    
    In the same way, $Y\cap Z$ is again a union of coordinate linear subspaces, whose maximal components are the intersection of the maximal components of $Z$ with $Y$, so again coming from another simplicial complex $\triangle''$, that is easier (with smaller dimension and with less facets) than before.
    
    Thanks to the Lemma above, we know that the radical ideal defining $Y$ and the radical ideal defining $Z$ have trivial intersection in $\KK[\triangle]$. We can then apply Proposition~\ref{proposition:exact-sequence-units-ideals-trivial-intersection} to obtain the short exact sequence of sheaves
    \[
    \begin{tikzcd}[cramped, column sep = 2em]
    1\rar&\O^*_X\rar["\varphi"]&i_*\O^*_Y\oplus i_*\O^*_Z\rar["\psi"] & i_*\O^*_{Y\cap Z}\rar&1.
    \end{tikzcd}
    \]
    When we take cohomology, we obtain the long exact sequence of cohomology on $X$ (we omit the space for clarity)
    \[
    \begin{tikzcd}[row sep=2ex, column sep = 2em, cramped]
    \dots\rar&\H^j(\O^*_X)\rar&\H^j(i_*\O^*_Y)\oplus \H^j(i_*\O^*_Z)\rar & \H^j(i_*\O^*_{Y\cap Z})\arrow[dll, overlay, start anchor = east, end anchor = west, to path={ .. controls +(4, -0.8) and +(-2,.8).. (\tikztotarget)}]\\
    &\H^{j+1}(\O^*_X)\rar&\dots
    \end{tikzcd}
    \]
    where, if $j\geq1$, we have that $\H^j(i_*\O^*_Z)=\H^j(Z, \O^*_Z)$ because $Z$ is a closed quotient, and this in turn is $0$ by induction on the number of facets, and $\H^j(i_*\O^*_{Y\cap Z})=\H^j(Y\cap Z, \O^*_{Y\cap Z})=0$. Since $Y$ is an affine space, we already know that $\H^j(i_*\O^*_Y)=\H^j(Y, \O^*_Y)=0$. So for $j>1$ we squeeze $\H^j(\O^*_X)$ between two zeros, and this proves that it is zero itself.
    
    For $j=1$, we have a look at
    \[
    \begin{tikzcd}[row sep=2ex, column sep = 2em, cramped]
    \H^0(\O^*_X)\rar&\H^0(i_*\O^*_Y)\oplus \H^0(i_*\O^*_Z)\rar & \H^0(i_*\O^*_{Y\cap Z})\rar&\H^1(\O^*_X)\rar&\dots
    \end{tikzcd}
    \]
    Since $X$, $Y$ and $Z$ are all defined by Stanley-Reisner ideals, whose units are just the units of the field, this sequence becomes
    \[
    \begin{tikzcd}[row sep=2ex, column sep = 2em, cramped]
    \KK^*\rar&\KK^*\oplus \KK^*\rar["\psi"] & \KK^*\rar&\H^1(\O^*_X)\rar&\dots.
    \end{tikzcd}
    \]
    But $\psi(s, t)=s^{-1}t$, so it is surjective, so also $\H^1(\O^*_X)=0$.
\end{proof}

\begin{corollary}
    The Picard group of a Stanley-Reisner algebra is always trivial.
\end{corollary}

\begin{theorem}\label{Thm:vanishingspecstanleyreisnerA*}
    $\H^j(\KK[\triangle][y_1^{\pm1}, \dots, y_m^{\pm1}], \O^*)=0$ for all $j\geq 1$.
\end{theorem}
\begin{proof}
    The proof is essentially the same as in Theorem~\ref{Thm:vanishingspecstanleyreisner}, modulo some small tweakings. In particular,
    \begin{itemize}
        \item instead of taking $Y=\AA^n$ we take $Y=\AA^n\times(\AA^*)^m$, where we can apply again Lemma~\ref{Lemma:aciclicityforaffinespace},
        \item in the same way, we consider $Z=\KK[\triangle']\times\AA^*$,
        \item we apply Proposition~\ref{proposition:exact-sequence-units-ideals-trivial-intersection} on the ring $\KK[\triangle][y_1^{\pm1}, \dots, y_m^{\pm1}]$ and its quotients,
        \item we can apply Lemma~\ref{lemma:trivial-intersection-ideals-SR-ring} essentially in the same way, because the ideals do not involve the units of the new ring. 
    \end{itemize} 
    A special observation has to be made for $\H^1$, because the sequence of global units becomes now
    \[
    \begin{tikzcd}[row sep=2ex, column sep = 2em, cramped]
    \KK^*\oplus(\ZZ)^m\rar&(\KK^*\oplus(\ZZ)^m) \bigoplus (\KK^*\oplus(\ZZ)^m)\rar["\pi"] & \KK^*\oplus(\ZZ)^m\rar&\dots
    \end{tikzcd}
    \]
    but again it is easy to see that the last map is surjective also on the $\ZZ$'s.
\end{proof}

\begin{corollary}
    The \'etale cohomology group $\H^1_{\mathrm{et}}(\KK[\triangle], \ZZ)$ is zero for every simplicial complex.
\end{corollary}
\begin{proof}
    From \cite{WeibelPic}, we know that for every ring $A$
    \[
    \Pic(A[x, x^{-1}])=\Pic(A)\oplus\mathrm{NPic}(A)\oplus\mathrm{NPic}(A)\oplus\H^1_{\mathrm{et}}(A, \ZZ)
    \]
    and now we know that $\Pic(\KK[\triangle][x,x^{-1}])=0$, so the result follows easily.\footnote{$\mathrm{NPic}(A)$ appears two times in the expression, and it is $\faktor{\Pic A[t]}{\Pic A}$. The idea behind it, is that Weibel treats separately $x$ and $x^{-1}$, as if they were two different variables, and then with the Étale cohomology with coefficients in $\ZZ$ he takes in account their relation.}
\end{proof}

\section{The \texorpdfstring{\v{C}}{C}ech-Picard complex on the punctured spectrum}
Using what we have proved in the previous section, we will show that $\{D(X_i)\}$ is an acyclic covering for $\Spec^\bullet\KK[\triangle]$ with respect to the sheaf $\O^*$ and we will describe the groups and the maps appearing in the \v{C}ech complex relative to this covering.

In the following theorems, we prove that $\{D(X_i)\}$ is an acyclic covering and we describe explicitly the groups appearing in the \v{C}ech complex
\begin{equation*}
\begin{tikzcd}[baseline=(current  bounding  box.center), row sep=2ex, column sep = 2em, cramped, 
/tikz/column 1/.append style={anchor=base east},
/tikz/column 2/.append style={anchor=base west}]
\displaystyle \vC: \bigoplus_{1\leq i\leq n}\O^*(D(X_i))\rar &  \displaystyle \bigoplus_{1\leq i_0<i_2\leq n}\O^*(D(X_{i_0})\cap D(X_{i_1}))\longrightarrow \dots \arrow[out=0, in=180, overlay, d, start anchor = east, end anchor = west, to path={ .. controls +(4, -0.8) and +(-4,.8).. (\tikztotarget)}] \\
& \displaystyle \bigoplus_{1\leq i_0 < i_1 <\dots < i_j \leq n }\O^*(D(X_{i_0})\cap D(X_{i_1}) \cap \dots\cap D(X_{i_j}) ) \arrow[out=0, in=180, overlay, d, start anchor = east, end anchor = west, to path={ .. controls +(4, -0.8) and +(-4,.8).. (\tikztotarget)}]\\
&\dots    
\end{tikzcd}
\end{equation*}

\begin{lemma}\label{theorem:localization-stanley-reisner-combinatorial-multiple} Let $F$ be a face of $\triangle$. The localization of the Stanley-Reisner ring of $\triangle$ at $X_F=\{X_{i_0}, \dots, X_{i_j}\}$ is
    \[
    \KK[\triangle]_{X_F}\cong \KK[\triangle'][\ZZ^F]
    \]
    where $\triangle'=\lk_\triangle(F)$.
\end{lemma}
\begin{proof}
    This follows from the equivalent theorem on binoids, Theorem~\ref{theorem:localization-simplicial-binoid-multiple}, together with the fact that the functor $\KK[\hspace{1em}]$ respects localizations, stated in Lemma~\ref{lemma:functor-respects-localizations}.
\end{proof}

\begin{corollary}\label{corollary:covering-cohomology-stanley-reisner}
    The cohomology of $\O^*$ on the punctured spectrum $\Spec^\bullet(\KK[\triangle])$ can be computed using as \v{C}ech covering the one given by the fundamental combinatorial open subsets $\{D(X_i)\}$.
\end{corollary}
\begin{proof}
    From Lemma~\ref{theorem:localization-stanley-reisner-combinatorial-multiple} and Theorem~\ref{Thm:vanishingspecstanleyreisnerA*} we get that the sheaf of units is acyclic on this covering and, by applying the Theorem of Leray~\ref{theorem:leray}, we get the result.
\end{proof}

\begin{remark}
    For a Stanley-Reisner ring, it does not make a difference whether we compute the cohomology of the sheaf of units $\H^j(\Spec^\bullet \KK[\triangle], \O^*)$ on the Zariski or in the combinatorial topology, since the covering $\{D(X_i)\}$ is acyclic in both topologies, and this yields the same \v{C}ech complex. Since the combinatorial topology is simpler, we can then restrict to work with it.
\end{remark}

\begin{theorem}\label{theorem:SR-splits-comb-top}
    In the combinatorial topology of $\Spec^\bullet\KK[M_\triangle]$ we have that the sheaf of units splits
    \[
    \O^*_{\KK[\triangle]}=\KK^*\oplus i_*\O^*_{M_\triangle}
    \]
    where $\KK^*$ is the usual constant sheaf.
\end{theorem}
\begin{proof}
    This is just an application of Proposition~\ref{proposition:split-sheaves-comb-top}, because $\KK[\triangle]=\KK[M_\triangle]$ is reduced and $M_\triangle$ is semifree, so torsion-free and cancellative.
\end{proof}

\begin{remark}
    In the Zariski topology the sheaf of units does not split, because there are non-combinatorial open subsets on which $i_*\O^*$ does not yields the right units. This can be easily see already in the $\AA^1$, for any $D(P)$ when $P$ is not a monomial.
\end{remark}

\begin{remark}
    Since we are able to split $\O^*_M$ into smaller subsheaves (see Theorem~\ref{theorem:semifree-decomposition-of-sheaf}), we can do the same here, and obtain that in the combinatorial topology
    \[
    \O^*_{\KK[\triangle]}=\KK^*\oplus i_*\left(\displaystyle\bigoplus_{i=1}^n\O^*_{x_i}\right)=\KK^*\oplus \left(\displaystyle\bigoplus_{i=1}^ni_*\O^*_{x_i}\right).
    \]
\end{remark}

\begin{corollary}\label{corollary:cech-complex-groups-algebraic-case}
    Let $\triangle$ be a simplicial complex on the vertex set $V$, let $\KK[\triangle]$ be its Stanley-Reisner algebra and $\O^*=\O^*_{\KK[\triangle]}$ the sheaf of units. Then
    \begin{equation*}
    \O^*\left(\bigcap_{j\in F}D(X_j)\right)\cong\left\{\begin{aligned}
    &\KK^*\times\ZZ^F, &\text{ if } F\in\triangle,\\
    &1,  & \text{ otherwise}.
    \end{aligned}
    \right.
    \end{equation*}
\end{corollary}
\begin{proof}
    From Theorem~\ref{theorem:SR-splits-comb-top} we know that
    \begin{equation*}
    \O^*\left(\bigcap_{j\in F}D(X_j)\right)\cong \KK^*\left(\bigcap_{j\in F}D(X_j)\right)\oplus \O^*\left(\bigcap_{j\in F}D(X_j)\right).
    \end{equation*}
    Clearly $\KK^*\left(\bigcap_{j\in F}D(X_j)\right)$ is either $\KK^*$ or $1$, depending on the intersection being empty or not.
    Finally, by applying Theorem~\ref{theorem:value-sheaf} on $\O^*_{M_\triangle}$, we obtain the wanted result
    \begin{equation*}
    \O^*\left(\bigcap_{j\in F}D(X_j)\right)\cong\left\{\begin{aligned}
    &\KK^*\times\ZZ^F, &\text{ if } F\in\triangle,\\
    &1,  & \text{ otherwise}.
    \end{aligned}
    \right.\qedhere
    \end{equation*}
\end{proof}

In particular, this tells us that a unit $P$ in $\left(\KK[M_{\triangle'}][X_{i_1}^{\pm1}, \dots, X_{i_j}^{\pm1}]\right)^*$ is a monomial and can be written as
\[
P=a_P{X_{i_1}^{e_1}\cdots X_{i_j}^{e_j}}
\]
with $a_P\in\KK^*$ and exponents in $\ZZ$.

\begin{remark}We can see that the maps split more explicitly.
    It is easy to see that the maps in the \v{C}ech complex, defined for example in \cite[Section III.4]{hartshorne1977algebraic} using a additive notation as
    \[
    (d\alpha)_{i_0, \dots, i_{p+1}}=\sum_{k=0}^{p+1}(-1)^k\alpha_{i_0, \dots, \widehat{i_k}, \dots, i_{p+1}}\restriction_{U_{i_0, \ldots, i_{p+1}}},
    \]
    become the following when using our multiplicative notation, the combinatorial covering $\{D(X_i)\}$ and assuming that $\{i_0, \ldots, i_{p+1}\}$ is still a face,
    \[
    (d\alpha)_{i_0, \dots, i_{p+1}}=\prod_{k=0}^{p+1}\left(\alpha_{i_0, \dots, \widehat{i_k}, \dots, i_{p+1}}\restriction_{U_{i_0, \ldots, i_{p+1}}}\right)^{(-1)^k},
    \]
    but
    \[
    \alpha_{i_0, \dots, \widehat{i_k}, \dots, i_{p+1}}\restriction_{U_{i_0, \ldots, i_{p+1}}}=\alpha_{i_0, \dots, \widehat{i_k}, \dots, i_{p+1}}
    \]
    because it is just the restriction
    \[
    \alpha_{i_0, \dots, \widehat{i_k}, \dots, i_{p+1}}=a_{i_0, \dots, \widehat{i_k}, \dots, i_{p+1}}X_{i_0}^{e_{i_0}}\cdots X_{i_{p+1}}^{e_{i_{p+1}}}=a_{i_0, \dots, \widehat{i_k}, \dots, i_{p+1}}X_{i_0}^{e_{i_0}}\cdots X_{i_k}^0\cdots X_{i_{p+1}}^{e_{i_{p+1}}}
    \]
    for some $a_{i_0, \dots, \widehat{i_k}, \dots, i_{p+1}}\in\KK^*$ and $e_{i_k}\in\ZZ$. If $\{i_0, \ldots, i_{p+1}\}$ is not a face, the restriction $\alpha_{i_0, \dots, \widehat{i_k}, \dots, i_{p+1}}\restriction_{U_{i_0, \ldots, i_{p+1}}}$ is $1$, because of Corollary~\ref{corollary:cech-complex-groups-algebraic-case}.
    
    For ease of notation, we denote with $e_{i_0, \dots, \widehat{i_k}, \dots, i_{p+1}}$ the vector with a $0$ in position $k$, $(e_{i_0}, \dots, e_{i_{k-1}}, 0, e_{i_{k+1}}, \dots, e_{i_{p+1}})$ and with $X^{e_{i_0, \dots, \widehat{i_k}, \dots, i_{p+1}}}$ the corresponding monomial $X_{i_0}^{e_{i_0}}\cdots X_{i_k}^0\cdots X_{i_{p+1}}^{e_{i_{p+1}}}$.
    
    So, going back to the map, we can split the part belonging to the field and the combinatorial part (involving only the variables) and describe it as
    \begin{align*}
    (d\alpha)_{i_0, \dots, i_{p+1}} & = \prod_{k=0}^{p+1}\left(a_{i_0, \dots, \widehat{i_k}, \dots, i_{p+1}}X^{e_{i_0, \dots, \widehat{i_k}, \dots, i_{p+1}}}\right)^{(-1)^k}\\
    & = \prod_{k=0}^{p+1}\left(a_{i_0, \dots, \widehat{i_k}, \dots, i_{p+1}}\right)^{(-1)^k}\cdot \prod_{k=0}^{p+1}\left(X^{e_{i_0, \dots, \widehat{i_k}, \dots, i_{p+1}}}\right)^{(-1)^k}.
    \end{align*}
    
    The part involving the variables is the part that comes from the cohomology of the simplicial binoid, and as such it can be rewritten exclusively in terms of operations on the exponents, that belong to $M_\triangle$
    \begin{align*}
    \prod_{k=0}^{p+1}\left(X^{e_{i_0, \dots, \widehat{i_k}, \dots, i_{p+1}}}\right)^{(-1)^k} = X^{m}
    \end{align*}
    with
    \[
    m=\sum_{k=0}^{p+1}(-1)^k e_{i_0, \dots, \widehat{i_k}, \dots, i_{p+1}} \in M_\triangle
    \]
    and we recovered the same map that we had before when computing cohomology for the combinatorial case, in Chapter~\ref{chapter:simplcial-binoids}.
\end{remark}

\begin{corollary}\label{corollary:cech-complex-split-stanley-reisner}
    The complex for computing \v{C}ech cohomology of $\O^*_{\KK[\triangle]}$ on $\Spec^\bullet\KK[\triangle]$ with respect to the combinatorial covering given by $\{D(X_i)\}$ can be split as a direct sum of two other complexes
    \[
    \vC\left(\{D(X_i)\}, \O^*_{\KK[\triangle]}\right)=\vC\left(\{D(x_i)\}, \KK^*\right)\oplus \vC\left(\{D(x_i)\}, \O^*_{M_\triangle}\right).
    \]
\end{corollary}
\begin{proof}
    The only thing that we have to notice is that 
    \[
    \vC\left(\{D(X_i)\}, \KK^*\right)=\vC\left(\{D(x_i)\}, \KK^*\right),
    \]
    because $\{D(X_i)\}$ and $\{D(x_i)\}$ have the same intersection patterns, thanks to Proposition~\ref{proposition:intersection-open-subsets} and Proposition~\ref{proposition:covering-algebraic-spectrum}.
\end{proof}

\begin{example}
    Consider again
    \[
    \KK[M] = \faktor{\KK[X, Y, Z]}{(XYZ)}.
    \]
    Its punctured spectrum is then covered by $D(X)$, $D(Y)$ and $D(Z)$, and the \v{C}ech complex w.r.t.\ this acyclic covering for $\O^*_{\KK[M]}$ is
    \begin{equation*}
    \begin{tikzcd}[baseline=(current  bounding  box.center), row sep=2ex, column sep = 2em, cramped, 
    /tikz/column 1/.append style={anchor=base east},
    /tikz/column 2/.append style={anchor=base west}]
    \displaystyle \C: & (\KK^*\oplus\ZZ)\bigoplus (\KK^*\oplus\ZZ)\bigoplus(\KK^*\oplus\ZZ)\arrow[out=0, in=180, overlay, d, start anchor = east, end anchor = west, to path={ .. controls +(4, -0.8) and +(-4,.8).. (\tikztotarget)}, "\partial^0"]\\
    &(\KK^*\oplus\ZZ^2)\bigoplus (\KK^*\oplus\ZZ^2)\bigoplus(\KK^*\oplus\ZZ^2) \rar & 1
    \end{tikzcd}
    \end{equation*}
    and we have the components
    \begin{equation*}
    \begin{tikzcd}[row sep=2ex, 
    /tikz/column 1/.append style={anchor=base east},
    /tikz/column 2/.append style={anchor=base east},
    /tikz/column 3/.append style={anchor=base west}]
    \displaystyle \C(\KK^*): \KK^*\oplus\KK^*\oplus\KK^*\rar["\partial^0_{\KK^*}"] &\KK^*\oplus \KK^*\oplus\KK^* \rar & 1\\
    (\alpha, \beta, \gamma)\rar[mapsto] &\left(\dfrac{\beta}{\alpha}, \dfrac{\gamma}{\alpha}, \dfrac{\gamma}{\beta}\right)\\
    \\
    \displaystyle \C(\O^*_M): \ZZ\oplus\ZZ\oplus\ZZ\rar["\partial^0_M"]&\ZZ^2\oplus\ZZ^2\oplus\ZZ^2 \rar["\partial^1_M"] & 0\\
    (a, b, c)\rar[mapsto]& (-a, b, -a, c, -b, c)
    \end{tikzcd}
    \end{equation*}
    that give us the decomposition.
\end{example}

\begin{remark}
    Thanks to the fact that the nerve of the covering $\{D(X_i)\}$ is the simplicial complex we started with, the previous Corollary tells us that the punctured \v{C}ech-Picard complex of a Stanley Reisner ring splits in two parts, that are both completely determined by the combinatorics of the simplicial complex.
\end{remark}

\section{Cohomology}
In this last section of the Chapter we sum up our results and give the explicit formulas for computing cohomology of the sheaf of units of a Stanley Reisner ring on the punctured spectrum.

\begin{lemma}
    For any constant sheaf of abelian groups $G$ on $\Spec^\bullet\KK[M_\triangle]$, we have that
    \[
    \vH^j\left(\left\{D(X_i)\right\}, G\right)=\H^j(\triangle, G).
    \]
\end{lemma}
\begin{proof}
    Thanks to Proposition~\ref{proposition:covering-algebraic-spectrum} and Proposition~\ref{proposition:intersection-open-subsets},
    \[
    \nerve\{D(X_i)\}=\triangle=\nerve\{D(x_i)\}
    \]
    the nerve of the combinatorial covering of $\Spec^\bullet M_\triangle$. Thanks to Corollary~\ref{corollary:cech-covering-nerve-simplicial-cohomology} the simplicial cohomology of the latter is \v{C}ech cohomology of the covering, and since the \v{C}ech complexes look the same for $\{D(X_i)\}$ and $\{D(x_i)\}$, we get our thesis.
\end{proof}

\begin{theorem}\label{theorem:cohomology-stanley-reisner}
    Let $\KK[\triangle]$ be the Stanley-Reisner ring of a simplicial complex $\triangle$ on a finite vertex set $V$. We have the following explicit formula for the computation of the cohomology groups of the sheaf of units, $\O^*_{\KK[\triangle]}$ restricted to the punctured spectrum $\Spec^\bullet\KK[\triangle]$.
    \begin{equation}\label{formula:cohomology-SR}
    \H^j(\Spec^\bullet(\KK[\triangle]), \O^*_{\KK[\triangle]})=\H^j(\triangle, \KK^*)\oplus \bigoplus_{v\in V}\widetilde{\H}^{j-1}(\lk_\triangle(v), \ZZ),
    \end{equation}
    where $\H^j(\triangle, \KK^*)$ is the $i$-th simplicial cohomology group with coefficients in $\KK^*$, for $j\geq0$.
\end{theorem}
\begin{proof}
    Recall from the results in this chapter that we can use the covering of the open subsets defined by the variables $\{D(X_i)\}$ as an acyclic covering for computing this cohomology via \v{C}ech cohomology (Corollary~\ref{corollary:covering-cohomology-stanley-reisner})
    \[
    \H^j(\Spec^\bullet\KK[\triangle], \O^*_{\KK[\triangle]})=\vH^j(\{D(X_i)\}, \O^*_{\KK[\triangle]})=\H_j\left(\vC\left(D(X_i), \O^*_{\KK[\triangle]}\right)\right).
    \]
    We observed in Corollary~\ref{corollary:cech-complex-split-stanley-reisner} that we can rewrite this \v{C}ech complex as a direct sum
    \[
    \C\left(\left\{D(X_i)\right\}, \O^*_{\KK[\triangle]}\right)=\C\left(\left\{D(x_i)\right\}, \KK^*\right)\oplus \C\left(\left\{D(x_i)\right\}, \O^*_{M_\triangle}\right)
    \]
    and the same holds for the cohomology
    \[
    \vH^j\left(\left\{D(X_i)\right\}, \O^*_{\KK[\triangle]}\right)=\vH^j\left(\left\{D(X_i)\right\}, \KK^*\right)\oplus \vH^j\left(\left\{D(x_i)\right\}, \O^*_{M_\triangle}\right).
    \]
    Thanks to the fact that the nerve of the covering $\{D(X_i)\}$ is the simplicial complex $\triangle$ we started with (Corollary~\ref{corollary:nerve-covering-stanley-reisner}), we can apply the previous Lemma and see the first summand as simplicial cohomology
    \[
    \vH^j\left(\left\{D(X_i)\right\}, \KK^*\right)=\H^j(\triangle, \KK^*)
    \]
    and, finally, we can use Theorem~\ref{theorem:cohomology-simplicial-complex} to recall the fact that
    \[
    \vH^j\left(\left\{D(x_i)\right\}, \O^*_{M_\triangle}\right)=\H^j(\Spec^\bullet M_\triangle, \O^*_{M_\triangle})=\bigoplus_{v\in V}\widetilde{\H}^{j-1}(\lk_\triangle(v), \ZZ)
    \]
    and obtain our thesis
    \[
    \H^j(\Spec^\bullet(\KK[\triangle]), \O^*_{\KK[\triangle]})=\H^j(\triangle, \KK^*)\oplus \bigoplus_{v\in V}\widetilde{\H}^{j-1}(\lk_\triangle(v), \ZZ).\qedhere
    \]
\end{proof}

\begin{example}
    Consider again 
    \[
    \KK[M] = \faktor{\KK[X, Y, Z]}{(XYZ)}.
    \] Given the \v{C}ech complex and its decomposition that we computed in the Example above, it is easy to explicitly compute its cohomology, since
    \begin{align*}
    \H^0(\C)=\ker\partial^0&=\ker\partial^0_{\KK^*}\oplus\ker\partial^0_M\\
    &\cong\{(\alpha, \beta, \gamma)\in(\KK^*)^3\mid\alpha=\beta=\gamma\}\oplus\{(0,0,0)\}\cong\KK^*
    \end{align*}
    and the image of $\partial^0$ is
    \begin{align*}
    \im\partial^0&=\im\partial^0_{\KK^*}\oplus\im\partial^0_M\\
    &\cong\{(\rho, \sigma,\tau)\in(\KK^*)^3\mid\sigma=\rho\tau\}\\
    &\hspace{2em}\oplus\{(p, q, r, s, t, u)\in\ZZ^6\mid p=r, q=-t, s=u\}\cong(\KK^*)2\oplus \ZZ^3
    \end{align*}
    and clearly
    \[
    \H^1(\C)=\faktor{\ker\partial^1}{\im\partial^0}\cong\KK^*\oplus\ZZ^3\cong\H^1(\triangle, \KK^*)\oplus\H^1(\O^*_M).
    \]
    We recover that $\H^0(\triangle, \KK^*)\cong\H^1(\triangle, \KK^*)\cong\KK^*$, with a description in \v{C}ech cohomology given by $(1, 1, \lambda)$ for $\lambda\in\KK^*$, that we could expect because our simplicial complex is topologically a circle.
\end{example}

\subsection{Line bundles}
A line bundle on $X=\Spec \KK[M]$ is a locally free $\O_{\KK[M]}$-module of constant rank 1. Similarly, a line bundle on $U=\Spec^\bullet \KK[M]$ is a locally free $\O_{\KK[M]}\restriction_U$-module of constant rank 1.

From Chapter~\ref{chapter:injections} we know that there is an injection from $\Pic^{\loc} M$ to $\Pic^{\loc} \KK[M]$. In particular, any combinatorial line bundle defines an algebraic line bundle. In general, the second is much bigger than the first, due to the contribution of the first cohomology of the sheaf of constant units $\KK^*$.

\begin{example}\label{example:cohomology-xyz=0}
    Let us go back to our favourite Example~\ref{example:spectrum-simplicial-ring-2},
    \[
    M_\triangle=(x, y, z\mid x+y+z=\infty) \hspace{4em} \KK[M_\triangle]=\faktor{\KK[X, Y, Z]}{(XYZ)}
    \]
    From what we proved in Theorem~\ref{theorem:cohomology-simplicial-complex} and Theorem~\ref{theorem:cohomology-stanley-reisner} we know that
    \[
    \Pic^{\loc} M_\triangle = \ZZ^3 \hspace{4em} \Pic^{\loc} \KK[M_\triangle]=\KK^*\oplus\ZZ^3
    \]
    so we can easily see that there are many more algebraic line bundles, as there are combinatorial ones.
    
    From the extensive study conducted in Section~\ref{subsection:x+y+z=infty}, we can recall that a line bundle in $\Pic^{\loc} M$ is represented (up to isomorphism) by a locally free $M_\triangle$-set of rank 1 of the form
    \[
    S=\left(e_1, e_2, e_3\midd \begin{aligned}
    e_1+ay&=e_2+bx,\\
    e_1+cz&=e_3+dx,\\
    e_2+ez&=e_3+fy
    \end{aligned}\right)
    \]
    with $a, b, c, d, e$ and $f$ positive natural numbers. 
    We can think of this isomorphism class as represented by $(a-f, b-d, c-e)\in\ZZ^3$.

    We know automatically that such $M_\triangle$-set gives rise to a locally free $\KK[M_\triangle]$-module of rank 1 in $\Pic^{\loc}\KK[M_\triangle]$
    \[
    \KK[S]=\left(E_1, E_2, E_3\midd \begin{aligned}
    E_1Y^a&=E_2X^b,\\
    E_1Z^c&=E_3X^d,\\
    E_2Z^e&=E_3Y^f
    \end{aligned}\right)
    \]
    with $a, b, c, d, e$ and $f$ natural numbers. This indeed identifies an isomorphism class in $\Pic^{\loc}$. In order to see the contribution of $\KK^*$, it is enough to notice that
    \[
    \left(E_1, E_2, E_3\midd \begin{aligned}
    \lambda_1 E_1Y^a&=E_2X^b,\\
    \lambda_2 E_1Z^c&=E_3X^d,\\
    E_2Z^e&=E_3Y^f
    \end{aligned}\right)
    \quad \text{ and }\quad
    \left(E_1, E_2, E_3\midd \begin{aligned}
    \lambda'_1 E_1Y^a&=E_2X^b,\\
    \lambda'_2 E_1Z^c&=E_3X^d,\\
    E_2Z^e&=E_3Y^f
    \end{aligned}\right)
    \]
    with $\lambda_i, \lambda'_i\in\KK^*$ define two different line bundles if $\dfrac{\lambda_1}{\lambda_2}\neq\dfrac{\lambda'_1}{\lambda'_2}$.
    
    With an argument similar to the one that we used on page~\pageref{example:the-star-graph} for the star graph, we can see that $a, b, c, d, e, f$ have to be different from $0$.
    
    In particular, we can think of the bundles with exponents all $1$'s as the bundles that come from the units of the field
    \[
    \left(E_1, E_2, E_3\midd \begin{aligned}
    \lambda'_1 E_1Y&=E_2X,\\
    \lambda'_2 E_1Z&=E_3X,\\
    E_2Z&=E_3Y
    \end{aligned}\right),
    \]
    that is in the same class as the bundle
    \[
    \left(E_1, E_2, E_3\midd \begin{aligned}
    \gamma E_1Y&=E_2X,\\
    E_1Z&=E_3X,\\
    E_2Z&=E_3Y
    \end{aligned}\right),
    \]
    thus giving us an easy decomposition of any bundle in what comes from $\KK^*$ and what is combinatorial.
\end{example}

\section{The general monomial case}
In this Section we present and discuss a couple of examples that show how the general non-reduced monomial case can be already much harder than the reduced case, but can still be handled with our methods.

If $I$ is any monomial ideal in $S=\KK[X_1, \dots, X_n]$, its radical ideal is a Stanley Reisner ideal, $\sqrt{I}=I_\triangle$. Let $R=\faktor{S}{I}$ and $R_{\red}=\KK[M_\triangle]=\faktor{S}{I_\triangle}$.
In Theorems~\ref{Thm:vanishingspecstanleyreisner} and \ref{Thm:vanishingspecstanleyreisnerA*} we proved that the covering of $\Spec^\bullet\faktor{S}{I}$ generated by $\{D(X_i)\}$ is acyclic for $\O^*_{R_{\red}}$ with the Zariski topology.

\begin{theorem}\label{theorem:split-non-reduced-monomial}
    Let $I$ be a monomial ideal in $S=\KK[X_1, \dots, X_n]$, let $R=\faktor{S}{I}$, let $M$ be its binoid ($R=\KK[M]$) and let $X=\Spec^\bullet R$.
    
    Let $I_\triangle=\sqrt{I}$ be the radical of $I$, let $R_{\red}=\faktor{S}{I_\triangle}$ be the reduction of $R$, let $M_\triangle$ be the simplicial binoid associated to $R_{\red}$ and let $\triangle$ be the respective simplicial complex.
    
    We can compute the cohomology of $\O^*_X$ on $X$ with the Zariski topology as
    \[
    \H^j(X, \O^*_X)=\H^j(\triangle, \KK^*)\oplus \bigoplus_{v\in V}\widetilde{\H}^{j-1}(\lk_\triangle(v), \ZZ) \oplus \vH^j(\{D(X_i)\}, 1+\N)
    \]
    for any $j\geq 0$.
\end{theorem}
\begin{proof}
    Thanks to the exact sequence presented in Proposition~\ref{proposition:exact-sequence-units-1+nilpotents-reduction}, the sheaf of abelian groups $(1+\N)$ fits into the short exact sequence
    \begin{equation*}
    \begin{tikzcd}[baseline=(current  bounding  box.center),
    /tikz/column 1/.append style={anchor=base east},
    /tikz/column 2/.append style={anchor=base west}, 
    row sep = 0pt]
    1\rar &1+\N \rar& \O^*_X\rar &\O^*_{X_{\red}}\rar&1
    \end{tikzcd}
    \end{equation*}
    in the Zariski topology, and $(1+\N)$ is acyclic for affine schemes, so thanks to Corollary~\ref{corollary:combinatorial-covering-non-reduced} we can then use any covering by combinatorial open subsets to compute the cohomology of $\O^*_{R}$.
    
    By applying Remark~\ref{remark:units-non-reduced} to the proof of Proposition~\ref{proposition:split-sheaves-comb-top} we can easily that even in the non reduced case we can split the sheaf of units as a direct sum in the combinatorial topology as
    \[
    \O^*_X=\KK^*\oplus i_*\O^*_M\oplus (1+\N).
    \]
    In particular, we have again a direct sum of complexes when we look at the \v{C}ech cohomology on the combinatorial covering $\{D(X_i)\}$ of $\Spec^\bullet \KK[M]$.
    
    The sheaves, obviously, do not split in the Zariski topology but, thanks to the fact that the \v{C}ech complexes for $\O^*$ on the covering $\{D(X_i)\}$ are the same in the combinatorial and in the Zariski topology, we can use the smaller complexes, that come from the decomposition above in the combinatorial topology, to compute cohomology in the Zariski topology as
    \begin{align*}
    \H^j_{\mathrm{Zar}}(X, \O^*_X)&=\H^j_{\comb}(X, \O^*_X)=\H^j_{\comb}(X, \KK^*\oplus i_*\O^*_M\oplus (1+\N))\\
    &=\H^j_{\comb}(X, \KK^*)\oplus \H^j_{\comb}(X, i_*\O^*_M)\oplus \H^j_{\comb}(X, 1+\N)\\
    &=\vH^j(\{D(x_i)\}, \KK^*)\oplus \vH^j(\{D(x_i)\}, \O^*_M)\oplus \H^j(\{D(X_i)\}, 1+\N)\\
    &=\H^j(\triangle, \KK^*)\oplus \bigoplus_{v\in V}\widetilde{\H}^{j-1}(\lk_\triangle(v), \ZZ) \oplus \vH^j(\{D(X_i)\}, 1+\N)\qedhere
    \end{align*}
\end{proof}

\begin{remark}
    This theorem explains well why our methods can be useful to compute of the cohomology in the non-reduced monomial case. Indeed, most of the computations now become again simplicial cohomology, that in general is easier to compute.
\end{remark}

\begin{remark}\label{remark:rewrite-non-reduced}
    Since $\O^*_M=\O^*_{M_{\red}}$, we can rewrite the result above as
    \begin{align*}
    \H^j_{\mathrm{Zar}}(X, \O^*_X)&=\H^j(X_{\red}, \O^*_{X_{\red}})\oplus \vH^j(\{D(X_i)\}, 1+\N)\\
    &=\H^j(\triangle, \KK^*)\oplus\H^j(\Spec^\bullet M, \O^*_M)\oplus \vH^j(\{D(X_i)\}, 1+\N).
    \end{align*}
\end{remark}

\begin{remark}
    Let $M$ be a binoid such that the radical ideal $I_\triangle$ corresponds to a simplicial complex $\triangle$ of dimension 0. Then $\Pic^{\loc} \KK[M]$ and all the higher cohomology groups are trivial. This is true because of the intersection pattern. Indeed, if $\triangle$ has dimension 0, then all the possible intersections $D(X_i)\cap D(X_j)$ are empty, so 
    \[
    \vC^{j\geq 1}(\{D(X_i)\}, \O^*)=0,
    \] thus proving our claim.
\end{remark}

\begin{example}
    Let $M$ be the non reduced binoid
    \[
    M=(x, y, z\mid x+y+2z=\infty)
    \]
    whose associated rings is
    \[
    \KK[M]=\faktor{\KK[X, Y, Z]}{\langle XYZ^2\rangle}.
    \]
    We want to compute $\Pic^{\loc}(\KK[M])$ and we can do it by meaning of \v{C}ech cohomology on the combinatorial covering. We have first to compute the localizations at the elements of the covering and their groups of units
    \begin{align*}
    \KK[M]_X &\cong \faktor{\KK[X^{\pm 1}, Y, Z]}{\langle YZ^2\rangle}\\
    \KK[M]_X^*&\cong\left\{aX^r\left(1+P\right)\midd a\in\KK^*, r\in\ZZ, P\in\N(D(X))=\langle YZ\rangle\KK[M]_X\right\}
    \end{align*}
    \begin{align*}
    \KK[M]_Y &\cong \faktor{\KK[X, Y^{\pm 1}, Z]}{\langle XZ^2\rangle}\\
    \KK[M]_Y^*&\cong\left\{bY^s\left(1+Q\right)\midd b\in\KK^*, s\in\ZZ, Q\in\langle XZ\rangle\KK[M]_Y\right\}
    \end{align*}
    \begin{align*}
    \KK[M]_Z &\cong \faktor{\KK[X, Y, Z^{\pm 1}]}{\langle XY\rangle}\\
    \KK[M]_Z^*&\cong\left\{cZ^t \midd c\in\KK^*, t\in\ZZ\right\}
    \end{align*}
    \begin{align*}
    \KK[M]_{XY} &\cong \faktor{\KK[X^{\pm 1}, Y^{\pm 1}, Z]}{\langle Z^2\rangle}\\
    \KK[M]_{XY}^*&\cong\left\{\alpha X^mY^n\left(1+T\right)\midd \alpha \in\KK^*, m, n\in\ZZ, T\in\langle Z\rangle\KK[M]_{XY}\right\}
    \end{align*}
    \begin{align*}
    \KK[M]_{XZ} &\cong \faktor{\KK[X^{\pm 1}, Y, Z^{\pm 1}]}{\langle Y\rangle}\cong\KK[X^{\pm 1}, Z^{\pm 1}]\\
    \KK[M]_{XZ}^*&\cong\left\{\beta X^mY^n\midd \beta\in\KK^*, m, n\in\ZZ\right\}
    \end{align*}
    \begin{align*}
    \KK[M]_{YZ} &\cong \faktor{\KK[X, Y^{\pm 1}, Z^{\pm 1}]}{\langle X\rangle}\cong\KK[Y^{\pm 1}, Z^{\pm 1}]\\
    \KK[M]_{XY}^*&\cong\left\{\gamma Y^mZ^n\midd \gamma\in\KK^*, m, n\in\ZZ\right\}
    \end{align*}
    The \v{C}ech complex looks like
    \begin{equation*}
    \begin{tikzcd}[baseline=(current  bounding  box.center), row sep=2ex, column sep = 2em, cramped, 
    /tikz/column 1/.append style={anchor=base east},
    /tikz/column 2/.append style={anchor=base west}]
    \KK[M]^*_X\oplus\KK[M]^*_Y\oplus\KK[M]^*_Z\rar["\partial^0"] &  \KK[M]^*_{XY}\oplus\KK[M]^*_{XZ}\oplus\KK[M]^*_{YZ}\rar["\partial^1"] & 0,
    \end{tikzcd}
    \end{equation*}
    where
    \[
    \partial^0\left(aX^r\left(1+P\right), bY^s\left(1+Q\right), cZ^t\right)=\left(\dfrac{bY^s\left(1+Q\right)}{aX^r\left(1+P\right)},\ \dfrac{cZ^t}{aX^r\left(1+P\right)},\ \dfrac{cZ^t}{bY^s\left(1+Q\right)}\right).
    \]
    Now we can easily compute the inverses, since $\N^2=0$, so $1+\N\cong \N$ as sheaves of groups, and
    \begin{align*}
    \left(aX^r\left(1+P\right)\right)^{-1}&=\dfrac{1}{a}X^{-r}\left(1-P\right)\\
    \left(bY^s\left(1+Q\right)\right)^{-1}&=\dfrac{1}{b}Y^{-s}\left(1-Q\right)
    \end{align*}
    so
    \begin{align*}
    \partial^0&\left(aX^r\left(1+P\right), bY^s\left(1+Q\right), cZ^t\right)\\
    &=\left(\dfrac{b}{a}Y^sX^{-r}\left(1+Q-P\right),\ \dfrac{c}{a}Z^tX^{-r}\left(1-P\right),\ \dfrac{c}{b}Z^tY^{-s}\left(1-Q\right)\right).
    \end{align*}
    
    When we split this map on the components, we already know the cohomology for $\KK^*$ and for $\O^*_{M_\triangle}$, so we have only to compute it for $1+\N\cong \N$
    \begin{equation*}
    \begin{tikzcd}[baseline=(current  bounding  box.center), row sep=2ex, column sep = 2em, cramped, 
    /tikz/column 1/.append style={anchor=base east},
    /tikz/column 2/.append style={anchor=base west}]
    (1+\N(D(X)))\oplus (1+\N(D(Y)))\oplus 1\rar["\partial^0_{1+\N}"] &  (1+\N(D(X)\cap D(Y))\oplus 1 \oplus 1\rar["\partial^1_{1+\N}"] & 0\\
    (1+P, 1+Q, 1)\rar[mapsto]& (1+Q-P, 1, 1)
    \end{tikzcd}
    \end{equation*}
    
    Since $P\in\N(X)$ and $Q\in\N(Y)$ we know that $P=fYZ$ and $Q=gXZ$ for some polynomials $f\in\KK[M]_X$ and $g\in\KK[M]_Y$. So $\partial^0(Q-P)=0\in\R_{XY}$ if and only if $\partial^0(Q-P)\in\langle Z^2\rangle\KK[M]_{XY}$, so $Z^2\mid P$ and $Z^2\mid Q$ in the respective rings, so $YZ^2\mid P$, so $P=0\in\KK[M]_X$ and similarly for $Q$. So $\partial^0_{1+\N}$ is injective and $\H^0(1+\N)=0$.
    
    The image of this map is $\im\partial^0_{1+\N}=\{1+(gX-fY)Z\mid f\in\KK[M]_X, g\in\KK[M]_Y\}$ and the quotient is then
    \[
    \H^1(1+\N)\cong \faktor{\langle Z\rangle \KK[M]_{XY}}{\langle (gX-fY)Z\rangle\KK[M]_{XY}}.
    \]
    As a $\KK$-vector space, this quotient is generated by the monomials of the type $X^iY^jZ$ such that $i, j\leq0$, because as soon as one of these exponents is positive, it comes from before.
    
    Summing up, the local Picard group of this ring will then be
    \[
    \Pic^{\loc}(\KK[M])\cong\KK^*\oplus\ZZ^3\oplus \faktor{\langle Z\rangle \KK[M]_{XY}}{\langle (gX-fY)Z\rangle\KK[M]_{XY}}.\qedhere
    \]
\end{example}

\begin{example}
    Let $M$ be the non reduced binoid
    \[
    M=(x, y, z\mid 2x+y+3z=\infty, x+2y+2z=\infty) 
    \]
    whose associated ring is
    \[
    \KK[M]=\faktor{\KK[X, Y, Z]}{\langle X^2YZ^3, XY^2Z^2\rangle}.
    \]
    We want to compute $\Pic^{\loc}(\KK[M])$ and we can do it by meaning of \v{C}ech cohomology on the combinatorial covering. We have first to compute the localizations at the elements of the covering and their groups of units
    \begin{align*}
    \KK[M]_X &\cong \faktor{\KK[X^{\pm 1}, Y, Z]}{\langle YZ^3, Y^2Z^2\rangle}\\
    \KK[M]_X^*&\cong\left\{aX^r\left(1+\sum_{i\geq 1}b_iY^iZ+cYZ^2\right)\midd r\in\ZZ, \text{ $b_i=0$ for almost all $i$}\right\}
    \end{align*}
    \begin{align*}
    \KK[M]_Y &\cong \faktor{\KK[X, Y^{\pm 1}, Z]}{\langle XZ^2\rangle}\\
    \KK[M]_Y^*&\cong\left\{dY^s\left(1+\sum_{i\geq 1}e_iX^iZ\right)\midd s\in\ZZ, \text{ $e_i=0$ for almost all $i$}\right\}
    \end{align*}
    \begin{align*}
    \KK[M]_Z &\cong \faktor{\KK[X, Y, Z^{\pm 1}]}{\langle X^2Y, XY^2\rangle}\\
    \KK[M]_Z^*&\cong\left\{fZ^t\left(1+gXY\right)\midd t\in\ZZ\right\}
    \end{align*}
    \begin{align*}
    \KK[M]_{XY} &\cong \faktor{\KK[X^{\pm 1}, Y^{\pm 1}, Z]}{\langle Z^2\rangle}\\
    \KK[M]_{XY}^*&\cong\left\{\alpha X^mY^n\left(1+hZ\right)\midd m, n\in\ZZ, h\in\KK[M]_{XY}\right\}
    \end{align*}
    \begin{align*}
    \KK[M]_{XZ} &\cong \faktor{\KK[X^{\pm 1}, Y, Z^{\pm 1}]}{\langle Y\rangle}\cong\KK[X^{\pm1}, Z^{\pm 1}]\\
    \KK[M]_{XZ}^*&\cong\left\{\beta X^mY^n\midd m, n\in\ZZ\right\}
    \end{align*}
    \begin{align*}
    \KK[M]_{YZ} &\cong \faktor{\KK[X, Y^{\pm 1}, Z^{\pm 1}]}{\langle X\rangle}\cong\KK[Y^{\pm1}, Z^{\pm 1}]\\
    \KK[M]_{XY}^*&\cong\left\{\gamma Y^mZ^n\midd m, n\in\ZZ\right\}
    \end{align*}
    where $a, d, f, \alpha, \beta, \gamma$ are in $\KK^*$, $b_i, c\in\KK[X^{\pm1}]$, $e_i\in\KK[Y^{\pm1}]$, $g\in\KK[Z^{\pm1}]$. The \v{C}ech complex looks like
    \begin{equation*}
    \begin{tikzcd}[baseline=(current  bounding  box.center), row sep=2ex, column sep = 2em, cramped, 
    /tikz/column 1/.append style={anchor=base east},
    /tikz/column 2/.append style={anchor=base west}]
    \KK[M]^*_X\oplus\KK[M]^*_Y\oplus\KK[M]^*_Z\rar["\partial^0"] &  \KK[M]^*_{XY}\oplus\KK[M]^*_{XZ}\oplus\KK[M]^*_{YZ}\rar["\partial^1"] & 0,
    \end{tikzcd}
    \end{equation*}
    where
    \begin{align*}
    \partial^0&\left(aX^r\left(1+\sum_{i\geq 1}b_iY^iZ+cYZ^2\right), dY^s\left(1+\sum_{i\geq 1}e_iX^iZ\right), fZ^t\left(1+gXY\right)\right)\\
    &=\left(\dfrac{dY^s\left(1+\sum_{i\geq 1}e_iX^iZ\right)}{aX^r\left(1+\sum_{i\geq 1}b_iY^iZ+cYZ^2\right)}, \dfrac{fZ^t\left(1+gXY\right)}{aX^r\left(1+\sum_{i\geq 1}b_iY^iZ+cYZ^2\right)},\right.\\
    &\hspace{4em} \left.\dfrac{fZ^t\left(1+gXY\right)}{dY^s\left(1+\sum_{i\geq 1}e_iX^iZ\right)}\right)
    \end{align*}
    We can notice that we can separate $\partial^0$ into different parts in each component, that are the one relative to $\KK^*$, the one relative to the combinatorial units $X^r, Y^s, Z^t$ and finally the one corresponding to the nilpotents. We can explicitly compute the inverses (all computations are done in the appropriate localization, with the appropriate quotients)
    \begin{align*}
    \left(aX^r\left(1+\sum_{i\geq 1}b_iY^iZ+cYZ^2\right)\right)^{-1}&=\dfrac{1}{a}X^{-r}\left(1-\sum_{i\geq 1}b_iY^iZ-cYZ^2\right)\\
    \left(dY^s\left(1+\sum_{i\geq 1}e_iX^iZ\right)\right)^{-1}&=\dfrac{1}{d}Y^{-s}\left(1-\sum_{i\geq 1}e_iX^iZ\right)
    \end{align*}
    so
    \begin{align*}
    \partial^0&\left(aX^r\left(1+\sum_{i\geq 1}b_iY^iZ+cYZ^2\right), dY^s\left(1+\sum_{i\geq 1}e_iX^iZ\right), fZ^t\left(1+gXY\right)\right)\\
    &=\left(\dfrac{d}{a}X^{-r}Y^s\left(1+\sum_{i\geq 1}e_iX^iZ\right)\left(1-\sum_{i\geq 1}b_iY^iZ-cYZ^2\right),\right.\\
    &\hspace{2em} \dfrac{f}{a}X^{-r}Z^t\left(1+gXY\right)\left(1-\sum_{i\geq 1}b_iY^iZ-cYZ^2\right),\\
    &\hspace{4em} \left.\dfrac{f}{d}Y^{-s}Z^t\left(1+gXY\right)\left(1-\sum_{i\geq 1}e_iX^iZ\right)\right)\\
    &=\left(\dfrac{d}{a}X^{-r}Y^s\left(1-\sum_{i\geq 1}b_iY^iZ+\sum_{i\geq 1}e_iX^iZ \right), \dfrac{f}{a}X^{-r}Z^t, \dfrac{f}{d}Y^{-s}Z^t\right).
    \end{align*}
    
    An element of the kernel of $\partial^0$ is given by $a=d=f\in\KK^*$, $b_i=e_i=0$ for all $i$, $r=s=t=0$ and $c$ and $g$ disappeared from the image, so they are free in the kernel
    \[
    \H^0(\O^*_{\KK[M]})\cong\KK^*\oplus\KK[X^{\pm1}]\oplus\KK[Z^{\pm1}].
    \]
    
    The kernel of $\partial^1$ is the whole group in the complex, and we should quotient out the image of $\partial^0$ from it, to obtain the local Picard group.
    
    To understand the image of $\partial^0$, we can look again separately at the components. By looking at $a, d, f$ we can easily see that the quotient of $(\KK^*)^3$ modulo their relations in the image is a $\KK^*$, because $\im(\partial^0\restriction_{\{\KK^*\}})=\{(\alpha, \beta, \gamma)\in(\KK^*)^3\mid \beta=\alpha\gamma\}$, as we already computed in Example~\ref{example:cohomology-xyz=0}.
    For the combinatorial units, we computed in Subsection~\ref{subsection:3-vertices} that $\H^1(i_*\O^*_M)\cong\H^1(\O^*_M)\cong\ZZ^3$. We are left with the last part, the units that come from the nilpotents.
    The local Picard group of this ring will then be
    \[
    \Pic^{\loc}(\KK[M])\cong\H^1(\Spec^\bullet\KK[M], \O^*_{\KK[M]})\cong\KK^*\oplus\ZZ^3\oplus G
    \]
    where, thanks to the fact that $\nil(\KK[M]_{XY})^2=0$, $G$ is the appropriate quotient of $\nil(\KK[M]_{XY})\cong\langle Z\rangle_{\KK[M]_{XY}}$, viewed as an abelian group with the usual sum, modulo the subgroup generated by all the polynomials of the form $-\sum_{i\geq 1}b_iY^iZ+\sum_{i\geq 1}e_iX^iZ$ with $b_i\in\KK[X^{\pm1}]$ and $e_i\in\KK[Y^{\pm1}]$.
    
    The description of this group is more complicated than before, because $\N^2\neq0$, so $1+\N\ncong \N$ as sheaves of abelian groups.
\end{example}

These examples clearly show how, even when we are able to use the powerful \v{C}ech cohomology on a very nice covering, the computation of cohomology groups might be extremely hard in the non reduced case.\\
However, it is now easy to prove the following results about the non vanishing of the local Picard group of the general monomial case.
We use the same hypothesis and notation of Theorem~\ref{theorem:split-non-reduced-monomial}.

\begin{corollary}\label{corollary:non-vanishing}
    If 
    \begin{itemize}
        \item $\Pic^{\loc}(R_{\red})\neq 0$ or
        \item $\Pic^{\loc}(M)\neq 0$ or
        \item $\H^1(\triangle, \KK^*)\neq0$
    \end{itemize}
    then $\Pic^{\loc}(R)\neq 0$.
\end{corollary}
\begin{proof}
    Thanks to Remark~\ref{remark:rewrite-non-reduced} above, we know that
    \begin{align*}
    \Pic^{\loc}(R)&=\H^1(X, \O^*_R)\\
    &=\H^1(X_{\red}, \O^*_{X_{\red}})\oplus \H^1(X, 1+\N)\\
    &=\Pic^{\loc}(R_{\red})\oplus\H^1(X, 1+\N)\\
    &=\H^1(\triangle, \KK^*)\oplus\H^1(\Spec^\bullet M, \O^*_M)\oplus \vH^1(\{D(X_i)\}, 1+\N)\\
    &= \H^1(\triangle, \KK^*) \oplus\Pic^{\loc}(M)\oplus \vH^1(\{D(X_i)\}, 1+\N).\qedhere
    \end{align*}
\end{proof}

\newpage
\section{Further developments}

Although we successfully proved some interesting results, there are still many open questions that can lead to further developments of the techniques and theories here presented. We list here some of them, that are not too far from this work and that can probably be achieved in the near future.

\begin{itemize}
    \item The techniques developed and used when talking about cohomology of the sheaf of units in the monomial case might be easily adapted to the case of toric face rings and to more general mixed situations arising from other torsion-free and cancellative binoids.
    \item On the other hand, completely new techniques might be necessary for addressing the cases in which the binoid is not torsion-free or, even worse, non-cancellative.
    \item Since the binoids that we considered are $\NN$-graded, the results here presented for the punctured spectrum of $\KK[M]$ can give some insights at some combinatorial properties and description of invariants of $\mathrm{Proj}\, \KK[M]$.
    \item The combinatorial topology proved to be useful to answer questions related to the algebraic invariants, but it might be also useful to address topological questions; for example about topological vector bundles and the fundamental group in the case $\KK=\RR$ or $\CC$. To the extent of our knowledge, these objects are not yet defined for binoids, and being able to define them by looking at the combinatorial topology might be a fruitful way to address the problem of studying them combinatorially.
    \item It might be interesting to have an interpretation of the well-known exponential sequence for a binoid, that again might arise through the study of the combinatorial topology and a suitable field.
\end{itemize}

    \printglossary[title={List of Symbols}]
    \addcontentsline{toc}{chapter}{List of Symbols}


    \bibliographystyle{alpha}
    \bibliography{bibliography}
        
\end{document}